\documentclass[twoside]{article}
\usepackage{amsmath,amssymb,amsfonts,amsthm,amsxtra,color,enumerate,datetime,xcolor,}
\usepackage{graphicx}
\usepackage{imakeidx}
\usepackage[normalem]{ulem}
\usepackage{cancel}
\newtheorem{theorem}{Theorem}[section]
\newtheorem*{theoremA*}{Theorem A}

\newtheorem{proposition}[theorem]{Proposition}
\newtheorem{lemma}[theorem]{Lemma}
\newtheorem{corollary}[theorem]{Corollary}
\newtheorem{definition}[theorem]{Definition}

\newtheorem{remark}[theorem]{Remark}
\newtheorem{remarks}[theorem]{Remarks}

\numberwithin{equation}{section}

\theoremstyle{definition}

\makeindex[title=Index of symbols, intoc]

\begin{document}

\def\mp{\marginpar{\textcolor{red}{\qquad XX}}}
\def\bmp{\marginpar{\textcolor{blue}{\qquad XX}}}
\def\brmp{\marginpar{\textcolor{brown}{\qquad XX}}}

\def\a{\mathbf a}
\def\b{\mathbf b}
\def\c{\mathbf c}
\def\d{\mathbf d}
\def\e{\mathbf e}
\def\f{\mathbf f}
\def\g{\mathbf g}
\def\h{\mathbf h}
\def\i{\mathbf i}
\def\j{\mathbf j}
\def\k{\mathbf k}
\def\m{\mathbf m}
\def\n{\mathbf n}
\def\o{\mathbf o}
\def\p{\mathbf p}
\def\q{\mathbf q}
\def\r{\mathbf r}
\def\s{\mathbf s}
\def\t{\mathbf t}
\def\u{\mathbf u}
\def\v{\mathbf v}
\def\w{\mathbf w}
\def\x{\mathbf x}
\def\y{\mathbf y}
\def\z{\mathbf z}

\def\0{{\mathbf 0}}

\def\Abf{\mathbf A}

\def\A{\mathbb A}
\def\B{\mathbb B}
\def\C{\mathbb C}
\def\D{\mathbb D}
\def\E{\mathbb E}
\def\F{\mathbb F}
\def\G{\mathbb G}
\def\H{\mathbb H}
\def\I{\mathbb I}
\def\J{\mathbb J}
\def\K{\mathbb K}
\def\L{\mathbb L}
\def\M{\mathbb M}
\def\N{\mathbb N}
\def\O{\mathbb O}
\def\P{\mathbb P}
\def\Q{\mathbb Q}
\def\R{\mathbb R}
\def\S{\mathbb S}
\def\T{\mathbb T}
\def\U{\mathbb U}
\def\V{\mathbb V}
\def\W{\mathbb W}
\def\X{\mathbb X}
\def\Y{\mathbb Y}
\def\Z{\mathbb Z}

\def\AA{\mathcal A}
\def\BB{\mathcal B}
\def\CC{\mathcal C}
\def\DD{\mathcal D}
\def\EE{\mathcal E}
\def\FF{\mathcal F}
\def\GG{\mathcal G}
\def\HH{\mathcal H}
\def\II{\mathcal I}
\def\JJ{\mathcal J}
\def\KK{\mathcal K}
\def\LL{\mathcal L}
\def\MM{\mathcal M}
\def\NN{\mathcal N}
\def\OO{\mathcal O}
\def\PP{\mathcal P}
\def\QQ{\mathcal Q}
\def\RR{\mathcal R}
\def\SS{\mathcal S}
\def\TT{\mathcal T}
\def\UU{\mathcal U}
\def\VV{\mathcal V}
\def\WW{\mathcal W}
\def\XX{\mathcal X}
\def\YY{\mathcal Y}
\def\ZZ{\mathcal Z}

\def\CZ{Calder\'on-Zygmund }

\def\alphab{\boldsymbol{\alpha}}
\def\betab{\boldsymbol{\beta}}
\def\gammab{\boldsymbol{\gamma}}
\def\deltab{\boldsymbol{\delta}}
\def\epsilonb{\boldsymbol{\epsilon}}
\def\zetab{\boldsymbol{\zeta}}
\def\etab{\boldsymbol{\eta}}
\def\thetab{\boldsymbol{\theta}}
\def\iotab{\boldsymbol{\iota}}

\def\kappab{\boldsymbol{\kappa}}
\def\lambdab{\boldsymbol{\lambda}}
\def\mub{\boldsymbol{\mu}}
\def\nub{\boldsymbol{\nu}}
\def\xib{\boldsymbol{\xi}}

\def\pib{\boldsymbol{\pi}}
\def\rhob{\boldsymbol{\rho}}
\def\sigmab{\boldsymbol{\sigma}}
\def\taub{\boldsymbol{\tau}}
\def\upsilonb{\boldsymbol{\upsilon}}
\def\phib{\boldsymbol{\phi}}
\def\varphib{\boldsymbol{\varphi}}
\def\chib{\boldsymbol{\chi}}
\def\psib{\boldsymbol{\psi}}
\def\omegab{\boldsymbol{\omega}}
\def\varthetab{\boldsymbol{\vartheta}}

\def\dbar{\bar\partial}
\def\bx{\square}
\def \RE {\Re\text{\rm e}}
\def \IM {\Im\text{\rm m}}
\def\|{{|\!|}}
\def\[{{[\![}}
\def\]{{]\!]}}

\def\be{\begin{equation}}
\def\ee{\end{equation}}
\def\bes{\begin{equation*}}
\def\ees{\end{equation*}}
\def\bea{\begin{equation}\begin{aligned}}
\def\eea{\end{aligned}\end{equation}}
\def\beas{\begin{equation*}\begin{aligned}}
\def\eeas{\end{aligned}\end{equation*}}

\def\AAA{\mathbf A}
\def\BBB{\mathbf B}
\def\CCC{\mathbf C}
\def\DDD{\mathbf D}
\def\EEE{\mathbf E}
\def\EEEt{\widetilde{\EEE}}
\def\FFF{\mathbf F}
\def\tI{\tilde I}
\def\xt{\tilde{\x}}
\def\cdott{\,\tilde\cdot\,}


\def \bA {\mathbb A}
\def \bB {\mathbb B}
\def \bC {\mathbb C}
\def \bD {\mathbb D}
\def \bE {\mathbb E}
\def \bF {\mathbb F}
\def \bG {\mathbb G}
\def \bH {\mathbb H}
\def \bI {\mathbb I}
\def \bJ {\mathbb J}
\def \bK {\mathbb K}
\def \bL {\mathbb L}
\def \bM {\mathbb M}
\def \bN {\mathbb N}
\def \bO {\mathbb O}
\def \bP {\mathbb P}
\def \bQ {\mathbb Q}
\def \bR {\mathbb R}
\def \bS {\mathbb S}
\def \bR {\mathbb R}
\def \bR {\mathbb R}
\def \bT {\mathbb T}
\def \bU {\mathbb U}
\def \bV {\mathbb V}
\def \bW {\mathbb W}
\def \bX {\mathbb X}
\def \bY {\mathbb Y}
\def \bZ {\mathbb Z}

\def \cA {\mathcal A}
\def \cB {\mathcal B}
\def \cC {\mathcal C}
\def \cD {\mathcal D}
\def \cE {\mathcal E}
\def \cF {\mathcal F}
\def \cG {\mathcal G}
\def \cH {\mathcal H}
\def \cI {\mathcal I}
\def \cJ {\mathcal J}
\def \cK {\mathcal K}
\def \cL {\mathcal L}
\def \cM {\mathcal M}
\def \cN {\mathcal N}
\def \cO {\mathcal O}
\def \cP {\mathcal P}
\def \cQ {\mathcal Q}
\def \cR {\mathcal R}
\def \cS {\mathcal S}
\def \cR {\mathcal R}
\def \cR {\mathcal R}
\def \cT {\mathcal T}
\def \cU {\mathcal U}
\def \cV {\mathcal V}
\def \cW {\mathcal W}
\def \cX {\mathcal X}
\def \cY {\mathcal Y}
\def \cZ {\mathcal Z}

\def \fa {\mathfrak a}
\def \fb {\mathfrak b}
\def \fc {\mathfrak c}
\def \fd {\mathfrak d}
\def \fe {\mathfrak e}
\def \ff {\mathfrak f}
\def \fg {\mathfrak g}
\def \fh {\mathfrak h}
\def \mfi {\mathfrak i}
\def \fj {\mathfrak j}
\def \fk {\mathfrak k}
\def \fl {\mathfrak l}
\def \fm {\mathfrak m}
\def \fn {\mathfrak n}
\def \fo {\mathfrak o}
\def \fp {\mathfrak p}
\def \fq {\mathfrak q}
\def \fr {\mathfrak r}
\def \fs {\mathfrak s}
\def \ft {\mathfrak t}
\def \fu {\mathfrak u}
\def \fv {\mathfrak v}
\def \fw {\mathfrak w}
\def \fx {\mathfrak x}
\def \fy {\mathfrak y}
\def \fz {\mathfrak z}

\def \fA {\mathfrak A}
\def \fB {\mathfrak B}
\def \fC {\mathfrak C}
\def \fD {\mathfrak D}
\def \fE {\mathfrak E}
\def \fF {\mathfrak F}
\def \fG {\mathfrak G}
\def \fH {\mathfrak H}
\def \fI {\mathfrak I}
\def \fJ {\mathfrak J}
\def \fK {\mathfrak K}
\def \fL {\mathfrak L}
\def \fM {\mathfrak M}
\def \fN {\mathfrak N}
\def \fO {\mathfrak O}
\def \fP {\mathfrak P}
\def \fQ {\mathfrak Q}
\def \fR {\mathfrak R}
\def \fS {\mathfrak S}
\def \fR {\mathfrak R}
\def \fR {\mathfrak R}
\def \fT {\mathfrak T}
\def \fU {\mathfrak U}
\def \fV {\mathfrak V}
\def \fW {\mathfrak W}
\def \fX {\mathfrak X}
\def \fY {\mathfrak Y}
\def \fZ {\mathfrak Z}

\def \Sp {\text{\rm Sp}}
\def \sp {\frak {sp}}
\def \RE {\Re\text{\rm e}\,}
\def \IM {\Im\text{\rm m}\,}
\def \al {\alpha}
\def \la {\lambda}
\def \ph {\varphi}
\def \del {\delta}
\def \eps {\varepsilon}
\def \lan {\langle}
\def \ran {\rangle}
\def \de {\partial}
\def \trans{\,{}^t\!}
\def \half{\frac12}
\def \inv{^{-1}}
\def \rinv#1{^{(#1)}}
\def \supp {\text{\rm supp\,}}
\def \inter {\overset \text{\rm o} \to}
\def \deg {\text{\rm deg\,}}
\def \dim {\text{\rm dim\,}}
\def \span {\text{\rm span\,}}
\def \rad {\text{\rm rad\,}}
\def \tr {\text{\rm tr\,}}

\def\mbh{{\mathbf h}}
\def\mbi{{\mathbf i}}
\def\mbj{{\mathbf j}}
\def\mbk{{\mathbf k}}
\def\mbx{{\mathbf x}}
\def\mby{{\mathbf y}}
\def\mbz{{\mathbf z}}
\def\mbw{{\mathbf w}}
\def\mbu{{\mathbf u}}
\def\mbv{{\mathbf v}}
\def\mbs{{\mathbf s}}
\def\mbt{{\mathbf t}}
\def\mbX{{\mathbf X}}
\def\mb0{{\mathbf 0}}
\def\balpha{{\bar\alpha}}
\def\bxi{{\boldsymbol\xi}}
\def\beeta{{\boldsymbol\eta}}
\def \coB {{}^c\!B}

\def\|{\vert}


\title{Algebras of singular integral operators\\with kernels controlled by multiple norms}

\pagestyle{myheadings}

\author{Alexander Nagel\thanks{University of Wisconsin-Madison, Madison WI 53706, \texttt{nagel@math.wisc.edu}} \and Fulvio Ricci\thanks{Scuola Normale Superiore, Piazza dei Cavalieri 7, 56126 Pisa, \texttt{fricci@sns.it}} \and Elias M.~Stein \thanks{Princeton University, Princeton NJ 08544. \texttt{stein@math.princeton.edu}}\and Stephen Wainger \thanks{University of Wisconsin-Madison, Madison WI 53706, \texttt{wainger@math.wisc.edu}}}

%
%
%
%

\date{}
\maketitle
\thispagestyle{empty}

\setcounter{tocdepth}{2}
\tableofcontents

\section{Introduction}

The purpose of this paper is to study algebras of singular integral operators on $\R^{n}$ and nilpotent Lie groups that arise when one considers the composition of  Calder\'on-Zygmund operators with different homogeneities, such as operators that occur in sub-elliptic problems and those arising in elliptic problems.  For example, one would like to describe the  algebras containing the operators related to the Kohn-Laplacian for appropriate domains, or those related to inverses of H\"ormander sub-Laplacians, when these are composed with the more standard class of pseudo-differential operators. The algebras we study can be characterized in a number of different but equivalent ways, and consist of operators that are pseudo-local and bounded on $L^{p}$ for $1<p<\infty$. While the usual class of \CZ operators is invariant under a one-parameter family of dilations, the operators we study fall outside this class, and reflect a multi-parameter structure.

\smallskip
This paper is the second in the series begun with \cite{MR2949616}.

\medskip

\subsection{Background}\quad

\medskip
 
An initial impetus for the study of composition of singular integral operators of different types came from the study of the $\overline\partial$-Neumann problem in complex analysis. In the case of domains in $\C^{n}$ where matters are sufficiently well understood, the relevant ``Calder\'on operator'' for this boundary-value problem can be viewed as a composition of a sub-elliptic type operator with a standard pseudo-differential operator of order 0. Early studies of such compositions in the context of the Heisenberg group can be found \cite{MR0499319} and in \cite{MR648484}. In \cite{MR1312498} more general such operators appear as singular integrals with ``flag kernels''. The corresponding algebra was broad enough to contain, for example, all operators of the form $m(\LL,iT)$, where $\LL$ is the sub-Laplacian and $T$ is the central invariant vector field, with $m$ a multiplier of Marcinkiewicz-type. The notion of flag kernels (having singularities on appropriate flag varieties) and the properties of the corresponding singular integrals were then extended to the higher step case in \cite{MR1818111}, largely in the Euclidean setting, and then in \cite{MR2949616} in the context of automorphic flags on a general homogeneous group.  In this connection we should mention the general theory of Brian Street \cite{MR3241740} which treats aspects of both singular integrals and singular Radon transforms in the context of multi-parameter analysis.

However, the general nature of these results did not provide an entirely satisfactory answer to a main question that arises when composing singular integrals of different homogeneities: that of characterizing the resulting class of kernels. It turns out that the class in question is indeed narrower than the class of flag kernels: the kernels satisfy stricter differential inequalities, and in particular they are all smooth away from the origin, and the corresponding operators are thus pseudo-local.

A different approach to the study of flag kernels can be found in the work of G{\l}owacki; see \cite{MR2602167}, \cite{MR2679042}, \cite{MR3185206}. There has been recent interest in the study of Hardy spaces associated to flags and flag kernels (see for example \cite{MR2648281}, \cite{MR2789471}, \cite{MR2810164}, \cite{MR2984061}, \cite{MR3293443},  \cite{MR3192289}) and in the study of weighted norm inequalities (see for example \cite{MR3188508}).  Further recent references that deal with flag kernels include   \cite{MR2499336}, \cite{MR2591640} \cite{MR2607286}, \cite{MR2923798}, \cite{MR3001008}, \cite{MR3228630},   \cite{MR3188508},     and \cite{MR3289035}.

\subsection{Some motivating examples}\label{SME}\quad

\medskip
 
We begin by describing a particularly simple situation occurring in studying the heat equation on $\R^{n}\times\R$, or in studying convolution of operators arising on the non-Abelian  Heisenberg group $\H^{n}=\{(\z,t):\z=(z_{1}, \ldots, z_{n})\in \C^{n},\,t\in \R\}$. In the latter case, Calder\'on-Zygmund kernels adapted to the automorphic dilations $\delta\cdot(\z,t)=(\delta \z,\delta^{2}t)$ are distributions $\KK$ on $\H^{n}$ which, away from the origin, are given by integration against a smooth function $K$ satisfying
\be \label{1.1aa}
\left|\partial^{\beta}_{t}\partial^{\alpha}_{z,\bar z}K(\z,t)\right|\lesssim \big(|\z|+|t|^{\frac{1}{2}}\big)^{-2n-2-|\alpha|-2\beta}
\ee
and which satisfy appropriate cancellation conditions. Similarly the isotropic variants of these kernels (the building blocks of the standard pseudodifferential operators) satisfy instead the inequalities
\be \label{1.2aa}
\left|\partial^{\beta}_{t}\partial^{\alpha}_{z,\bar z}K(\z,t)\right|\lesssim \left(|\z|+|t|\right)^{-2n-1-|\alpha|-\beta}.
\ee

If we consider kernels with compact support, it turns out that the relevant algebra of kernels containing both kinds of distributions are those given by the mixed differential inequalities
\be \label{1.3aa}
\left|\partial^{\beta}_{t}\partial^{\alpha}_{z,\bar z}K(\z,t)\right|\lesssim \left(|\z|+|t|\right)^{-2n-|\alpha|}(|\z|^{2}+|t|)^{-1-\beta}
\ee 
for $(\z,t)$ in the unit ball, in addition to cancellation conditions.

\medskip

There are several ways of thinking about the estimates in equation (\ref{1.3aa}). One notices that the inequalities in (\ref{1.3aa}) are the best that can be satisfied on the unit ball for kernels $K$ that are assumed to be either of type (\ref{1.1aa}) or of type (\ref{1.2aa}), but this rather simple situation is not repeated in the higher step case\footnote{See Subsection \ref{2-Flag} below for an interesting example, and Section \ref{Integrability} for a more complete discussion}. However, one can make two more productive observations.
\begin{enumerate}[a)]

\smallskip

\item Operators satisfying (\ref{1.3aa}) can be understood in terms of the theory of \textit{flag kernels} studied in \cite{MR1818111} and \cite{MR2949616}.  Consider the following two sets of differential inequalities:
\begin{equation}\label{1.4aa}
\Big\vert\partial^{\beta}_{t}\partial^{\alpha}_{z,\bar z}K(\z,t)\Big\vert \lesssim |\z|^{-2n-|\alpha|}\big(|\z|^{2}+|t|\big)^{-1-\beta},
\end{equation}
\begin{equation}\label{1.5aa}
\Big\vert\partial^{\beta}_{t}\partial^{\alpha}_{z,\bar z}K(\z,t)\Big\vert \lesssim \big(|\z|+|t|\big)^{-2n-|\alpha|}|t|^{-1-\beta}.
\end{equation}
Equation (\ref{1.4aa}) gives the differential inequalities satisfied by flag kernels for the flag $(0)\subset \C^{n}\subset \C^{n}\oplus\R$, while equation (\ref{1.5aa}) give the differential inequalities for the opposite flag  $(0)\subset \R\subset \C^{n}\oplus\R$. Operators of type (\ref{1.1aa}) satisfy (\ref{1.4aa}), operators of type (\ref{1.2aa}) satisfy (\ref{1.5aa}), and locally (for $(\z,t)$ in compact sets), operators of type (\ref{1.1aa}) or (\ref{1.2aa}) satisfy \textit{both} (\ref{1.4aa}) and (\ref{1.5aa}). Moreover, we will see that a kernel satisfying both (\ref{1.4aa}) and (\ref{1.5aa}) automatically satisfies (\ref{1.3aa}). Thus one is led to the study of \textit{two-flag kernels} which locally are simultaneously flag kernels for two opposite flags.

\smallskip

\item A second related perspective is to view kernels $K$ satisfying the conditions of (\ref{1.3aa}) as satisfying differential inequalities appropriate to two different families of dilations on $\R^{N}$. Derivatives in $z$ or $\bar z$ are controlled by the dilations $(\z,t)\to (\delta \z,\delta t)$, while derivatives with respect to $t$ are controlled by the dilations $(\z,t) \to (\delta \z,\delta^{2}t)$. This suggests that, more generally, we should study operators whose kernels satisfy differential inequalities in which different derivatives are controlled by different families of dilations.
\end{enumerate}

\smallskip

\noindent In this paper we begin by adopting the second point of view. It will then turn out that the resulting class of distributions includes the class which belong to two flags.  One of the main goals of this paper is the study the composition of convolution operators with kernels of this form. In the Euclidean context this reduces to the study of the product of the Fourier multipliers, but in the non-Abelian situation the composition becomes a more serious issue.

\medskip

A second goal of this paper is to establish regularity of two-flag kernels. The inequalities in equations (\ref{1.4aa}) and (\ref{1.5aa}) are the differential inequalities for a kernel belonging simultaneously to two opposite $2$-step flags, and it is easy to see that these imply the inequalities in equation (\ref{1.3aa}). However in the case of higher step, the results are not so simple, and in fact are rather surprising. Consider, for example, the case of two opposite $3$-step flags in $\R^{3}$:
\beas
\FF:&& &(0) \subset \{(x,y,z):y=z=0\}\subset \{(x,y,z):z=0\}\subset \R^{3}, \\
\FF^{\perp}:&& &(0) \subset \{(x,y,z):x=y=0\}\subset \{(x,y,z):x=0\}\subset \R^{3}.
\eeas
A flag kernel $\KK$ for the flag $\FF$ with homogeneity $\lambda\cdot(x,y,z)=(\lambda x, \lambda y, \lambda z)$ satisfies the differential inequalities
\bea\label{1.11iou}
\big\vert\partial^{a}_{x}\partial^{b}_{y}\partial^{c}_{z}K(x,y,z)\big\vert\leq C_{a,b,c}\,|x|^{-1-a}(|x|+|y|)^{-1-b}(|x|+|y|+|z|)^{-1-c}.
\eea
A flag kernel $\KK$ for the flag $\FF^{\perp}$ with homogeneity $\lambda\cdot(x,y,z)=(\lambda x, \lambda^\frac12 y, \lambda^\frac13 z)$ satisfies the differential inequalities
\bea\label{1.12iou}
\big\vert\partial^{a}_{x}\partial^{b}_{y}\partial^{c}_{z}K(x,y,z)\big\vert\leq C_{a,b,c}\,(|x|+|y|^2+|z|^3)^{-1-a}(|y|+|z|^\frac32)^{-1-b}|z|^{-1-c}.
\eea
Suppose that $\KK$ is a distribution with compact support which is a flag kernel for both flags. Looking only at the inequalities in equations (\ref{1.11iou}) and (\ref{1.12iou}), these estimates provide no information about $K(x,y,z)$ if $x=z=0$ and $y\neq 0$;   in terms of the assumed size estimates alone, the distribution $\KK$ could be singular away from the origin. Unexpectedly, it turns out that
\beas
\big\vert\partial^{a}_{x}\partial^{b}_{y}\partial^{c}_{z}K(x,y,z)\big\vert\leq C_{a,b,c}\,(|x|+|y|^2+|z|^3)^{-1-a}(|x|+|y|+|z|^\frac32)^{-1-b}(|x|+|y|+|z|)^{-1-c},
\eeas
and the distribution $\KK$ is indeed pseudo-local. It is important to note that this is a consequence of both the differential inequalities and the cancellation conditions.

\smallskip

\subsection{Plan of the paper}\quad

\medskip

Our results can be divided roughly into two kinds: those concerning properties of the kernels and those concerning the resulting convolution operators.

\begin{enumerate}[1.]

\smallskip

\item We begin by studying a class $\PP(\EEE)$ of distributions on $\R^{N}$ given away from the origin by integration against a smooth function. These functions are required to satisfy differential inequalities adapted to a set of $n\leq N$ families of dilations on $\R^{N}$ prescribed by an $n\times n$ matrix $\EEE$. Distributions $\KK\in \PP(\EEE)$ must also satisfy appropriate cancellation conditions. 

\smallskip

\item We are primarily interested in the local behavior of these distributions\footnote{We will see below in Section \ref{Integrability} that if the rank of the matrix $\EEE$ is greater than $1$, then the kernels $\KK\in \PP(\EEE)$ are integrable at infinity.}, and it will be convenient to modify the class of kernels outside the unit ball so that they and all their derivatives are rapidly
decreasing at infinity. Under these hypotheses, we shall denote the class of modified kernels  by $\PP_{0}(\EEE)$.

\smallskip

\item We then characterize these distributions in terms of their Fourier transform, and also in terms of decompositions as infinite dyadic sums of dilates of normalized bump functions.  This material is discussed in Sections \ref{Kernels} - \ref{Decompositions}. In  Section \ref{Integrability} we show that these distributions fall outside the class of standard \CZ kernels. In Section \ref{CZKernels} we find the smallest class $\PP_0(\EEE)$ containing the convolution of Calder\'on-Zygmund kernels with different homogeneities. In Section \ref{Multi-Flag} we show that two-flag kernels provide examples of the classes $\PP(\EEE)$.

\smallskip

\label{proper}

\medskip

After establishing the basic properties of the distributions $\KK\in \PP_{0}(\EEE)$, we fix a homogeneous nilpotent Lie group $G$ whose underlying space is $\R^{N}$ and thus has a given automorphic family of dilations. (Of course this includes the special case in which the group $G$ is the abelian $\R^{n}$.)

\medskip

\item 

If the $n$ families of dilations on $\R^{N}$ given by the matrix $\EEE$ are compatible\footnote{A precise definition is given in Subsection \ref{Compatibility} below.} with the automorphic dilations, we show that convolution on $G$ with a distribution $\KK\in \PP(\EEE)$ defines a bounded operator on $L^{p}(G)$ for $1<p<\infty$, and that the collection of these convolution operators form an algebra under composition. This material is discussed in Sections \ref{Groups} - \ref{Composition}.

\smallskip

\item

We also characterize the smallest algebra of convolution operators which arise when one composes \CZ operators of different homogeneities. This material appears in Section \ref{CZKernels}.

\smallskip

\item

Finally we describe a generalization of the class $\PP(\EEE)$ to allow variable coefficients, and we investigate the commutation properties of these operators. In particular, we study the role of the classical pseudo-differential operators. This material is considered in Sections \ref{6:kernels} and \ref{7:pseudo}.
\end{enumerate}
In the following subsections we provide a more detailed summary of the results in this paper.

\subsubsection{Section \ref{Kernels}: The class of kernels}\quad
\medskip

In Section \ref{Kernels} we provide a precise definition of a class of distributions and their associated multipliers. We start with a fixed decomposition $\R^{N}=\R^{C_{1}}\oplus\cdots \oplus\R^{C_{n}}$. Write $\x\in\R^{N}$ as $(\x_{1}, \ldots, \x_{n})$ with  $\x_{j}=\big(x_{k}:k\in C_{j}\big)\in \R^{C_{j}}$. (Throughout this paper we use boldface, such as $\x_{j}$, to indicate a tuple of coordinates, and standard type, such as $x_{k}$, to indicate a single coordinate.) Each component $\R^{C_{j}}$ is equipped with a family of dilations, denoted by $\lambda\cdot \x_{j}$ for $\lambda >0$.   For simplicity of exposition in this Introduction, we assume that these are the standard isotropic dilations $\lambda\cdot \x_{j}= \big(\lambda x_{k}:k\in C_{j}\big)$ although we will allow more general non-isotropic dilations in Section \ref{Kernels} below. If $Q_{j}$ is the homogeneous dimension of $\R^{C_{j}}$, then for isotropic dilations, $Q_{j}$ is the dimension of $\R^{C_{j}}$.\footnote{If $\lambda\cdot \x_{j}=\exp[A_{j}\log\lambda](\x_{j})$ the homogeneous dimension of $\R^{C_{j}}$ is the trace of $A_{j}$.} Let $n_{j}$ be a smooth homogeneous norm on $\R^{C_{j}}$ so that $n_{j}(\lambda\cdot\x_{j}) = \lambda\,n_{j}(\x_{j})$. Thus for isotropic dilations, we can take $n_{j}$ to be the standard Euclidean norm on $\R^{C_{j}}$.

The homogeneities on each component $\R^{C_{j}}$ can be weighted in various ways to give a global family of dilations on $\R^{N}$. Let $\EEE=\{e(j,k)\}$ be an $n\times n$ matrix with strictly positive entries, and for $1 \leq j \leq n$, define a family of dilations on $\R^{N}$ by setting
\be\label{1.6aa}
\delta_{j}(\lambda)[\x]= \big(\lambda^{1/e(j,1)}\cdot\x_{1}, \ldots ,\lambda^{1/e(j,j)}\cdot \x_{j}, \ldots, \lambda^{1/e(j,n)}\cdot\x_{n}\big).
\ee
Put
\bea\label{1.7aa}
N_{j}(\x) = n_{1}(\x_{1})^{e(j,1)}+ \cdots + n_{j}(\x_{j})^{e(j,j)}+ \cdots n_{n}(\x_{n})^{e(j,n)}
\eea
so that $N_{j}\big(\delta_{j}(\lambda)[\x]\big) = \lambda N_{j}(\x)$. Thus $N_{1}, \ldots, N_{n}$ are homogeneous norms\footnote{This is standard terminology, although we only have $N_{j}(\x+\y)\leq A_{j}[N_{j}(\x)+N_{j}(\y)]$ for some constant $A_{j}$ instead of the usual triangle inequality.} on $\R^{N}$  for the different dilation structures. We require that the components of the matrix $\EEE$ satisfy
\bea\label{1.8aa}
e(j,j) &=1 & 1&\leq j \leq n,\\
e(j,k) & \leq e(j,l)e(l,k) & 1&\leq j,k,l\leq n.
\eea
We study a family $\PP=\PP(\EEE)$ of distributions $\KK$ on $\R^{N}$ which are given away from the origin $0\in \R^{N}$ by integration against a smooth function $K$ satisfying the following differential inequalities\footnote{We use standard multi-index notation; see Section \ref{Kernels} for more details.}:
\be\label{1.9aa}
\left|\partial^{\alphab_{1}}_{\x_{1}}\cdots \partial^{\alphab_{n}}_{\x_{n}}K(\x)\right|\lesssim \prod_{j=1}^{n}N_{j}(\x)^{-Q_{j}-|\alpha_{j}|}.
\ee
(Note that derivatives of $K$ with respect to any of the variables in $\x_{j}$ are controlled by the family of dilations $\delta_{j}(\lambda)$ and the homogeneous norm $N_{j}$.) In addition, we impose certain cancellation conditions on the distributions $\KK$. These are analogous to cancellation conditions for product kernels or flag kernels. Roughly speaking, they require that if $\x=(\x',\x'')$ is a decomposition of the coordinates  $\x_{1}, \ldots, \x_{n}$ into two subsets, then $\KK(\x')=\int \KK(\x',\x'')\psi(\x'')\,d\x''$ is a distribution in the $\x''$ variables of the same type. The precise conditions are specified in part (\ref{Def2.2B}) of Definition \ref{Def2.2} below.

\medskip
Note that in the example of the Heisenberg group given above, we have the natural decomposition $\H^{n}= \C^{n}\oplus \R\cong \R^{2n}\oplus \R$. The isotropic dilations on each component are given by $\lambda\cdot (z_{1}, \ldots, z_{n}) = (\lambda z_{1}, \ldots, \lambda z_{n})$ and $\lambda\cdot t= \lambda t$, so that the homogeneous dimension of $\C^{n}$ is $2n$ and the homogeneous dimension of $\R$ is $1$. Set $n_{1}(\z) = |\z|=\big[\sum_{j=1}^{n}|z_{j}|^{2}\big]^{\frac{1}{2}}$ and $n_{2}(t) = |t|$, so that $n_{1}(\lambda\cdot \z)= \lambda n_{1}(\z)$ and $n_{2}(\lambda\cdot t) = \lambda n_{2}(t)$. If we take $\EEE = \left[\begin{matrix}1&1\\2&1\end{matrix}\right]$ then
\bea\label{1.10ty}
N_{1}(\z,t) &= n_{1}(\z) +n_{2}(t)\,\,\,\approx \,\,\,|\z| +|t|,\\
N_{2}(\z,t) &= n_{1}(\z)^{2} +n_{2}(t)\,\,\,\approx |\z|^{2}+|t|,
\eea
and the corresponding dilations are given by
\beas
\delta_{1}(\lambda)(\z,t) &= (\lambda \z, \lambda t),\\
\delta_{2}(\lambda)(\z,t)&= (\lambda^{\frac{1}{2}} \z, \lambda t).
\eeas
The differential inequalities required by (\ref{1.9aa}) are then
\beas
\big\vert\partial^{\alpha}_{z,\bar z}\partial^{\beta}_{t}K(\z,t)\big\vert \lesssim N_{1}(\z,t)^{-2n-|\alpha|}N_{2}(\z,t)^{-1-\beta}\approx (|\z| +|t|)^{-2n-|\alpha|}(|\z|^{2}+|t|)^{-1-\beta}
\eeas
which is in agreement with (\ref{1.3aa}).

\subsubsection{Section \ref{Partitions}: Marked partitions}\quad

\medskip

Our analysis of the distributions $\KK\in \PP_{0}(\EEE)$ is based on the decomposition of the unit ball into the regions where one summand in each norm $N_{j}(\x) = n_{1}(\x_{1})^{e(j,1)}+ \cdots + n_{n}(\x_{n})^{e(j,n)}$ is strictly larger than all the others. (There is an analogous decomposition of the Fourier transform $\xib$-space for $\xib$ large, which is needed to understand the differential inequalities of the multiplier $m=\widehat {\KK}$.) Matters are simplest in the case when $n=2$ exemplified by the Heisenberg group, where the two norms are given in equation (\ref{1.10ty}). In this case there are three regions near the origin.
\begin{enumerate}[1.]
\item The set $|\z|^{2}\gtrsim |t|\gtrsim |t|^{2}$, where the isotropic dilations are controlling. Here $N_{1}(\z,t)\approx |\z|\approx |\z|+|t|$ and $N_{2}(z,t)\approx |\z|^{2}\approx (|\z|+|t|)^{2}$. 
\item The set $|\z|\lesssim |t|\lesssim |t|^{\frac{1}{2}}$, where the automorphic dilations are controlling. Here $N_{1}(\z,t)\approx |t|\approx |\z|^{2}+|t|$ and $N_{2}(\z,t)\approx |t|\approx |\z|^{2}+|t|$. 
\item The intermediate region where $|t|\lesssim |\z|\lesssim |t|^{\frac{1}{2}}$ where kernels of the type (\ref{1.3aa}) behave like ``product kernels''. Here $ N_{1}(\z,t)\approx |\z|$ and $N_{2}(\z,t)\approx |t|$. 
\end{enumerate}

The case $n\geq 3$ is more intricate. The required systematic decomposition of the $\x$-space depends on the notion of a `marked partition' of the set $\{1, \ldots, n\}$, which is a collection of disjoint non-empty subsets $I_{1}, \ldots, I_{s}$ of $\{1, \ldots, n\}$ with $\bigcup_{r=1}^{s}I_{r}=\{1, \ldots, n\}$, together with a `marked' element $k_{r}$ in each subset $I_{r}$. This concept arises as follows. 

Suppose the matrix $\EEE$ satisfies the conditions given in equation (\ref{1.8aa}), and let $\x$ be a point in the unit ball such that no two terms in any norm $N_{j}(\x)$ are equal. We will see in Proposition \ref{Prop5.3} below that if $n_{k}(\x_{k})$ is the dominant term in $N_{j}(\x)$ for some $j\neq k$ then $n_{k}(\x)$ is also the dominant term in $N_{k}(\x)$. Thus given such a point $\x$, there is a well-defined set $\{k_{1}, \ldots, k_{s}\}$ of indices such that $n_{k_{r}}(\x_{k_{r}})$ is the dominant term in $N_{k_{r}}(\x)$. For $1 \leq r\leq s$, we let $I_{r}$ be the set of indices $j$ such that $n_{k_{r}}(\x_{k_{r}})$ is dominant in $N_{j}(\x)$. Then $k_{r}\in I_{r}$ and this produces a marked partition. The details are given below in Section
\ref{Partitions}.

One may note that the decomposition given via marked partitions is somewhat akin to what happens in the theory of ``resolutions of singularities'', here applied to (say) a smooth version of a power of the function $F(\x) = \prod_{j=1}^{n}N_{j}(\x)$. Using that theory (see e.g. \cite{MR3117305}) one would be led to decompose a neighborhood of the origin into a number of parts, and in each part make an appropriate change of variables so that $F$ becomes essentially a monomial in the new variables. However, in our situation the decomposition by marked partitions is much more explicit and tractable, and no auxiliary change of variables is required.

\subsubsection{Section \ref{Duality}: Characterization by Fourier transforms} \quad

\medskip

For $\xib\in \R^{N}=\R^{C_{1}}\oplus\cdots \oplus\R^{C_{n}}$ write $\xib=(\xib_{1}, \ldots,\xib_{n})$ with $\xib_{k}\in\R^{C_{k}}$ and set
\bea\label{1.12hj}
\widehat N_{j}(\xib) =|\xib_{1}|^{1/e(1,j)}+\cdots +|\xib_{j}|+\cdots + |\xib_{n}|^{1/e(j,n)}.\eea
This is the `dual norm' to $N_{j}(\x)$. Denote by $\MM_{\infty}({\EEE})$ the class of multipliers $m\in \CC^{\infty}(\R^{N})$ satisfying
\bea\label{1.11hj}
\left|\partial^{\alphab_{1}}_{\xib_{1}}\cdots \partial^{\alphab_{n}}_{\xib_{n}}m(\xib)\right|\lesssim \prod_{j=1}^{n}\left[1+\widehat N_{j}(\xib)\right]^{-|\alphab_{j}|}.
\eea
On page \pageref{proper} we indtroduced the subclass $\PP_{0}(\E)$ of proper distributions. In Section \ref{Duality} we show that $\KK\in \PP_{0}(\EEE)$ if and only if its Fourier transform $m=\widehat\KK\in \MM_{\infty}({\EEE})$.

\subsubsection{Sections \ref{Schwartz sums} and \ref{Decompositions}: Characterization by dyadic decompositions} \quad

\medskip

In Section \ref{Schwartz sums} we show that kernels $\KK\in \PP_{0}(\EEE)$ can be characterized by their dyadic decomposition into dilates of \emph{bump functions}. Let $\{\varphi^{I}\}$ be a uniformly bounded family of $\CC^{\infty}$-functions supported in the unit ball, where $I$ runs over the set of all $n$-tuples $I=(i_{1}, \ldots, i_{n})\in\Z^{n}$. Let 
\be\label{1.15iou}
[\varphi^{I}]_{I}(\x) = 2^{-\sum_{j=1}^{n}Q_{j}i_{j}}\varphi(2^{-i_{1}}\x_{1}, \ldots, 2^{-i_{n}}\x_{n})
\ee
be the $L^{1}$-invariant dilation of $\varphi^{I}$, and let 
\bea\label{1.16iou}
\Gamma_{\Z}(\EEE)&=\Big\{I=(i_{1}, \ldots, i_{n})\in \Z^{n}: e(j,k)i_{k}\leq i_{j}< 0\Big\}.
\eea
If
\bea \label{1.17iou}
\int_{\R^{C_{j}}}\varphi^{I}(\x_{1}, \ldots, \x_{j}, \ldots, \x_{n})\,d\x_{j}=0
\eea
for each $I\in \Gamma_{\Z}(\EEE)$ and for all $1 \leq j \leq n$, then 
the infinite series $\sum_{I\in \Gamma_{\Z}(\EEE)}[\varphi^{I}]_{I}$ converges in the sense of distributions to a distribution $\KK\in \PP_{0}(\EEE)$. More generally, we show that the `strong' cancellation condition given in (\ref{1.17iou}) can be replaced by a variant of the \emph{weak} cancellation condition that was introduced in \cite{MR2949616}. Roughly speaking\footnote{The precise statement is given in Definition \ref{Def4.4}.}, we require that if $\varphi^{I}$ does not have cancellation in a variable $\x_{j}$, then there is a compensating gain in the size of $\varphi^{I}$. 

This more general result is then used in our study of convolutions of distributions in Section \ref{Composition}. In this study we also need a converse statement:  if $\KK\in\PP_{0}(\EEE)$, it can be written, modulo distributions corresponding to coarser decompositions of $\R^{N}$, as a sum $\sum_{I\in \Gamma_{\Z}(\EEE)}[\varphi^{I}]_{I}$. This is the subject of Section \ref{Decompositions}.

\subsubsection{Section \ref{Integrability}: Significance of the rank of $\EEE$}\quad

\medskip

In Section \ref{Integrability} we show that the class of classical \CZ kernels corresponds exactly to the case in which the rank of the matrix $\EEE$ is equal to $1$. If the rank is greater than $1$, distributions $\KK\in \PP(\EEE)$ differ from \CZ kernels in two ways:

\begin{enumerate}[1)]

\smallskip

\item The distribution $\KK$ is integrable at infinity, and so we can write $\KK=\KK_{0}+\KK_{\infty}$ with $\KK_{0}, \KK_{\infty}\in \PP(\EEE)$, where $\KK_{0}$ has support in the unit ball, and where $\KK_{\infty}$ is given by integration against a function $K_{\infty}\in \CC^{\infty}(\R^{N})\cap L^{1}(\R^{N})$. Thus for rank$(\EEE)\geq 2$, we focus on the local behavior of distributions $\KK\in \PP(\EEE)$. We generally restrict our attention to the class of proper distributions $\PP_{0}(\EEE)\subset \PP(\EEE)$ such that all derivatives are rapidly decreasing at infinity.

\smallskip

\item If $\KK$ is a \CZ kernel, then $\big\vert\big\{\x\in \R^{N}:|K(\x)|>\lambda\big\}\big\vert\lesssim \lambda^{-1}$. If $\KK\in \PP(\EEE)$ and the rank of $\EEE$ is greater than $1$, such weak type estimates may no longer be true. We show, for example, that if rank$(\EEE)=n$ there exist distributions $\KK\in \PP(\EEE)$ such that $\big\vert\big\{\x\in \R^{N}:|K(\x)|>\lambda\big\}\big\vert\approx \lambda^{-1}\big[\log \lambda\big]^{n-1}$.

\smallskip

\item In general, there is a notion of `reduced rank' which determines optimal bounds for  estimates of the sort considered in 2).  This is discussed in Section \ref{Weak type}.

\end{enumerate}

\subsubsection{Sections \ref{Groups} and \ref{Composition}: Convolution operators on a Lie group}\quad

\medskip

In Sections \ref{Groups} and \ref{Composition} we turn to the second part of our analysis dealing with convolution operators $T_{\KK}f=f*\KK$ on a homogeneous nilpotent Lie group $G\cong \R^{N}$. Our two main results are the following:
\begin{enumerate}[(A)]
\smallskip

\item If $\KK\in \PP_{0}(\EEE)$ then  each operator $T_{\KK}$ extends to a bounded operator on $L^{p}(G)$ for $1<p<\infty$.

\smallskip

\item If $\KK, \LL\in \PP_{0}(\EEE)$ then there exists $\MM\in \PP_{0}(\EEE)$ such that $T_{\LL}\circ T_{\KK}=T_{\MM}$. Formally $\MM= \KK*\LL$, although the convolution of tempered distributions is not always defined. Thus the space $\PP_{0}(\EEE)$ is an algebra under composition (or convolution).
\end{enumerate}

\smallskip

The first result follows rather easily from the observation that every distribution $\KK\in\PP_{0}(\EEE)$ is a flag kernel, since it was established in \cite{MR2949616} that convolution with a flag kernel extends to a bounded operator on $L^{p}$. The second result is analogous to the result in the same paper that the space of kernels is closed under convolution. However this case is more difficult and requires additional ideas. 

We remark that both results are quite simple for Euclidean convolution, since $\widehat{T_{\KK}f}(\xib) = m(\xib)\widehat f(\xib)$ where $m$ is the Fourier transform of $\KK$ and hence $m\in \MM_{\infty}(\EEE)$. It is not hard to check that  a multiplier $m\in \MM_{\infty}(\EEE)$ satisfies the conditions of the Marcinkiewicz multiplier theorem, and the $L^{p}$-boundedness of the operator $T_{\KK}$ follows immediately. Since the product of two such multipliers is again a multiplier of the same class, the result on composition of operators also follows easily in this case.

To prove the results on a more general group, we rely very heavily on the decomposition of distributions $\KK\in \PP(\EEE)$ as dyadic sums of dilates of bump functions. The heart of the argument is then a careful analysis of sums of the form $\sum_{I,J}[\varphi^{I}]_{I}*[\psi^{J}]_{J}$.

\subsubsection{Section \ref{CZKernels}: Convolution of \CZ kernels}\quad

\medskip

In Section \ref{CZKernels} we  study the convolution of  \CZ kernels $\KK_{1}, \ldots, \KK_{m}$ with compact support  having different homogeneities. If the homogeneities are suitably adapted to the decomposition on $\R^{N}$ we  show that there is an $s\times s$ matrix $\EEE$ so that the algebra $\PP_{0}(\EEE)$ contains each $\KK_{j}$. We also establish a kind of converse result: If $\EEE$ is an $n\times n$ matrix satisfying the basic hypothesis given in (\ref{1.8aa}), then there are finitely many different homogeneities on $\R^{N}$ so that $\PP_{0}(\EEE)$ is the minimal algebra of the above type containing compactly supported \CZ kernels of each homogeneity.

\subsubsection{Section \ref{Multi-Flag}: Regularity of two-flag kernels}\label{2-Flag}\quad

\medskip

In Section \ref{Multi-Flag} we return to the study of kernels belonging to two opposite flags, and show that these belong to $\PP_{0}(\EEE)$ for an appropriate matrix $\EEE$. In the example studied in section \ref{SME}, where the kernel satisfied the two estimates
\beas
\big\vert\partial^{a}_{x}\partial^{b}_{y}\partial^{c}_{z}K(x,y,z)\big\vert\leq C_{a,b,c}
\begin{cases}
|x|^{-1-a}(|x|^{2}+|y|)^{-1-b}(|x|^{3}+|y|^{3/2}+|z|)^{-1-c}\\\\
(|x|+|y|^2+|z|^3)^{-1-a}(|y|+|z|^\frac32)^{-1-b}|z|^{-1-c}
\end{cases},
\eeas
it turns out that $\KK\in \PP_{0}(\EEE)$ where 
\beas
\EEE&=\left[\begin{matrix}
1&1&1\\2&1&1\\3&\frac{3}{2}&1
\end{matrix}\right]&&\text{and} &&\begin{cases}
N_{1}(x,y,z) &= |x|+|y|+|z|,\\
N_{2}(x,y,z) &= |x|^{2}+|y|+|z|,\\
N_{3}(x,y,z) &= |x|^{3}+|y|^{3/2}+|z|
\end{cases}.
\eeas

\subsubsection{Sections \ref{6:kernels} and \ref{7:pseudo}: Operators with variable coefficients}\quad

\medskip

In Sections \ref{6:kernels} and \ref{7:pseudo} we describe a wider class of operators with ``variable coefficients''. An operator $T$ belongs to this extended class if it is of the form
\be
T[f](\x) = \int_{G}K(\x,\z)f(\x\cdot\z^{-1})\,d\z
\ee
where for each $\x$, the distribution $K(\x,\cdot)$ is a kernel of the type we have been considering that depends in an appropriately smooth way on $\x$. One can then assert:
\begin{enumerate}[a)]

\smallskip

\item Operators of that type form an algebra, and each is bounded on $L^{p}(G)$ for $1 < p < \infty$.

\smallskip

\item The standard pseudo-differential operators of order $0$ belong to this algebra, and the sub-algebra of these operators is ``central'' in the sense that an operator of the extended algebra commutes with a standard pseudo-differential operator of order $0$, up to an error which is an appropriate ``smoothing operator." 
\end{enumerate}

\section{The Classes $\PP(\EEE)$ and $\MM({\EEE})$}\label{Kernels}

In this section we define a class of distributions $\PP(\EEE)$ and the corresponding class of Fourier multipliers $\MM(\EEE)$ on the space $\R^{N}$ and its dual. We begin by introducing notation that will be used throughout the paper.

\subsection{Notation}\label{Notation}\quad

\subsubsection{Decompositions of $\R^{N}$} We begin with an initial decomposition of $\R^{N}$ into  direct sums of subspaces  where only some of the coordinates $x_{1}, \ldots, x_{N}$ are non-zero. Thus if $C\subset \{1, \ldots, N\}$ is non-empty, set
\bea\label{Eqn2.1bb}
\R^{C}=\Big\{\x=(x_{1}, \ldots, x_{N})\in \R^{N}:\text{$x_{j}=0$ for all $j\notin C$}\Big\}\subset\R^{N}.
\eea
We will abuse of this notation by allowing $\R^C$ to also indicate the set of $c$-tuples ($c$ being the cardinality of $C$) with entries indexed by  elements of $C$. We assume that an initial decomposition of $\R^N$ is given, \index{R1Cj@$\R^{C_j}$}
\bea\label{Eqn2.2bb}
\R^{N}=\R^{C_{1}}\oplus\cdots\oplus \R^{C_{n}},
\eea
where $C_{1}, \ldots ,C_{n}\subset \{1, \ldots, N\}$ are disjoint non-empty subsets with $\bigcup_{j=1}^{n}C_{j}=\{1, \ldots, N\}$. Given the decomposition (\ref{Eqn2.2bb}), if $\x=(x_{1}, \ldots, x_{N})\in \R^{N}$ we write $\x=(\x_{1}, \ldots, \x_{n})$ where $\x_{j}=\{x_{l}:l\in C_{j}\}$. \index{x1j@$\x_{C_j}$}

Unfortunately there are further notational complications. We shall sometimes need to consider the direct sum of only a subset of the subspaces $\R^{C_{1}}, \ldots, \R^{C_{n}}$, and we will also need to partition the set $\{C_{1}, \ldots, C_{n}\}$ into a union of disjoint subsets and consider the corresponding \emph{coarser} decomposition of $\R^{N}$. We shall then use the following notation.
\begin{enumerate}[(a)]
\smallskip

\item If $L=\{l_{1}, \ldots, l_{r}\}\subset\{1, \ldots, n\}$, write\label{exponent}\index{R2L@$\R^{L}$}
\bea\label{Eqn2.3bb}
\R^{L}=\R^{C_{l_{1}}}\oplus \cdots \oplus\R^{C_{l_{r}}}=\Big\{\x\in\R^{N}:\text{$x_{j}=0$ for all $j\notin \bigcup_{k=1}^{r}C_{l_{k}}$}\Big\}.
\eea
An element of $\R^{L}$ can then be written $\x_{L}=\{\x_{k}:k\in L\}$.\index{x2L@$\x_{L}$}

\smallskip

\item Let $I_{1}, \ldots, I_{s}\subset \{1, \ldots, n\}$ be disjoint, non-empty subsets with $\bigcup_{r=1}^{s}I_{r}=\{1, \ldots, n\}$. Write
\bea\label{Eqn2.4bb}
\R^{I_{r}}=\bigoplus_{k\in I_{r}}\R^{C_{k}}=\Big\{\x\in\R^{N}:\text{$x_{j}=0$ for all $j \notin \bigcup_{k\in I_{r}}C_{k}$}\}.
\eea
We then have the coarser decomposition $\R^{N}=\R^{I_{1}}\oplus\cdots \oplus \R^{I_{s}}$. If $\x\in \R^{N}$ we write $\x=(\x_{I_{1}}, \ldots, \x_{I_{s}})$ where $\x_{I_{r}}\in \R^{I_{r}}$.
\end{enumerate}
Note that we have slightly abused notation since the symbols $\R^{L}$ and $\R^{I_{r}}$ in equations (\ref{Eqn2.3bb}) and (\ref{Eqn2.4bb}) have a different meaning than the symbol $\R^{C}$ in eqwuation (\ref{Eqn2.1bb}). However $C$ is a subset of $\{1, \ldots, N\}$ while $L$ and $I_{r}$ are subsets of $\{1, \ldots, n\}$.

\bigskip

\subsubsection{Dilations on $\R^{N}$} The underlying space $\R^{N}$  is equipped with a family of dilations given by \index{l1ambdax@$\lambda\cdot\x$}
\bea\label{Eqn2.1aa}
\lambda\cdot\x=\lambda\cdot(x_{1}, \ldots, x_{N}) = \big(\lambda^{d_{1}}x_{1}, \ldots, \lambda^{d_{N}}x_{N}\big)
\eea
where each $d_{j}>0$.  We can restrict these dilations in equation (\ref{Eqn2.1aa}) to each subspace $\R^{C_{j}}$ and write $\lambda\cdot\x_{j}=\big\{\lambda^{d_{l}}x_{l}:l\in C_{j}\big\}$. The homogeneous dimension of the subspace $\R^{C_{j}}$ is then \index{Q1j@$Q_j$} $
Q_{j}= \sum_{l\in C_{j}}d_{l}$. Let $n_{j}:\R^{C_{j}}\to [0,\infty)$ be a smooth homogeneous norm for this family of dilations so that \index{n1j@$n_{j}$}
\be\label{2.2}
n_{j}(\x_{j}) \approx \sum_{l\in C_{j}}|x_{l}|^{1/d_{l}}.
\ee
Thus $n_{j}(\lambda\cdot\x_{j}) = \lambda n_{j}(\x_{j})$ for $\lambda>0$. Put \index{B1rho@$\B(\rho)$}
\bes
\B(\rho) = \Big\{\x\in \R^{N}:n_{j}(\x_{j})< \rho,\,1\leq j \leq n\Big\}.
\ees

We use standard multi-index notation for derivatives. If $\gammab =(\gamma_{1}, \ldots, \gamma_{N})\in \N^{N}$ then $\partial^{\gammab}_{\x}=\prod_{l=1}^{N}\partial^{\gamma_{l}}_{x_{l}}$ and $|\gammab|=\sum_{l=1}^{N}\gamma_{l}$. Using the decomposition \eqref{Eqn2.2bb}, we can also write $\gammab=(\gammab_{1}, \ldots, \gammab_{n})\in \N^{C_{1}}\times \cdots \times \N^{C_{n}}$. When estimating derivatives $\partial_{\x}^{\gammab}f(\x)$ in terms of homogeneous norms on each subspace $\R^{C_{j}}$ we need to use an appropriately weighted length. Thus put \index{g1amma@$\[\gammab\]$} \label{weighted length} 
\bea\label{2.3iou}
\[\gammab_{j}\]&= \sum_{l\in C_{j}}\gamma_{l}d_{l}&&&&\text{and}&&&\[\gammab\]&=\sum_{j=1}^{n}\[\gammab_{j}\]=\sum_{l=1}^{N}\gamma_{l}d_{l} .
\eea

The space of infinitely differentiable real-valued functions on $\R^{N}$ with compact support is denoted by $\CC^{\infty}_{0}(\R^{N})$ and the space of Schwartz functions is denoted by $\SS(\R^{N})$. The basic semi-norms on these spaces are given by  
\index{n2orms@$\|\varphi\|_{(m)},\|\psi\|_{M}$}
\beas
||\varphi||_{(m)} &= \sup_{\x\in \R^{N}}\Big\{\big\vert\partial^{\gammab}_{\x}\varphi(\x)\big\vert: |\gamma|\leq m\Big\} &&\text{for $\varphi\in \CC^{\infty}_{0}(\R^{N})$},\\
||\psi||_{M}\,\, &= \sup_{\x\in \R^{N}}\big\{\big\vert(1+|\x|)^{\alphab}\partial^{\betab}_{\x}\psi(\x)\big\vert:|\alphab|+|\betab|\leq M\big\}&&\text{for $\psi\in \SS(\R^{N})$}.
\eeas

Let $\AA$ be an index set. We use the following terminology from Section 4 in \cite{MR2949616}.
\begin{enumerate}[1)]
\smallskip

\item If  $\big\{\varphi_{\alpha}:\alpha\in \AA\big\}\subset \CC^{\infty}_{0}(\R^{N})$, then the family $\{\varphi_{\alpha}\}$ is \textit{normalized} in terms of a function $\Phi\in \CC^{\infty}_{0}(\R^{N})$ if there are constants $C, C_{m}>0$ and integers $p_{m}$ with the following properties.
\begin{enumerate}[a)]

\smallskip

\item If the support of $\Phi$ is contained in the ball $\B(\rho)$ then the support of each $\varphi_{\alpha}$ is contained in the ball $\B(C\rho)$.

\smallskip

\item For every $m\in \N$ and every $\alpha\in \AA$, $||\varphi_{\alpha}||_{(m)}\leq C_{m}||\Phi||_{(m+p_{m})}$.
\end{enumerate}
\label{uniform bounded}
\smallskip

\item A family $\{\varphi_{\alpha}:\alpha\in \AA\}\subset \CC^{\infty}_{0}(\R^{N})$ is \textit{uniformly bounded} if all the members are supported in a fixed ball $\B(\rho)$ and if there are constants $C_{m}>0$ so that $||\varphi_{\alpha}||_{(m)}\leq C_{m}$ for all $\alpha\in \AA$.

\smallskip

\item If $\{\psi_{\alpha}:\alpha\in \AA\}\subset \SS(\R^{N})$, then the family $\{\psi_{\alpha}\}$ is \textit{normalized} in terms of a function $\Psi\in \SS(\R^{N})$ if there are constants $C_{M}>0$ and integers $p_{M}\in \N$ so that $||\psi_{\alpha}||_{[M]}\leq C_{M}||\Psi||_{[M+p_{M}]}$ for every $M\in \N$.
\smallskip

\item A family $\{\psi_{\alpha}:\alpha\in \AA\}\subset \SS(\R^{N})$ is \textit{uniformly bounded} if there are constants $C_{M}>0$ so that $||\psi_{\alpha}||_{M}\leq C_{M}$ for all $\alpha\in \AA$.
\end{enumerate}

Given the decomposition $\R^{N}=\R^{C_{1}}\oplus\cdots\oplus\R^{C_{n}}$, define an $n$-parameter family of dilations by setting $(\lambda_{1}, \ldots, \lambda_{n})\cdot \x= (\lambda_{1}\cdot\x_{1}, \ldots, \lambda_{n}\cdot \x_{n})$. If $(\lambda_{1}, \ldots, \lambda_{n})=(2^{-t_{1}}, \ldots, 2^{-t_{n}})$ with $\t=(t_{1}, \ldots, t_{n})\in \R^{n}$ then for $f\in L^{1}(\R^{N})$, we define \index{f1t@$(f)_{\t},[f]_{\t}$}
\bea\label{2.3}
(f)_{\t}(\x) &=f(2^{-t_{1}}\cdot\x_{1}, \ldots, 2^{-t_{n}}\cdot \x_{n}),\\
[f]_{\t}(\x)&= 2^{-\sum_{j=1}^{n}t_{j}Q_{j}}f(2^{-t_{1}}\cdot\x_{1}, \ldots, 2^{-t_{n}}\cdot \x_{n}).
\eea
Note that $||(f)_{\t}||_{L^{\infty}(\R^{N})}= ||f||_{L^{\infty}(\R^{N})}$ while $||[f]_{\t}||_{L^{1}(\R^{N})}= ||f||_{L^{1}(\R^{N})}$.

\subsection{Global norms}\label{Global}\quad

\medskip

The classes $\PP(\EEE)$ and $\MM({\EEE})$ of distributions and multipliers are defined using families of norms $\{N_{1},\ldots, N_{n}\}$ and dual norms $\{\widehat N_{1}, \ldots, \widehat N_{n}\}$. These in turn are defined in terms of an $n\times n$ matrix $\EEE=\{e(j,k)\}$, where each entry $e(j,k)\in (0,\infty)$. For $1\leq j \leq n$ set 

\bea\label{2.4}\index{N2j@$N_{j},\widehat N_{j}$}
N_{j}(\t_{1}, \ldots, \t_{n}) &= 
n_{1}(\t_{1})^{e(j,1)}+\cdots +n_{n}(\t_{n})^{e(j,n)}
\approx 
\sum_{k=1}^{n}\sum_{l\in C_{k}}|t_{l}|^{e(j,k)/d_{l}},\\
\widehat N_{j}(\t_{1}, \ldots, \t_{n})&= 
n_{1}(\t_{1})^{\frac{1}{e(1,j)}}+\cdots +n_{n}(\t_{n})^{\frac{1}{e(n,j)}}
\approx 
\sum_{k=1}^{n}\sum_{l\in C_{k}}|t_{l}|^{1/e(k,j)d_{l}}.
\eea
The entries of $\EEE$ are usually subject to certain natural constraints.\footnote{The motivation for these constraints is discussed in Section \ref{Cones} below.}

\begin{definition}\label{Def2.2Z}
The matrix $\EEE$ satisfies the {\rm basic hypothesis} if :
\bea\label{2.5}
e(j,j)&=1,&&\text{for $1\leq j \leq n$},\\
e(j,l) &\leq e(j,k)e(k,l)&&\text{for $1\leq j,k,l\leq n$.}
\eea
In particular, $1\leq e(j,k)e(k,j)$ for $1 \leq j,k\leq n$. 
\end{definition}

We will need to restrict the norms $\{N_{j}\}$ to certain subspaces of $\R^{N}$. Thus if $L=\{l_{1},\ldots,l_{s}\}\subset \{1, \ldots, n\}$, we use the notation introduced on page \pageref{exponent} and set
 $\R^{L}=\bigoplus_{l\in L}\R^{C_{l}}$. 
For $\t_{L} = (\t_{l_{1}},\ldots,\t_{l_{s}})\in \R^{L}$, and for each $l\in L$, put\label{x_C} \bes
N_{l}(\t_{L})= \sum_{m\in L}n_{m}(\t_{m})^{e(l,m)}.
\ees
If $R=\{R_{1}, \ldots, R_{s}\}$ are positive real numbers and if $\psi\in \mathcal C^{\infty}_{0}(\R^{L})$, define $\psi_{R}\in \CC^{\infty}_{0}(\R^{L})$ by setting 
\bes
\psi_{R}(\t_{L}) = \psi(R_{1}\cdot\t_{l_{1}}, \ldots, R_{s}\cdot \t_{l_{s}}).
\ees

\subsection{Classes of distributions and multipliers}\label{classes}\quad

\medskip

Let $\R^{N}= \R^{C_{1}}\oplus \cdots \oplus\R^{C_{n}}$, let $\EEE = \{e(j,k)\}$ be an $n\times n$ matrix satisfying (\ref{2.5}), and let $\{N_{1}, \ldots, N_{n}\}$ and $\{\widehat N_{1}, \ldots, \widehat N_{n}\}$ be the norms and dual defined in (\ref{2.4}). 

\begin{definition}\label{Def2.2}
A distribution $\KK$ on $\R^{N}$ belongs to the class $\PP(\EEE)$ if it satisfies the following differential inequalities and cancellation conditions. \index{P1@$\PP(\EEE)$}

\begin{enumerate}[{\rm(a)}] 

\medskip

\item \label{Def2.2A} {\rm[Differential Inequalities]} 
Away from the origin the distribution $\mathcal K$ is given by integration against a $\mathcal C^{\infty}$-function $K$, and for every $\gammab=(\gammab_{1},\ldots,\gammab_{n})\in\mathbb N^{C_{1}}\times\cdots\times \mathbb N^{C_{n}}$, there is a constant $C_{\gammab}$ so that
\bes
|\partial^{\gammab}K(\x)|\leq C_{\gammab}\prod_{j=1}^{n}N_{j}(\x)^{-(Q_{j}+\[\gammab_{j}\])}.
\ees

\item \label{Def2.2B} {\rm[Cancellation Conditions]} Let $L=\{l_{1},\ldots,l_{r}\}$ and $M=\{m_{1},\ldots, m_{s}\}$ be any pair of complementary subsets of $\{1,\ldots, n\}$, let $R=\{R_{1}, \ldots, R_{s}\}$ be any positive real numbers, and let $\psi\in \mathcal C^{\infty}_{0}(\R^{M})$ have support in the unit ball. Define a distribution $\mathcal K_{\psi,R}$ on $\R^{L}$ by setting 
\bes
\big\langle \mathcal K_{\psi,R},\varphi\big\rangle = \big\langle \mathcal K, \varphi\otimes\psi_{R}\big\rangle
\ees
for every $\varphi\in \mathcal C^{\infty}_{0}(\R^{L})$. Then $\mathcal K_{\psi,R}$ uniformly satisfies the analogue of the estimates in {\rm(\ref{Def2.2A})} on the space $\R^{L}$: precisely, this means that away from the origin of $\R^{L}$ the distribution $\KK_{\psi,R}$ is given by integration against a smooth function $K_{\psi,R}$ and for every $\gammab=(\gammab_{l_{1}}, \ldots, \gammab_{l_{r}}) \in \mathbb N^{C_{l_{1}}}\times\cdots\times\mathbb N^{C_{l_{r}}}$ there is a constant $C_{\gammab}^{\prime}$ depending only on the constants $\{C_{\gammab}\}$ in {\rm(\ref{Def2.2A})} and the norms $\{||\psi||_{(m)}\}$ (but independent of $R$) so that 
\beas
\big|\partial^{\gammab}K_{\psi,R}(\x)\big|\leq C_{\gammab}^{\prime}\prod_{t=1}^{r}N_{l_{t}}(\x_{L})^{-(Q_{l_{t}}+\[\gammab_{l_{t}}\])}.
\eeas
In particular, if $L=\emptyset$ then $|\big\langle\KK,\psi_{R}\big\rangle|$ is bounded independently of $R$.
\end{enumerate}
\end{definition}

\begin{remark}\label{Rem2.2} 
Since $N_{j}(\x)^{-\[\gammab_{j}\]}\leq |\x_{j}|^{-\[\gammab_{j}\]}$, it follows that a  distribution $\KK\in \PP(\EEE)$ is a product kernel on $\R^{N}=\bigoplus_{j=1}^{n}\R^{C_{j}}$ in the sense of \cite{MR1818111}. (The cancellation conditions follow from the Fourier transform characterization. See \cite{MR1818111}, Remark 2.16.) However, the distributions in $\PP(\EEE)$ are in general more regular; they are singular only at the origin.
\end{remark}

\begin{definition}\label{Def2.4} \index{M1@$\MM(\EEE)$}
If $m\in \CC^{\infty}(\R^{N}\setminus\{0\})$ then $m$ belongs to the class of multipliers $\MM(\EEE)$ if for every $\gammab=(\gammab_{1}, \ldots, \gammab_{n}) \in \N^{C_{1}}\times \cdots \times \N^{C_{n}}$ there is a constant $C_{\gammab}$ so that
\bes
\left|\partial^{\gammab}m(\xib)\right| \leq C_{\gammab}\prod_{j=1}^{n}\widehat N_{j}(\xib)^{-\[\gammab_{j}\]}.
\ees
\end{definition}

We shall see in Lemma \ref{Lem4.2} that unless $e(j,k)=e(j,l)e(l,k)$ for all $j,k,l\in \{1, \ldots, n\}$, the distribution $\KK\in \PP(\EEE)$ is integrable at infinity. We are primarily interested in the local behavior of these kernels and the following restricted classes of distributions and multipliers. 

\begin{definition}\label{Def2.5} Let $|\x|$ denote the usual Euclidean norm of a vector $\x\in\R^{N}$. 
\begin{enumerate}[{\rm(a)}] 

\medskip

\item \label{Def2.5a} \index{P2zero@$\PP_{0}(\EEE)$}
A distribution $\KK$ on $\R^{N}$ belongs to the class $\PP_{0}(\EEE)$ if it belongs to the class $\PP(\EEE)$ and away from the origin it is given by a smooth function $K$ satisfying the differential inequalities 
\be\label{2.6iou}
\big\vert\partial^{\gammab}_{\x}K(\x)\big\vert \leq C_{M,\gammab}\prod_{j=1}^{n}N_{j}(\x)^{(Q_{j}+\[\gammab_{j}\])}\,\big(1+|\x|\big)^{-M}
\ee
for every  multi-index $\gammab=(\gammab_{1}, \ldots, \gammab_{n}) \in \N^{C_{1}}\times \cdots \times \N^{C_{n}}$ and every $M\geq 0$ . 
Moreover  the quotient distributions $\KK_{\psi,R}$ defined in part {\rm(\ref{Def2.2B})} of Definition \ref{Def2.2} satisfy 
\be\label{2.7iou}
\big|\partial^{\gammab}K_{\psi,R}(\x)\big|\leq D_{M,\gammab}\big(1+|\x|\big)^{-M}\prod_{t=1}^{r}N_{l_{t}}^{L}(\x_{L})^{-(Q_{l_{t}}+\[\gammab_{l_{t}}\])}.
\ee

\item \label{Def2.5b} \index{M2infty@$\MM_\infty(\EEE)$}
A function $m\in \CC^{\infty}(\R^{N})$ belongs to $\MM_{\infty}(\EEE)$ if there is a constant $C_{\gammab}$ so that
\bes
\left|\partial^{\gammab}_{\xib}m(\xib)\right| \leq C_{\gammab}\prod_{j=1}^{n}\big(1+\widehat N_{j}(\xib)\big)^{-\[\gammab_{j}\]}
\ees
for every multi-index $\gammab=(\gammab_{1}, \ldots, \gammab_{n}) \in \N^{C_{1}}\times \cdots \times \N^{C_{n}}$.
\end{enumerate}
\end{definition}

\noindent We will see in Section \ref{Duality} below that the multipliers $m\in \MM_{\infty}(\EEE)$ are precisely the Fourier transforms of the distributions $\KK\in \PP_{0}(\EEE)$, and conversely, that distributions $\KK\in \PP_{0}(\EEE)$ are the inverse Fourier transforms of multipliers $m\in \MM_{\infty}(\EEE)$. The family of norms
\be\label{2.8iou}
|\KK|_{M}= \sum_{|\gammab|,|\gammab'|\leq M}\inf\Big\{C_{M,\gammab}+D_{M,\gammab'}:\text{ (\ref{2.6iou}) and (\ref{2.7iou}) hold}\Big\}
\ee
with $M\in \N$ induces a Fr\'echet space topology on the space $\PP_{0}(\EEE)$.

\section{Marked partitions and decompositions of $\R^{N}$}\label{Partitions}

\subsection{Dominant terms in $N_{j}(\t)$ and $\widehat N_{j}(\t)$}\quad

\medskip

To study the kernels and multipliers introduced in Section \ref{Kernels}, we need to partition the space $\R^{N}$ into regions where the norms $N_{j}(\t)=\sum_{k=1}^{n}n_{k}(\t_{k})^{e(j,k)}$ or $\widehat N_{j}(\t)=\sum_{k=1}^{n}n_{k}(\t_{k})^{1/e(k,j)}$ are comparable to a single term $n_{k}(\t_{k})^{e(j,k)}$ or $n_{k}(\t_{k})^{1/e(k,j)}$. We begin by introducing the notion of strict dominance and $A$-dominance. 

\begin{definition}\label{Def5.1}
Let $A\geq 1$.
\begin{enumerate}[{\rm(a)}]

\smallskip

\item Let $\t\in \B(1)$. The term $\t_{k}$ is {\rm $A$-dominant in $N_{j}(\t)$} if $\displaystyle n_{l}(\t_{l})^{e(j,l)}<\big[A\,n_{k}(\t_{k})\big]^{e(j,k)}$ for every $l\neq k$. The term $\t_{k}$ is {\rm strictly dominant in $N_{j}(\t)$} if we can take $A=1$.

\smallskip

\item Let $\t\in \B(1)^{c}$. The term $\t_{k}$ is {\rm $A$-dominant} in $\widehat N_{j}(\t)$ if $\displaystyle n_{l}(\t_{l})^{1/e(l,j)}<\big[A\,n_{k}(\t_{k})\big]^{1/e(k,j)}$ for every $l\neq k$. The term $\t_{k}$ is {\rm strictly dominant in $\widehat N_{j}(\t)$} if we can take $A=1$.
\end{enumerate}
\end{definition}

\goodbreak

\begin{remarks} \label{Rem5.2}\quad
\begin{enumerate}[{\rm 1)}]
\item If $\t\in \B(1)$ and if $\t_{k}$ is $A$-dominant in $N_{j}(\t)$ then $N_{j}(\t)\approx n_{k}(\t_{k})^{e(j,k)}$ where the implied constant depends only on $A$, $\EEE$, and $n$. 

\item If $\t\in \B(1)^{c}$ and if $\t_{k}$ is $A$-dominant in $\widehat N_{j}(\t)$ then $\widehat N_{j}(\t)\approx n_{k}(\t)^{1/e(k,j)}$ where the implied constant depends only on $A$, $\EEE$, and $n$. 

\item If $\t\in \B(1)^{c}$ then each $\widehat N_{j}(\t)\geq 1$.

\item If $\t\in \B(1)^{c}$ and if $\t_{k}$ is $A$-dominant in any $\widehat N_{j}(\t)$ then $An_{k}(\t_{k})>1$.
\end{enumerate}
\end{remarks} 

The following is a simple but key fact about dominance.

\goodbreak

\begin{proposition}\label{Prop5.3}
Suppose the matrix $\EEE$ satisfies the basic hypothesis (\ref{2.5}). 
\begin{enumerate}[{\rm(a)}]

\smallskip

\item \label{Prop5.3a}
Let $\t\in \B(1)$. If $\t_{k}$ is $A$-dominant in $N_{j}(\t)$ then $\t_{k}$ is also $A$-dominant in $N_{k}(\t)$. In particular, if $\t_{k}$ is strictly dominant in $N_{j}(\t)$, then $\t_{k}$ is strictly dominant in $N_{k}(\t)$.

\smallskip

\item \label{Prop5.3b} 
Let $\t\in \B(1)^{c}$. If $\t_{k}$ is $A$-dominant in $\widehat N_{j}(\t)$ then $\t_{k}$ is also $A$-dominant in $\widehat N_{k}(\t)$. In particular if $\t_{k}$ is strictly dominant in $\widehat N_{j}(\t)$  then $\t_{k}$ is also strictly dominant in $\widehat N_{k}(\t)$.
\end{enumerate}
\end{proposition}

\begin{proof}
Let $\t\in \B(1)$ and $A\geq 1$. By the basic hypothesis, $e(k,l)\geq e(j,l)/e(j,k)$, and since $n_{l}(\t_{l})\leq 1$ it follows that $n_{l}(\t_{l})^{e(k,l)}\leq n_{l}(\t_{l})^{e(j,l)/e(j,k)}$. If $\t_{k}$ is $A$-dominant in $N_{j}(\t)$ and if $l\neq k$, then $n_{l}(\t_{l})^{e(j,l)}<\big[A\,n_{k}(\t_{k})\big]^{e(j,k)}$, and so 
$$
n_{l}(\t_{l})^{e(k,l)}\leq n_{l}(\t_{l})^{e(j,l)/e(j,k)}< A\,n_{k}(\t_{k}).
$$ 
Thus $\t_{k}$ is $A$-dominant in $N_{k}(\t)$. On the other hand, if $\t\in \B(1)^{c}$ and $\t_{k}$ is $A$-dominant in $\widehat N_{j}(\t)$, then every $l \neq k$ we have  $n_{l}(\t_{l})^{1/e(l,j)}<[A n_{k}(\t_{k})]^{1/e(k,j)}$. Since $e(l,j)\leq e(l,k)e(k,j)$ and $An_{k}(\t_{k})>1$ we have 
$$
n_{l}(\t_{l})^{1/e(l,k)} = \Big[n_{l}(\t_{l})^{1/e(l,j)}\Big]^{e(l,j)/e(l,k)} < \Big[A n_{k}(t_{k})\Big]^{e(l,j)/e(l,k)e(k,j)}\leq A\,n_{k}(\t_{k}).
$$ 
Thus  $\t_{k}$ is $A$-dominant in $\widehat N_{k}(\t)$.
\end{proof}

\subsection{Marked partitions and the sets $E_{S}$ and $\widehat E_{S}$}\quad

\medskip

The notion of strict dominance allows us to decompose (up to a closed set of measure zero) the ball $\B(1)=\big\{\t\in \R^{N}:n_{j}(\t_{j})\leq 1\big\}$ into sets where one term in each $N_{j}$ is strictly dominant. The same is true for the complement $\B(1)^{c}$ and the norms $\widehat N_{j}$. The decomposition depends on the notion of a \textit{marked partition}.

\begin{definition}\label{Def3.4iou} \quad
\begin{enumerate}[{\rm(a)}]

\item
A \textbf{marked partition} $S=\big\{(I_{1},k_{1});\ldots; (I_{s};k_{s})\big\}$ of the set $\{1, \ldots, n\}$ is a collection $\{I_{1}, \ldots, I_{s}\}$ of disjoint non-empty subsets of $\{1,\ldots,n\}$ such that $\bigcup _{r=1}^{s}I_{r}=\{1,\ldots,n\}$, together with a choice of a particular element $k_{r}\in I_{r}$ in each subset.   

\smallskip

\item
The partition $S_{0}\!=\!\big\{(\{1\},1);\ldots;(\{n\},n)\big\}$  is called the \textbf{principal partition} of $\{1, \ldots, n\}$. 

\smallskip

\item
Any marked partition $S=\big\{(I_{1},k_{1});\ldots; (I_{s};k_{s})\big\}$ with $s<n$ is called a \textbf{non-principal partition}. 

\smallskip

\item
The set of all marked partitions of $\{1,\ldots,n\}$ is denoted by $\SS(n)$. \index{S1n@$\SS(n)$}
\end{enumerate}
\end{definition}

We want to associate a unique marked partition to every point $\t\in \R^{N}$, but to do this we need to exclude the closed set of measure zero where two different terms in $N_{j}(\t)$ or $\widehat N_{j}(\t)$ can be equal. Let
\beas
\Xi = \Big\{\t\in \R^{N}: \text{$\big(\exists j\big)\big(\exists l\neq m\big)$ $n_{l}(\t_{l})^{e(j,l)}= n_{m}(\t_{m})^{e(j,m)}$ or $n_{l}(\t_{l})^{1/e(l,j)}= n_{m}(\t_{m})^{1/e(m,j)}$}\Big\}.
\eeas
The set $\Xi$ is closed and has measure zero, and $\R^{N}\setminus \Xi$ is dense in $\R^{N}$. Note that if $t\notin \Xi$, all the summands in $N_{j}(\t)$ and $\widehat N_{j}(\t)$ are distinct. 

We associate to every point $\t\in\R^{N}\setminus\Xi$ a marked partition $S(\t)\in \SS(n)$. There are two cases, depending on whether the point $\t$ belongs to $\B(1)$ or to $\B(1)^{c}$.

\smallskip

\noindent\textbf{Case 1:} If $\t\in \B(1)\setminus \Xi$ then for every $j$ the terms $n_{1}(\t_{1})^{e(j,1)}, \ldots, n_{n}(\t_{n})^{e(j,n)}$ in $N_{j}(\t)$ are distinct. Thus for each $j\in \{1, \ldots, n\}$ there exists a unique integer $k(j)\in \{1, \ldots, n\}$ such that $\t_{k(j)}$ is strictly dominant in $N_{j}(\t)$. Let $\{k_{1}, \ldots, k_{s}\}$ be the set of distinct integers $k(j)$ which arise in this way. For $1 \leq r \leq s$ set 
\bes
I_{r}=\Big\{j\in \{1, \ldots, n\}: \text{$\t_{k_{r}}$ is strictly dominant in $N_{j}(\t)$}\Big\}.
\ees 
Every $j$ belongs to a unique $I_{r}$ since there is only one term in $N_{j}(\t)$ which is dominant, and according to Proposition \ref{Prop5.3}, we have $k_{r}\in I_{r}$. Thus $S(\t) = \big\{(I_{1},k_{1});\ldots;(I_{s},k_{s})\big\}\in \SS(n)$. 

\smallskip

\noindent\textbf{Case 2:} If $\t\in \B(1)^{c}\setminus \Xi$, the terms $n_{1}(\t_{1})^{1/(1,j)}, \ldots, n_{n}(\t_{n})^{1/e(n,j)}$ in $\widehat N_{j}(\t)$ are distinct, and again for each $j$ there exists a unique integer $k(j)\in \{1,\ldots,n\}$ such that $\t_{k(j)}$ is strictly dominant in $\widehat N_{j}(\t)$. Let $\{k_{1}, \ldots, k_{s}\}$ be the set of distinct integers that arise, and again set 
\bes
I_{r}=\Big\{j\in \{1,\ldots,n\}: \text{$\t_{k_{r}}$ is strictly dominant in $\widehat N_{j}(\t)$}\Big\}.
\ees 
We have $k_{r}\in I_{r}$ by Proposition \ref{Prop5.3}, and thus $S=S(\t) = \{(I_{1},k_{1});\ldots;(I_{s},k_{s})\}\in\SS(n)$.

\begin{definition}\label{Def5.4}
Let $S=\big\{(I_{1},k_{1});\ldots;(I_{s},k_{s})\big\}\in\SS(n)$, and let $A\geq 1$. Then \index{E1AS@$E_{S},\widehat E_{S},E^{A}_{S},\widehat E^{A}_{S}$}
\beas
E^{A}_{S}&= \Big\{\t=(\t_{1}, \ldots, \t_{n})\in \B(1): 
\text{$\forall r\in \{1, \ldots, s\},\,\forall j\in I_{r},\quad\t_{k_{r}}$ is $A$-dominant in $N_{j}(\t)\big)$}\Big\}\\
&=
\Big\{\t\in\B(1):\text{$\forall r\in \{1, \ldots, s\},\,\forall j\in I_{r},\,\forall l\neq k_{r}$,\,$n_{l}(\t_{l})^{e(j,l)}<\big[A\,n_{k_{r}}(\t_{k_{r}})\big]^{e(j,k_{r})}$  }\Big\}.
\\\\
\widehat E^{A}_{S}&= \Big\{\t=(\t_{1}, \ldots, \t_{n})\in \B(1)^{c}: 
\text{$\forall r\in \{1, \ldots, s\},\,\forall j\in I_{r},\,\t_{k_{r}}$ is $A$-dominant in $\widehat N_{j}(\t)\big)$}\Big\}\\
&=
\Big\{\t\in\B(1)^{c}:\text{$\forall r\in \{1, \ldots, s\},\,\forall j\in I_{r},\,\forall l\neq k_{r}$,\,$n_{l}(\t_{l})^{1/e(l,j)}<\big[A\,n_{k_{r}}(\t_{k_{r}})\big]^{1/e(k_{r},j)}$}\Big\}.
\eeas

\noindent When $A=1$ we write $E_{S}^{1}=E_{S}$ and $\widehat E_{S}^{1}=\widehat E_{S}$.
\end{definition}

\begin{proposition}\label{Prop5.5}
The set $E_{S}^{A}$ and $\widehat E_{S}^{A}$ are open. Moreover,
\beas
\B(1)\setminus \Xi &= \bigcup_{S\in \SS(n)}E_{S},&&&&&&& \B(1)^{c}\setminus \Xi &= \bigcup_{S\in \SS(n)}\widehat E_{S}.
\eeas
If $A>1$ then
\beas
\B(1) &\subset \bigcup_{S\in \SS(n)}E^{A}_{S},&&&&&&&\B(1)^{c} \subset \bigcup_{S\in \SS(n)}\widehat E^{A}_{S}.
\eeas
\end{proposition}
\begin{proof}
The first line follows from the construction of $E_{S}$ and $\widehat E_{S}$. The second line follows since the closure of $\R^{N}\setminus \Xi$ is all of $\R^{N}$, and if $A>1$ then $E_{S}^{A}$ is an open neighborhood of the closure of $E_{S}$, and $\widehat E_{S}^{A}$ is an open neighborhood of the closure of $\widehat E_{S}$.
\end{proof}

\subsection{Characterizing the sets $E_{S}^{A}$ and $\widehat E_{S}^{A}$}\quad

\medskip

Let $S=\big\{(I_{1},k_{1});\ldots;(I_{s},k_{s})\big\}\in \SS(n)$. The definitions of the sets $E_{S}^{A}$ and $\widehat E_{S}^{A}$ involve inequalities between the norms $n_{j}(\t_{j})$ and $n_{k_{r}}(\t_{k_{r}})$ for all $r\in \{1, \ldots, s\}$ and \textit{all} $j\neq k_{r}$. In this section we show that the sets are characterized by a smaller number of inequalities relating, for given $r$, $n_{k_{r}}(\t_{k_{r}})$ to $n_{j}(\t_{j})$   \textit{only} for $j\in I_{r} \cup\{k_{1}, \ldots, k_{s}\}$. 

\begin{definition}\label{Def3.7iou} 
Let $\big\{(I_{1},k_{1});\ldots;(I_{s},k_{s})\big\} \in \SS(n)$. For $1 \leq r\leq s$ and $1 \leq l\leq n$ set \index{s2igmaS@$\sigma_{S}(k_{r},l)$} \index{t1auS@$\tau_{S}(l,k_{r})$}
\beas
\sigma_{S}(k_{r},l) &= \min\limits_{\substack{j\in I_{r}}}\frac{e(j,l)}{e(j,k_{r})}, &&&\tau_{S}(l,k_{r}) &= \min_{j\in I_{r}}\frac{e(l,j)}{e(k_{r},j)}.
\eeas
\end{definition}

\begin{proposition}\label{Prop5.8} Suppose that the matrix $\EEE=\{e(j,k)\}$ satisfies the basic hypothesis (\ref{2.5}) and that $S=\big((I_{1},k_{1});\ldots;(I_{s},k_{s})\big)\in \SS(n)$. 
\begin{enumerate}[{\rm(a)}]

\smallskip

\item\label{Prop5.8a}
If $1 \leq r \leq s$ and if $l\in I_{r}$ then 
\beas
\sigma_{S}(k_{r},l) &= e(l,k_{r})^{-1},&&&\tau_{S}(l,k_{r})&=e(k_{r},l)^{-1}.
\eeas

\smallskip

\item\label{Prop5.8b} 
Let $1 \leq r, p \leq s$. If $m\in I_{p}$ then 
\beas
\sigma_{S}(k_{r},m)&\geq \sigma_{S}(k_{r},k_{p})\sigma_{S}(k_{p},m),&&&\tau_{S}(m,k_{r})&\geq \tau_{S}(m,k_{p})\tau_{S}(k_{p},k_{r}).
\eeas
\end{enumerate}
\end{proposition}

\begin{proof}
It suffices to prove the statements about $\sigma_{S}$; the statements about $\tau_{S}$ follow in the same way. If $l\in I_{r}$ we can take $j=l$ in the definition of $\sigma_{S}(k_{r},l)$ to see that $\sigma_{S}(k_{r},l)\leq e(l,k_{r})^{-1}$. On the other hand it follows from the basic assumption (\ref{2.5}) that $e(j,l)/e(j,k_{r}) \geq e(l,k_{r})^{-1}$ for any $l$, and this gives the opposite inequality and establishes the first equality in (\ref{Prop5.8a}). 

Next let $m\in I_{p}$. Then using (\ref{Prop5.8a}) and the basic hypothesis, if $j_{0}\in I_{r}$
\beas
\sigma_{S}(k_{p},m)\sigma_{S}(k_{r},k_{p}) 
&= 
e(m,k_{p})^{-1} \sigma_{S}(k_{r},k_{p}) 
=
e(m,k_{p})^{-1}\min\limits_{\substack{j\in I_{r}}}\frac{e(j,k_{p})}{e(j,k_{r})}\,\\
&\leq 
e(m,k_{p})^{-1}\frac{e(j_{0},k_{p})}{e(j_{0},k_{r})}
\leq
\frac{e(j_{0},m)}{e(j_{0},k_{r})}.
\eeas
Taking the minimum over $j_{0}\in I_{r}$ it follows that $\sigma_{S}(k_{r},m)\geq \sigma_{S}(k_{r},k_{p})\sigma_{S}(k_{p},m)$, which establishes the first inequality in (\ref{Prop5.8b}) and completes the proof. 
\end{proof}

\goodbreak

\begin{lemma}\label{Thm5.6}
Let $S=\big((I_{1},k_{1});\ldots;(I_{s},k_{s})\big)\in \SS(n)$. 
\begin{enumerate}[{\rm (a)}] 

\smallskip

\item \label{Thm5.6A}
Let $\t\in \B(1)$. If $\t\in E_{S}^{A}$ then for every $r\in \{1, \ldots, s\}$
\bea\label{5.2aa}
n_{k_{r}}(\t_{k_{r}})&>A^{-1}n_{j}(\t_{j})^{1/e(j,k_{r})}=A^{-1}n_{j}(\t_{j})^{\sigma_{S}(k_{r},j)}& &\forall j\in I_{r}, j\neq k_{r},\quad\text{and}\\
n_{k_{r}}(\t_{k_{r}})&>A^{-1}n_{k_{p}}(\t_{k_{p}})^{\sigma_{S}(k_{r},k_{p})} & &\forall p\in \{1, \ldots, s\},\,\, p\neq r.
\eea
Conversely, if $\t\in \B(1)$ satisfies (\ref{5.2aa}) then $\t\in E_{S}^{A^{\eta}}$ where $\eta$ is a constant that depends only on the matrix $\EEE$. 

\smallskip

\item \label{Thm5.6B}
Let $\t\in \B(1)^{c}$. If $\t\in \widehat E_{S}^{A}$ then for every $r\in \{1, \ldots, s\}$
\bea\label{5.3aa}
n_{k_{r}}(\t_{k_{r}})&>A^{-1}n_{j}(\t_{j})^{e(k_{r},j)}=A^{-1}n_{j}(\t_{j})^{1/\tau_{S}(j,k_{r})}& &\forall j\in I_{r}, j\neq k_{r},\quad\text{and}\\
n_{k_{r}}(\t_{k_{r}})&>A^{-1}n_{k_{p}}(\t_{k_{p}})^{1/\tau_{S}(k_{p},k_{r})} & &\forall p\in \{1, \ldots, s\}, \,\,p\neq r.
\eea
Conversely, if $\t\in \B(1)^{c}$ satisfies (\ref{5.3aa}) then $\t\in E_{S}^{A^{\eta}}$ where $\eta$ is a constant that depends only on the matrix $\EEE$. 
\end{enumerate}
\end{lemma}

\noindent Letting $A=A^{\eta}=1$, we obtain the following characterization of the sets $E_{S}$ and $\widehat E_{S}$.

\begin{corollary}\label{Cor5.7}
Let $S=\big((I_{1},k_{1});\ldots;(I_{s},k_{s})\big)\in \SS(n)$. 
\begin{enumerate}[{\rm (a)}] 
\item \label{Cor5.7A} If $\t\in \B(1)$ then $\t\in E_{S}$ if and only if 
\beas
n_{k_{r}}(\t_{k_{r}})&>n_{j}(\t_{j})^{1/e(j,k_{r})}=n_{j}(\t_{j})^{\sigma_{S}(k_{r},j)}& &\text{for every $j\in I_{r}$ with $j\neq k_{r}$, and}\\
n_{k_{r}}(\t_{k_{r}})&>n_{k_{p}}(\t_{k_{p}})^{\sigma_{S}(k_{r},k_{p})} & &\text{for every $p\in \{1, \ldots, s\}$ with $p\neq r$.}
\eeas
\item \label{Cor5.7B} If $\t \in \B(1)^{c}$ then $\t\in \widehat E_{S}$ if and only if
\beas
n_{k_{r}}(\t_{k_{r}})&>n_{j}(\t_{j})^{e(k_{r},j)}=n_{j}(t_{j})^{1/\tau_{S}(j,k_{r})}& &\text{for every $j\in I_{r}$ with $j\neq k_{r}$, and}\\
n_{k_{r}}(\t_{k_{r}})&>n_{k_{p}}(\t_{k_{p}})^{1/\tau_{S}(k_{p},k_{r})} & &\text{for every $p\in \{1, \ldots, s\}$ with $p\neq r$.}
\eeas
\end{enumerate}
\end{corollary}

\begin{proof}[Proof of Lemma \ref{Thm5.6}] \quad

\smallskip

If $\t\in \B(1)$ then $\t\in E_{S}^{A}$ if and only if $\displaystyle n_{k_{r}}(\t_{k_{r}})>A^{-1}n_{l}(\t_{l})^{e(j,l)/e(j,k_{r})}$ for $1\leq r \leq s$, for $l\neq k_{r}$, and for all $j\in I_{r}$. But since $n_{l}(\t_{l})\leq 1$ this is equivalent to saying that 
\beas
n_{k_{r}}(\t_{k_{r}})&> 
A^{-1}\max_{j\in I_{r}}n_{l}(\t_{l})^{e(j,l)/e(j,k_{r})}= 
A^{-1}n_{l}(\t_{l})^{\min_{j\in I_{r}}e(j,l)/e(j,k_{r})}
= 
A^{-1}n_{l}(\t_{l})^{\sigma_{S}(k_{r},l)},
\eeas
and this establishes both inequalities in (\ref{5.2aa}).  To prove the converse, let $l\in \{1, \ldots, n\}$ with $l\neq k_{r}$. Then there exists an index $p\in \{1, \ldots, s\}$ so that $l\in I_{p}$, and by the first inequality in 
(\ref{5.2aa}) it follows that $n_{l}(\t_{l})^{\sigma_{S}(k_{p},l)}<A\,n_{k_{p}}(\t_{k_{p}})$. On the other hand, the second inequality in (\ref{5.2aa}) says that $n_{k_{p}}(\t_{k_{p}})^{\sigma_{S}(k_{r},k_{p})}\leq A\,n_{k_{r}}(\t_{k_{r}})$. Also, Proposition \ref{Prop5.8}, part (\ref{Prop5.8b}) says that $\sigma_{S}(k_{r},l)\geq \sigma_{S}(k_{r},k_{p})\sigma_{S}(k_{p},l)$. Since $n_{l}(\t_{l})< 1$, it follows that 
\beas
n_{l}(\t_{l})^{\sigma_{S}(k_{r},l)}&\leq n_{l}(\t_{l})^{\sigma_{S}(k_{p},l)\sigma_{S}(k_{r},k_{p})}<\big[A\,n_{k_{p}}(\t_{k_{t}})\big]^{\sigma_{S}(k_{r},k_{p})} <
A^{1+\sigma_{S}(k_{r},k_{p})}n_{k_{r}}(\t_{k_{r}}).
\eeas
But then if $j\in I_{r}$ and $l\neq k_{r}$,
\beas
n_{l}(\t_{l})^{e(j,l)}
&=
\Big[n_{l}(\t_{l})^{e(j,l)/e(j,k_{r})}\Big]^{e(j,k_{r})}
\leq 
n_{l}(\t_{l})^{\sigma_{S}(k_{r},l)e(j,k_{r})}
<
\Big[A^{(1+\sigma_{S}(k_{r},k_{p}))}n_{k_{r}}(\t_{k_{r}})\Big]^{e(j,k_{r})}\\
&\leq 
\Big[A^{\eta}n_{k_{r}}(\t_{k_{r}})\Big]^{e(j,k_{r})}
\eeas
where we take $\eta = 1+\sup_{r,p}\sigma_{S}(k_{r},k_{p})+\sup_{r,p}\tau_{S}(k_{p},k_{r})^{-1}$. This shows that $\t_{k_{r}}$ is $A^{\eta}$-dominant in $N_{j}(\t)$, and completes the proof of (\ref{Thm5.6A}).

The proof of (\ref{Thm5.6B}) is very similar. Suppose that $\t\in \B(1)^{c}$. Then $\t\in \widehat E_{S}^{A}$ if and only if $n_{k_{r}}(\t_{k_{r}})> A^{-1}\,n_{l}(\t_{l})^{e(k_{r},j)/e(l,j)}$ for $1\leq r\leq s$, for $l\neq k_{r}$, and for all $j\in I_{r}$. Suppose first that $n_{l}(\t_{l})\geq 1$. In this case it follows that
\beas
n_{k_{r}}(\t_{k_{r}})&> A^{-1}\,\max_{j\in I_{r}}\,n_{l}(\t_{l})^{e(k_{r},j)/e(l,j)}\\&=A^{-1}n_{l}(\t_{l})^{\max_{j\in I_{r}}\,e(k_{r},j)/e(l,j)}=A^{-1}n_{l}(\t_{l})^{1/\tau_{S}(l,k_{r})}.
\eeas 
On the other hand, if $n_{l}(\t_{l})< 1$ then we still have $n_{k_{r}}(\t_{k_{r}})>A^{-1}n_{l}(\t_{l})^{1/\tau_{S}(l,k_{r})}$ since $n_{k_{r}}(\t_{k_{r}})>A^{-1}$ by Remarks \ref{Rem5.2}. In either case, this establishes both inequalities of (\ref{5.3aa}). To prove the converse, let $l\in \{1, \ldots, n\}$ with $l\neq k_{r}$. If $n_{l}(\t_{l})<1$, then since $n_{k_{r}}(\t_{k_{r}})>A^{-1}$ it follows that $n_{l}(\t_{l})^{1/e(l,j)}<1<\big[A\,n_{k_{r}}(\t_{k_{r}})\big]^{1/e(k_{r},j)}$. If $n_{l}(\t_{l})\geq 1$. There exists $p\in \{1, \ldots, s\}$ so that $l\in I_{p}$, and by the first inequality in 
(\ref{5.3aa}) it follows that $n_{l}(\t_{l})^{1/\tau_{S}(l,k_{p})}<A\,n_{k_{p}}(\t_{k_{p}})$. On the other hand, the second inequality in (\ref{5.3aa}) says that $n_{k_{p}}(\t_{k_{p}})\leq A^{\tau_{S}(k_{p},k_{r})}\,n_{k_{r}}(\t_{k_{r}})$. Also, Proposition \ref{Prop5.8}, part (\ref{Prop5.8b}) says that $\tau_{S}(l,k_{r})\geq \tau_{S}(l,k_{p})\tau_{S}(k_{p},k_{r})$. It follows that 
\beas
n_{l}(\t_{l})^{1/\tau_{S}(l,k_{r})}&\leq 
n_{l}(\t_{l})^{1/\tau_{S}(l,k_{p})\tau_{S}(k_{p},k_{r})}<\big[A\,n_{k_{p}}(\t_{k_{p}})\big]^{1/\tau_{S}(k_{p},k_{r})} 
<
A^{1+1/\tau_{S}(k_{p},k_{r})}n_{k_{r}}(\t_{k_{r}}).
\eeas
If $j\in I_{r}$ and $l\neq k_{r}$,
\beas
n_{l}(\t_{l})^{1/e(l,j)}\leq \big[n_{l}(\t_{l})^{1/\tau_{S}(l,k_{r})}\big]^{1/e(k_{r},j)} < \big[A^{1+1/\tau_{S}(k_{p},k_{r})}n_{k_{r}}(\t_{k_{r}})\big]^{1/e(k_{r},j)}.
\eeas 
Thus in all cases we have $n_{l}(\t_{l})^{1/e(k_{r},l)}<\big[A^{\eta}\,n_{k_{r}}(\t_{k_{r}})\big]^{1/e(k_{r},j)}$, and this says that $\t_{k_{r}}$ is $A^{\eta}$-dominant in $\widehat N_{j}(\t)$, which completes the proof of (\ref{Thm5.6B}) and of Lemma \ref{Thm5.6}.
\end{proof}

\subsection{Estimates of kernels and multipliers on $E_{S}$ and $\widehat E_{S}$}\label{Sec3.4iou}\quad

\medskip

The differential inequalities for kernels and multipliers take simpler forms on the sets $E_{S}$ and $\widehat E_{S}$. First consider the principal marked partition $S_{0}=\big\{(\{1\},1);\ldots;(\{n\},n)\big\}$. If $\x\in \B(1)\cap E_{S_{0}}$ then $N_{k}(\x)\approx n_{k}(\x_{k})$, and if $\xib\in \widehat E_{S_{0}}$ then $\widehat N_{k}(\xib)\approx n_{k}(\xib_{k})$.  Thus the differential inequalities for $\KK\in \PP(\EEE)$ and $m\in \MM(\EEE)$ take the following form on the sets $E_{S_{0}}$ and $\widehat E_{S_{0}}$:
\beas
\x\in E_{S_{0}}&&&\Longrightarrow&\big|\partial^{\gammab}K(\x)\big|&\leq C_{\gammab}\prod_{k=1}^{n}n_{k}(\x_{k})^{-(Q_{k}+\[\gammab_{k}\])},\\
\xib\in\widehat E_{S_{0}}&&&\Longrightarrow& \big\vert\partial^{\gammab}m(\xib)\big\vert&\leq C_{\gammab}\prod_{k=1}^{n}n_{k}(\xib_{k})^{-\[\gammab_{k}\]}.
\eeas
Note that these are precisely the differential inequalities satisfied by product kernels or product multipliers  on $\R^{N}= \bigoplus_{j=1}^{n}\R^{C_{j}}$. (See \cite{MR1818111}, page 34 and page 37.) Thus although a distribution $\KK\in \PP(\EEE)$ is pseudo-local, if the  set $E_{S_{0}}\neq \emptyset$ the function $K$ satisfies nothing better than the estimates for a product kernel on this set.  

We have the following extension for non-empty sets $E_{S}$ and $\widehat E_{S}$.

\begin{proposition}\label{Lem5.15} Let $S=\big\{(I_{1},k_{1});\ldots; (I_{s};k_{s})\big\}\in \SS(n)$. 
\begin{enumerate}[{\rm(a)}]

\item\label{Lem5.15a}
If $E_{S}\cap \B(1) \neq \emptyset$, if $\KK\in \PP(\EEE)$, 
and if $K\in \CC^{\infty}(\R^{N}\setminus\{0\})$ is the associated smooth function, then for every $\gammab=(\gammab_{1},\ldots,\gammab_{n})$, there is a constant $C_{\gammab}$ so that for $\x\in E_{S}\cap \B(1)$ the inequality in Definition \ref{Def2.2}, part (\ref{Def2.2A}) is equivalent to 
\bes
|\partial^{\gammab}K(\x)|\leq C_{\gammab}\prod_{r=1}^{s}n_{k_{r}}(\x_{k_{r}})^{-\sum_{j\in I_{r}}e(j,k_{r})(Q_{j}+\[\gammab_{j}\])}.
\ees

\item \label{Lem5.15b}
If $\widehat E_{S}\cap \B(1)^{c}\neq \emptyset$ and if $m\in \MM(\EEE)$ then for every $\gammab=(\gammab_{1}, \ldots, \gammab_{n})$ there is a constant $C_{\gammab}$ so that for $\xib\in \widehat E_{S}$ the inequality in  Definition \ref{Def2.4} is equivalent to
\bes
\left|\partial^{\gammab}m(\xib)\right| \leq C_{\gammab}\prod_{r=1}^{s} n_{k_{r}}(\xib_{k_{r}})^{-\sum_{j\in I_{r}}e(j,k_{r})\[\gammab_{j}\]}.
\ees
\end{enumerate}
\end{proposition}
\begin{proof}
If $\x\in E_{S}$, then $N_{j}(\x)\approx n_{k_{r}}(\x_{k_{r}})^{e(j,k_{r})}$ for every $j\in I_{r}$. Thus if $\x\in E_{S}$ 
\beas
|\partial^{\gammab}K(\x)|
&\leq 
C_{\gammab}\prod_{j=1}^{n}N_{j}(\x)^{-(Q_{j}+\[\gammab_{j}\])}
\approx
C_{\gammab}\prod_{r=1}^{s}\prod_{j\in I_{r}}n_{k_{r}}(\x_{k_{r}})^{-e(j,k_{r})(Q_{j}+\[\gammab_{j}\])}.
\eeas
This establishes part (\ref{Lem5.15a}), and again part (\ref{Lem5.15b}) is done the same way.
\end{proof} 

\subsection{A coarser decomposition of $\R^{N}$ associated to $S\in \SS(n)$}\label{Sec3.5iou}\quad

\medskip

Proposition \ref{Lem5.15} suggests that, on each non-empty set $E_S$, a kernel $\KK\in \PP_0(\EEE)$ satisfies the same differential inequalities of a product kernel adapted to a coarser decomposition of $\R^N$ depending on $S$ and, in a similar way, that on each non-empty set $\widehat E_S$,  a multiplier $m\in \MM_\infty(\EEE)$ behaves like a product multiplier.

For each $1 \leq r \leq s$ we gather the coordinates $\{\t_{j}:j\in I_{r}\}$ into a single subspace as discussed on page \pageref{exponent}. Thus for $1 \leq r \leq s$ set 
\be\label{5.9aa}
\R^{I_{r}}= \bigoplus_{j\in I_{r}} \R^{C_{j}}
\ee
so that $\R^{N}= \R^{I_{1}}\oplus \cdots \oplus \R^{I_{s}}$. Note that if $s<n$, this decomposition is strictly coarser than the original decomposition $\R^{N}= \R^{C_{1}}\oplus \cdots \oplus \R^{C_{n}}$.
If $\t\in \R^{N}$ write $\t= (\t_{I_{1}}, \ldots, \t_{I_{S}})$ where $\t_{I_{r}}=\big\{\t_{j}:j\in I_{r}\big\}$. 

Next we introduce two new families of dilations $\x_{I_{r}}\to \lambda\cdot_{S}\x_{I_{r}}$ and $\xib_{I_{r}}\to\lambda\,\hat\cdot_{S}\,\xib_{I_{r}}$  on each component $\R^{I_{r}}$. These should have the property that if $\x\in E_{S}$ then the associated norm of $\x_{I_{r}}$ should be comparable to $n_{k_{r}}(\x_{k_{r}})$, and if $\xib\in \widehat E_{S}$, then the associated norm of $\xib_{I_{r}}$ should be comparable to $n_{k_{r}}(\xib_{k_{r}})$. Recall from Corollary \ref{Cor5.7} that
\bea\label{njvsnkr}
\x\in E_{S}&&&\Longrightarrow &&&n_{j}(\x_{j})^{1/e(j,k_{r})}&<n_{k_{r}}(\x_{r})\\
\xib\in \widehat E_{S}&&&\Longrightarrow &&&n_{j}(\xib_{j})^{\,e(k_{r},j)\,}\,\,&<n_{k_{r}}(\xib_{k_{r}}),
\eea
and from Proposition \ref{Prop5.8} that if $j\in I_{r}$ then $e(j,k_{r})=\sigma_{S}(k_{r},j)^{-1}$ and $e(k_{r},j)=\tau(j,k_{r})^{-1}$.
\begin{definition}\label{Def3.12iou}
Let $S=\big\{(I_{1},k_{1});\ldots;(I_{s},k_{s})\big\}\in \SS(n)$. \index{l2ambdaS@$\lambda\cdot_{S}\x_{I_{r}},\lambda \,\hat\cdot_{S}\,\xib_{I_{r}}$}

\begin{enumerate}[{\rm(a)}]
\smallskip

\item If $\x=(\x_{I_{1}}, \ldots, \x_{I_{s}})\in \R^{N}$ with $\x_{I_{r}}=\big\{\x_{j}:j\in I_{r}\}\in \R^{I_{r}}$ set
\bes
\lambda\cdot_{S}\x_{I_{r}}=\Big(\lambda^{e(j,k_{r})}\cdot \x_{j}:j\in I_{r}\Big)=\Big(\lambda^{1/\sigma_{S}(k_{r},j)}\cdot \x_{j}:j\in I_{r}\Big).
\ees
Let $n_{S,r}$ be a smooth homogeneous norm on $\R^{I_{r}}$ for this dilation.

\smallskip

\item If $\xib=(\xib_{I_{1}}, \ldots, \xib_{I_{s}})\in \R^{N}$ with $\xib_{I_{r}}=\big\{\xib_{j}:j\in I_{r}\big\}$ let
\bes
\lambda \,\hat\cdot_{S}\,\xib_{I_{r}}=\Big(\lambda^{1/e(k_{r},j)}\cdot \xib_{j}:j\in I_{r}\Big)=\Big(\lambda^{\tau_{S}(j,k_{r})}\cdot \xib_{j}:j\in I_{r}\Big).
\ees
Let $\widehat n_{S,r}$ be a smooth homogeneous norm on $\R^{I_{r}}$ for this dilation.
\end{enumerate}
\end{definition}

\noindent Observe that \index{n3Sr@$n_{S,r},\widehat n_{S,r}$}
\bea\label{5.11aa}
n_{S,r}(\x_{I_{r}}) &\approx \sum_{j\in I_{r}}n_{j}(\x_{j})^{1/e(j,k_{r})}=\sum_{j\in I_{r}}n_{j}(\x_{j})^{\sigma_{S}(k_{r},j)},\\
\widehat n_{S,r}(\xib_{I_{r}}) & \approx \sum_{j\in I_{r}}n_{j}(\xib_{j})^{e(k_{r},j)}=\sum_{j\in I_{r}}n_{j}(\xib_{j})^{1/\tau_{S}(j,k_{r})},
\eea
and that the homogeneous dimensions of $\R^{I_{r}}$ under the dilations $\lambda\cdot_{S}$ and $\lambda\,\hat\cdot_{S}$ are given by
\bea\label{5.12aa} \index{Q2Sr@$Q_{S,r},\widehat Q_{S,r}$}
Q_{S,r}&= \sum_{j\in I_{r}}Q_{j}e(j,k_{r})= \sum_{j\in I_{r}}Q_{j}\sigma_{S}(k_{r},j)^{-1},\\
\widehat Q_{S,r}&= \sum_{j\in I_{r}}Q_{j}e(k_{r},j)^{-1}= \sum_{j\in I_{r}}Q_{j}\tau_{S}(j,k_{r}).
\eea

Moreover, it follows directly from  and \eqref{2.4} and \eqref{5.11aa} that, for $1\le r\le s$,
\bea\label{new}
n_{S,r}(\t_{I_r})&\lesssim \widehat N_{k_r}(\t),\\
\widehat n_{S,r}(\t_{I_r})&\lesssim  N_{k_r}(\t),
\eea
for all $\t\in\R^N$.

\section{Fourier transform duality of kernels and multipliers}\label{Duality}

In this section we show that the Fourier transforms of kernels $\KK\in \PP_{0}(\E)$ are multipliers $m\in\MM_{\infty}(\EEE)$, and conversely, that the inverse Fourier transform of multipliers $m\in \MM_{\infty}(\EEE)$ are kernels $\KK\in \PP_{0}(\EEE)$. Such results are well known for \CZ kernels and multipliers, and more generally, for product kernels and product multipliers or for flag kernels and flag multipliers. (See \cite{MR1818111} for details.) 

\subsection{Fourier transforms of multipliers} \quad

\medskip

Let $m\in L^{\infty}(\R^{N})$ and let $\KK = \widehat m$ be its Fourier transform in the sense of tempered distributions. Thus if $\varphi\in\SS(\R^{N})$ we have $\big\langle \KK,\varphi\big\rangle = \big\langle m,\widehat\varphi\big\rangle = \int_{\R^{N}}m(\xib)\widehat\varphi(\xib)\,d\xib$. Let $\chi\in \CC^{\infty}_{0}(\R^{N})$ be identically equal to one in a neighborhood of the origin, and let $\chi_{\epsilon}(\xib) = \chi(\epsilon\xib)$. Then 
$$
\big\langle \KK,\varphi\big\rangle
=
\int_{\R^{N}}m(\xib)\widehat\varphi(\xib)\,d\xib
=
\lim_{\epsilon\to 0}\int_{\R^{N}}\chi_{\epsilon}(\xib)m(\xib)\widehat\varphi(\xib)\,d\xib 
= 
\lim_{\epsilon\to0}\big\langle \widehat{\chi_{\epsilon}m},\varphi\big\rangle
$$
so $\KK = \lim_{\epsilon\to0}\widehat{\chi_{\epsilon}m}$ in the sense of distributions. Thus in making estimates of the Fourier transform $\widehat m$, we can assume that $m$ has compact support (and hence can freely integrate by parts) provided that we do not use any information about the size of this support in our estimates. 

\begin{theorem}\label{Thm6.1} 
Let $m\in \MM_{\infty}(\EEE)$. Then the Fourier transform $\widehat m$ is a distribution $\KK$ belonging to the class $ \PP_{0}(\EEE)$.
\end{theorem}

\begin{proof}
Assume as discussed above that $m$ has compact support. The Fourier transform is given by $K(\x) = \widehat m(\x) = \int_{\R^{N}}e^{2\pi i \langle\x,\xib\rangle}m(\xib)\,d\xib$, and for every $\gammab\in\N^{N}$
\be\label{6.1aa}
\partial^{\gammab}K(\x) = (2\pi i)^{|\gamma|}\int_{\R^{N}}e^{2\pi i \langle\x,\xib\rangle}\,\xib^{\gammab}\,m(\xib)\,d\xib.
\ee
To show that $\KK = \widehat m \in \PP_{0}(\EEE)$, we must verify the differential inequalities of Definition \ref{Def2.2} and Definition \ref{Def2.5} as well as the the cancellation conditions of Definition \ref{Def2.2}. We begin with the differential inequalities.

If $\gammab\in\N^{C_{1}}\times \cdots \times \N^{C_{n}}$ and $M\geq 0$, we must show that there are constants $C_{\gammab},\,C_{\gammab,M}>0$ (independent of the support of $m$) so that
\begin{equation} \label{6.2aa}
\begin{aligned}
\big\vert\partial^{\gammab}K(\x)\big\vert&\leq C_{\gammab}\prod_{j=1}^{n}N_{j}(\x)^{-(Q_{j}+\[\gammab_{j}\])}\quad\text{if $\x\in \B(1)$,}\\
\big\vert\partial^{\gammab}K(\x)\big\vert&\leq C_{\gammab,M}(1+|\x|)^{-M}\quad\text{if $\x\notin \B(1)$,}
\end{aligned}
\end{equation}
where $|\x|$ is the ordinary Euclidean norm of $\x\in \R^{N}$. 

The second inequality in (\ref{6.2aa})  follows from a standard integration by parts. Let $\Delta$ denote the ordinary Laplace operator. Then
\beas
\int_{\R^{N}}e^{2\pi i \langle\x,\xib\rangle}\,\xib^{\gammab}\,m(\xib)\,d\xib
&=
(4\pi^{2}|\x|)^{-2M}\int_{\R^{N}}(-\Delta_{\xib})^{M}\left[e^{2\pi i \langle\x,\xib\rangle}\right]\,\xib^{\gammab}\,m(\xib)\,d\xib\\
&=
(4\pi^{2}|\x|)^{-2M}\int_{\R^{N}}e^{2\pi i \langle\x,\xib\rangle}\,(-\Delta_{\xib})^{M}\left[\xib^{\gammab}\,m(\xib)\right]\,d\xib.
\eeas
If $M$ is large enough, $\big\vert\Delta_{\xib}^{M}\left[\xib^{\gammab}\,m(\xib)\right]\big\vert \leq C_{M} \big(1+|\xib|\big)^{-N-1}$,  so the last integral converges absolutely independently of the support of $m$. 

We turn to the heart of the matter which is the first estimate in (\ref{6.2aa}). Suppose that $\x\in \B(1) \cap E_{S}$ where $S=\big\{(I_{1},k_{1});\ldots;(I_{s},k_{s})\big\}\in \SS(n)$.  By \eqref{njvsnkr}, for $1 \leq r\leq s$, if $j\in I_{r}$ we have
\be\label{NjapproxnS}
N_{j}(\x) \approx n_{k_{r}}(\x_{k_{r}})^{e(j,k_{r})}\approx n_{S,r}(\x_{I_{r}}),
\ee
where $n_{S,r}(\x_{I_{r}}) \approx \sum_{j\in I_{r}}n_{j}(\x_{j})^{1/e(j,k_{r})}$ is the norm defined in \eqref{5.11aa}.
 Hence, for $\x\in \B(1) \cap E_{S}$, the estimate (\ref{6.2aa}) is equivalent to the estimate
\be\label{6.4aa}
\big|\partial^{\gammab}K(\x)\big|\lesssim \prod_{r=1}^{s}n_{k_{r}}(\x_{k_{r}})^{-\sum_{j\in I_{r}}e(j,k_{r})(Q_{j}+\[\gammab_{j}\])},
\ee 
which, by \eqref{NjapproxnS}, may also be interpreted as the differential inequality of a product kernel for the decomposition $\R^N=\R^{I_1}\oplus\cdots\oplus\R^{I_s}$ with the dilations $\cdot_S$ such that
$\lambda\cdot_{S}\x_{I_{r}}= \big\{\lambda^{e(j,k_{r})}\cdot\x_{j}:j\in I_{r}\big\}$.

To obtain \eqref{6.4aa} we adapt the standard method used for proving that product kernels and multipliers are related via the Fourier transform, cf. Theorem 2.1.11 in  \cite{MR1818111}. This consists in splitting the Fourier integral \eqref{6.1aa} into $2^s$ regions depending on whether $n_{S,r}(\xib_{I_{r}})$ is larger or smaller\footnote{It may seem strange that we measure the size of $\xib_{I_{r}}$ with $n_{S,r}(\xib_{I_{r}})$ instead of $\widehat n_{s,r}(\xib_{I_{r}})$. The point is that we have estimates of derivatives of $m$ in terms of $\widehat N_{k_{r}}(\xib)$ and we want a bound from below on this quantity that involves a norm on $\R^{I_{r}}$. By \eqref{new}, we need to use $n_{S,r}(\xib_{I_{r}})$.} than $n_{k_{r}}(\x_{k_{r}})^{-1}$ for $1\leq r \leq s$. For any $r$, we take advantage of the smallness of the domain of integration if we are considering  $n_{S,r}(\xib_{I_{r}})<n_{k_{r}}(\x_{k_{r}})^{-1}$, while we integrate  by parts in $\xib_{k_r}$ in the other case\footnote{  We are not claiming that $m$ is a product multiplier for the coarser decomposition. If it were so, the estimate 
 $$
 \big|\partial^{\gammab}K(\x)\big|\lesssim \prod_{r=1}^{s}n_{S,r}(\x_{I_{r}})^{-\sum_{j\in I_{r}}e(j,k_{r})(Q_{j}+\[\gammab_{j}\])},
$$ would hold for every $\x\in\R^N$, which is not true. Also notice that formula \eqref{3.2.7a} below, which gives a differential inequality typical of a product multiplier, only holds for derivatives in the $\xib_{k_r}$.}.

 Let $\varphi_{0}$ and $\varphi_{1}$ be smooth functions on the positive half-line with $\varphi_{0}(t)+\varphi_{1}(t) \equiv 1$, $\varphi_{0}(t) \equiv 1$ for $t\leq \frac{1}{2}$ and $\varphi_{0}(t) \equiv 0$ for $t\geq 2$. For $\epsilon =(\epsilon_{1},\ldots, \epsilon_{s})\in \{0,1\}^{s}$ set
\beas
m_{\epsilon}(\xib) &= m(\xib)\,\prod_{r=1}^{s}\varphi_{\epsilon_{r}}\big( n_{S,r}(\xib_{I_{r}})\,n_{k_{r}}(\x_{k_{r}})\big),\\
K_{\epsilon}(\x) &= \int_{\R^{N}}e^{2\pi i \langle \x,\xib\rangle}m_{\epsilon}(\xib)\,d\xib.
\eeas
Note that $m(\xib) = \sum_{\epsilon\in \{0,1\}^{s}}m_{\epsilonb}(\xib)$ and $K(\x) = \sum_{\epsilon\in\{0,1\}^{s}}K_{\epsilon}(\x)$, so it will suffice to show that
\be\label{6.5aa}
\big|\partial^{\gammab}K_{\epsilon}(\x)\big|\lesssim \prod_{r=1}^{s}n_{k_{r}}(\x_{k_{r}})^{-\sum_{j\in I_{r}}e(j,k_{r})(Q_{j}+\[\gammab_{j}\])}.
\ee

Since $m\in \MM(\EEE)$, and since $\widehat N_{k_{r}}(\xib)\gtrsim n_{S,r}(\xib_{I_{r}})$ by \eqref{new},
we have 
\bea\label{3.2.7a}
\big\vert\partial^{\deltab_{k_{1}}}_{\xib_{k_{1}}}\cdots \partial^{\deltab_{k_{s}}}_{\xib_{k_{s}}}m(\xib)\big\vert
&\lesssim
\prod_{r=1}^{s}\widehat N_{k_{r}}(\xib)^{-\[\deltab_{k_{r}}\]}
\lesssim
\prod_{r=1}^{s} n_{S,r}(\xib_{I_{r}})^{-\[\deltab_{k_{r}}\]}.
\eea
Since $n_{S,r}(\xib_{I_{r}})$ is homogeneous of degree one under the dilations $\lambda\,\cdot_{S}\,\xib_{I_{r}}$, and since $e(k_{r},k_{r}) = 1$, it follows from Proposition \ref{Prop12.1} in Appendix II that  $\big\vert\partial^{\gamma}_{\xib_{k_{r}}}n_{S,r}(\xib_{I_{r}})\big\vert \leq C_{\gamma}n_{S,r}(\xib_{I_{r}})^{1-\[\gamma\]}$ for all multi-indices $\gamma$,. Using the chain rule, we see that $\partial^{\deltab_{k_{r}}}_{\xib_{k_{r}}}\Big[\varphi_{\epsilon_{r}}\big( n_{S,r}(\xib_{I_{r}})\,n_{k_r}(\x_{k_{r}})\big)\Big]$ is a sum of terms of the form
\beas
\varphi_{\epsilon_{r}}^{(m)}\big(n_{S,r}(\xib_{I_{r}})n_{k_r}(\x_{k_{r}})\big)\,n_{k_r}(\x_{k_{r}})^{m}\,\partial^{\alphab_{1}}_{\xib_{k_{r}}}n_{S,r}(\xib_{I_{r}})\cdots \partial^{\alphab_{m}}_{\xib_{k_{r}}}n_{S,r}(\xib_{I_{r}})
\eeas
where $1\leq m \leq |\deltab_{k_{r}}|$ and $\alphab_{1}+ \cdots +\alphab_{m}= \delta_{k_{r}}$. Since $n_{k_r}(\x_{k_{r}})\approx n_{S,r}(\xib_{I_{r}})^{-1}$ on the support of $\varphi^{(m)}_{\epsilon_{r}}$, $m\geq 1$, it follows that
\beas
\Big\vert
\varphi_{\epsilon_{r}}^{(m)}\big(n_{S,r}(\xib_{I_{r}})n_{k_r}(\x_{k_{r}})\big)\,n_{k_r}(\x_{k_{r}})^{m}\,\partial^{\alphab_{1}}_{\xib_{k_{r}}}n_{S,r}(\xib_{I_{r}})\cdots \partial^{\alphab_{m}}_{\xib_{k_{r}}}n_{S,r}(\xib_{I_{r}})\Big\vert
&\lesssim
n_{S,r}(\xib_{I_{r}})^{-\[\deltab_{k_{r}}\]}.
\eeas
Combining these estimates gives $\big|\partial^{\deltab_{k_{1}}}_{\xib_{k_{1}}}\cdots \partial^{\deltab_{k_{s}}}_{\xib_{k_{s}}}m_{\epsilon}(\xib)\big|\leq C_{\deltab}\,\prod_{r=1}^{s} n_{S,r}(\xib_{I_{r}})^{-\[\deltab_{k_{r}}\]}$.

Now for each $\epsilon=(\epsilon_{1}, \ldots, \epsilon_{s})\in \{0,1\}^{s}$ let $A(\epsilon)=\{r\in \{1,\ldots, s\}:\epsilon_{r}=1\}$. Note that $m_{\epsilon}$ is supported where
\beas
n_{S,r}(\xib_{I_{r}})&\leq 2n_{k_r}(\x_{k_{r}})^{-1} &&&\text{if $r\notin A(\epsilon)$},\\
n_{S,r}(\xib_{I_{r}})& \geq\frac{1}{2}n_{k_r}(\x_{k_{r}})^{-1} &&&\text{if $r\in A(\epsilon)$}.
\eeas
Choose $\deltab=(\deltab_{1}, \ldots, \deltab_{n})\in \N^{C_{1}}\times \cdots \times\N^{C_{n}}$ with $\deltab_{j}= 0$ unless $j=k_{r}$ for some $r\in A(\epsilon)$. If $r\in A(\epsilon)$, we will choose the entries of each $\deltab_{k_{r}}=(\delta_{k_{r},1}, \ldots, \delta_{k_{r},c_{k_{r}}})$ to be much larger than the corresponding entries of $\gammab_{k_{r}}=(\gamma_{k_{r},1}, \ldots, \gamma_{k_{r},c_{k_{r}}})$. Using (\ref{6.1aa}) we integrate by parts to obtain
\bea\label{3.2.8a}
\big(\prod_{r\in A(\epsilonb)}\x_{k_{r}}^{\deltab_{k_{r}}}\big)\partial^{\gammab}_{\x}K_{\epsilonb}(\x) 
&=
(2\pi i)^{|\gammab|}
\int_{\R^{N}}\big(\prod_{r\in A(\epsilonb)}\x_{k_{r}}^{\deltab_{k_{r}}}\big)e^{2\pi i \langle \x,\xib\rangle}\xib^{\gammab}m_{\epsilon}(\xib)\,d\xib\\
&=
(2\pi i)^{|\gammab|-|\deltab|}
\int_{\R^{N}}\big[\prod_{r\in A(\epsilonb)}\partial^{\deltab_{k_{r}}}_{\xib_{k_{r}}}\big]\big(e^{2\pi i \langle \x,\xib\rangle}\big)\,\xib^{\gammab}m_{\epsilon}(\xib)\,d\xib\\
&=
(-1)^{|\deltab|}(2\pi i)^{|\gammab|-|\deltab|}
\int_{\R^{N}}e^{2\pi i \langle \x,\xib\rangle}\,\big[\prod_{r\in A(\epsilonb)}\partial^{\deltab_{k_{r}}}_{\xib_{k_{r}}}\big]\big(\xib^{\gammab}m_{\epsilon}\big)(\xib)\,d\xib.
\eea
The product rule shows that $\big[\prod_{r\in A(\epsilonb)}\partial^{\deltab_{k_{r}}}_{\xib_{k_{r}}}\big]\big(\xib^{\gammab}m_{\epsilonb}\big)(\xib)$ can be written as a finite sum of terms of the form
\bea\label{6.8aa}
\Big(\prod_{r\in A(\epsilonb)}\partial^{\mub_{k_{r}}}_{\xib_{k_{r}}}\Big)\big[\xib^{\gammab}\big]\,\,\Big(\prod_{r\in A(\epsilonb)}\partial^{\deltab_{k_{r}}-\mub_{k_{r}}}_{\xib_{k_{r}}}\Big)\big[m_{\epsilon}\big](\xib)
\eea
where $\mub = (\mub_{1}, \ldots, \mub_{n})\leq (\deltab_{1},\ldots, \deltab_{n})$.
The first factor is zero if the entries of $\mub_{k_{r}}$ are larger than the corresponding terms $\gammab_{k_{r}}$. Thus our choice of $\deltab_{k_{r}}$ guarantees that the entries of $\deltab_{k_{r}}-\mub_{k_{r}}$ are large enough to make the integrals below converge independently of the compact support of $m$. Note that $\mub_{j}=\0$ unless $j=k_{r}$ for some $r\in A(\epsilon)$. Then
\beas
\Big(\prod_{r\in A(\epsilonb)}\partial^{\mub_{k_{r}}}_{\xib_{k_{r}}}\Big)\big[\xib^{\gammab}\big]
&= 
C_{\mub,\gammab}
\prod_{\substack {r\in A(\epsilon)}}\Big(\prod_{\substack{j\in I_{r}}}\xib_{j}^{\gammab_{j}-\mub_{j}}\Big)\,
\prod_{r\notin A(\epsilon)}\big(\prod_{\substack{j\in I_{r}}}\xib_{j}^{\gammab_{j}}\big)
\eeas
According to (\ref{3.2.7a}) we also have
\beas
\Big|\Big(\prod_{r\in A(\epsilonb)}\partial^{\deltab_{k_{r}}-\mub_{k_{r}}}_{\xib_{k_{r}}}\Big)\big[m_{\epsilon}\big](\xib)\Big|
\leq 
C_{\deltab,\mub}\prod_{r\in A(\epsilon)} n_{S,r}(\xib_{I_{r}})^{-\[\deltab_{k_{r}}\]+\[\mub_{k_{r}}\]}
\eeas
Thus
\bea\label{6.9aa}
\Big\vert \Big(&\prod_{r\in A(\epsilonb)}\partial^{\mub_{k_{r}}}_{\xib_{k_{r}}}\Big)\big[\xib^{\gammab}\big]\,\,\Big(\prod_{r\in A(\epsilonb)}\partial^{\deltab_{k_{r}}-\mub_{k_{r}}}_{\xib_{k_{r}}}\Big)\big[m_{\epsilon}\big](\xib)\Big\vert\\
&\leq 
C_{\deltab,\gammab}\,
\prod_{r\in A(\epsilon)} n_{S,r}(\xib_{I_{r})^{-\[\deltab_{k_{r}}\]+\[\mub_{k_{r}}\]} }
\prod_{\substack {r\in A(\epsilon)}}\Big(\prod_{\substack{j\in I_{r}}}|\xib_{j}^{\gammab_{j}-\mub_{j}}|\Big)\,
\prod_{r\notin A(\epsilon)}\big(\prod_{\substack{j\in I_{r}}}|\xib_{j}^{\gammab_{j}}|\big).
\eea
Now using (\ref{3.2.8a}), (\ref{6.8aa}), and (\ref{6.9aa}), we can estimate $\displaystyle \big|\big(\prod_{r\in A(\epsilonb)}\x_{k_{r}}^{\deltab_{k_{r}}}\big)\partial^{\gammab}_{\x}K_{\epsilonb}(\x)\big|$ by a finite sum of terms of the form
\beas
&\prod_{r\in A(\epsilon)}
\Bigg[\,\,
\int\limits_{n_{S,r}(\xib_{I_{r}})\geq n_{k_r}(\x_{k_{r}})^{-1}}
\Big(\prod_{\substack{j\in I_{r}}}|\xib_{j}^{\gammab_{j}-\mub_{j}}|\Big)\,  n_{S,r}(\xib_{I_{r}})^{-\[\deltab_{k_r}\]+\[\mub_{k_r}\]} \,d\xib_{I_{r}}\Bigg]\\
&\qquad\qquad\qquad\qquad\qquad\qquad\qquad\qquad
\times\prod_{r\notin A(\epsilon)}
\Bigg[\,\,
\int\limits_{n_{S,r}(\xib_{I_{r}})<n_{k_r}(\x_{k_{r}})^{-1}}\Big(\prod_{j\in I_{r}}|\xib_{j}^{\gammab_{j}}|\Big)\,d\xib_{I_{r}}\Bigg].
\eeas 
According to Proposition \ref{Prop12.3} in the appendix, this last expression is bounded by
\beas
\prod_{r\notin A(\epsilon)}&n_{k_r}(\x_{k_{r}})^{-\sum_{j\in I_{r}}e(k_{r},j)(Q_{j}+\[\gammab_{j}\])}\,\,
\prod_{r\in A(\epsilon)}n_{k_r}(\x_{k_{r}})^{-\sum_{j\in I_{r}}e(k_{r},j)(Q_{j}+\[\gammab_{j}\]-\[\deltab_{r}\])}\\
&=
\prod_{r\in A(\epsilon)}n_{k_r}(\x_{k_{r}})^{\[\deltab_{r}\]}\prod_{r=1}^{s}n_{k_r}(\x_{k_{r}})^{-\sum_{j\in I_{r}}e(k_{r},j)(Q_{j}+\[\gammab_{j}\])}.
\eeas
Comparing with (\ref{3.2.8a}), this shows that $|\partial^{\gammab}_{\x}K(\x)|$ satisfies (\ref{6.4aa}), so the kernel $K$ satisfies the correct differential inequalities.

\bigskip

We also need to verify the cancellation conditions given in Definition \ref{Def2.2}, part (\ref{Def2.2B}). Let $\x=(\x',\x'')\in \R^{N_{1}}\times\R^{N_{2}}$ be a decomposition of the variables into two subsets, with $\x'=(\x_{p_{1}}, \ldots, \x_{p_{r}})$ and $\x''=(\x_{q_{1}}, \ldots, \x_{q_{t}})$. Let $\psi$ be a normalized bump function in the $\x''$ variables. Then
\beas
K_{\psi,R}(\x') &=
\int_{\R^{N_{2}}}K(\x',\x'')\psi(R\cdot \x'')\,d\x''\\
&=
\int_{\R^{N}}\int_{\R^{N_{2}}}e^{2\pi i <\x,\xib>}m(\xib)\psi(R\cdot \x'')\,d\x''\,d\xib\\
&=
\int_{\R^{N_{1}}}
\Bigg[\int_{\R^{N_{2}}}
\left[\int_{\R^{N_{2}}}e^{2\pi i <\x'',\xib''>}\psi(R\cdot\x'')\,d\x''\right]\,m(\xib',\xib'')\,d\xib''\Bigg]\,e^{2\pi i <\x',\xib'>}
\,d\xib'\\
&=
\int_{\R^{N_{1}}}
\left[\int_{\R^{N_{2}}}
\left[R^{-Q''}\widehat \psi(R^{-1}\cdot\xib'')\right]m(\xib',\xib'')\,d\xib''\right]
\,e^{2\pi i <\x',\xib'>}\,d\xib'\\
&=
\int_{\R^{N_{1}}}e^{2\pi i<\x',\xib'>}m_{\psi,R}(\xib')\,d\xib'
\eeas
where 
\bes
m_{\psi,R}(\xib') = \int_{\R^{N_{2}}}\left[R^{-Q''}\widehat \psi(R^{-1}\cdot\xib'')\right]m(\xib',\xib'')\,d\xib''=\int_{\R^{N_{2}}}\widehat \psi_{R}(\xib'')m(\xib',\xib'')\,d\xib''.
\ees 
Since $||\widehat\psi_{R}||_{L^{1}}$ is finite and independent of $R$, it is easy to see that $m_{\psi,R}$ belongs to the class $\MM_{\infty}(\widehat N_{p_{1}}, \ldots, \widehat N_{p_{r}})$ on the space $\R^{N_{1}}$. Thus the estimates for $K_{\psi,R}$ follow from the same arguments. This completes the proof of Theorem \ref{Thm6.1}.
\end{proof}

\subsection{Fourier transforms of kernels}\quad

\smallskip

\begin{theorem}\label{Thm6.2} 
Let $\KK\in \PP_{0}(\EEE)$. Then the Fourier transform of $\KK$ is a smooth function belonging to the class $\MM_{\infty}(\EEE)$.
\end{theorem}

\begin{proof}
Suppose that $\KK\in\PP_{0}(\EEE)$. According to Definition \ref{Def2.5} we can write $\KK=\KK_{0}+\psi$ where $\KK_{0}\in \PP_{0}(\EEE)$ has compact support in $\B(1)$ and $\psi\in \SS(\R^{N})$. Since $\widehat \psi\in \MM_{\infty}(\EEE)$, it suffices to show that $\widehat {\KK_{0}}\in \MM_{\infty}(\EEE)$. Thus without loss of generality we can assume that $\KK$ has compact support in the ball $\B(1)$. The distribution $\KK$ then acts on $\CC^{\infty}(\R^{N})$, and in particular, we set $m(\xib)=\big\langle\KK,\exp[2\pi i<\,\cdot\,,\xib>]\big\rangle$. From general principles, this is a smooth function and
\bea\label{6.10aa}
\partial^{\gammab}_{\xib}m(\xib)&=(2\pi i)^{|\gammab|}\big\langle\KK,\x^{\gammab}\exp[2\pi i<\,\cdot\,,\xib>]\big\rangle= (2\pi i)^{|\gammab|}\big\langle \x^{\gammab}\KK,\exp[2\pi i<\,\cdot\,,\xib>]\big\rangle\\
&=(2\pi i)^{|\gammab|}\int_{\R^{N}}e^{2\pi i <\x,\xib>}\x^{\gammab}K(\x)\,d\x.
\eea
Note that if $|\gamma|\geq 1$, the distribution $\x^{\gammab}\KK$ coincides with integration against the locally integrable function $\x^{\gammab}K(\x)$, (see Remark 2.1.7 on page 36 of \cite{MR1818111}), and this justifies the last equality.

We want to show
\be\label{6.11aa}
\big|\partial^{\gammab}_{\xib}m(\xib)\big|\leq C_{\gammab}\prod_{j=1}^{n} \big(1+\widehat N_{j}(\xib)\big)^{-\[\gammab_{j}\]}.
\ee
Since $\KK$ has compact support, this inequality follows if $\xib\in \B(1)$. Thus we only need to consider $\xib\in \B(1)^{c}$. Suppose $\xib$ belongs to $\widehat E_{S}$ where $S=\big\{(I_{1},k_{1});\ldots;(I_{s},k_{s})\big\}\in\SS(n)$. Then for $j\in I_{r}$, $1+\widehat N_{j}(\xib) \approx n_{k_{r}}(\xib_{k_{r}})^{1/e(k_{r},j)}$ and we want to prove
\bea\label{6.12aa}
\big|\partial^{\gammab}_{\xib}m(\xib)\big|
&\leq 
C_{\gammab}\prod_{r=1}^{s}\prod_{j\in I_{r}} \big(1+\widehat N_{j}(\xib)\big)^{-\[\gammab_{j}\]}
\approx
C_{\gammab}\prod_{r=1}^{s}n_{k_{r}}(\xib_{k_{r}})^{-\sum_{j\in I_{r}}\[\gammab_{j}\]/e(k_{r},j)}.
\eea

The proof now follows the same pattern as the proof of Theorem \ref{Thm6.1}.
To establish (\ref{6.12aa}) we split the integral in (\ref{6.10aa}) into $2^{s}$ regions depending on whether $\widehat n_{S,r}(\x_{I_{r}})$ is larger or smaller than $n_{k_{r}}(\xib_{k_{r}})^{-1}$. 

Let $\varphi_{0}$ and $\varphi_{1}$ be smooth functions on the positive half-line with $\varphi_{0}(t)+\varphi_{1}(t) \equiv 1$, $\varphi_{0}(t) \equiv 1$ for $t$ near zero, and $\varphi_{0}(t) \equiv 0$ for $t$ large. For $\epsilon =(\epsilon_{1},\ldots, \epsilon_{s})\in \{0,1\}^{s}$ set
\bea
K_{\epsilon}(\x) &= 
K(\x)\prod_{r=1}^{s}\varphi_{\epsilon_{r}}\big(n_{k_{r}}(\xib_{k_{r}})\widehat n_{S,r}( \x_{I_{r}})\big),\\
m_{\epsilon}(\xi) &= \int_{\R^{N}}e^{2\pi i \langle \x,\xib\rangle}K_{\epsilon}(\x)\,d\x.
\eea
Note that $m(\xib) = \sum_{\epsilon\in \{0,1\}^{s}}m_{\epsilonb}(\xib)$ and $K(\x) = \sum_{\epsilon\in\{0,1\}^{s}}K_{\epsilon}(\x)$. Recall that $\widehat n_{S,r}(\xib_{r})=\sum_{j\in I_{r}}n_{j}(\xib_{j})^{e(k_{r},j)}$ so $\widehat n_{S,r}$ is homogeneous relative to the family of dilations $\{\xi_{j};j\in I_{r}\}\to \{\lambda^{1/e(k_{r},j)}\xi_{j}:j\in I_{r}\}$. It follows from Proposition \ref{Prop12.1} that for $\deltab_{k_{r}}\in \N^{C_{k_{r}}}$ and $\epsilonb\in\{0,1\}^{s}$ we have
\be\label{4.3.6d}
\big|\partial^{\deltab_{k_{1}}}_{\x_{k_{1}}}\cdots \partial^{\deltab_{k_{s}}}_{\x_{k_{s}}}K_{\epsilon}(\x)\big|
\leq 
C_{\deltab}\,\prod_{j=1}^{n}N_{j}(\x)^{-Q_{j}}\prod_{r=1}^{s}\widehat n_{S,r}(\x_{k_{r}})^{-\[\deltab_{k_{r}}\]}
\ee
since $e(k_{r},k_{r})=1$.
Now fix $\epsilon\in \{0,1\}^{s}$ and let $A(\epsilon)=\{r\in \{1,\ldots, s\}:\epsilon_{r}=1\}$. Let $\deltab_{k_{r}}\in \N^{c_{k_{r}}}$.
We will choose the entries of each $\deltab_{k_{r}}$ to be sufficiently large integers. Then
\bea\label{4.3.7d}
\big(\prod_{r\in A(\epsilonb)}\xib_{k_{r}}^{\deltab_{r}}\big)\partial^{\gammab}_{\xib}m_{\epsilonb}(\xib) &=
(-2\pi i)^{|\gammab|}
\int_{\R^{N}}\big(\prod_{r\in A(\epsilonb)}\xib_{k_{r}}^{\deltab_{r}}\big)e^{-2\pi i \langle \x,\xib\rangle}\x^{\gammab}K_{\epsilon}(\x)\,d\x\\
&=
(-2\pi i)^{|\gammab|-|\deltab|}
\int_{\R^{N}}\big[\prod_{r\in A(\epsilonb)}\partial^{\deltab_{r}}_{\x_{k_{r}}}\big]\big(e^{-2\pi i \langle \x,\xib\rangle}\big)\,\x^{\gammab}K_{\epsilon}(\x)\,d\x\\
&=
(-1)^{|\deltab|}(-2\pi i)^{|\gammab|-|\deltab|}
\int_{\R^{N}}e^{-2\pi i \langle \x,\xib\rangle}\,\big[\prod_{r\in A(\epsilonb)}\partial^{\deltab_{r}}_{\x_{k_{r}}}\big]\big(\x^{\gammab}K_{\epsilon}\big)(\x)\,d\x.
\eea
The product rule shows that the derivative $\big[\prod_{r\in A(\epsilonb)}\partial^{\deltab_{r}}_{\x_{k_{r}}}\big]\big(\x^{\gammab}K_{\epsilonb}\big)(\x)$ can be written as a sum of terms of the form $\big(\prod_{r\in A(\epsilonb)}\partial^{\mub_{r}}_{\x_{k_{r}}}\big)\big(\x^{\gammab}\big)\,\,\big(\prod_{r\in A(\epsilonb)}\partial^{\deltab_{r}-\mub_{r}}_{\x_{k_{r}}}\big(K_{\epsilon}(\x)\big)\big)$. In the first factor, the entries of $\mub_{r}$ can only be as large as the corresponding terms $\gammab_{r}$. Thus we will choose $\deltab_{r}$ to have entries much larger than those of $\gammab_{r}$ so that the entries of $\deltab_{r}-\mub_{r}$ are large enough to make the integrals below converge independently of the compact support of $m$. Then
\beas
\Big(\prod_{r\in A(\epsilonb)}\partial^{\mub_{r}}_{\x_{k_{r}}}\Big)\big(\x^{\gammab}\big)
&=
C_{\mub,\gammab}
\prod_{\substack {r\in A(\epsilon)}}\big(\prod_{\substack{j\in I_{r}\\j\neq k_{r}}}\x_{j}^{\gammab_{j}}\big)\x_{k_{r}}^{\gammab_{k_{r}}-\mub_{k_{r}}}\,
\prod_{r\notin A(\epsilon)}\big(\prod_{\substack{j\in I_{r}}}\x_{j}^{\gammab_{j}}\big)
\eeas
Note that $e(j,l)\leq e(j,k_{r})e(k_{r},l)$. Thus
\bea\label{4.3.9d}
N_{j}(\x)&=\sum_{l=1}^{n}|\x_{l}|^{e(j,l)}
\geq
\sum_{l\in I_{r}}|\x_{l}|^{e(j,l)}
\geq
\sum_{l\in I_{r}}|\x_{l}|^{e(k_{r},l)e(J,k_{r})}\\
&\gtrsim
\Big[\sum_{L\in I_{r}}|\x_{l}|^{e(k_{r},l)}\Big]^{e(j,k_{r})}
=
\widehat n_{S,r}(\x_{I_{r}})^{e(j,k_{r})}.
\eea
According to (\ref{4.3.6d})
\beas
\Big|\Big(\prod_{r\in A(\epsilonb)}&\partial^{\deltab_{r}-\mub_{r}}_{\x_{k_{r}}}\big(K_{\epsilon}\big)(\x)\Big)\Big|
\leq 
C_{\deltab,\mub}\prod_{j=1}^{n}N_{j}(\x)^{-Q_{j}}\prod_{r\in A(\epsilon)}\widehat n_{S,r}(\x_{I_{r}})^{-\[\deltab_{r}\]+\[\mub_{r}\]}\\
&\leq
C_{\deltab,\mub}\prod_{r=s}\widehat n_{S,r}(\x_{I_{r}})^{-\sum_{j\in I_{r}}Q_{j}/e(k_{r},j)}\prod_{r\in A(\epsilon)}\widehat n_{S,r}(\x_{I_{r}})^{-\[\deltab_{r}\]+\[\mub_{r}\]}.
\eeas
Thus according to (\ref{3.2.8a}) and (\ref{4.3.9d}), we can estimate $\big\vert\big(\prod_{r\in A(\epsilonb)}\x_{k_{r}}^{\deltab_{r}}\big)\partial^{\gammab}_{\x}K_{\epsilonb}(\x)\big\vert$ by a sum of terms of the form
\beas
&\Big[\prod_{r\notin A(\epsilon)}\!\!\!\!\int\limits_{\substack{\R^{I_{r}}\\\widehat n_{S,r}(\x_{I_{r}})<n_{k_{r}}(\xib_{k_{r}})^{-1}}}\Big|\prod_{j\in I_{r}}\x_{j}^{\gammab_{j}}\Big|\,\widehat n_{S,r}(\x_{I_{r}})^{-\sum_{j\in I_{r}}Q_{j}e(j,k_{r})}\,d\x_{I_{r}}\Big]\times\\
&
\Big[\prod_{r\in A(\epsilon)}\!\!\!\!\int\limits_{\substack{\R^{I_{r}}\\\widehat n_{S,r}(\x_{I_{r}})\geq n_{k_{r}}(\xib_{k_{r}})^{-1}}}\Big|\big(\prod_{\substack{j\in I_{r}\\j\neq k_{r}}}\x_{j}^{\gammab_{j}}\big)\x_{k_{r}}^{\gammab_{k_{r}}-\mub_{k_{r}}}
\Big|\,\widehat n_{S,r}(\x_{I_{r}})^{-\sum_{j\in I_{r}}Q_{j}e(j,k_{r})-\[\deltab_{r}\]+\[\mub_{r}\]}\,d\x_{I_{r}}\Big]\\
&\approx
\Big[\prod_{r\notin A(\epsilon)}\!\!\!\!\int\limits_{\substack{\R^{I_{r}}\\\widehat n_{S,r}(\x_{I_{r}})<n_{k_{r}}(\xib_{k_{r}})^{-1}}}\Big|\prod_{j\in I_{r}}\x_{j}^{\gammab_{j}}\Big|\,\widehat n_{S,r}(\x_{I_{r}})^{-\sum_{j\in I_{r}}Q_{j}e(j,k_{r})}\,d\x_{I_{r}}\Big]\times\\
&
\Big[\prod_{r\in A(\epsilon)}\!\!\!\!\int\limits_{\substack{\R^{I_{r}}\\\widehat n_{S,r}(\x_{I_{r}})\geq n_{k_{r}}(\xib_{k_{r}})^{-1}}}\Big|\big(\prod_{\substack{j\in I_{r}}}\x_{j}^{\gammab_{j}}\big)\Big|
\widehat n_{S,r}(\x_{I_{r}})^{-\sum_{j\in I_{r}}Q_{j}e(j,k_{r})-\[\deltab_{r}\]}\,d\x_{I_{r}}\Big].
\eeas
It may happen that the integrals over the regions where $\widehat n_{S,r}(\x_{I_{r}})<n_{k_{r}}(\xib_{k_{r}})^{-1}$ diverge if all the appropriate $\gammab_{j}$ are equal to zero. However, for these terms, we use the cancellation properties of the kernels since $\varphi_{\epsilon_{r}}\big(n_{k_{r}}(\xib_{k_{r}})\widehat n_{S,r}(\x_{I_{r}})\big)$ are dilates of normalized bump functions. Thus according to Proposition \ref{Prop12.3}, this last expression is bounded by
\beas
\prod_{r\notin A(\epsilon)}&n_{k_{r}}(\xib_{k_{r}})^{-\sum_{j\in I_{r}}\[\gammab_{j}\]e(j,k_{r})}\,\,
\prod_{r\in A(\epsilon)}n_{k_{r}}(\xib_{k_{r}})^{\sum_{j\in I_{r}}(-(\[\gammab_{j}\]-\[\deltab_{r}\])e(j,k_{r}))}\\
&=
\prod_{r\in A(\epsilon)}n_{k_{r}}(\xib_{k_{r}})^{\[\deltab_{r}\]e(j,k_{r})}\prod_{r=1}^{s}n_{k_{r}}(\xib_{k_{r}})^{-\sum_{j\in I_{r}}\[\gammab_{j}\]e(j,k_{r})}.
\eeas
Comparing with (\ref{3.2.8a}), this shows that $|\partial^{\gammab}_{\xib}m(\xib)|$ satisfies (\ref{6.12aa}), so the multiplier $m$ satisfies the correct differential inequalities, and this completes the proof of Theorem \ref{Thm6.2}.\end{proof}

\medskip

\section{Dyadic sums of Schwartz functions}\label{Schwartz sums}

In Section \ref{Duality} we showed that distributions $\KK\in \PP_{0}(\EEE)$ can be characterized in terms of their Fourier transforms. We now begin the study of a different kind of characterization in terms of dyadic sums of normalized bump functions. In this section we show that appropriate sums of dilates do converge to distributions $\KK\in \PP_{0}(\EEE)$. In Theorem \ref{Thm8.5} below, we will see that, modulo `coarser' kernels, every distribution $\KK\in \PP_{0}(\EEE)$ can be decomposed in this way.

If $\varphi\in\SS(\R^{N})$, the dilates $[\varphi]_{\t}$ and $(\varphi)_{\t}$ were defined in equation (\ref{2.3}). In this section we show that appropriate dyadic sums of dilates $[\varphi]_{\t}$ of Schwartz functions having suitable cancellation properties converge to distributions $\KK\in \PP_{0}(\EEE)$. Similarly we show that dyadic sums of dilates $(\varphi)_{\t}$ of Schwartz functions with appropriate vanishing conditions converge to multipliers $\MM_{\infty}(\EEE)$. The main results appear in subsections \ref{SubSec5.4}-\ref{SubSec5.6}. However, before turning to the main results of this section, we begin with some further remarks about cones associated to matrices, and then discuss inclusions among classes of kernels and multipliers defined by different matrices.

\goodbreak

\subsection{Cones associated to a matrix $\EEE$}\label{CONES}\quad

\medskip

Let $\EEE=\{e(j,k)\}$ be an $n\times n$ matrix satisfying the basic hypotheses in (\ref{2.5}) and let $\mu \geq 0$. We introduce the following notation. \index{G2ammaE@$\Gamma(\EEE),\Gamma_{\mu}(\EEE),\Gamma^o(\EEE),\Gamma_{\Z,\mu}(\EEE),\Gamma_{\Z}(\EEE)$} 
\bea\label{5.1qwe}
\Gamma(\EEE)&=\Big\{\t=(t_{1}, \ldots, t_{n})\in \R^{n}:\forall j\neq k,\,e(j,k)t_{k}\leq t_{j}<0\Big\},\\
\Gamma_{\mu}(\EEE)&=\Big\{\t=(t_{1}, \ldots, t_{n})\in \R^{n}:\forall j\neq k,\,e(j,k)t_{k}-\mu\leq t_{j}<0\Big\},\\
\Gamma^{o}(\EEE)&= {\rm int\,}\big(\Gamma(\EEE)\big)=\Big\{\t=(t_{1}, \ldots, t_{n})\in \R^{n}:\forall j\neq k,\,e(j,k)t_{k}< t_{j}<0\Big\},\\
\Gamma_{\Z,\mu}(\EEE)&=\big\{I=(i_{1}, \ldots, i_{n})\in \Z^{n}:\forall j\neq k,\, e(j,k)i_{k}-\mu\leq i_{j}<0\Big\},\\
\Gamma_{\Z}(\EEE)&=\Gamma_{\Z,0}(\EEE).
\eea

 We  will refer to properties of the cone $\Gamma(\EEE)$ which are proved in Appendix I (Section \ref{Cones}). In particular, Lemma \ref{Lem5.1ghj} summarizes the content of Lemma \ref{Lem3.5} and Lemma \ref{Lem3.7}.

Since $\EEE$ satisfies the basic hypothesis, $\Gamma(\EEE)$ is not empty.\footnote{See Proposition \ref{Prop3.3} in Appendix I.} However its interior, which is an open convex cone, can be empty. For example, if $\EEE=\left[\begin{matrix}1&2\\\frac{1}{2}&1\end{matrix}\right]$, $\Gamma(\EEE)$ is the one-dimensional ray given by $x_{1}=2x_{2}$, $x_{2}< 0$, which has no interior. Note that in this case $1=e(1,2)e(2,1)$. The following Lemma shows that, more generally,  any such equality is the obstruction to the open cone being non-empty.

\begin{lemma}\label{Lem5.1ghj}
Let $\EEE$ be a matrix satisfying the basic hypotheses \eqref{2.5}.
\begin{enumerate}
\item[\rm(a)] The interior of $\Gamma(\EEE)$ is non-empty if and only if $\EEE$ is reduced, i.e., $1<e(j,k)e(k,j)$ for every pair $1\leq j\neq k \leq n$.
\smallskip
\item[\rm(b)] $\Gamma(\EEE)$ is contained in the intersection of the hyperplanes $\big\{\t\in\R^n:e(j,k)t_k=t_j\big\}$ for all pairs $(j,k)$ such that $e(j,k)e(k,j)=1$ and has non-empty interior in this subspace of $\R^n$.
\end{enumerate}  
\end{lemma}
\smallskip

\subsection{Inclusions among classes of kernels associated to different matrices}\label{subs.inclusions}\quad
\medskip

The cone $\Gamma(\EEE)$ contains the relevant information on the structure of kernels and multipliers associated with the matrix $\EEE$.

\begin{proposition}\label{inclusion}
Let $\EEE,\EEE'$ be two $n\times n$ matrices,  both satisfying the basic hypotheses. The following are equivalent:
\begin{enumerate}
\item[\rm(i)] $\cP_0( \EEE)\subseteq \cP_0(\EEE')$;
\item[\rm(ii)] $\cM_\infty(\EEE)\subseteq \cM_\infty(\EEE')$;
\item[\rm(iii)] for every $j=1,\dots,n$ and $\xib\in B^c$, $\widehat N'_j(\xib)\le \widehat N_j(\xib)$;
\item[\rm(iv)] for every $j,k=1,\dots,n$, $e(j,k)\le e'(j,k)$;
\item[\rm(v)] $\Gamma(\EEE)\subseteq\Gamma(\EEE')$.
\end{enumerate}
\end{proposition}

\begin{proof}
The equivalence between (i) and (ii) is given by Theorems~\ref{Thm6.1} and~\ref{Thm6.2}. The implications (v)$\Rightarrow$(iv)$\Rightarrow$(iii)$\Rightarrow$(ii) are obvious. In order to prove that (ii) implies (v), assume by contradiction that $\Gamma(\EEE)$ in not contained in $\Gamma(\EEE')$. Then there exists a half-line $\R^
-\alphab\in \Gamma(\EEE)\setminus \Gamma(\EEE')$, where $\alphab=(\alpha_1,\dots,\alpha_n)$ has strictly positive components.

Let $N_{\alphab}=\sum_{j=1}^nn_j(\xib)^{1/\alpha_j}$ and consider multipliers $m$ which satisfy the inequalities
\bea\label{m=N^i}
\big|\de^{\gammab} m(\xib)\big|\lesssim \big(1+N_{\alphab}(\xib)\big)^{-\sum_j\alpha_j\[\gamma_j\]}
\eea
for all $\gammab\in\N^N$.
We show that every such $m$ is in $\MM_\infty(\EEE)$ but there exist some which are not in $\MM_\infty(\EEE')$.

To prove the first statement it suffices to show that, for every $j$,
$$
\big(1+N_{\alphab}(\xib)\big)^{\alpha_j}\gtrsim 1+\widehat N_j(\xib).
$$
This condition is equivalent to the inequalities $\alpha_k/\alpha_j\le e(k,j)$ for every $j,k$, and in turn these are equivalent to the condition $-\alphab\in\Gamma(\EEE)$.

Similarly, since $-\alphab\not\in\Gamma(\EEE')$, there exist $j,k$ such that $\alpha_k/\alpha_j> e'(k,j)$.
Suppose in \eqref{m=N^i} we take $\gammab=(0,\dots,0,\gamma_j,\dots,0)$ and $\xib=(0,\dots,0,\xib_k,\dots,0)$. The inequality becomes
$$
\big|\de^{\gamma_j} m(0,\dots,\xib_k,\dots,0)\big|\lesssim\big(1+n_k(\xib_k)^{1/\alpha_k}\big)^{-\alpha_j\[\gamma_j\]}
$$
where the right-hand side  is not dominated by
$$
 \big(1+n_k(\xib_k)\big)^{-\frac1{e'(k,j)}\[\gamma_j\]}\approx \big(1+\widehat N'_j(\xib)\big)^{-\[\gamma_j\]}.
$$

It is then sufficient to construct $m$ satisfying \eqref{m=N^i} and such that, for some $l\in C_j$ it satisfies the opposite inequality
$$
\big|\de_{\xi_l} m(0,\dots,\xib_k,\dots,0)\big|\gtrsim\big(1+n_k(\xib_k)^{1/\alpha_k}\big)^{-\alpha_jd_l},
$$
e.g.,
$$
m(\xib)=\frac{\xi_l}{\widetilde N_{\alphab}(\xib)^{\alpha_j d_l}},
$$
where $\widetilde N_{\alphab}$ is a smooth homogeneous norm equivalent to $N_{\alphab}$ and even in $\xi_l$. This $m$ is not in $\MM_\infty(\EEE')$.
\end{proof}

\subsection{Coarser decompositions and lower dimensional matrices}\label{subs.coarser}\quad
\medskip

We will often need to consider classes $\PP_0(\EEE)$ where $\EEE$ is a lower-dimensional (say $s\times s$) matrix and the variables $\x_1,\x_2,\dots,\x_n$ have correspondingly been grouped together into $s$ blocks with given dilation exponents within each block.

Precisely, consider a coarser decompositions of $\R^N$,
$$
\R^N=\R^{A_1}\oplus\R^{A_2}\oplus\cdots\oplus\R^{A_s},
$$
where $\AA=\{A_1,A_2,\dots,A_s\}$ is a partition of $\{1,2,\dots,n\}$ and $\R^{A_r}$ is defined according to \eqref{Eqn2.3bb}. 
We also assign an $n$-tuple $\alphab=(\alpha_1,\alpha_2,\dots,\alpha_n)$ of positive exponents to define dilations on each $\R^{A_r}$, with homogeneous norm
\be\label{n_A}
n_{A_r}(\mbt_{A_r})=\sum_{j\in A_r}n_j(\mbt_j)^{\alpha_j},\qquad r=1,\dots,n.
\ee

Given an $s\times s$ matrix $\EEE=\big(e(t,r)\big)_{1\le t,r\le s}$ satisfying the basic hypotheses, with $s<n$, let $\PP_0(\EEE,\AA,\alphab)$ be the class of associated kernels and  $\MM_\infty(\EEE,\AA,\alphab)$ the class of multipliers. The associated  norms are, for $1\le t\le s$,
\bea\label{globalnorms}
N_t(\x)&=\sum_{r=1}^sn_{A_r}(\x_{A_r})^{e(t,r)}\approx \sum_{r=1}^s\sum_{k\in A_r}n_k(\x_k)^{\alpha_ke(t,r)}\\
\widehat N_t(\xib)&=\sum_{r=1}^sn_{A_r}(\xib_{A_r})^{\frac1{e(r,t)}}\approx \sum_{r=1}^s\sum_{k\in A_r}n_k(\xib_k)^{\frac{\alpha_k}{e(r,t)}}.
\eea

Define the $n\times n$ matrix $\EEE^\sharp=\big(e^\sharp(j,k)\big)$ where, for $j\in A_t$ and $k\in A_r$,
\be\label{esharp}
e^\sharp(j,k)=\frac{\alpha_{k}}{\alpha_j}e\big(t,r).
\ee
\index{E2sharp@$\EEE^\sharp$}

Then $\EEE^\sharp$ satisfies the basic assumptions and the associated dual norms are
\beas
N^\sharp_{j}(\x) &=\sum_{r=1}^s \sum_{k\in A_r}n_k(\x_k)^{\frac{\alpha_k}{\alpha_{j}}e(t,r)}\approx  N_t(\x)^{1/\alpha_{j}}\\
\widehat N^\sharp_{j}(\xib) &=\sum_{r=1}^s \sum_{k\in A_r}n_k(\xib_k)^{\frac{\alpha_k}{\alpha_{j}}\frac1{e(r,t)}}\approx \widehat N_t(\xib)^{1/\alpha_{j}}.
\eeas

\begin{lemma}\label{sharp}
We have the equalities $\MM_\infty(\EEE,\AA,\alphab)=\MM_\infty(\EEE^\sharp)$,  $\PP_0(\EEE,\AA,\alphab)=\PP_0(\EEE^\sharp)$.
\end{lemma}

\begin{proof}
Let $m\in \MM_\infty(\EEE,\AA,\alphab)$.  Then, for every $\boldsymbol\gamma=(\boldsymbol\gamma_1,\dots,\boldsymbol\gamma_s)\in\bN^{A_1}\times\cdots\times\bN^{A_s}$,
$$
\big|\de^{\boldsymbol\gamma}_{\boldsymbol\xi}m(\boldsymbol\xi)\big|\lesssim\prod_{t=1}^s\big(1+\widehat N_t(\xib)\big)^{-\[\boldsymbol\gamma_t\]},
$$
where
$$
[\![\boldsymbol\gamma_t]\!]=\sum_{j\in A_t}\frac1{\alpha_j}\sum_{l\in C_j}\gamma_ld_l,
$$
by \eqref{2.3iou}, taking into account the dilations on $\R^{A_t}$ which make \eqref{n_A} a homogeneous norm.

If $\gammab_t=(\gammab'_j)_{j\in A_t}$, $\gammab$ can be written as an element of $\bN^{C_1}\times\cdots\times \bN^{C_n}$  as $\gammab=(\gammab'_1,\dots,\gammab'_n)$. Then
$$
\sum_{l\in C_j}\gamma_ld_l=\[\gammab'_j\]
$$
and therefore
$$
\begin{aligned}
\prod_{t=1}^s\big(1+\widehat N_t(\xib)\big)^{-[\![\gammab_t]\!]}&=\prod_{t=1}^s\prod_{j\in A_t}\big(1+\widehat N_t(\xib)\big)^{-\frac1{\alpha_j}[\![\gammab'_j]\!]}\\
&\approx\prod_{t=1}^s\prod_{j\in A_t}\big(1+\widehat N^\sharp_j(\xib)\big)^{-\[\gammab'_j\]}=\prod_{j=1}^n\big(1+\widehat N^\sharp_j(\xib)\big)^{-\[\gammab'_j\]}\ .
\end{aligned}
$$

Then $m\in\cM_\infty(\mathbf E)$. The opposite implication is proved in the same way. The equality $\cP_0(\mathbf E)=\cP_0(\mathbf E^\sharp)$ follows from Theorems \ref{Thm6.1} and \ref{Thm6.2}.
\end{proof}

Notice that Lemma \ref{sharp} implies that kernels belonging to $\PP_{0}(\EEE,\AA,\alphab)$ satisfy stronger cancellation properties than those required by the definition.

Combining together Proposition \ref{inclusion} and Lemma \ref{sharp} we obtain the following statement.
\begin{corollary}\label{sharp-inclusion}
Let $\AA$ and $\alphab$ be as above, and let $\EEE$ an $s\times s$ matrix satisfying the basic hypotheses. $\EEE'$ an $n\times n$ matrix also satisfying the basic hypotheses. We have the inclusion $\PP_0(\EEE,\AA,\alphab)\subseteq \PP_0(\EEE')$ if and only if the matrices $\EEE^\sharp$ and $\EEE'$ satisfy the condition
$$
e^\sharp(j,k)\le e'(j,k),\qquad\forall\,j,k,\ 1\le j,k\le n,
$$
together with the other equivalent conditions in Proposition \ref{inclusion}.
\end{corollary}

This result can be applied in particular to the case where an $n\times n$ matrix $\EEE$ satisfying the basic hypotheses is not reduced, i.e., there exist $j\ne k$ such that $e(j,k)e(k,j)=1$.  As shown in Appendix I (Section \ref{Reduced}), in this case the $j$-th row of $\EEE$ equals $e(j,k)$ times the $k$-th row, and the same holds for the $k$-th and $j$-th column.

Following the notation of Lemma \ref{Lem3.7}, we denote by $k_1,\dots, k_s$ representatives of the equivalence classes $A_1,\dots,A_s$ under the relation
$$
j\sim k\ \Longleftrightarrow\ e(j,k)e(k,j)=1.
$$

Then, for all $j\in A_t$,
\bea\label{reducednorms}
N_j(\mathbf x)&\approx N_{k_t}(\mathbf x)^{e(j,k_t)} &&\text{and}& \widehat N_j(\xib)&\approx \widehat N_{k_t}(\xib)^{e(j,k_t
)}.
\eea

Moreover, for every $r\in\{1,\dots,s\}$,
$$
N_{k_t}(\x)\approx\sum_{r=1}^s\Big[\sum_{k\in A_r}n_k(\x_k)^{e(k_r,k)}\Big]^{e(k_t,k_r)}\quad \text{and}\quad
  \widehat N_{k_t}(\xib)\approx\sum_{r=1}^s\Big[\sum_{k\in A_r}n_k(\xib_k)^{e(k_r,k)}\Big]^{1/e(k_r,k_t)}.
$$
Therefore we are in the previous situation, with $\alpha_k=e(k_r,k)$ if $k\in A_r$ and with the reduced matrix
$$
\mathbf E^\flat=\big(e(k_t,k_r)\big)_{1\le t,r\le s}
$$ 
defined in  Section \ref{Reduced}. Notice that $(\EEE^\flat)^\sharp=\EEE$.
By Lemma \ref{sharp} we have the following

\begin{proposition}\label{PEflat}
$\PP_0(\EEE)=\PP_0(\EEE^\flat)$ and $\MM_\infty(\EEE)=\MM_\infty(\EEE^\flat)$.
\end{proposition}

Proposition \ref{PEflat} makes it possible to replace, whenever it is appropriate, $\EEE$  by $\EEE^\flat$ and reduce matters to the case of a reduced matrix.

\medskip
\subsection{Size estimates}\label{SubSec5.4}\quad

\medskip

We now turn to the construction of the dyadic decompositions of kernels $K$ belonging to classes $\PP_{0}(\EEE)$. In the rest of this section, we shall assume that $\Gamma^{o}(\EEE)$ is non-empty, and consequently that $\Gamma_{\Z}(\EEE)$ is non-empty, where $\Gamma^{o}(\EEE)$ and $\Gamma_{\Z}(\EEE)$ were defined in equation (\ref{5.1qwe}). Let $\big\{\varphi^{I}:I\in\Gamma_{\Z}(\EEE)\big\}$ be a uniformly bounded family in $\CC^{\infty}_{0}(\R^{N})$ or $\SS(\R^{N})$. We will show that if all the functions $\varphi^{I}$ satisfy a suitable cancellation condition, then the sum 
\be\label{4.2}
K=\sum_{I\in \Gamma_{\Z}(\EEE)}[\varphi^{ I}]_{ I}
\ee 
converges in the sense of distributions to an element $\KK\in \PP_{0}(\EEE)$. The precise statement is given in Theorem \ref{Thm3.7} below.

We first show that, even without any cancellation hypotheses on the functions $\varphi^{I}$, the sum in equation (\ref{4.2}) satisfies the differential inequalities of Definition \ref{Def2.2} for the class $\PP_{0}(\EEE)$.

\begin{lemma}\label{Lem4.1}
Suppose that $\Gamma^{o}(\EEE)$ is non-empty, and let $\big\{\varphi^{I}:I\in\Gamma_{\Z}(\EEE)\big\}\subset \SS(\R^{N})$ be a uniformly bounded family of Schwartz functions. Then the series $K(\x)=\sum_{ I\in \Gamma_{\Z}(\EEE)} [\varphi^{ I}]_{ I}(\x)$ converges absolutely for all $\x\neq 0$, $K\in\CC^{\infty}(\R^{N}\setminus\{0\})$, and
\bea\label{4.3}
\partial^{\gammab}K(\x) &= 
\sum_{I\in \Gamma_{\Z}({\EEE})}2^{-\sum_{j=1}^{n}i_{j}(Q_{j}+\[ \gammab_{j}\])}(\partial^{\gamma}\varphi^{I})\big(2^{-i_{1}}\,\cdot\,\x_{1}, \ldots, 2^{-i_{n}}\,\cdot\,\x_{n}\big).
\eea
Moreover, for every $\gammab=(\gammab_{1}, \ldots, \gammab_{n})\in \N^{N}$ and $M\geq 0$ there exists a constant $C_{\gammab,M}>0$ so that
\be\label{4.5wer}
\big\vert\partial^{\gammab}K(\x)\big\vert\leq C_{\gammab,M}\,\prod_{j=1}^{n}N_{j}(\x)^{-(Q_{j}+\[\gammab_{j}\])}(1+|\x|)^{-M}.
\ee
\end{lemma}

\begin{proof}
Since the family $\{\varphi^{I}\}$ is uniformly bounded in $\SS(\R^{n})$, for each $M>0$ there is a constant $C_{\gamma,M}$ so that

\beas
\Big\vert2^{-\sum_{j=1}^{n}i_{j}(Q_{j}+\[\gammab_{j}\])}(\partial^{\gamma}\varphi^{I})\big(2^{-i_{1}}\,&\cdot\,\x_{1}, \ldots, 2^{-i_{n}}\,\cdot\,\x_{n}\big)\Big\vert
\\&\leq 
C_{\gamma,M}\prod_{j=1}^{n}2^{-i_{j}(Q_{j}+\[\gammab_{j}\])}(1+\sum_{j=1}^{n}2^{-i_{j}} n_{j}({\x}_{j}))^{-M}.
\eeas
It follows from Proposition \ref{Prop13.1} in Appendix III that the series on the right hand side of equation (\ref{4.3}) converges absolutely and is bounded by $C\,\prod_{j=1}^{n} N_{j}(\x)^{-(Q_{j}+\[\gammab_{j}\])}\prod_{j=1}^{n}(1+n_{j}(\x))^{-M'}$. This shows that the formal differentiation of equation (\ref{4.2}) yielding (\ref{4.3}) is justified, and also establishes the inequality in (\ref{4.5wer}). 
\end{proof}

As an application of Lemma \ref{Lem4.1} we check that kernels satisfying the differential inequalities in part (\ref{Def2.2A}) of Definition \ref{Def2.2} just fail to be integrable near the origin. See also Lemma \ref{Lem4.4}.

\begin{corollary}\label{Cor4.2}
Let $\EEE$ satisfy (\ref{2.5}) and let $N_{j}(\x)$ be defined in equation (\ref{2.4}). Suppose that $\Gamma^{o}(\EEE)$ is non-empty. Then
\begin{enumerate}[{\rm(a)}]
\item \label{Cor4.2a}$\displaystyle \int_{\B(1)}\prod_{j=1}^{n}N_{j}(\x)^{-Q_{j}}\,d\x =+\infty$; 

\item \label{Cor4.2b}If $\alpha>0$ and $1 \leq k \leq n$ then $\displaystyle \int_{\B(1)}N_{k}(\x)^{\alpha}\prod_{j=1}^{n}N_{j}(\x)^{-Q_{j}}\,d\x <+\infty$ .
\end{enumerate}
\end{corollary}
\begin{proof}
Choose $\varphi\in \CC^{\infty}_{0}(\R^{N})$ supported in $\B(1)$ with $\int_{\R^{N}}\varphi(\x)\,d\x = 1$ and $\varphi(\x)\geq 0$. Then if $\varphi^{I}=\varphi$ for every $I$, the family $\{\varphi^{I}\}$ is normalized in $\SS(\R^{N})$. According to equation (\ref{4.5wer}), if $F\subset \Gamma(\EEE)$ is a finite set with $M$ elements , then $\sum_{I\in F}[\varphi]_{I}(\x)\leq C\prod_{j=1}^{N}N_{j}(\x)^{-Q_{j}}$, and hence
\bes
\int_{\B(1)}\prod_{j=1}^{N}N_{j}(\x)^{-Q_{j}}\,d\x \geq C^{-1}\sum_{I\in F}\int_{\R^{N}}[\varphi]_{I}(\x)\,d\x = C^{-1}M.
\ees
Letting $M\to \infty$ gives statement (\ref{Cor4.2a}). To prove statement (\ref{Cor4.2b}), observe that there is a constant $c_{k,l}>0$ so that if $\x\in \B(1)$, then $N_{k}(\x)\leq N_{l}(\x)^{c_{k,l}}$. Thus given $\alpha>0$ and $1 \leq k \leq n$ there exist positive numbers $\epsilon_{1}, \ldots, \epsilon_{n}$ so that $N_{k}(\x) ^{\alpha}\leq \prod_{l=1}^{n}N_{l}(\x)^{\epsilon_{l}}$. Then
\beas
\int\limits_{\B(1)}\!\!N_{k}(\x)^{\alpha}\prod_{j=1}^{n}N_{j}(\x)^{-Q_{j}}\,d\x&\leq 
\int\limits_{\B(1)}\prod_{j=1}^{n}N_{j}(\x)^{-Q_{j}+\epsilon_{j}}\,d\x\\
&
\leq \prod_{j=1}^{n}\int\limits_{n_{j}(\x_{j})\leq 1}\!\!\!\!n_{j}(\x_{j})^{-Q_{j}+\epsilon_{j}}\,d\x_{j}<+\infty.
\eeas
This completes the proof.
\end{proof}

\subsection{Cancellation properties}\quad

We next turn to a discussion of cancellation properties. Let $\R^{N}=\bigoplus_{j=1}^{n}\R^{C_{j}}$ and suppose $\EEE$ is an $n\times n$ matrix satisfying (\ref{2.5}). 

\begin{definition}\label{Def4.3} Let $\varphi\in\SS(\R^{N})$.
\begin{enumerate}[{\rm(a)}]

\smallskip

\item $\varphi$ has {\rm cancellation in the variable ${\x}_{j}$} if $\displaystyle \int_{\R^{C_{j}}}\varphi({\x}_{1}, \ldots, {\x}_{j}, \ldots, {\x}_{n})\,d{\x}_{j}=0$. 

\smallskip

\item $\varphi$ has {\rm strong cancellations} if it has cancellation in ${\x}_{j}$ for $1 \leq j \leq n$. 
\end{enumerate}
\end{definition}

\noindent This terminology was used in Definition 5.3 of \cite{MR2949616}. We will also need a variant of the notion of \textit{weak-cancellation} that was introduced in Definition 5.5 of \cite{MR2949616}. 

\begin{definition}\label{Def4.4}
A function $\varphi\in \CC^{\infty}_{0}(\R^{N})$ has weak cancellation with parameter $\epsilon>0$ relative to the $n$-tuple $ I=(i_{1}, \ldots, i_{n})\in \Gamma_{\Z}(\EEE)$ provided that one can write
\beas
\varphi= \sum_{A\subset\{1, \ldots, n-1\}}\Big[\prod_{j\in A}2^{-\epsilon\Lambda( I, j)}\Big]\varphi_{A}
\eeas
where the sum is taken over all subsets $A\subset\{1, \ldots, n-1\}$, and for each subset $A$,
\begin{enumerate}[{\rm (1)}]

\smallskip

\item $\varphi_{A}\in\CC^{\infty}_{0}(\R^{N})$, and is normalized relative to $\varphi$;

\smallskip

\item $\varphi_{A}$ has cancellation in the variable $\x_{j}$ for each $j \notin A$, and in particular for $j=n$;

\smallskip

\item $\tau_{A}:A\to \{1, \ldots, n\}$ is a mapping with $\tau_{A}(j) \neq j$  for all $j\in A$;

\smallskip

\item $\displaystyle \Lambda(I, j) = i_{j}-e\big(j,\tau_{A}(j)\big)\,i_{\tau_{A}(j)}.$

\end{enumerate}
\end{definition}

\noindent Note that if $ I\in \Gamma_{\Z}(\EEE)$, then $i_{j}-e\big(j,\tau_{A}(j)\big)\,i_{\tau_{A}(j)}\ge0$. Thus if $\varphi$ has weak cancellation, it can be written as a finite sum of functions, each of which has the property that for each index $j$, either there is cancellation in $\x_{j}$ or the lack of cancellation is compensated by a gain $2^{-\epsilon[i_{j}-e(j,\tau_{A}(j))\,i_{\tau_{A}(j)}]}$.
\goodbreak

The following is Lemma 5.1 in \cite{MR2949616}:

\begin{proposition}\label{Prop3.5}
If $\varphi \in\SS(\R^{N})$ and $\{J_{1}, \ldots, J_{r}\}$ are disjoint subsets of $\{1, \ldots, N\}$, then the following are equivalent:
\begin{enumerate}[{\rm(a)}]

\medskip

\item $\int\varphi(\x)\,d\x_{J_{k}}=0$ for $1\leq k \leq r$.

\medskip

\item For every $(j_{1}, \ldots, j_{r})\in J_{1}\times\cdots\times J_{r}$ there is a function $\varphi_{j_{1}, \ldots, j_{r}}\in\SS(\R^{n})$, normalized with respect to $\varphi$, so that $\varphi(\x) = \sum_{j_{1}\in J_{1}}\cdots \sum_{j_{r}\in J_{r}}\partial_{j_{1}}\cdots\partial_{j_{r}}\varphi_{j_{1},\ldots, j_{r}}(\x)$.
\medskip
\end{enumerate}
Moreover, if $\varphi\in \CC^{\infty}_{0}(\R^{N})$, then we can choose the functions $\varphi_{j_{1}, \ldots, j_{r}}\in\CC^{\infty}_{0}(\R^{N})$ normalized relative to $\varphi$.
\end{proposition}

We will also need the following result.

\begin{proposition}\label{Prop3.6}
Let $L=\{l_{1},\ldots,l_{a}\}$ and $M=\{m_{1},\ldots, m_{b}\}$ be complementary subsets of $\{1, \ldots,n\}$, let $A$ be a subset of $M$, and let $R=\{R_{m_{1}}, \ldots, R_{m_{b}}\}$ be positive real numbers. Write any $\x\in \R^{N}$ as $\x=(\x',\x'')$ where $\x'=(\x_{l_{1}}, \ldots, \x_{l_{a}})\in \R^{L}$ and $\x'' =(\x_{m_{1}}, \ldots, \x_{m_{b}})\in \R^{M}$. Let $\varphi\in\SS(\R^{N})$ and let $\psi\in \CC^{\infty}_{0}(\R^{M})$ have support in $\B(1)$. Put
\bes
\Phi(\x') =\int_{\R^{M}}\varphi(\x',\x'')\psi(R_{m_{1}}\,\cdot \,\x_{m_{1}}, \ldots, R_{m_{b}}\,\cdot\,\x_{m_{b}})\,d\x_{m_{1}}\cdots d\x_{m_{b}}.
\ees
\begin{enumerate}[{\rm(a)}]

\item \label{Prop3.6a} $\Phi\in\SS\left(\R^{L}\right)$ and is normalized relative to $\varphi$ and $\psi$ with constants independent of $R$. If $\varphi\in \CC^{\infty}_{0}(\R^{N})$, then $\Phi \in \CC^{\infty}_{0}\left(\R^{L}\right)$ and is normalized relative to $\varphi$ and $\psi$ with constants independent of $R$.
\medskip

\item \label{Prop3.6b} If some of the parameters $\{R_{m_{j}}\}$ are large, there is an improvement over statement {\rm (\ref{Prop3.6a}):} there exists $\epsilon >0$ independent of $\varphi$, $\psi$ and $R$ so that 
\bes
\Phi(\x') = \Big[\prod_{m_{j}\in M}\min\big\{1,\,R_{m_{j}}^{-\epsilon}\big\}\Big]\,\widetilde \Phi(\x')
\ees
where $\widetilde \Phi\in \SS(\R^{N})$ is normalized relative to $\varphi$ and $\psi$, with constants independent of $R$. If $\varphi\in \CC^{\infty}_{0}(\R^{N})$ is supported in the unit ball, then $\widetilde \Phi \in \CC^{\infty}_{0}\left(\R^{L}\right)$ and is normalized relative to $\varphi$ and $\psi$ with constants independent of $R$.

\medskip

\item \label{Prop3.6c} If $\varphi$ has cancellation in some of the variables in $M$ there is an improvement over statement {\rm (\ref{Prop3.6b})} even if $R_{m_{j}}$ is small. There exists $\epsilon >0$ independent of $\varphi$, $\psi$ and $R$ so that if $\varphi$ has cancellation in $\x_{j}$ for each $j\in A$ then
\bes
\Phi(\x') = \Big[\prod_{m_{j}\in M\setminus A}\min\big\{1,\,R_{m_{j}}^{-\epsilon}\big\}\Big]\,\Big[\prod_{m_{j}\in A}\min\big\{R_{m_{j}}^{\epsilon}, R_{m_{j}}^{-\epsilon}\big\}\Big]\,\widetilde \Phi(\x'),
\ees
where $\widetilde \Phi$ is normalized relative to $\varphi$ and $\psi$, with constants independent of of $R$. 
\end{enumerate}
\end{proposition}

\begin{proof}
Statement (\ref{Prop3.6a}) is clear. Next, since $\psi(\x'')$ is supported where $ n_{m_{j}}(\x_{m_{j}})\leq 1$, we have $ n_{m_{j}}(\x_{m_{j}})\leq R_{m_{j}}^{-1}$ if $\psi(R_{m_{1}}\cdot \x_{m_{1}}, \ldots, R_{m_{b}}\cdot\x_{m_{b}})\neq 0$. The integration in the definition of $\Phi(\x')$ takes place over the set $\big\{(\x_{m_{1}}, \ldots, \x_{m_{b}}): n_{m_{j}}(\x_{m_{j}})\leq R_{m_{j}}^{-1}\big\}$, and we can estimate the integral by the (small) size of the region of integration in the variables $\{\x_{m_{j}}\}$ for which $R_{m_{j}}$ is large. This gives the estimate in part (\ref{Prop3.6b}). To establish (\ref{Prop3.6c}) we use Proposition \ref{Prop3.5}: $\Phi(\x')$ is a sum of terms of the form
\bes
\int_{\R^{M}}\partial_{j_{1}}\cdots \partial_{j_{a}}\varphi_{j_{1},\ldots, j_{a}}(\x',\x'')\psi(R_{m_{1}}\cdot \x_{m_{1}}, \ldots, R_{m_{b}}\cdot\x_{m_{b}})\,d\x_{m_{1}}\cdots d\x_{m_{b}}
\ees
where $\partial_{j_{k}}$ is a derivative with respect to a variable in $\x_{m_{k}}$ for each $m_{k}\in A$. Integrating by parts in the variables in $A$ for which $R_{m_{k}}<1$ then gives the estimate in part (\ref{Prop3.6c}), and this completes the proof.
\end{proof}

We will want to show that sums of dilates of normalized bump functions with weak cancellation satisfy the cancellation properties of the distributions $\KK\in \PP_{0}(\EEE)$, and we will use the following.

\begin{lemma}\label{Lem3.7jj}
Suppose that $\varphi^{I}$ has weak cancellation with parameter $\epsilon >0$. Let $L=\{l_{1},\ldots,l_{a}\}$ and $M=\{m_{1},\ldots, m_{b}\}$ be complementary subsets of $\{k_{1},\ldots, k_{s}\}$, and let $R=\{R_{m_{1}}, \ldots, R_{m_{b}}\}$ be positive real numbers. Write $\x\in \R^{N}$ as $\x=(\x',\x'')$ where $\x'=(\x_{l_{1}}, \ldots, \x_{l_{a}})$ and $\x'' =(\x_{m_{1}}, \ldots, \x_{m_{b}})$, and similarly write $ I=( I', I'')$. Let $\psi\in \CC^{\infty}_{0}(\R^{M})$ be a normalized bump function. Then the function 
\beas
\Phi^{I}_{\psi,R}(\x')=\int_{\R^{M}}\varphi^{I}(\x',\x'')\psi\big((2^{i_{m_{1}}}R_{m_{1}})\cdot \x_{m_{1}}, \ldots, (2^{i_{m_{b}}}R_{m_{b}})\cdot\x_{m_{b}}\big)\,d\x_{m_{1}}\cdots d\x_{m_{b}}.
\eeas
can be written as a finite sum of terms, indexed by subsets $A\subset \{k_{1}, \ldots, k_{s}\}$, of the form
\beas
\Big[\prod_{j\in A\cap M}2^{-\epsilon[i_{j}-e(j,\tau(j))i_{\tau(j)}]}\Big]
\Big[\prod_{j\in M\setminus A}\min\Big\{(2^{j}R_{j})^{+\epsilon}, (2^{j}R_{j})^{-\epsilon}\Big\}\Big]\,\widetilde \varphi^{I}_{A}(\x').
\eeas
The family $\{\widetilde\varphi^{I}_{A}\}\subset \CC^{\infty}_{0}(\R^{L})$ is normalized relative to $\varphi^{I}$.
\end{lemma}

\begin{proof}
To make the exposition and the notation simpler, assume without loss of generality that $\x'=(\x_{k_{1}}, \ldots, \x_{k_{a}})$ and $\x''=(\x_{k_{a+1}}, \ldots, \x_{k_{s}})$ and write $ I=( I', I'')$ with $ I'=(i_{1}, \ldots, i_{k_{a}})$ and $ I''=(i_{k_{a+1}}, \ldots, i_{k_{s}})$. According to Definition \ref{Def4.4}, since $\varphi^{I}$ has weak cancellation,
\beas
\varphi^{I}(\x',\x'')= \sum_{A\subset\{1, \ldots, n\}}\Big[\prod_{j\in A}2^{-\epsilon[i_{j}-e(j,\tau(j))i_{\tau(j)}]}\Big]\varphi^{I}_{A}(\x',\x'').
\eeas
Then $\Phi^{I}_{R,\psi}$ can be written as a finite sum of terms, indexed by $A\subset\{1, \ldots,n\}$, of the form
\beas
\!\Big[\prod_{j\in A}2^{-\epsilon[i_{j}-e(j,\tau(j))i_{\tau(j)}]}\Big]\!
\int
\varphi^{I}_{A}(\x',\x'')\psi\big((2^{i_{m_{1}}}R_{m_{1}})\cdot \x_{m_{1}}, \ldots, (2^{i_{m_{b}}}R_{m_{b}})\cdot\x_{m_{s}}\big)\,d\x''.
\eeas
Using Proposition \ref{Prop3.6}, and incorporating $\prod_{j\in M\setminus A}2^{-\epsilon[i_{j}-e(j,\tau(j))i_{\tau(j)}]}$ into the bump function, it follows this function can be written
\beas
\Big[\prod_{j\in A\cap M}2^{-\epsilon[i_{j}-e(j,\tau(j))i_{\tau(j)}]}\Big]
\Big[\prod_{j\in M\setminus A}\min\Big\{(2^{i_{m_{j}}}R_{m_{j}})^{+\epsilon}, (2^{i_{m_{j}}}R_{m_{j}})^{-\epsilon}\Big\}\Big]\widetilde\varphi^{I}_{A}(\x'),
\eeas
which completes the proof.
\end{proof}

\subsection{Dyadic sums with weak cancellation}\label{SubSec5.6}\quad

Fix the decomposition $\R^{N}= \R^{C_{1}}\oplus \cdots \oplus \R^{C_{n}}$ and let $\EEE$ be an $n\times n$ matrix satisfying the basic hypothesis given in equation (\ref{2.5}). Suppose that $\Gamma^{o}(\EEE) \neq \emptyset$. Let  $\Gamma_{\Z}(\EEE)$ be as in equation (\ref{5.1qwe}).

\begin{theorem}\label{Thm3.7} 
Let $\left\{\varphi^{ I}: I\in \Gamma_{\Z}(\EEE)\right\}\subset \SS(\R^{N})$ be a uniformly bounded family. If each function $\varphi^{ I}$ has weak cancellation with parameter $\epsilon>0$ relative to $I$, then $ K=\sum_{ I\in\Gamma_{\Z}(\EEE)}[\varphi^{ I}]_{I}$ converges in the sense of distributions to a kernel $\KK\in \PP_{0}(\EEE)$.
\end{theorem}

The arguments are similar to those in the proofs of Theorem 6.8 and Propositions 6.9, 6.10, and 6.11 of \cite{MR2949616}. The proof consists of two steps:
\begin{enumerate}[(1)]

\smallskip

\item If $F\subset \Gamma_{\Z}(\EEE)$ is a finite set, the distribution $\KK_{F}=\sum_{ I\in F}[\varphi^{ I}]_{ I}$ belongs to the class $\PP_{0}(\EEE)$ with constants independent of the set $F$. (Lemma \ref{Lem4.1} shows that $\KK_{F}$ satisfies the appropriate differential inequalities, so we  only need to verify that $\KK_{F}$ satisfies the cancellation conditions.) 

\smallskip

\item If $\{F_{m}\}$ is any increasing sequence of finite subsets of $\Gamma_{\Z}(\EEE)$ with $\bigcup F_{m}= \Gamma_{\Z}(\EEE)$, then $\lim_{m\to\infty}\KK_{F_{m}}$ exists in the sense of distributions.
\end{enumerate}

\begin{proof}[Proof of Step 1]\quad

\medskip

Let $F\subset\Gamma_{\Z}(\EEE)$ be a finite set and let $\KK_{F}= \sum_{ I\in F}[\varphi^{ I}]_{ I}$. To verify the cancellation conditions, let $L=\{l_{1},\ldots,l_{a}\}$ and $M=\{m_{1},\ldots, m_{b}\}$ be complementary subsets of $\{1, \ldots, n\}$. If $\x=(\x_{1}, \ldots, \x_{n})\in \R^{N}$ write $\x=(\x',\x'')$ where $\x'=(\x_{l_{1}}, \ldots, \x_{l_{a}})\in \R^{L}$ and $\x''=(\x_{m_{1}}, \ldots, \x_{m_{b}})\in \R^{M}$. Let $\psi\in \mathcal C^{\infty}_{0}(\R^{L})$ be a normalized bump function with support in the unit ball, and let $R=\{R_{l_{1}}, \ldots, R_{l_{a}}\}$ be positive real numbers. We must show that the function
\beas
K_{F,\psi,R}(\x'')
&=
\sum_{ I\in F}\int_{\R^{M}}[\varphi^{ I}]_{ I}(\x_{1}, \ldots, \x_{n})\psi(R_{l_{1}}\cdot \x_{l_{1}}, \ldots, R_{l_{a}}\cdot\x_{l_{a}})\,d\x_{l_{1}}\cdots d\x_{l_{a}}
\eeas
satisfies the correct differential inequalities in the $\x''$ variables.

To simplify the notation, and without any real loss of generality, we will only consider the case when $L=\{1, \ldots, a\}$ and $M=\{a+1, \ldots,n\}$, so that $\x'=(\x_{1}, \ldots, \x_{a})$ and $\x''=(\x_{a+1}, \ldots, \x_{n})$. Let $\EEE''$ be the sub-matrix of $\EEE$ consisting of those entries $e(j,k)$ for which $a+1\leq j,k \leq n$, and let
\bes
\Gamma_{\Z}(\EEE'')=\Big\{(i_{a+1}, \ldots, i_{n})\in \Z^{n-a}:e(j,k)i_{k} \leq i_{j}<0, \, a+1\leq j,k \leq n\Big\}.
\ees
For each $ I''=(i_{a+1}, \ldots, i_{n})\in \Gamma_{\Z}(\EEE'')$ there are only finitely many elements $ I\in \Gamma_{\Z}(\EEE)$ which have the same last $n-a$ entries. Let $F''\subset \Gamma_{\Z}(\EEE'')$ be a finite set and suppose (without loss of generality) that 
\bes
F=\Big\{ I=(i_{1}, \ldots, i_{n})\in \Gamma_{\Z}(\EEE): I''=(i_{a+1}, \ldots, i_{n})\in F''\Big\}.
\ees
For each $ I''\in \Gamma_{\Z}(\EEE'')$ let 
\bes
\Lambda(I'')= \Big\{ I'=(i_{1}, \ldots, i_{a})\in \Z^{a}: ( I', I'')\in \Gamma_{\Z}(\EEE)\Big\}.
\ees
Then for each $ I\in F$ we can write $ I = ( I',  I'')$ with $ I''\in F''$ and $ I'\in \Lambda( I'')$.

Let $\psi\in \mathcal C^{\infty}_{0}(\R^{M})$ be a normalized bump function with support in the unit ball, and let $R=\{R_{1}, \ldots, R_{a}\}$ be positive real numbers. To complete step 1 we need to estimate
\beas
K_{F,\psi,R}&(\x'')
=
\sum_{ I''\in F''}\sum_{I'\in\Lambda( I'')}\int_{\R^{M}}[\varphi^{ I}]_{ I}(\x_{1}, \ldots, \x_{n})\psi(R_{1}\cdot \x_{1}, \ldots, R_{a}\cdot\x_{a})\,d\x_{1}\cdots d\x_{a}
\eeas
and its derivatives. Let
\beas
\Phi^{I}_{\psi,R}(\x'')
&= \int_{\R^{M}}\varphi^{ I}(\x_{1}, \ldots, \x_{n})\psi\big((2^{i_{1}}R_{1})\cdot \x_{1}, \ldots, (2^{i_{a}}R_{a})\cdot\x_{a})\,d\x_{1}\cdots d\x_{a}.
\eeas
Then 
\beas
\int_{\R^{M}}[\varphi^{ I}]_{ I}(\x_{1}, \ldots, \x_{n})\psi(R_{1}\cdot \x_{1}, \ldots, R_{a}\cdot\x_{a})\,d\x_{1}\cdots d\x_{a}
&=
\big[\Phi^{(I',I'')}_{\psi,R}\big]_{ I''}(\x'')
\eeas
and so
\beas
K_{F,\psi,R}(\xt'') 
&=
\sum_{I''\in F'}\Big[\sum_{ I'\in \Lambda(I'')}\Phi^{(I',I'')}_{\psi,R}\Big]_{ I''}(\x'').
\eeas
If we can show that for each $I''\in F''$ the inner sum $\sum_{ I'\in \Lambda(I'')}\Phi^{(I',I'')}_{\psi,R}$ converges to a function $\widetilde \Phi^{ I''}_{\psi,R}\in\CC^{\infty}_{0}(\R^{L})$, and this family is uniformly bounded independent of the choice of $F$, then the required differential inequality estimates for $K_{F,\psi,R}$ will follow from Lemma \ref{Lem4.1}. But according to Lemma \ref{Lem3.7jj}, $\Phi^{(I',I'')}_{\psi,R}$ is a sum of normalized bump functions multiplied by exponential gains. Thus the proof of step 1 will be complete if we can show that for each fixed $ I''\in F''\subset\Gamma_{\Z}(\EEE'')$ and each subset $A\subset \{1, \ldots, n\}$,
\beas
\sum_{(i_{1}, \ldots, i_{a})\in \Lambda( I'')}\Bigg[\prod_{\substack{j=1\\j\in A}}^{a}2^{-\epsilon\left[i_{j}-e(j,\tau_{A}(j))i_{\tau_{A}(j)}\right]}\Bigg]
\Bigg[\prod_{\substack{j=1\\j\notin A}}^{a}\min\Big\{(2^{i_{j}}R_{j})^{+\epsilon}, (2^{i_{j}}R_{j})^{-\epsilon}\Big\}\Bigg]
\eeas
converges independently of $ I''$ and the finite set $F''$. Without loss of generality, we can assume that $\{1, \ldots,a\}\cap A=\{1,\ldots, b\}$ and $\{1, \ldots,a\}\setminus A=\{b+1,\ldots, a\}$. Thus we need to estimate
\beas
\sum_{(i_{1}, \ldots, i_{a})\in \Lambda( I'')}\Bigg[\prod_{\substack{j=1}}^{b}2^{-\epsilon\left[i_{j}-e(j,\tau_{A}(j))i_{\tau_{A}(j)}\right]}\Bigg]
\Bigg[\prod_{\substack{j=b+1}}^{a}\min\Big\{(2^{i_{j}}R_{j})^{+\epsilon}, (2^{i_{j}}R_{j})^{-\epsilon}\Big\}\Bigg].
\eeas
The required estimate follows from Proposition \ref{Prop16.1wer} in Appendix III.
\end{proof}

\begin{proof}[Proof of Step 2]\quad

\medskip 

We need to check that if $\{F_{m}\subset\Gamma_{\Z}(\EEE)\}$ is a sequence of finite sets with $\bigcup F_{m}=\Gamma_{\Z}(\EEE)$, then the sequence $\big\{\KK_{F_{m}}=\sum_{ I\in F_{m}}[\varphi^{ I}]_{ I}\big\}$ converges in the sense of distributions to an element of $\PP_{0}(\EEE)$. The argument is similar to that on pages 667-668 of \cite{MR2949616}. Since each $\varphi^{ I}$ has cancellation in $\x_{n}$, it follows from Proposition \ref{Prop3.6} that we can write $\varphi^{ I}=\sum_{k\in C_{n}}\partial_{x_{k}}\varphi^{ I}_{k}$ where each $\varphi^{ I}_{k}\in \CC^{\infty}_{0}(\R^{N})$ is normalized relative to $\varphi^{ I}$.  It follows that
\bes
[\varphi^{ I}]_{ I}
=
\sum_{k\in C_{n}}[\partial_{x_{k}}\varphi^{ I}_{k}]_{ I}
=
\sum_{k\in C_{n}}2^{i_{k_{s}}d_{k}}\partial_{x_{k}}[\varphi^{ I}_{k}]_{ I}.
\ees
Thus if $\psi\in\SS^{\infty}_{0}(\R^{N})$ is a test function, we can integrate by parts to get 
\beas
\big\langle \KK_{F_{m}},\psi\big\rangle
&=
\sum_{k\in C_{n}}\big\langle \sum_{ I\in F_{m}}2^{i_{k}d_{k}}\partial_{x_{k}}[\varphi^{ I}_{k}]_{ I},\psi\big\rangle
=
-\sum_{k\in C_{n}}\big\langle \sum_{ I\in F_{m}}2^{i_{k}d_{k}}[\varphi^{ I}_{k}]_{ I},\partial_{x_{k}}\psi\big\rangle\\
&=
-\sum_{k\in C_{n}}\int_{\R^{N}}\Big[\sum_{ I\in F_{m}}2^{i_{k}d_{k}}2^{-\sum_{k=1}^{n}i_{k}Q_{k}}
 \varphi^{ I}_{k}(2^{-i_{1}}\cdot\x_{1}, \ldots, 2^{-i_{n}}\cdot\x_{n})\Big]\partial_{x_{k}}\psi(\x)\,d\x.
\eeas
Let $\d = \min\big\{d_{k}:1\leq k \leq N\big\}>0$. Since $\{\varphi^{ I}_{k}\}$ is a normalized family in $\SS(\R^{N})$, it follows from Proposition \ref{Prop13.1} that
\beas
\Big|\sum_{ I\in F_{m}}&2^{i_{k}d_{k}}2^{-\sum_{k=1}^{n}i_{k} Q_{k}}
 \varphi^{ I}_{k}(2^{-i_{1}}\cdot\x_{1}, \ldots, 2^{-i_{n}}\cdot\x_{n})\Big|\\
&\leq
C_{M}\sum_{ I\in \Gamma(\EEE)}\Big[\prod_{k=1}^{n}2^{-i_{k} Q_{k}}\Big]2^{i_{k}\d}\Big(1+\sum_{j=1}^{n}2^{-i_{k}} n_{k}(\x_{k})\Big)^{-M}\\
&\leq
C_{M}N_{n}(\x)^{\d}\,\prod_{k=1}^{n}N_{k}(\x)^{- Q_{k}} \leq C_{M}N_{n}(\x)^{\d}\prod_{j=1}^{n}N_{j}(\x)^{-Q_{j}}.
\eeas
It follows from Corollary \ref{Cor4.2} that $N_{n}(\x)^{\d}\,\prod_{j=1}^{n}N_{j}(\x)^{-Q_{j}}$ is locally integrable. Thus by the Lebesgue dominated convergence theorem it follows that as $m$ tends to infinity, $\langle \KK_{F_{m}},\psi\rangle$ tends to
\beas
-\sum_{k\in C_{n}}\int\limits_{\R^{N}}\Big[\sum_{ I\in F_{m}}2^{i_{k}d_{k}}2^{-\sum_{k=1}^{n}i_{k} Q_{k}}
 \varphi^{ I}_{k}(2^{-i_{1}}\cdot\x_{1}, \ldots, 2^{-i_{n}}\cdot\x_{n})\Big]\partial_{x_{k}}\psi(\x)d\x.
\eeas
This completes the proof.
\end{proof}

\section{Decomposition of multipliers and kernels}\label{Decompositions}

Let $\KK\in \PP_{0}(\EEE)$. The object of this section is to decompose $\KK$ into a finite sum of distributions, each of which can be written as a dyadic sum of dilates of normalized Schwartz functions (or normalized bump functions) with appropriate cancellation. It is difficult to directly decompose $\KK$ and maintain the cancellation, so we work on the Fourier transform side.  The effect of this procedure is that the following steps will be governed by the decomposition of $\B(1)^{c}$ in the $\xib$-space according to the sets $\widehat E_S$, while the decomposition of $\B(1)$ in the $\x$-space will no longer play any r\^ole.

\subsection{New matrices $\EEE_{S}$}\label{NewMatrices}\quad

\medskip

Let $\SS'(\EEE)$ \index{S3'n@$\SS'(\EEE)$} denote the set of marked partitions $S$ such that $\widehat E_{S}\cap \B(1)^{c}$ is non-empty. From Corollary \ref{Cor5.7} we obtain the following equivalent condition.

\begin{corollary}\label{Remark3.10}
Let $S=\big((I_{1},k_{1});\ldots;(I_{s},k_{s})\big)\in \SS(n)$.  Then $\widehat E_{S}\cap \B(1)^{c}\neq \emptyset$ if and only if the cone
\bea
\widehat \Gamma_{S}=\Big\{\a=(a_{1}, \ldots, a_{s})\in \R^{s}:\text{$\tau_{S}(k_{r},k_{p})a_{p}\leq a_{r}<0$ for all $1\leq p,r\leq s$}\Big\},
\eea
has non-empty interior.
\end{corollary}

With the decomposition introduced in Section \ref{Sec3.5iou} and in the notation of Section \ref{subs.coarser}, for each $S\in\SS'(\EEE)$ we want to  define new norms on $\R^{N}$ which describe the behaviour of multipliers on the set $\widehat E_{S}\cap \B(1)^{c}$ in the same way that the original norms $\widehat N_{1}(\t), \ldots, \widehat N_{n}(\t)$ describe it on the principal region $\widehat E_{S_0}\cap \B(1)^{c}$. 
\smallskip

Let $S=\big\{(I_1,k_1);\dots;(I_s,k_s)\big\}\in\SS'(\EEE)$. On each $\R^{I_r}$ we put the norm
$$
\widehat n_{S,r}(\xib_{I_{r}}) = \sum_{j\in I_{r}}n_{j}(\xib_{j})^{e(k_{r},j)}=\sum_{j\in I_{r}}n_{j}(\xib_{j})^{1/\tau_{S}(j,k_{r})},
$$
defined in \eqref{5.11aa} and homogeneous with respect to the dilations~$\hat\cdot_S$. 

We now want  an $s\times s$ matrix $\EEE_S$ satisfying the basic hypotheses and such that $\Gamma(\EEE_S)=\widehat\Gamma_S$. We cannot simply use the coefficients $\tau_S(k_r,k_p)$ because in general they do not satisfy the inequalities $\tau_{S}(k_{r},k_{p})\leq \tau_{S}(k_{r},k_{u})\,\tau_{S}(k_{u},k_{p})$. 

Luckily we can replace the coefficients $\tau_{S}(k_{p},k_{r})$ by new ones which do satisfy the basic hypothesis. According to Lemma \ref{Lem3.2} in Appendix I, if the cone is non-empty, the coefficients in the inequalities defining the cone can be replaced by new coefficients which satisfy the basic hypothesis. Thus we have the following result.

\begin{lemma}\label{Lem5.10}

Let $S=\big\{(I_{1},k_{1});\ldots, (I_{s},k_{s})\big\}\in \SS'(\EEE)$. \index{E3ES@$\EEE_{S}$}
There is a unique $s\times s$ matrix $\EEE_{S}=\{e_{S}(r,p)\}$ such that $e_{S}(r,p)\leq \tau_{S}(k_{r},k_{p})$, and
\bes
\Gamma(\EEE_{S})=\Big\{\a\in\R^{s}:e_{S}(r,p)a_{p}\leq a_{r}<0\Big\} = \widehat\Gamma_{S}.
\ees
\end{lemma}

According to \eqref{globalnorms}, we introduce global norms $N_{S,r}$ and $\widehat N_{S,r}$.

\begin{definition}
Let $S=\big\{(I_{1},k_{1});\ldots, (I_{s},k_{s})\big\}\in \SS'(\EEE)$.  Set \index{N4Sr@$N_{S,r}\widehat N_{S,r}$}
\bea\label{3.6iou}
N_{S,r}(\x) &= \sum_{p=1}^{s}\widehat n_{S,p}(\x_{I_{p}})^{e_{S}(r,p)}\,\,\approx \,\,\sum_{p=1}^{s}\sum_{k\in I_{p}}n_{k}(\x_{k})^{e(k_{p},k)e_{S}(r,p)},\\
\widehat N_{S,r}(\xib) &= \sum_{p=1}^{s}\widehat n_{S,p}(\xib_{I_{p}})^{1/e_{S}(p,r)}\approx\sum_{p=1}^{s}\sum_{k\in I_{p}}n_{k}(\xib_{k})^{e(k_{p},k)/e_{S}(p,r)}.
\eea
\end{definition}

These norms coincide with the norms defined in \eqref{globalnorms}, relative to the partition $\II_S=\{I_1,\dots,I_s\}$ of $\{1,\dots,n\}$, to the matrix $\EEE_S$ and to the vector of exponents $\alphab_S=(\alpha_1,\dots,\alpha_n)$ with $\alpha_j=e(k_r,j)$ if $j\in I_r$.

We will use the shortened notation $\MM_\infty(\EEE_S)$ instead of $\MM_\infty(\EEE_S,\II_S,\alphab_\S)$.

\begin{proposition}\label{Lem8.7mm}
We have the inclusion $\MM_\infty(\EEE_S)\subset \MM_\infty(\EEE)$.
\end{proposition}

\begin{proof}
By Corollary \ref{sharp-inclusion}, it is sufficient to check that the entries of the matrix $\EEE_S^\sharp$ are not greater than the corresponding entries of the matrix $\EEE$. By \eqref{esharp}, if $j\in I_r$ and $k\in I_p$,
$$
e_S^\sharp(j,k)=\frac{e(k_p,k)}{e(k_r,j)}e_S(r,p)\le \frac{e(k_p,k)}{e(k_r,j)}\tau_S(k_r,k_p)\le \frac{e(k_p,k)}{e(k_r,j)}\frac{e(k_r,k)}{e(k_p,k)}\le e(j,k).
$$
\end{proof}

\subsection{Road map for the dyadic decomposition}\quad

Let $m=\widehat{\KK}$ be the Fourier transform so that $m\in \MM_{\infty}(\EEE)$. 
We proceed as follows. 
\begin{enumerate}[(1)]
\smallskip

\item  In Section \ref{SSPartitions} we will construct a partition of unity on $\R^{N}$ consisting of functions $\Psi_{0}$ and $\{\Psi_S\}_{S\in \SS'(\EEE)}$, where $\Psi_0\in \CC^{\infty}_{0}(\R^N)$ has support in $\B(2)$, and where each $\Psi_S \in \MM_{\infty}(\EEE)$ and has support in $\widehat E_{S}^{A}\cap \B(1)^{c}$. Here $A>1$ depends only on the matrix $\EEE$.

We will write 
\beas
m=\Psi_{0}m+\sum_{S\in \SS'(\EEE)}\Psi_{S}m= m_{0}+\sum_{S\in \SS'(\EEE)}m_{S}.
\eeas 
Then $m_{0}\in\CC^{\infty}_{0}(\B(2))$ and each $m_{S}\in \MM_{\infty}(\EEE)$ 
 is supported in $\widehat E_{S}^{A}\cap\B(1)^{c}$.  

\smallskip

\item For each $S=\big\{(I_{1},k_{1});\ldots;(I_{s},k_{s})\big\}\in \SS'(\EEE)$, we show that 
\beas
m_{S}(\xib_{I_{1}}, \ldots, \xib_{I_{s}})= \sum_{J\in \Gamma(\EEE_S)}m_{S}^{J}(2^{-j_{1}}\,\hat\cdot_{S}\xib_{I_{1}}, \ldots, 2^{-j_{s}}\,\hat\cdot_{S}\xib_{I_{s}})
\eeas
where
 $\big\{m_{S}^{J}:J\in \Gamma(\EEE_S)\big\}\subset\CC^{\infty}_{0}(\R^{N})$ is a uniformly bounded family with supports in the set where $4^{-1}<\widehat n_{S,r}(\xib_{I_{r}})<4$ for $1 \leq r \leq s$. 
\smallskip

\item Writing $[m_{S}^{J}]^{J}(\xib_{I_{1}}, \ldots, \xib_{I_{s}}) = m_{S}^{J}(2^{-j_{1}}\,\hat\cdot_{S}\xib_{I_{1}}, \ldots, 2^{-j_{s}}\,\hat\cdot_{S}\xib_{I_{s}})$, it then follows that 
\beas
m(\xib)=m_{0}(\xib)+\sum_{S\in \SS'([n])}\sum_{J\in\Gamma(\EEE_S)}[m_{S}^{J}]^{J}(\xib).
\eeas
This is done in Section \ref{DyadicDecomp}.

\smallskip

\item
Taking inverse Fourier transform, denoted by $^{\vee}$, we obtain 
\beas
\KK = m_{0}^{\vee}+\sum_{S\in \SS'([n])}\sum_{J\in\Gamma(\EEE_S)}\big([m_{S}^{J}]^{J}\big)^{\vee}.
\eeas 
This gives a decomposition of $\KK$ into sums of dilates of Schwartz functions, which in turn can be written as sums of dilates of normalized bump functions. (See Lemma \ref{Lem8.7qq} below.) 
\end{enumerate}

\subsection{Partitions of unity}\label{SSPartitions}\quad

\medskip

In this section, we construct a partition of unity on the `Fourier transform side'. We begin by slightly modifying our notation. As usual, $\R^{N}=\bigoplus_{j=1}^{n}\R^{C_{j}}$ and $\EEE=\{e(j,k)\}$ is an $n\times n$ matrix satisfying (\ref{2.5}). Let $S=\big\{(I_{1},k_{1});\ldots;(I_{s},k_{s})\big\}\in   \SS'(\EEE)$  and  $\R^{I_{r}}=\bigoplus_{j\in I_{r}}\R^{C_{j}}$. Instead of taking $\widehat N_{j}(\xib)=\sum_{l=1}^{n}n_{l}(\xib_{l})^{1/e(l,j)}$ which might not be differentiable, we choose a smooth version of this norm. Then
\beas
\widehat N_{j}(\xib) &\approx \sum_{l=1}^{n}n_{l}(\xib_{l})^{1/e(l,j)},&&&\widehat n_{S,r}(\xib_{I_{r}}) &\approx \sum_{j\in I_{r}}n_{j}(\xib_{j})^{e(k_{r},j)},\\
\lambda\,\hat\cdot_{S}\,\xib_{I_{r}}&= \big\{\lambda^{1/e(k_{r},j)}\cdot\xib_{j}:j\in I_{r}\big\},&&&
\widehat Q_{S,r}&=\sum_{j\in I_{r}}Q_{j}e(k_{r},j)^{-1},\\
\tau_{S}(k_{p},k_{r})&=\min_{j\in I_{r}}\big\{e(k_{p},j)/e(k_{r},j)\big\}.
\eeas
 Let $\widehat E_{S}^{A}$ be the set introduced in Definition \ref{Def5.4} for $A\ge1$. Recall from Lemma \ref{Thm5.6} that there is a constant $\eta>1$ so that
\beas
\xib\in \widehat E_{S}^{A}\cap \B(1)^{c}
&\Longrightarrow
\begin{cases}
\text{$n_{j}(\xib_{j})^{e(k_{r},j)}<A n_{k_{r}}(\xib_{k_{r}})$ \quad\quad $\forall j\neq k_{p}$ in $I_{r}$}\\
\text{$n_{k_{p}}(\xib_{k_{p}})<An_{k_{r}}(\xib_{k_{r}})^{\tau_{S}(k_{p},k_{r})}\,\forall 1\leq p\neq r\leq s$}
\end{cases}\\
&
\Longrightarrow
\xib\in \widehat E_{S}^{A^{\eta}}\cap \B(1)^{c}.
\eeas

\medskip

Now pick $\psi\in\CC^{\infty}(\R)$ and $\chi\in \CC^{\infty}(\R^{N})$ such that
\beas
\psi(t) &=
\begin{cases}
1 &\text{if $t\leq 2n$}\\
0 &\text{if $t\geq 4n$}
\end{cases},
&&\text{and}&
\chi(\xib) &=
\begin{cases}
0 &\text{if $|\xib|\leq 1$}\\
1 &\text{if $|\xib|\geq \frac{3}{2}$}
\end{cases}.
\eeas
For $S=\big\{(I_{1},k_{1});\ldots;(I_{s},k_{s})\big\}\in \SS'(\EEE)$, set
\be\label{8.1bb}
\Psi_{S}^{\#}(\xib) = \chi(\xib)\,\prod_{r=1}^{s} \prod_{j\in I_{r}}\psi\big(\widehat N_{j}(\xib)n_{k_{r}}(\xib_{k_{r}})^{-1/e(k_{r},j)}\big).
\ee

\begin{proposition}\quad
\begin{enumerate}[{\rm(a)}]

\item
Let $\overline{\widehat E_{S}}$ be the closure of $\widehat E_{S}$. Then $\Psi_{S}^{\sharp}(\xib)\equiv 1$ in an open neighborhood of the set  $\overline{\widehat E_{S}}\setminus \B(2)$. 

\smallskip

\item There is a constant $A>1$ so that $\Psi_{S}^{\sharp}$ is supported in $\widehat E_{S}^{A}\setminus \overline{\B(1)}$ .

\smallskip

\item $\Psi_{S}^{\sharp}\in \CC^{\infty}(\R^{N})$.

\end{enumerate}
\end{proposition}

\begin{proof}
Suppose that $\xib\in \overline{\widehat E_{S}}\setminus \B(2)$. Then $\chi(\etab) = 1$ for $\etab$ in a neighborhood of $\xib$. Moreover,  $n_{l}(\xib_{l})^{1/e(l,j)}\leq n_{k_{r}}(\xib_{k_{r}})^{1/e(k_{r},j)}$ for every $j\in I_{r}$, and so 
\beas
\widehat N_{j}(\xib)n_{k_{r}}(\xib_{k_{r}})^{-1/e(k_{r},j)} \approx n_{k_{r}}(\xib_{k_{r}})^{-1/e(k_{r},j)}\sum_{l=1}^{n}n_{l}(\xib_{l})^{1/e(l,j)}.
\eeas 
It follows that $\psi\big(\widehat N_{j}(\etab)n_{k_{r}}(\etab_{k_{r}})^{-1/e(k_{r},j)}\big)=1$ for $\etab$ in an open neighborhood of $\xib$, and so the same is true for $\Psi_{S}^{\sharp}(\etab)$.

Next suppose $\Psi_{S}^{\#}(\xib)\neq 0$.  Then  $\chi(\xib)\neq 0$, and $\psi\big(\widehat N_{j}(\xib)n_{k_{r}}(\xib_{k_{r}})^{-1/e(k_{r},j)}\big)\neq 0$ for $1\leq r \leq s$ and every $j\in I_{r}$. This shows that $\xib\in \overline{B(1)}^{c}$, and since $n_{j}(\xib)\lesssim \widehat N_{j}(\xib_{j})$ and $n_{k_{p}}(\xib_{k_{p}})\lesssim \widehat N_{j}(\xib)^{e(k_{p},j)}$, it follows from the assumption on the support of $\psi$ that
\beas
n_{j}(\xib_{j})&\lesssim n_{k_{r}}(\xib_{k_{r}})^{1/e(k_{r},j)} && \text{for $j\in I_{r}$,}\\
n_{k_{p}}(\xib_{k_{p}})&\lesssim n_{k_{r}}(\xib_{k_{r}})^{e(k_{p},j)/e(k_{r},j)}&&\text{for $1\leq p,r\leq s$}.
\eeas
It follows that there exists a constant $A>0$ (depending on $n$ and the matrix $\EEE$) so that $n_{j}(\xib_{j})\leq A\,n_{k_{r}}(\xib_{k_{r}})^{1/e(k_{r},j)}$ and $n_{k_{p}}(\xib_{k_{p}})\leq A\,n_{k_{r}}(\xib_{k_{r}})^{\tau_{S}(k_{p},k_{r})}$. But this means that the support of $\Psi_{S}^{\sharp}$ is contained in the set $\widehat E_{S}^{A^{\eta}}$ where $\eta$ is the constant from Lemma \ref{Thm5.6}.

Finally, if $\xib\notin \B(1)$ and if $\Psi_{S}^{\sharp}(\xib)\neq 0$, then $\widehat N_{j}(\xib)\geq 1$ for $1\leq j \leq n$ and it follows that $n_{k_{r}}(\xib_{k_{r}})$ is bounded away from $0$. Since the function $\widehat N_{j}(\xib)n_{k_{r}}(\xib_{k_{r}})^{-1/e(k_{r},j)}$ is smooth on the set where $\xib_{k_{r}}\neq \0$, and it follows that $\Psi_{S}^{\#}$ is smooth. 
\end{proof}

\begin{lemma}\label{Lem8.1} 
$\Psi_{S}^{\#}\in \MM_{\infty}(\widehat E)$; \textit{i.e.} for every $\gammab\in \N^{C_{1}}\times\cdots\times \N^{C_{n}}$ there is a constant $C_{\gammab}>0$ so that
\bes
\big|\partial^{\gammab}_{\xib}\Psi_{S}^{\#}(\xib)\big|\leq C_{\gammab}\prod_{j=1}^{n}\Big[1+\widehat N_{j}(\xib)\Big]^{-\[\gammab_{j}\]}.
\ees
\end{lemma}

\begin{proof}
The function $\widehat N_j(\xib)$ is homogeneous of degree one relative to the dilations
\beas
(\xib_{1},\ldots,\xib_{n}) \to (\lambda^{e(1,j)}\cdot\xib_{1},\ldots,\lambda^{e(n,j)}\cdot \xib_{n}).
\eeas
and the function $\rho_{j,k_{r}}(\xib)=\widehat N_{j}(\xib)n_{k_{r}}(\xib_{k_{r}})^{-1/e(k_{r},j)}$ is homogeneous of degree zero. The homogeneous dimension of $\R^{C_{l}}$ is $e(l,j)Q_{l}$, and if $\gammab=(\gammab_{1}, \ldots, \gammab_{n})\in \N^{C_{1}}\times\cdots\times \N^{C_{n}}$, the adapted length of $\gammab$ is $\sum_{l=1}^{n}e(l,j)\[\gammab_{j}\]$. It follows from Proposition \ref{Prop12.1} in Appendix II that for $\xib\in \B(1)^{c}$
\beas
|\partial^{\gammab}\rho_{j,k_{r}}(\xib)|
\lesssim 
\Big[1+\widehat N_{j}(\xib)\Big]^{-\sum_{l=1}^{n}e(l,j)\[\gammab_{l}\]}.
\eeas
But on the support of $\Psi_{S}^{\#}$ which is contained in $\widehat E_{S}^{A}$, 
\beas
\Big[1+&\widehat N_{j}(\xib)\Big]^{-\sum_{l=1}^{n}e(l,j)\[\gammab_{l}\]}\\
&=
\prod_{p=1}^{s}\prod_{l\in I_{p}}\Big[1+\widehat N_{j}(\xib)\Big]^{-\[\gammab_{l}\]e(l,j)}
\leq
\prod_{p=1}^{s}\prod_{l\in I_{p}}\Big[1+n_{k_{p}}(\xib_{k_{p}})\Big]^{-\[\gammab_{l}\]e(l,j)/e(k_{p},j)}\\
&\leq 
\prod_{p=1}^{s}\prod_{l\in I_{p}}\Big[1+n_{k_{p}}(\xib_{k_{p}})\Big]^{-\[\gammab_{l}\]/e(k_{p},l)}
\lesssim
\prod_{p=1}^{s}\prod_{l\in I_{p}}\Big[1+N_{l}(\xib)\Big]^{-\[\gammab_{l}\]}
\eeas
since $e(k_{p},j)\leq e(k_{p},l)e(l,j)$, $n_{k_{p}}(\xib_{k_{p}})>1$, and $N_{l}(\xib) \approx n_{k_{p}}(\xib_{k_{p}})^{1/e(k_{p},l)}$ on the support of $\Psi_{S}^{\#}$ if $l\in I_{p}$. The estimates of the Lemma now follow from the chain rule and product rule.
\end{proof}

Let $\Psi_{0}^{\#}$ be a compactly supported function which is identically equal to $1$ on $\B(2)$. We set \index{P3si0S@$\Psi_0,\Psi_S$}
\bea\label{8.3bb}
\Psi_{0}&= \Psi_{0}^{\#}\Big[\Psi_{0}^{\#} + \sum_{S\in \SS'(\EEE)}\Psi_{S}^{\#}\Big]^{-1},&&&&
\Psi_{S}&= \Psi_{S}^{\#}\Big[\Psi_{0}^{\#} + \sum_{S\in \SS'(\EEE)}\Psi_{S}^{\#}\Big]^{-1}.
\eea

Using the fact that, if $m\in\MM_\infty(\EEE)$ is bounded away from zero, then also $1/m\in\MM_\infty(\EEE)$, we then have

\begin{corollary} \label{Cor8.2}\quad 
\begin{enumerate}[{\rm(a)}]
\smallskip

\item $\Psi_{0}(\xib)+\sum_{S\in \SS'(\EEE)}\Psi_{S}(\xib) \equiv 1$ for every $\xib\in \R^{N}$.

\smallskip

\item $\Psi_{S}\in \MM_{\infty}(\EEE)$ for $S=0$ and for every $S\in \SS'(\EEE)$.

\smallskip

\item For $S\in \SS'(\EEE)$ the function $\Psi_{S}$ is supported in $\widehat E_{S}^{A}$.

\smallskip
\end{enumerate}
\end{corollary}

We now use Corollary \ref{Cor8.2} to obtain a decomposition of a multiplier $m\in \MM_{\infty}(\EEE)$.

\index{m30S@$m_0,m_S$}
\begin{corollary}\label{Cor8.3}
Let $m\in \MM_\infty(\EEE)$. Put $m_{0}= \Psi_{0}m$, and $m_{S} = \Psi_{S}m$ for each $S\in \SS'(\EEE)$. Then 
\begin{enumerate}[{\rm(a)}]

\smallskip

\item $m(\xib) = m_{0}(\xib) + \sum_{S\in \SS'(\EEE)}m_{S}(\xib)$ for every $\xib \in \R^{N}$;

\smallskip

\item $m_{S}\in \MM_{\infty}(\widehat E)$ for every $S\in \SS'(\EEE)$;

\smallskip

\item $m_{S}$ is supported in $\widehat E_{S}^{A}$ where $A$ depends only on the matrix $\EEE$;

\smallskip

\item $m_{0}\in \CC^{\infty}_{0}(\R^{N})$ has support in $\B(2)$.
\end{enumerate}
\end{corollary}

\begin{remark}\label{Rem8.5ww}
Since $m_{S}$ is supported on $\widehat E_{S}^{A}$, we have
\beas
\big|\partial^{\gammab}_{\xib}( m_{S})(\xib)\big|
&\lesssim
\prod_{r=1}^{s}\prod_{j\in I_{r}}\widehat N_{j}(\xib)^{-\[\gammab_{j}\]}
\approx
\prod_{r=1}^{s}\prod_{j\in I_{r}}n_{k_{r}}(\xib_{k_{r}})^{-\[\gammab_{j}\]e(k_{r},j)}\\
&\approx
\prod_{r=1}^{s}\widehat n_{S,r}(\xib_{I_{r}})^{-\sum\limits_{j\in I_{r}}\[\gammab_{j}\]/e(k_{r},j)}.
\eeas
This says that the differential estimates for $ m_{S}$ are equivalent to product multiplier estimates relative to the norms $\{\widehat n_{S,r}\}$ on the decomposition $\R^{N}= \R^{I_{1}}\oplus\cdots\oplus\R^{I_{s}}$. 
\end{remark}

All the functions $m_{S}$ belongs to the class $\MM_{\infty}(\EEE)$. If $S=\big\{(I_{1},k_{1});\ldots;(I_{s},k_{s})\big\}\in \SS'(\EEE)$ is not the principal marked partition, one can say more:  the function $m_{S}$ in fact belongs to the class of multipliers $\MM_\infty(\EEE_{S})$ introduced in Section \ref{NewMatrices}.

\begin{proposition}\label{Prop8.4qq}
For $S\in \SS'(\EEE)$, the function $m_S$ belongs to the class $\MM_{\infty}(\EEE_{S})$.
\end{proposition}

\begin{proof}
We must show that for any $\gammab=(\gammab_{I_{1}}, \ldots, \gammab_{I_{s}}) \in \N^{I_{1}}\times \cdots \times \N^{I_{S}}$ there is a constant $C_{\gammab}$ so that
\beas
\big\vert\partial^{\gammab}m_{S}(\xib)\big\vert\leq C_{\gammab}\,\prod_{r=1}^{s}\big(1+\widehat N_{S,r}(\xib)\big)^{- \[\gammab_{I_{r}}\]_{S}},
\eeas
where $\[\gammab_{I_{r}}\]_{S}=\sum\limits_{j\in I_{r}}\[\gammab_{j}\]/e(k_{r},j)$ is the length of the multi-index $\gammab_{I_r}$ adapted to the dilations~$\hat\cdot_S$.
But this is a simple consequence of the estimate in Remark \ref{Rem8.5ww} and the fact that on the set $\widehat E_{S}^{A}$ we have $\widehat n_{S,r}(\xib_{I_r})\approx 1+\widehat N_{S,r}(\xib)$.
\end{proof}

\subsection{Dyadic decomposition of a multiplier}\label{DyadicDecomp}\quad

We now turn to the task of providing a dyadic decomposition for each multiplier $m_{S}$.  Fix $S=\big\{(I_{1},k_{1});\ldots;(I_{s},k_{s})\big\}\in\SS'(\EEE)$, and write $\R^{N}=\R^{I_{1}}\oplus\cdots\oplus\R^{I_{s}}$. On each factor $\R^{I_{r}}$ choose a function $\theta\in\CC^{\infty}_{0}(\R^{I_{r}})$ satisfying\footnote{To keep notation simpler, we do not indicate the dependence of $\theta$ on the index $r$.}
\beas
\theta(\xib_{I_{r}}) =
\begin{cases}
1&\text{if \,\,$\widehat n_{S,r}(\xib_{I_{r}})\leq 1$,}\\
0 & \text{if \,\,$\widehat n_{S,r}(\xib_{I_{r}})\geq 2$.}
\end{cases}
\eeas
Recall that $2^{-j}\,\hat\cdot_{S}\,\xib_{I_{r}}=\big\{2^{-j/e(k_{r},j)}\cdot\xib_{j}\big\}_{j\in I_{r}}$. Put $\theta_{j}(\xib_{I_{r}})= \theta(2^{-j}\,\hat\cdot_{S}\,\xib_{I_{r}})$, and
\beas
\Theta_j(\xib_{I_{r}})&=\theta_j(\xib_{I_{r}})-\theta_{j-1}(\xib_{I_{r}})=\theta(2^{-j}\,\hat\cdot_{S}\,\xib_{I_{r}})-\theta(2^{j-1}\,\hat\cdot_{S}\,\xib_{I_{r}})=\Theta_{0}(2^{-j}\,\hat\cdot_{S}\,\xib_{I_{r}}).
\eeas
Note that $\Theta_{j}$ is supported where $\widehat n_{S,r}(\xib_{I_{r}})\approx 2^{j}$. Precisely:
\bes
\Theta_{j}(\xib_{I_{r}}) \neq 0\quad\Longrightarrow\quad2^{j-1}\leq \widehat n_{S,r}(\xib_{I_{r}}) \leq 2^{j+2}.
\ees 
For $0\neq\xib_{I_{r}}\in \R^{I_{r}}$ we have 
\bes
\sum_{j=-\infty}^{\infty}\Theta_{j}(\xib_{I_{r}})=\theta_{0}(\xib_{I_{r}})+\sum_{j=1}^{\infty}\Theta_{j}(\xib_{I_{r}}) = 1,
\ees
and the second equality holds for $\xib_{I_{r}}=0$ as well. For $J=(j_{1}, \ldots, j_{s})\in \Z^{s}$ define
\bes
\Theta_{J}(\xib)=\Theta_{J}(\xib_{I_{1}}, \ldots, \xib_{I_{s}}) = \Theta_{j_{1}}(\xib_{I_{1}}) \cdots \Theta_{j_{s}}(\xib_{I_{s}}).
\ees
Then $\Theta_{J}$ is supported where each $\widehat n_{S,r}(\xib_{I_{r}})\approx 2^{j_{r}}$,
and $\sum_{J\in \Z^{s}}\Theta_{J}(\xib_{I_{1}}, \ldots,\xib_{I_{s}}) \equiv 1$ provided each $\xib_{I_{r}}\neq 0$. Recall that $m_{S}$ is supported on $\widehat E_{S}^{A}\cap \B(1)^{c}$, and $\xib_{I_{r}}\neq 0$ on this support. Thus we can write 
\beas
m_{S}(\xib) = \sum_{J\in \Z^{s}}\Theta_{J}(\xib)m_{S}(\xib)=\sum_{J\in \Z^{s}}\Theta_{J}(\xib)\Psi_{S}(\xib)m(\xib).
\eeas 

\begin{proposition}
There is a constant $\mu>0$ depending only on $\EEE$ so that if $J\in \Z^{s}$ and $\Theta_{J}(\xib)\Psi_{S}(\xib)m(\xib)\neq 0$, then $J$ belongs to 
$$
 -\Gamma_{\Z,\mu}(\EEE_S)=\Big\{J=(j_{1}, \ldots, j_{s})\in \Z^{s}: \text{$0< j_{p}\leq \tau_{S}(k_{p},k_{r})j_{r}+\mu$}\Big\}.
 $$
\end{proposition}

\begin{proof}
Suppose $\Theta_{J}(\xib)\Psi_{S}(\xib)m(\xib)\neq 0$. Then $\Theta_{J}(\xib)\neq 0$, and it follows that $2^{j_{r}-1}\leq \widehat n_{S,r}(\xib_{I_{r}})\leq 2^{j_{r}+2}$. Also $\Psi_{S}(\xib)\neq 0$, and it follows that $\xib\in \widehat E_{S}^{A}$ and so 
\beas
1\leq\widehat n_{S,p}(\xib_{I_{p}})^{1/\tau_{S}(k_{p},k_{r})}<A\,\widehat n_{S,r}(\xib_{I_{r}})
\eeas 
for all $r,p\in \{1, \ldots, s\}$. It follows that
\beas
\big(2^{j_{p}-1}\big)^{1/\tau_{S}(k_{p},k_{r})}&\leq \widehat n_{S,p}(\xib_{I_{p}})^{1/\tau_{S}(k_{p},k_{r})}<A\,\widehat n_{S,r}(\xib_{I_{r}})\leq A\,2^{j_{r}+1},
\eeas
and so 
\beas
j_{p}\leq \tau_{S}(k_{p},k_{r})j_{r}+\big[\tau_{S}(k_{p},k_{r})(1+\log_{2}(A))+1\big]\leq\tau_{S}(k_{p},k_{r})j_{r}+\mu
\eeas
for an appropriate constant $\mu$.
\end{proof}

It now follows that 
\beas
m_{S}(\xib) &= \sum_{J\in -\Gamma_{\Z,\mu}(\EEE_S)}\Theta_{J}(\xib)m_{S}(\xib)\\
&=
\sum_{J\in -\Gamma_{\Z,\mu}(\EEE_S)}\Theta_{\0}(2^{-j_{1}}\,\hat\cdot_{S}\,\xib_{I_{1}}, \ldots, 2^{-j_{s}}\,\hat\cdot_{S}\,\xib_{I_{s}}) \,m_{S}(\xib_{I_{1}}, \ldots, \xib_{I_{s}}).
\eeas
Put
\beas
m_{S}^{J}(\xib)
= \Theta_{0}(\xib)m_{S}(2^{j_{1}}\,\hat\cdot_{S}\,\xib_{I_{1}}, \ldots, 2^{j_{s}}\,\hat\cdot_{S}\,\xib_{I_{s}})
\eeas
so that we can write
\beas
m_{S}(\xib_{I_{1}}, \ldots, \xib_{I_{s}}) = \sum_{J\in -\Gamma_{\Z,\mu}(\EEE_S)}m_{S}^{J}(2^{-j_{1}}\hat\cdot_{S}\xib_{I_{1}}, \ldots, 2^{-j_{s}}\hat\cdot_{S}\xib_{I_{S}}).
\eeas
\begin{proposition}
The set of functions $\Big\{m_{S}^{J}:J\in -\Gamma_{\Z,\mu}(\EEE_S)\Big\}$ is a uniformly bounded family in $\CC^{\infty}_{0}(\R^{N})$, each supported where $\frac{1}{4}<\widehat n_{S,r}(\xib_{I_{r}})<4$ for $1 \leq r \leq s$. 
\end{proposition}

\begin{proof}
Since $m_{S}$ is supported on $\widehat E_{S}^{A}$,
\beas
|\partial^{\gammab}\big(m_{S}(2^{j_{1}}\,\hat\cdot_{S}\,\xib_{I_{1}}, \ldots, 2^{j_{s}}\,\hat\cdot_{S}\,\xib_{I_{s}})\big)(\xib)| 
&= 
2^{\sum_{r=1}^{s}\[\gammab_{r}\]/e(k_{r},j)}\,\left|\partial^{\gammab}_{\xib}( m_{S})(2^{j_{1}}\,\hat\cdot_{S}\, \xib_{I_{1}}, \ldots, 2^{j_{s}}\,\hat\cdot_{S}\, \xib_{I_{s}})\right|\\
&\lesssim
2^{\sum_{r=1}^{s}\[\gammab_{r}\]/e(k_{r},j)}\,\prod_{r=1}^{s}\widehat n_{S,r}(2^{j_{r}}\,\hat\cdot_{S}\,\xib_{I_{r}})^{-\[\gammab_{l}\]/e(k_{r},j)}\\
&=
C_{\gammab}\,\prod_{r=1}^{s}\widehat n_{S,r}(\xib_{I_{r}})^{-\[\gammab_{l}\]/e(k_{r},j)}
\eeas
But on the support of $\theta_{\0}$ the product $\prod_{r=1}^{s}\widehat n_{S,r}(\xib_{I_{r}})^{-\[\gammab_{l}\]/e(k_{r},j)}$ is uniformly bounded and bounded away from zero (depending on $\gammab$). 
\end{proof}

We have  written $m_{S}(\xib) = \sum_{J\in -\Gamma_{\Z,\mu}(\EEE_S)}m_{S}^{J}(2^{-J}\,\hat\cdot_{S}\,\xib)$. 
We want to replace the sum over $-\Gamma_{\Z,\mu}(\EEE_S)$ by a sum over $-\Gamma_{\Z}(\EEE_S)$.

For each $J\in -\Gamma_{\Z,\mu}(\EEE_S)$, choose $J'=p(J)\in -\Gamma_{\Z}(\EEE_S)$ at minimal distance from $J$. It is quite clear that there is a finite upper bound $M$ for the number of $J$ having the same $p(J)$ and that, if $p(J)=J'$, then $|j_r-j'_r|\le c(\mu)$ for each  $r=1,\dots,s$. It follows that we can rewrite
\bea
m_{S}(\xib) &= \sum_{J\in -\Gamma_{\Z,\mu}(\EEE_S)}m_{S}^{J}(2^{-J}\,\hat\cdot_{S}\,\xib) \\
&= \sum_{J'\in -\Gamma_{\Z}(\EEE_S)}\sum_{J:p(J)=J'} m_{S}^{J}(2^{-J}\,\hat\cdot_{S}\,\xib)\\
&=\sum_{J'\in -\Gamma_{\Z}(\EEE_S)}\widetilde m_{S}^{J'}(2^{-J'}\,\hat\cdot_{S}\,\xib)\ ,
\eea
where the functions
$$
\widetilde m_{S}^{J'}(\xib)=\sum_{J:p(J)=J'} m_{S}^{J}(2^{J'-J}\,\hat\cdot_{S}\,\xib)
 $$
form a a uniformly bounded family in $\CC^{\infty}_{0}(\R^{N})$.

Applying this decomposition for each marked partition $S\in \SS'(\EEE)$ we have the following decomposition for multipliers.

\begin{theorem}\label{Thm8.4}
Let $m\in \MM_{\infty}(\EEE)$. There is a uniformly bounded collection of functions $\{m_{S}^{J}\}\subset \CC^{\infty}_{0}(\R^{N})$ indexed by  
\beas
S\in \SS'(\EEE)=\big\{T\in \SS(n):\widehat E_{T}\cap \B(1)^{c}\neq\emptyset\big\}
\eeas 
and
$J\in \Gamma_{\Z}(\EEE_S)$ 
with the following properties.
\begin{enumerate}[{\rm(a)}]

\smallskip

\item If $S=\big\{(I_{1},k_{1});\ldots;(I_{s},k_{s})\big\}\in \SS'(\EEE)$ and $J\in \Gamma_{\Z}(\EEE_S)$, the function $m_{S}^{J}$ is supported where $c^{-1}<\widehat n_{S,r}(\xib_{I_{r}})<c$ for $1\leq r\leq s$ and an appropriate constant $c$. 

\smallskip

\item The multiplier $m$ can be written
\bea\label{8.8tt}
m(\xib) =m_{0}(\xib)+ \sum_{S\in \SS'(\EEE)}\sum_{J\in -\Gamma_{\Z}(\EEE_S)}m_{S}^{J}(2^{-j_{1}}\,\hat\cdot_{S}\,\xib_{i_{1}}, \ldots, 2^{-j_{s}}\,\hat\cdot_{S}\,\xib_{I_{s}})
\eea
where $m_{0}\in \CC^{\infty}_{0}(\R^N)$ is supported in $\B(2)$.
\end{enumerate}
\end{theorem}

\subsection{Dyadic decomposition of a kernel}\label{SSKernelDecomp}\quad

\medskip

We can now decompose a distribution $\KK\in \PP_{0}(\EEE)$. We know from Theorem \ref{Thm6.2} that the Fourier transform $m=\widehat {\KK} \in \MM_{\infty}(\EEE)$, and by Theorem \ref{Thm8.4} we can write $m$ as in equation (\ref{8.8tt}). Write
\beas[]
[\psi]^{J}(\xib_{I_{1}}, \ldots, \xib_{I_{s}})= \psi\big(2^{-j_{1}}\,\hat\cdot_{S}\,\xib_{I_{1}}, \ldots, 2^{-j_{s}}\,\hat\cdot_{S}\,\xib_{I_{s}}\big)
\eeas
so that 
\beas
\widehat \KK = m_{0}+\sum_{S\in \SS'(\EEE)}\sum_{J\in -\Gamma_{\Z}(\EEE_S)}[m_{S}^{J}]^{J}.
\eeas
Then if $\spcheck$ denotes the inverse Fourier transform, we have 

\beas
\KK = (m_{0})^{\spcheck}+\sum_{S\in \SS'(\EEE)}\sum_{J\in -\Gamma_{\Z}(\EEE_S)}([m_{S}^{J}]^{J})\spcheck.
\eeas 
Computing the inverse Fourier transform of $[m_{S}^{J}]^{J}$ we have 
\beas
([m_{S}^{J}]^{J})\spcheck(\x_{I_{1}}, \ldots, \x_{I_{S}})
&=
2^{\sum_{r=1}^{s}j_{r}\widehat Q_{I_{r}}}\,[m_{S}^{J}]\spcheck(2^{j_{1}}\,\hat\cdot_{S}\,\x_{I_{1}}, \ldots, 2^{j_{s}}\,\hat\cdot_{S}\,\x_{I_{s}}).
\eeas
Since the family $\{m_{S}^{J}\}$ is uniformly bounded in $\CC^{\infty}_{0}(\R^{N}) \subset\SS(\R^{N})$, it follows that the family $\{([m_{S}^{J}]^{J})\spcheck\}$ is also a uniformly bounded family in $\SS(\R^{N})$. Since each $m_{S}^{J}$ is supported where $c^{-1}<\widehat n_{S,r}(\xib_{I_{r}})<c$ for all $r$, and since $\widehat n_{S,r}(\xib_{I_{r}})\approx n_{k_{r}}(\xib_{k_{r}})$, it follows that 
\bea\label{8.10ff}
\int_{\R^{C_{k_{r}}}}(m_{S}^{J})\spcheck(\x_{1}, \ldots, \x_{k_{r}}, \ldots, \x_{n})\,d\x_{k_{r}}=0
\eea
for $1 \leq r \leq s$; \textit{i.e.} the function $(m_{S}^{J})\spcheck$ has cancellation in each of the variables $\{\x_{k_{1}}, \ldots, \x_{k_{s}}\}$.
Let us write
\bea
\psi_{0}&=(m_{0})^{\spcheck},&&&\psi_{S}^{J}&= ([m_{S}^{J}]^{J})\spcheck= [(m_{S}^{-J})^{\vee}]_{-J}.
\eea

\begin{theorem}\label{Thm8.5} 
Let $\KK\in \PP_{0}(\EEE)$. There is a function $\psi_{0}\in \SS(\R^{N})$ and for each $S\in \SS'(\EEE)$ there is a uniformly bounded collection of functions $\big\{\psi_{S}^{J}\big\}_{J\in \Gamma_{\Z}(\widehat\EEE_{S})}\subset\SS(\R^{N})$ so that:
\begin{enumerate}[{\rm(1)}]

\smallskip

\item 
Let $S=\big\{(I_{1},k_{1});\ldots;(I_{s},k_{s})\big\}\in \SS'(\EEE)$.

\begin{enumerate}[{\rm(a)}]

\smallskip

\item \label{Thm8.5a} Each function $\psi^{J}_{S}$ has cancellation in each of the variables $\{\x_{k_{1}}, \ldots, \x_{k_{s}}\}$, and so has strong cancellation relative to the decomposition $\R^{N}=\R^{I_{1}}\oplus\cdots \oplus\R^{I_{s}}$. 

\smallskip
\item \label{Thm8.5b} 
The series 
\bes
\KK_S(\x)=\sum_{J\in \Gamma_{\Z}(\EE_{S})}\big[\psi_{S}^{J}\big]_{J}(\x) =\sum_{J\in \Gamma_{\Z}(\EE_{S})}2^{-\sum_{r=1}^{s}i_{r}\widehat Q_{I_{r}}}\psi_{S}^{J}(2^{-J}\hat\cdot_{S}\x)
\ees
converges to an element of $\PP_{0}({\EEE}_{S})\subset \PP_{0}(\EEE)$ in the sense of distributions .
\end{enumerate}
\smallskip

\item\label{Thm8.5c} The distribution $\KK$ decomposes as $\displaystyle \KK(\x) = \psi_{0}(\x)+\sum_{S\in \SS'(\EEE)}\KK_{S}(\x)$.
\end{enumerate}
\end{theorem}

\begin{proof}
Let $m=\widehat\KK\in \MM_{\infty}(\EEE)$. Using Theorem \ref{Thm8.4}, write 
$$
m=m_{0}+\sum_{S\in \SS'(\EEE)}\sum_{J\in -\Gamma_{\Z}(\EE_{S})}[m_{S}^{J}]^{J}.
$$
  Then for each $S\in \SS'(\EEE)$ and each {$J\in \Gamma_{\Z}(\EE_{S})$}, let $\psi_{S}^{J}(\x)=(m_{S}^{-J})\spcheck(\x)$. We have already observed in equation (\ref{8.10ff}) that $\psi_{S}^{J}$ has cancellation in each variable $\x_{k_{r}}$, $1\leq r \leq s$, which gives assertion (\ref{Thm8.5a}).

 Next 
\beas
\KK_{S}&= \sum_{J\in \Gamma_{\Z}(\EE_{S})}[\psi_{S}^{J}]_{J}=\sum_{J\in -\Gamma_{\Z}(\EE_{S})}[\psi_{S}^{-J}]_{-J}=\sum_{J\in -\Gamma_{\Z}(\EE_{S})}([m_{S}^{J}]^{J})^{\vee}=m_{S}^{\vee}.
\eeas
According to Proposition \ref{Prop8.4qq}, $m_{S}\in \MM_{\infty}(\EEE_{S})$, and using Theorem \ref{Thm6.1} it follows that $\KK_{S}\in \PP_{0}(\EEE_{S})$. But according to Lemma \ref{Lem8.7mm}, $\MM_{\infty}(\EEE_{S})\subset \MM_{\infty}(\EEE)$, and every distribution in $\PP_{0}(\EEE_{S})$ belongs to $\PP_{0}(\EEE)$. This establishes assertion (\ref{Thm8.5b}).

Finally assertion (\ref{Thm8.5c}) is then a consequence of the identity $m=m_0+\sum_{S\in \SS'(\EEE)}m_{S}$, and this completes the proof.
\end{proof}

Theorem \ref{Thm8.5} shows that each kernel $\KK$ can be written as sums of dilates of uniformly bounded families of Schwartz functions. As in \cite{MR2949616}, we can improve this result and show that we can replace Schwartz functions by uniformly bounded families of functions compactly supported in the unit ball. We use the following result, which is Lemma 6.5 in \cite{MR2949616}.

\begin{lemma}\label{Lem8.7qq}
Let $\psi\in \SS(\R^N)$. Then there are functions $\{\varphi^{k}\}\subset\CC^\infty_{0}(\R^N)$, $k\in \N$ such that
\bes
\psi(\x)= \sum_{k=0}^{\infty} 2^{-kQ}\varphi^{k}(2^{-k}\cdot\x)
\ees
where $Q$ is the homogeneous dimension of $\R^N$ relative to the given dilations. The functions $\{\varphi^{k}\}$ have the following properties.
\begin{enumerate}[{\rm(a)}]
\smallskip

\item Each $\varphi^{k}$ is supported in $\B(1)$.

\smallskip

\item For any $\delta>0$ and any $\gammab\in \N^N$ there exists a positive integer $M$ so that 
\bes
\sup_{\x\in\R^N} |\partial^{\gammab}\varphi^{k}(\x)|\leq 2^{-k\delta}\sup_{\x\in\R^N}\sum_{|\alphab|\leq M}|\partial^{\alphab}\psi(\x)|.
\ees

\item If $\psi$ has strong cancellation, then each $\varphi^{k}$ also has strong cancellation.
\end{enumerate}
\end{lemma}

Now let $\KK\in \PP_{0}(\EEE)$  and let $\KK=\psi_{0}+\sum_{S\in \SS'(\EEE)}\KK_S$ be the decomposition given in Theorem~\ref{Thm8.5} where $\KK_S(\x)=\sum_{J\in \Gamma_{\Z}(\EE_{S})}[\psi_{S}^{J}]_{J}(\x)$, and each $\psi_{S}^{J}$ has cancellation in the variables $\{\x_{k_{1}}, \ldots, \x_{k_{s}}\}$. We apply Lemma \ref{Lem8.7qq} to each function $\psi_{S}^{J}$:
\bes
\psi_{S}^{J}(\x)=\sum_{k=0}^{\infty} 2^{-kQ_S}\varphi_{S}^{J,k}(2^{-k}\cdot\x)
\ees
where each $\{\varphi_{S}^{J,k}\}\subset \CC^{\infty}_{0}(\R^N)$ is a uniformly bounded family of functions supported in $\B(1)$, each having cancellation in the variables $\{\x_{k_{1}}, \ldots, \x_{k_{s}}\}$. Then 
\beas
\KK_S(\x) &=
\sum_{J\in \Gamma_{\Z}(\EE_{S})}\left[\sum_{k=0}^{\infty} [\varphi_{S}^{J,k}]_k\right]_{J}(\x).
\eeas
The argument on pages 661-663 of \cite{MR2949616} shows that this last sum can be rewritten to yield 
\beas
\KK_S(\x)=\sum_{J\in \Gamma_{\Z}(\EE_{S})} [\widetilde\varphi_{S}^{J}]_{J}(\x)
\eeas
where $\{\widetilde\varphi_{S}^{J}\}$ is a uniformly bounded family in $\CC^{\infty}_{0}(\R^{N})$, all supported in $\B(1)$, and all having cancellation in the variables $\{\x_{k_{1}}, \ldots, \x_{k_{s}}\}$. Thus we have

\goodbreak

\begin{corollary}\label{Cor8.8}
Let $\KK\in \PP_{0}(\EEE)$. For each $S\in \SS'(\EEE)$ there is a uniformly bounded collection of functions $\big\{\varphi_{S}^{J}\big\}_{J\in \Gamma_{\Z}(\EE_{S})}\subset\CC^{\infty}_{0}(\R^{N})$ and a function $\psi_{0}\in \SS(\R^N)$ so that  
\begin{enumerate}[{\rm(a)}]

\smallskip

\item \label{Cor8.8a} 
Each function $\varphi^{J}_{S}$ has cancellation in the variables $\x_{k_{1}}, \ldots, \x_{k_{s}}$.
 
\smallskip

\item \label{Cor8.8b} 
Let
\beas[]
[\varphi_{S}^{J}]^{J}(\x) = 2^{-}\varphi_{S}^{J}(2^{-j_{1}}\,\hat\cdot_{S}\,\x_{I_{1}}, \ldots, 2^{-j_{s}}\,\hat\cdot_{S}\,\x_{I_{s}}),
\eeas
where 
\beas
2^{-j_{r}}\,\hat\cdot_{S}\,\x_{I_{r}}= \Big\{2^{-j_{r}/e(k_{r},j)}\cdot\x_{j}\Big\}_{j\in I_{r}}.
\eeas
 Then the series $\KK_S=\sum_{J\in \Gamma_{\Z}( \EEE_{S})}\big[\varphi_{S}^{J}\big]_{J}$ converges in the sense of distributions to an element $\KK_{S} \in \PP_{0}({\EEE}_{S})\subset \PP_{0}(\EEE)$.

\smallskip

\item\label{Cor8.8c} 
$
\KK(\x) = \psi_{0}(\x)+\sum_{S\in \SS'(\EEE)}\KK_{S}(\x).
$
\end{enumerate}
\end{corollary}

\section{The rank of $\EEE$ and integrability at infinity}\label{Integrability}

 So far we have analyzed kernel in the restricted class $\PP_0(\EEE)$, which are the sum of a kernel in $\PP(\EEE)$ with compact support and  a Schwartz function.

In this Section we  discuss
the nature of a distribution $\KK\in \PP(\EEE)$ and show that it depends in the first place on the rank of the matrix $\EEE=\{e(j,k)\}$. 

We will see that if the rank equals $1$ then the distribution $\KK$ is a standard (non-isotropic) \CZ kernel. However, if the rank is greater than $1$, we will see that the kernel $K$ is integrable at infinity, but has worse behavior near the origin than a \CZ kernel.

We will also show that we have integrability at infinity in a more general situation. Suppose that the Fourier transform $m$ of a distribution $\KK$ satisfies the differential inequalities in Definition \ref{Def2.4} for pure derivatives; \textit{i.e.} for every $1 \leq j \leq n$ and for every $\gammab_{j}\in \N^{C_{j}}$ there is a constant $C_{\gammab_{j}}>0$ so that $|\partial^{\gammab_{j}}_{\xib_{j}}m(\xib)|\leq C_{\gammab_{j}}N_{j}(\xib)^{-\[\gammab_{j}\]}$. Then if the rank of $\EEE$ is greater than $1$, the function $K$ is integrable at infinity. This result will be used in Section \ref{Multi-Flag} below when we consider distributions which are simultaneously flag kernels for two different flags.

\subsection{The rank $1$ case: \CZ kernels}\label{Rank1}\quad

We begin by recalling the definition of a \CZ kernel\footnote{For reasons of compatibility with the framework of this paper, we restrict ourselves to kernels and multipliers which are $C^\infty$ away from the origin and to dilations which are compatible with the decomposition 
\ref{RNagain} and with the dilations in \eqref{Eqn2.1aa}.}. Start with the usual decomposition 
\be\label{RNagain}
\R^{N}= \R^{C_{1}}\oplus \cdots \oplus \R^{C_{n}},
\ee
 and use the standard notation of Section \ref{Kernels}. For $\a=(a_{1}, \ldots, a_{n})$ an $n$-tuple of positive numbers, define a one-parameter family of dilations by setting 
 \be\label{lambdaa}
 \lambda\cdot_{\a}\x=(\lambda^{a_{1}}\cdot \x_{1}, \ldots, \lambda^{a_{n}}\cdot \x_{n}).
\ee
 \index{l3ambdadota@$\lambda\cdot_{\a}\x$}
  With this homogeneity, $\R^{N}$ has homogeneous dimension $Q_\a= \sum_{j=1}^{n}a_{j}Q_{j}$, and 
  $$
  N_\a(\x) = n_{1}(\x_{1})^{1/a_{1}}+ \cdots + n_{n}(\x_{n})^{1/a_{n}}
  $$
  \index{N5a@$N_\a$} is a homogeneous norm. We can then define the class of \CZ kernels with this homogeneity.

\begin{definition}\label{CZa}
$\CC\ZZ_{\a}$\index{C1Z@$\CC\ZZ_\a$} is the space of tempered distributions $\KK$ on $\R^{N}$ such that
\begin{enumerate}[{\rm(a)}]
\smallskip

\item away from the origin, $\KK$ is given by integration against a smooth function $K$ satisfying the differential inequalities
\beas
\big\vert\partial^{\gammab}K(\x)\big\vert \leq C_{\gammab}N_\a(\x)^{-Q_\a-\[\gammab\]_\a},
\eeas
for every multi-index $\gammab$, with $\[\gammab\]_\a=\sum_{j=1}^na_j\[\gammab_j\]$;

\item there is a constant $C>0$ so that for any normalized bump function $\psi\in \CC^{\infty}_{0}(\R^{N})$ with support in the unit ball and for any $R>0$, $\big\vert\big\langle\KK,\psi_{R}\big\rangle\big\vert\leq C$, where $\psi_{R}(\x) = \psi(R\cdot_{\a}\x)$.
\end{enumerate}
\end{definition}

It is well known, cf. \cite{MR1818111}, that, for distributions $\KK\in \CC\ZZ_{\a}$, the corresponding multipliers  $m=\widehat \KK$ are characterized by the differential inequalities 
$$
\big\vert\partial^{\gammab}m(\xib)\big\vert\leq C_{\gammab}N_\a(\xib)^{-\[\gammab\]_{\a}}.
$$ 
Put 
$$
N_{j}(\xib)= n_{1}(\xib)^{\frac{a_{j}}{a_{1}}}+ \cdots + n_{n}(\xib)^{\frac{a_{j}}{a_{n}}}\approx N_\a(\xib)^{a_j}.
$$ 
Then these differential inequalities can be written as $\big\vert\partial^{\gammab}m(\xib)\big\vert\leq C_{\gammab}\prod_{j=1}^{n}N_{j}(\xib)^{-\[\gammab_j\]}$,
and so the multipliers are the elements of  the class $\MM(\EEE_{\a})$ where
\bea\label{Ea}
\EEE_{\a}= \left[\begin{matrix}
1&\frac{a_{1}}{a_{2}}&\frac{a_{1}}{a_{3}}&\cdots&\frac{a_{1}}{a_{n}}\\
\frac{a_{2}}{a_{1}}&1&\frac{a_{2}}{a_{3}}&\cdots&\frac{a_{2}}{a_{n}}\\
\vdots&\vdots&\vdots&\ddots&\vdots\\
\frac{a_{n}}{a_{1}}&\frac{a_{n}}{a_{2}}&\frac{a_{n}}{a_{3}}&\cdots&1
\end{matrix}\right].
\eea

\begin{proposition}\label{Prop4.1aa}
Let $\EEE=\{e(j,k)\}$ satisfy the basic hypotheses \eqref{2.5}. The following are equivalent:
\begin{enumerate}[{\rm(1)}]

\smallskip

\item $\text{rank}(\EEE) =1$;

\medskip

\item $e(j,k)=e(j,l)e(l,k)$ for all $1\leq j,k,l \leq n$;

\medskip

\item there is a dilation structure on $\R^{N}$, compatible with the decomposition \eqref{RNagain}, so that if $m\in \MM(\EEE)$, then $m$ is a Mihlin-H\"ormander multiplier relative to that structure.

\medskip

\item there is a dilation structure on $\R^{N}$, compatible with the decomposition \eqref{RNagain}, so that if $\KK\in \PP(\EEE)$, then $\KK$ is a \CZ kernel relative to that structure;

\end{enumerate}
\end{proposition}

\begin{proof}
The rank of $\EEE$ equals 1 if and only if all the rows are proportional. Due to the 1's along the diagonal this means that, for every $j\ne k$ the $j$-th row equals the $k$-th row multiplied by $e(j,k)$. Thus (1) and (2) are equivalent. 

If (2) holds, we have in particular that $e(j,k)e(k,j)=1$ for all $j,k$. Hence $\widehat N_j=N_j$ for all~$j$. Moreover, in analogy with \eqref{reducednorms},
$$
N_j(\x)\approx N_1(\x)^{e(j,1)}.
$$
Thus the differential inequalities in Definition \ref{Def2.4} can be written
\beas
\big|\partial^{\gammab}m(\xib)\big|&\leq C_{\gammab}\prod_{j=1}^{n}\widehat N_{j}(\xib)^{-\[\gammab_{j}\]}\\ &\approx \prod_{j=1}^{n}N_1(\xib)^{-e(j,1)\[\gammab_{j}\]}= N_1(\xib)^{-\sum_{j=1}^{n}e(j,1)\[\gammab_{j}\]}\\
&=N_1(\xib)^{-\[\gammab\]},
\eeas
where $\[\gammab\]=\sum_{j=1}^{n}e(j,1)\[\gammab_{j}\]=\sum_{j=1}^{n}e(1,j)^{-1}\[\gammab_{j}\]$ is the length of $\gamma$ in the dilation structure that makes $N_1$ a homogeneous norm.
Thus the multiplier $m(\xib)$ satisfies the Mihlin-H\"ormander differential inequalities relative to the family of dilations, and so (3) holds.

Conversely, assume that (3) holds. Then there are exponents $\beta_1=1,\beta_2,\dots,\beta_n$ such that, setting
$$
N(\xib)=\sum_{j=1}^n n_j(\xib)^{\beta_j},
$$
 every $m\in\MM(\EEE)$ satisfies the inequalities
$$
\big|\partial^{\gammab}m(\xib)\big|\lesssim N(\xib)^{-\sum_{j=1}^{n}\beta_j^{-1}\[\gammab_{j}\]}.
$$
This is equivalent to saying that $m\in\MM(\EEE')$ with
$$
e'(j,k)=\frac{\beta_k}{\beta_j}.
$$
Hence in our hypothesis we have the inclusion $\MM(\EEE)\subseteq\MM(\EEE')$, which implies that $\MM_\infty(\EEE)\subseteq\MM_\infty(\EEE')$. By Proposition \ref{inclusion}, this implies that $\Gamma(\EEE)\subseteq \Gamma(\EEE')$. By Lemma \ref{Lem3.7}, the dimension of $\Gamma(\EEE')$ equals the reduced rank of $\EEE'$, which is 1. Hence also $\Gamma(\EEE)$ has dimension 1, so it must coincide with $\EEE'$. This proves that (3) implies (1).

Finally, the equivalence of (3) and (4) is proved in \cite{MR1818111}.
\end{proof}
\medskip

\subsection{Higher rank and integrability at infinity}\label{HigerRank}\quad

In this section let $\EEE=\{e(j,k)\}$ be an $n\times n$ matrix satisfying the basic hypothesis (\ref{2.5}), and as usual let $N_{j}(\x) = \sum_{k=1}^{n}n_{k}(\x_{k})^{e(j,k)}$ and $\widehat N_{j}(\xib) = \sum_{k=1}^{n}n_{k}(\xib_{k})^{1/e(k,j)}$. Let $Q_{j}$ be the homogeneous dimension of $\R^{C_{j}}$. We begin by studying the integrability at infinity of a  measurable function $K$ on $\R^{N}$ which satisfies the inequality 
\bea\label{Eqn7.1wer}
|K(\x)|\leq C\,\prod_{j=1}^{n}N_{j}(\x)^{-Q_{j}}.
\eea
Note that nothing is said about the size of the derivatives of $K$.

\begin{lemma}\label{Lem4.2}
Suppose that the rank of $\EEE$ is greater than or equal to $2$. If $K$ satisfies the inequality in equation (\ref{Eqn7.1wer}), then $K$ is integrable over the complement of the unit ball $\B(1)$.
\end{lemma}
\begin{proof}
Since the rank of $\EEE$ is strictly larger than $1$, it follows from Proposition \ref{Prop4.1aa} that there are integers $j,k,l\in \{1, \ldots, n\}$ so that $e(j,k)<e(j,l)e(l,k)$. Moreover, for every $k\in \{1, \ldots, n\}$, we have
\begin{equation}\label{4.2aa}
|K(\x)|\leq C\prod_{j=1}^{n}N_{j}(\x)^{-Q_{j}}\leq C\prod_{j=1}^{n}|\x_{k}|^{-e(j,k)Q_{j}}=C|\x_{k}|^{-\sum_{j=1}^{n}e(j,k)Q_{j}}=C|\x_{k}|^{-\widetilde Q_{k}}
\end{equation}
where $\widetilde Q_{k}= \sum_{\substack{j=1}}^{n}e(j,k)Q_{j}$. Decompose the complement of $\B(1)$ as the union of sets $E_{\omega}$ where $\omega$ ranges over the collection of non-empty subsets of $\{1, \ldots, n\}$, and
\bes
E_{\omega}=\{\x\in\R^{N}:\text{$|\x_{k}|\geq 1$ for $k\in \omega$, $|\x_{k}|<1$ for $k\notin\omega$}\}.
\ees
Fix $\omega$, and let $\{\theta_{l}\}_{l\in\omega}$ be strictly positive numbers such that $\sum_{l\in\omega}\theta_{l}=1$. Then it follows from (\ref{4.2aa}) that $|K(\x)| \lesssim \prod_{l\in\omega}|\x_{l}|_{l}^{-\theta_{l}\widetilde Q_{l}}$, and if we can choose the constants $\{\theta_{l}\}$ such that $\theta_{l}\widetilde Q_{l}>Q_{l}$ for every $l$, then 
\bes
\int_{E_{\omega}}|K(\x)|\,d\x \lesssim \prod_{l\in\omega}\,\int\limits_{\substack{\x_{l}\in \R^{C_{l}}\\|\x_{l}|\geq 1}}|\x_{l}|^{-\theta_{l}\widetilde Q_{l}}\,d\x_{l}<+\infty.
\ees
We can choose $\{\theta_{l}\}$ with the required property if \, $\sum_{l\in \omega}Q_{l}\widetilde Q_{l}^{-1}<1$, and we can do this for every subset $\omega$ if \,$\sum_{l=1}^{n}Q_{l}\widetilde Q_{l}^{-1}<1$. But for any fixed $k\in \{1, \ldots, n\}$, we have $e(j,l)\geq e(j,k)/e(l,k)$ and so
\be\label{4.3aa}
\sum_{l=1}^{n}Q_{l}\widetilde Q_{l}^{-1}
= 
\sum_{l=1}^{n}Q_{l}\Big[\sum_{\substack{j=1}}^{n}e(j,l)Q_{j}\Big]^{-1}
\leq
\sum_{l=1}^{n}e(l,k)\,Q_{l}\Big[\sum_{j=1}^{n}e(j,k)Q_{j}\Big]^{-1}=1.
\ee
Moreover, we can make the inequality in (\ref{4.3aa}) strict if there exists a single triple $j,k,l\in\{1,\ldots,n\}$ such that $e(j,k)<e(j,l)e(l,k)$. This completes the proof.\end{proof}

\begin{corollary}\label{Cor5.3ww}
Suppose that $\text{rank}\,(\EEE)\geq 2$. If $\KK\in \PP(\EEE)$ we can write $\KK=\KK_{0}+\KK_{\infty}$ where $\KK_{0}, \KK_{\infty}\in \PP(\EEE)$, $\KK_{0}$ is a distribution with compact support in the unit ball, and $\KK_{\infty}$ is a distribution given by integration against a function $K_{\infty}\in \CC^{\infty}(\R^{N})\cap L^{1}(\R^{N})$ which satisfies the differential inequalities of Definition \ref{Def2.2}.
\end{corollary}

We also will need the following variant of Lemma \ref{Lem4.2}.

\begin{lemma}\label{Lem4.6}
Suppose that $\KK$ is a tempered distribution on $\R^{n}$ and suppose that its Fourier transform $m=\widehat\KK$ is a smooth function away from zero which satisfies the following differential inequalities for pure derivatives: for $1 \leq j \leq n$ and for every $\gammab_{j}\in \N^{C_{j}}$ there exists a constant $C_{\gammab_{j}}>0$ so that
\beas
\big\vert\partial^{\gammab_{j}}_{\xib_{j}}m(\xib)\big\vert \leq C_{\gammab_{j}}\widehat N_{j}(\xib)^{-\[\gammab_{j}\]}.
\eeas
Then $\KK$ is given by integration against a smooth function $K$ away from the origin, and if the rank of the matrix $\EEE$ is greater than $1$, then $K$ is integrable at infinity.\footnote{Note that we only make an hypothesis about ``pure'' derivatives of $m$. We are not assuming that $m \in \MM(\EEE)$.}
\end{lemma}
\begin{proof}
If we can show that $\KK$ is given by integration agains a function $K$ with $|K(\x)|\lesssim n_{k}(\x_{k})^{-\sum_{j=1}^{n}e(j,k)Q_{j}} $ for $1\leq k \leq n$, then the argument proceeds as in the proof of Lemma \ref{Lem4.2}. Thus fix $k$, and define $\lambda \cdot_{k}\x= (\lambda^{e(1,k)}\cdot\x_{1}, \ldots, \lambda^{e(n,k)}\cdot\x_{n})$.  This is a family of dilations on all of $\R^{N}$ and $\widehat N_{j}(\t) = \sum_{k=1}^{n}n_{k}(\t_{k})^{1/e(k,j)}$ is a homogeneous norm. The homogeneous dimension is $\sum_{j=1}^{n}e(j,k)Q_{j}$. Choose $\chi\in \CC^{\infty}_{0}(\R^{N})$ with $\chi(\xib) \equiv 1$ if $\widehat N_{j}(\xib)\leq 1$. Let $m_{R}(\xib) = \chi(R\cdot_{j}\xib)m(\xib)$. Then for all $\gammab_{j}\in\N^{C_{j}}$, the product and chain rules show that we have
\beas
\big\vert\partial^{\gammab_{j}}_{\xib_{j}}m_{R}(\xib)\big\vert \leq C_{\gammab_{j}}\widehat N_{j}(\xib)^{-\[\gammab_{j}\]}.
\eeas
Now $m_{R}\to m$ in the sense of distributions as $R\to 0$, and hence $\KK$ is the limit, in the sense of distributions, of the inverse Fourier transform of $m_{R}$. But it now follows from Proposition \ref{Prop12.4} in Appendix II that $|K(\x)|\lesssim n_{k}(\x_{k})^{-\sum_{j=1}^{n}e(j,k)Q_{j}}$, and this completes the proof.
\end{proof}

\subsection{Higher rank and weak-type estimates near zero}\label{Weak type}\quad

\medskip

If $K$ is the smooth function on $\R^{N}\setminus\{0\}$ corresponding to a Calder\'on-Zygmund kernel $\KK$, then $K$ satisfies the  estimate $\left|\left\{\x\in\R^{N}:|K(x)|>\lambda\right\}\right| \lesssim \lambda^{-1}$. Moreover, there are Calder\'on-Zygmund kernels $\KK$ such that $\left|\left\{\x\in\R^{N}:|K(x)|>\lambda\right\}\right| \gtrsim \lambda^{-1}$. In this section we show that the corresponding estimates for kernels $\KK\in \PP(\EEE)$ associated with the decomposition $\R^{N}=\R^{C_{1}}\oplus\cdots\oplus\R^{C_{n}}$ depend on the reduced rank of $\EEE$. Recall from the discussion in Section \ref{subs.coarser} and Proposition \ref{PEflat}
 that if the reduced rank of $\EEE$ is $m<n$ then there is an $m\times m$ matrix $\EEE^{\flat}$ so that $\PP_{0}(\EEE)$ coincides with the space of distributions $\PP_{0}(\EEE^{\flat})$ associated with a coarser decomposition decomposition $\R^{N}=\R^{A_{1}}\oplus\cdots\oplus \R^{A_{m}}$

\begin{lemma}\label{Lem4.4}
Suppose that the reduced rank of $\EEE$ is $m$.
\begin{enumerate}[{\rm(a)}]

\item \label{Lem4.4a}
If $\KK\in \PP(\EEE)$, then for $\lambda \geq 1$, $\displaystyle \Big|\Big\{\x\in\R^{N}:K(\x)>\lambda\Big\}\Big|\lesssim\lambda^{-1}\log(\lambda)^{m-1}$.

\smallskip

\item \label{Lem4.4b}
There exists $\KK\in \PP(\EEE)$ so that $\displaystyle \Big|\Big\{\x\in\R^{N}:K(\x)>\lambda\Big\}\Big|
\gtrsim
\lambda^{-1}\log(\lambda)^{m-1}$ for $\lambda \geq 1$.
\end{enumerate}
\end{lemma}

\noindent If $m=1$, we are dealing with a Calder\'on-Zygmund kernel and these facts are well-known, so we may assume $m>1$.

\begin{proof}[Proof of Lemma \ref{Lem4.4}, part {\rm(\ref{Lem4.4a})}]  As indicated above,  we can assume that $m=n$ by replacing $\EEE$ by its reduced matrix and applying Proposition \ref{PEflat}.
Suppose that $\lambda \geq 1$. For $1 \leq k \leq n$ we have $\prod_{j=1}^{n}N_{j}(\x)^{Q_{j}}\geq n_{k}(\x_{k})^{\sum_{l=1}^{n}e(l,k)Q_{l}}$, and also $\prod_{j=1}^{n}N_{j}(\x)^{Q_{j}}\geq \prod_{l=1}^{n}n_{l}(\x_{l})^{Q_{l}}$. Since $|K(\x)|\leq \prod_{j=1}^{n}N_{j}(\x)^{-Q_{j}}$, it follows that
\beas
\Big\{\x\in\R^{N}:|K(\x)|>\lambda\Big\}\subset \Bigg\{\x\in \R^{N}:
\begin{cases}
\prod_{l=1}^{n}n_{l}(\x_{l})^{Q_{l}} <\lambda^{-1}\,\,\text{and}\\
n_{l}(\x_{l})^{Q_{l}}<\lambda^{-Q_{l}/\sum_{k=1}^{n}e(k,l)Q_{k}}&1\leq l \leq n
\end{cases}
\Bigg\}.
\eeas
Note that, as in equation (\ref{4.3aa}), $\sum_{l=1}^{n}Q_{l}\,\Big[\sum_{k=1}^{n}e(k,l)Q_{k}\Big]^{-1}<1$ since the rank of $\EEE$ is strictly greater than $1$. Part (\ref{Lem4.4a}) of the Lemma then follows from the following calculation.
\end{proof}

\begin{proposition}\label{Prop4.5}
Let $A_{1}, \ldots, A_{n}, \delta \in (0,\infty)$ and put $A= \prod_{j=1}^{n}A_{j}$. Put
\beas
E(A_{1}, \ldots, A_{n},\delta)=\Big\{\x\in\R^{N}:\text{$\prod_{j=1}^{n}n_{j}(\x_{j})^{Q_{j}}<\delta$ and $n_{j}(\x_{j})^{Q_{j}}<A_{j}$ for $1 \leq j \leq n$}\Big\}.
\eeas
There is a constant $c_{n}$ depending only on $n$ so that 
\beas
\big\vert E(A_{1}, \ldots, A_{n},\delta)\big\vert \gtrsim c_{n}\,
\begin{cases}
A&\text{if $A<\delta$},\\
\delta\left[1+ \left(\log\left(\frac{A}{\delta}\right)\right)^{n-1}\right] &\text{if $A\geq\delta$}
\end{cases}
\eeas
\end{proposition}
\begin{proof}
We argue by induction on $n$, with the case $n=1$ following from Proposition \ref{Prop12.2} in Appendix II. Thus suppose the result is true for $n-1$. We divide the region $E(A_{1}, \ldots, A_{n},\delta)$ into two parts; the part where $n_{n}(\x_{n})^{Q_{n}}<\delta[A_{1}\cdots A_{n-1}]^{-1}$ is a product of $n$ balls, and its volume is $\delta$. For the region where $n_{n}(\x_{n})^{Q_{n}}\geq\delta[A_{1}\cdots A_{n-1}]^{-1}$ we can use the induction hypothesis in the variables $\{\x_{1}, \ldots, \x_{n-1}\}$. 
\end{proof}

\begin{proof}[Proof of Lemma \ref{Lem4.4}, part {\rm(\ref{Lem4.4b})}]
As before we assume that $\EEE$ is reduced, i.e.,   $m=n$. Recall that $\lambda\cdot\x=(\lambda^{d_{1}}x_{1}, \ldots, \lambda^{d_{N}}x_{N})$. Choose a closed interval $J_{k}\subset (1,2^{d_{k}})\subset \R$, and let $\theta_{k}\in \CC^{\infty}_{0}(\R)$ be an odd function such that $\theta_{k}(t) \geq 0$ if $t\geq 0$, $\theta_{k}(t) \equiv 0$ if $t\notin (1, 2^{d_{k}})$, and $\theta_{k}(t) \equiv 1$ if $t\in J_{k}$. Then $\int_{\R}\theta_{k}(t) \,dt =0$ and if we put $\theta_{k,m}(t)=\theta_{k}(2^{md_{k}}t)$, the functions $\theta_{k,m_{1}}$ and $\theta_{k,m_{2}}$ have disjoint supports if $m_{1}\neq m_{2}$. Put $\Theta(\t)=\prod_{k=1}^{N}\theta_{k}(t_{k})$.

The function $\Theta$ has strong cancellation in the sense of Definition \ref{Def4.3}. 
It follows from Theorem \ref{Thm3.7} that $\KK=\sum_{I\in \Gamma_{\Z}(\EEE)}[\Theta]_{I}$ converges in the sense of distributions to an element $\KK\in \PP_0(\EEE)$. Note that the supports of the functions $[\Theta]_{I}$ are disjoint. Let $\Sigma$ denote the set where $\Theta = 1$, so that if $2^{-I}\cdot \x\in \Sigma$ and $2^{-\sum_{j=1}^{n}i_{j}Q_{j}}>\lambda$ we have $K(\x)>\lambda$. Let $\Sigma_{I}=\{\x\in\R^{N}:2^{-I}\cdot \x\in \Sigma\}$. Then $\big\vert\Sigma_{I}\big\vert = 2^{+\sum_{j=1}^{n}i_{j}Q_{j}}\big\vert\Sigma\big\vert$. We have
\beas
\Big\{\x\in\R^{N}:\Big(\exists I\in  \Gamma_{\Z}(\EEE)\Big)\Big(2^{-I\cdot\x}\in \Sigma \,\,\text{and}\,\,2^{-\sum_{j=1}^{n}i_{j}Q_{j}}>\lambda\Big)\Big\}\subset \{\x\in\R^{N}:K(\x)>\lambda\},
\eeas
so 
\beas
\Big|\Big\{\x\in\R^{N}:K(\x)>\lambda\Big\}\Big|
&\geq
\Big|\Big\{\x\in\R^{N}:\Big(\exists I\in \Gamma_{\Z}(\EEE)\Big)\Big(2^{-I\cdot\x}\in \Sigma \,\,\text{and}\,\,2^{-\sum_{j=1}^{n}i_{j}Q_{j}}>\lambda\Big)\Big\}\Big|\\
&\gtrsim
\sum_{I\in \Gamma_{\Z}(\EEE,\lambda)}|\Sigma_{I}|
\gtrsim
\lambda^{-1}\#\Big(\Gamma_{\Z}(\EEE,\lambda)\Big)
\eeas
where  $\Gamma_{\Z}(\EEE,\lambda)$ is the set of $n$-tuples $I=(i_{1}, \ldots, i_{n}) \in  \Gamma_{\Z}(\EEE)$ such that $\lambda < 2^{-\sum_{j=1}^{n}i_{j}Q_{j}} \leq C\lambda$ with $C$ a large constant.  The cone $\Gamma(\EEE) = \Big\{t=(t_{1}, \ldots, t_{n}): e(j,k)t_{k}< t_{j}< 0\Big\}$ is open since $\EEE$ has rank $n$, and the number of lattice points $(i_{1}, \ldots, i_{n})$ in this cone satisfying
\beas
\log(\lambda) <\sum_{j=1}^{n}Q_{j}i_{j}<\log(\lambda)+1
\eeas
is of the order of $\big(\log(\lambda)\big)^{n-1}$. This shows that $\#\Big(\Gamma_{\Z}(\EEE,\lambda)\Big)\gtrsim \big(\log(\lambda)\big)^{n-1}$ which completes the proof of part {\rm(\ref{Lem4.4b})}.
\end{proof}
\medskip

\section{Convolution operators on homogeneous nilpotent Lie groups}\label{Groups}

We now turn to the study of convolution operators $f\to f*\KK$ with $\KK\in \PP_0(\EEE)$ and the convolution on a homogeneous nilpotent Lie group $G$. The underlying manifold of $G$ is $\R^{N}$ for some integer $N\geq 1$, and  there is a one-parameter group of  {\it automorphic dilations}, i.e., of automorphisms $\delta_{r}:G\to G$ which, in appropriate coordinates on $G$, take the form 
\be\label{automorphic}
\delta_{r}(x_{1}, \ldots, x_{N})=(r^{d_{1}}x_{1},\ldots,r^{d_{N}}x_{N}),
\ee 
with $d_i>0$ for every $i$.
The ordering of the coordinates and the parametrization by $r$ can be chosen so that 
\be\label{d1<d2}
1\leq d_{1}\leq d_{2}\leq\cdots\leq d_{N}.
\ee 
In these coordinates, the product on $G$ has the following form. For $\x=(x_{1}, \ldots, x_{N})$, $\y=(y_{1}, \ldots, y_{N})$,
\be\label{productN}
\x\y=\big(x_1+y_1,x_2+y_2+M_2(\x,\y),\dots, x_N+y_N+M_N(\x,\y)\big),
\ee
where $M_2,\dots,M_N$ are polynomials which vanish for $\x$ or $\y$ equal to 0 and are such that
\be\label{Ml}
M_l(\delta_r\x,\delta_r\y)=\delta_rM_l(\x,\y).
\ee
In particular, each $M_l$ only depends on the variables $x_m,y_m$ for which $d_m<d_l$.
Moreover, the Haar measure for $G$ is Lebesgue measure on $\R^{N}$.

 If $f,g\in L^{1}(G)$, the convolution $f*g$ is defined by
\bes
f*g(\x)=\int_{G\cong\
R^{N}}f(\x\y^{-1})g(\y)\,d\y=\int_{G\cong\R^{N}}f(\y)g(\y^{-1}\x)\,d\y.
\ees

\medskip

\subsection{ Convolution of scaled bump functions: compatibility of dilations and convolution}\label{Compatibility}\quad

\medskip

If $f\in L^{1}(G)$ and $\lambda=(\lambda_{1}, \ldots, \lambda_{N})$ with each $\lambda_{j}>0$, consider the $N$-parameter family of dilations of $f$ by $\lambda$ given by
\bea\label{9.1aaa}
f_{\lambda}(x_{1}, \ldots, x_{N}) = \Big[\prod_{j=1}^{N}\lambda_{j}^{-d_{j}}\Big]\,f(\lambda_{1}^{-d_{1}}x_{1}, \ldots, \lambda_{N}^{-d_{N}}x_{N}),
\eea
If $\varphi, \psi$ are normalized bump functions supported in the unit ball and if $ \lambda =(\lambda_{1}, \ldots,\lambda_{N})$ and $\mu=(\mu_{1}, \ldots,\mu_{N})$, it is easy to check that for Euclidean convolution $\varphi_{\lambda}*\psi_{\mu}=\theta_{\nu}$ where $\theta\in \CC^{\infty}_{0}(\R^{N})$ is normalized relative to $\varphi$ and $\psi$, and $\nu_{j}= \max\{\lambda_{j},\mu_{j}\}=\lambda_{j}\vee\mu_{j}$. This fact is very useful in studying the convolution of two sums of dilates of bump functions. However, on a more general nilpotent Lie group, this need not be true unless we suitably restrict the $N$-tuples $\lambda=(\lambda_{1},\ldots, \lambda_{N})$ and $\mu=(\mu_{1}, \ldots, \mu_{N})$. Let
\bea\label{9.2aaa} \index{E4N@$E_{N}$}
E_{N}=\Big\{\lambda=(\lambda_{1}, \ldots, \lambda_{N})\in (0,\infty)^{N}:\lambda_{1}\leq \lambda_{2}\leq \cdots\leq \lambda_{N}\Big\}.
\eea
We say that a mapping $(x_{1}, \ldots, x_{N})\to (\lambda_{1}^{-d_{1}}x_{1}, \ldots, \lambda_{N}^{-d_{N}}x_{N})$ is \emph{compatible with the group structure of $G$} if $\lambda\in E_{N}$. The significance of this notion comes from the following result, which is Lemma 6.17 in \cite{MR2949616}.

\begin{lemma}\label{Lem9.2}
Let $\varphi, \psi\in\CC^{\infty}_{0}(G)$ be supported in $\B(1)$. If $\lambda, \mu\in E_{N}$ there exists $\theta\in \CC^{\infty}_{0}(G)$ so that $\varphi_{\lambda}*\psi_{\mu}=\theta_{\nu}$ where $\nu_{j}=\lambda_{j}\vee\mu_{j}=\max\{\lambda_{j},\mu_{j}\}$. Moreover, $\theta$ is normalized relative to $\varphi$ and $\psi$: there are constants $\rho_{0}\geq 1$ and $C_{m}>0$ independent of $\varphi$ and $\psi$ so that $\theta$ is supported in $\B(\rho_{0})$ and $\displaystyle \sup_{|\gammab|\leq m}||\partial^{\gammab}\theta||_{\infty}\leq C_{m}\,\sup_{|\gammab|\leq m}||\partial^{\gammab}\varphi||_{\infty}\,\sup_{|\gammab|\leq m}||\partial^{\gammab}\psi||_{\infty}$.
\end{lemma}

\medskip

It will be useful for our purposes to have a formulation of Lemma \ref{Lem9.2} for dilations with fewer parameters and with the $N$ variables split into $n$ blocks according to a decomposition $G\cong \R^{N}= \R^{C_{1}}\oplus\cdots\oplus\R^{C_{n}}$. For reasons of compatibility with the group structure, we now require that $C_{1}, \ldots, C_{n}\subset\{1, \ldots, N\}$ be sets of consecutive integers, indexed according to the natural ordering:
\be\label{orderedC}
C_1=\{1,\dots,l_1\},\quad C_2=\{l_1+1,\dots,l_2\},\quad\dots\quad C_n=\{l_{n-1}+1,\dots,N\}.
\ee
This implies that each $G_j\cong \R^{C_{j}}\oplus\cdots\oplus\R^{C_{n}}$ is a subgroup of $G$. Formula \eqref{productN} 
remains true if the individual variables $x_l,y_l\in\R$ are replaced by the blocks of variables $\x_j,\y_j\in\R^{C_j}$ and the scalar-valued functions $M_l$ by the $\R^{C_j}$-valued functions $\mathbf M_j=(M_{l_{j-1}+1},\dots,M_{l_j})$. Notice that the functions $\mathbf M_j$ still satisfy \eqref{Ml}.\footnote{{In what follows below, after choosing marked partitions, we may have situations where we no longer have sequences of consecutive integers. However this is all right if we can refine to a situation of consecutive integers.}}

\smallskip

Let $\EEE=\{e(j,k)\}$ be an $n\times n$ matrix satisfying (\ref{2.5}). Associated to $\EEE$ is the cone of lattice points
\bes
\Gamma_{\Z}(\EEE)  = 
\Big\{I=(i_{1}, \ldots, i_{n})\in\mathbb Z^{n}: e(j,k)i_{k}\leq i_{j}< 0,\,\,1\leq j,k\leq n\Big\}
\ees
and we consider dilations 
\beas[]
[f]_{I}(\x)=2^{-\sum_{j=1}^{n}i_{j}Q_{j}}f(2^{-i_{1}}\cdot\x_{1}, \ldots, 2^{-i_{n}}\cdot \x_{n}).
\eeas 
Note that $(2^{i_{1}}, \ldots, 2^{i_{n}})$ is an $n$-tuple while elements of $E_{N}$ defined in equation (\ref{9.2aaa}) are $N$-tuples. However $[f]_{I}=f_{\lambda(I)}$ where $\lambda(I)$ is obtained from $I$ by setting $\lambda_{l}=2^{i_{k}}$ if $l\in C_{k}$; that is
\bea\label{9.6bb}
\lambda(I) = \big(\overbrace{2^{i_{1}}, \ldots, 2^{i_{1}}}^{C_{1}},\,\,\overbrace{2^{i_{2}}, \ldots, 2^{i_{2}}}^{C_{2}},\,\,\ldots, \,\overbrace{2^{i_{n}}, \ldots, 2^{i_{n}}}^{C_{n}}\,\big).
\eea
We want to impose conditions on the matrix $\EEE$, in addition to the basic hypothesis (\ref{2.5}), which guarantees that if $I\in \Gamma_{\Z}(\EEE)$, then $\lambda(I)\in E_{N}$. 

\begin{definition}\label{Def9.3}
An $n\times n$ matrix $\EEE=\{e(j,k)\}$ is {\rm doubly monotone} if each row is weakly increasing from left to right, and each column is weakly decreasing from top to bottom; \textit{i.e.}
\beas
e(j,k)&\leq e(j,k+1)&&&&\text{for $1\leq j \leq n$ and $1 \leq k <n$},\\
e(j,k)&\geq e(j+1,k)&&&&\text{for $1\leq j < n$ and $1 \leq k \leq n$}.
\eeas
\end{definition}

\begin{proposition}\label{Prop9.4}
Let $\EEE=\{e(j,k)\}$ be an $n\times n$ matrix satisfying the basic hypothesis (\ref{2.5}). 
\begin{enumerate}[{\rm(a)}]
\smallskip
\item \label{Prop9.4a}The matrix $\EEE$ is doubly monotone if and only if $e(j+1,j)\leq 1$ for $1 \leq j \leq n-1$.
\smallskip
\item \label{Prop9.4b}If the matrix $\EEE$ is doubly monotone and if $I\in \Gamma(\EEE)$ then $\lambda(I)\in E_{N}$. 
\end{enumerate}
\end{proposition}
\begin{proof}
If $\EEE$ is doubly monotone and each row is weakly increasing, then since $e(j+1,j+1)=1$ we certainly have $e(j+1,j)\leq 1$. Conversely, suppose that $e(k+1,k)\leq 1$ for $1\leq k \leq n-1$. Then by (\ref{2.5}) we have 
\beas
e(j,k)\leq e(j,k+1)e(k+1,k)&\leq e(j,k+1),\\
e(j+1,k)\leq e(j+1,j)e(j,k)&\leq e(j,k),
\eeas 
so $\EEE$ is doubly monotone. 

Next if $\EEE$ is doubly monotone and $I=(i_{1}, \ldots, i_{n})\in \Gamma(\EEE)$, it follows from the definition of $\lambda(I)$ in equation (\ref{9.6bb}) that we only need to show that $i_{j}\leq i_{j+1}$ for $1 \leq j \leq n-1$. However, since $e(j+1,j)\leq 1$ and $i_{j}<0$, we have $i_{j}\leq e(j+1,j)i_{j}\leq i_{j+1}$ since  $I\in \Gamma_{\Z}(\EEE)$. 
\end{proof}

\begin{corollary}
Suppose  $\EEE$ is doubly monotone, and let $I=(i_{1}, \ldots, i_{n}),\,J=(j_{1}, \ldots, j_{n})\in \Gamma_{\Z}(\EEE)$. If $\varphi, \psi\in \CC^{\infty}_{0}(\R^{n})$ are supported in the unit ball, then $[\varphi]_{I}*[\psi]_{J}=[\theta]_{K}$ where $\theta\in \CC^{\infty}_{0}(\R^{N})$ is normalized relative to $\varphi$ and $\psi$, and $K=(k_{1}, \ldots, k_{n})\in \Gamma_{\Z}(\EEE)$ with $k_{m}=\max\{i_{m},j_{m}\}$ for $1\leq m \leq n$.
\end{corollary}

Recall that $S\in \SS'(\EEE)$ means that $\widehat E_{S}\cap \B(1)^{c}\neq \emptyset$. If $S=\big((I_{1},k_{1});\ldots;(I_{s},k_{s})\big)\in \SS'(\EEE)$, the $s\times s$ matrix $\EEE_{S}=\{e_{S}(r,p)\}$ was defined in Lemma \ref{Lem5.10}. We observe that if we order the sets $\{I_{1}, \ldots, I_{s}\}$ so that $k_{1}\leq k_{2}\leq \cdots \leq k_{s}$ and if $\EEE$ is doubly monotone, then the same is true of the matrix $\EEE_{S}$. 

\begin{lemma}\label{Lem9.3pp}
Let $S=\big((I_{1},k_{1});\ldots;(I_{s},k_{s})\big)\in \SS'(\EEE)$ with $k_{1}<k_{2}<\cdots<k_{s}$, and let $\EEE_{S}=\{e_{S}(r,p)\}$ be the $s\times s$ matrix whose existence is established in Lemma \ref{Lem5.10}. If the matrix $\EEE$ is doubly monotone, then the matrix $ \EEE_{S}$ is doubly monotone.
\end{lemma}
\begin{proof}
From the construction of $ \EEE_{S}$ we have
\beas
 e_{S}(r+1,r) \leq \tau_{S}(r+1,r)=\min_{j\in I_{r}}\frac{e(k_{r+1},j)}{e(k_{r},j)}\leq e(k_{r+1},k_{r}).
\eeas
If $\{e(j,k)\}$ is doubly monotone, $e(k_{r+1},k_{r})\leq 1$ because of our ordering of the sets $\{I_{1}, \ldots, I_{s}\}$, and it follows that $e_{S}(r+1 ,r)\leq 1$. Since $\EEE_{S}$ satisfies the basic hypothesis (\ref{2.5}), it then follows from part (\ref{Prop9.4a}) of Proposition \ref{Prop9.4} that it is also doubly monotone.
\end{proof}

Continue to fix $S=\big((I_{1},k_{1});\ldots;(I_{s},k_{s})\big)\in \SS'(\EEE)$ with $k_{1}<k_{2}<\cdots<k_{s}$. In Definition \ref{Def3.12iou} we introduced the family of dilations
\beas
2^{-J}\hat\cdot_{S}\,\t &= \big(2^{-j_{1}}\hat\cdot_{S}\,\t_{I_{1}}, \ldots,2^{-j_{r}}\hat\cdot_{S}\,\t_{I_{r}},\ldots, 2^{-j_{s}}\hat\cdot_{S}\,\t_{I_{s}}\big)
\eeas
where
\beas
2^{-j_{r}}\,\hat\cdot_{S}\,\t_{I_{r}}=\Big\{2^{-j_{r}/e(k_{r},j)}\cdot\t_{j}:j\in I_{r}\Big\} = \Big\{2^{-j_{r}\tau_{S}(j,k_{r})}\cdot\t_{j}:j\in I_{r}\Big\}.
\eeas
Recall that $\Gamma_{\Z}({\mathbf{ E}}_{S}) =\big\{(j_{1}, \ldots, j_{s})\in \Z^{s}: e_{S}(p,r)j_{r}\leq j_{p} < 0\big\}$. 

\begin{lemma}\label{Lem9.6}
Suppose that $\EEE$ is doubly monotone. If $J=(j_{1}, \ldots, j_{s})\in \Gamma_{\Z}(\EEE_{S})$ then the mapping $\t\to 2^{-J}\,\hat\cdot_{S}\,\t$ is compatible with the group structure.  
\end{lemma}
\begin{proof}
Let $l,m\in \{1, \ldots, n\}$ with $l\leq m$ and $l\in I_{r}$, $m\in I_{p}$. According to Proposition \ref{Prop9.4}, it suffices to show that $j_{r}e(k_{r},l)^{-1}\leq j_{p}e(k_{p},m)^{-1}$.  If $J=(j_{1}, \ldots, j_{s})\in \Gamma_{\Z}(\EEE_{S})$ then $e_{S}(p,r)j_{r}\leq j_{p}$. But then since $j_{r}<0$,
\beas
\frac{j_{p}}{j_{r}}&\leq e_{S}(p,r) \leq \tau_{S}(k_{p},k_{r})=\min_{j\in I_{r}}\frac{e(k_{p},j)}{e(k_{r},j)}\leq \frac{e(k_{p},l)}{e(k_{r},l)}\leq \frac{e(k_{p},m)e(m,l)}{e(k_{r},l)}
\leq \frac{e(k_{p},m)}{e(k_{r},l)}
\eeas
where the last inequality follows since $l\leq m$ and the matrix $\EEE$ is (weakly) increasing across each row. This means that  $j_{r}e(k_{r},l)^{-1}\leq j_{p}e(k_{p},m)^{-1}$, completing the proof.
\end{proof}

\medskip

\subsection{ Automorphic flag kernels and $L^p$-boundedness of convolution operators}

 In this section we discuss  $L^p$-boundedness of  operators
$$
T_\KK f=f*\KK,
$$
given by convolution (relative to some homogeneous group structure on $\R^N$)  with a kernel $\KK\in\PP(\EEE)$. The matrix $\EEE$  is assumed to satisfy the basic hypotheses \eqref{2.5}.

The reader should be aware that there are two possible interpretations of the symbol $f*\KK$, depending on whether the underlying group is the standard abelian $\R^{n}$ or some non-abelian group $G$. Whenever clarification is required, we will write $T^{G}_{\KK}$ to specify the group. For  ordinary convolution on $\R^N$, it can be easily verified that a kernel in $\PP(\EEE)$ is a {\it product kernel}, as defined in Section 2 of   \cite{MR1818111}. In fact, the differential inequalities of a product kernel follow from the inequalities $N_j(\x)\ge n_j(\x_j)$, $j=1,\dots,n$, and the cancellation conditions are the same as in  Definition \ref{Def2.2} (\ref{Def2.2B}). By  \cite{MR1818111}, we have the following theorem.

\begin{theorem}\label{L^p-abelian}
Let $\delta_\lambda$ be any family of dilations on $\R^N$ and  $\R^N=\R^{C_{1}}\oplus\cdots\oplus \R^{C_{n}}$ be any decomposition of $\R^N$ into homogeneous subspace. Suppose that the matrix $\EEE$ satisfies the basic hypotheses. Then, for every $\KK\in\PP(\EEE)$, the operator $T_\KK$ is bounded on $L^p$ for $1<p<\infty$.
\end{theorem}

Considering now a general homogeneous group law as in \eqref{productN}, we must be more specific on all the ingredients needed to define a class $\PP(\EEE)$.

We fix a decomposition $\R^N=\R^{C_{1}}\oplus\cdots\oplus \R^{C_{n}}$ where the $C_j$ are as in \eqref{orderedC}.  
  In the notation of \cite{MR1818111}, the subgroups $G_j=\R^{C_{j}}\oplus\cdots\oplus \R^{C_{n}}$
form the {\it standard flag}
\beas
\FF:\qquad (0)\ \subset\ G_n\ \subset\ G_{n-1}\ \subset\ \cdots \subset\ G_2\ \subset\ G_1=G.
\eeas

On each $\R^{C_j}$ we also fix a homogeneous norm  $n_j$ for the automorphic dilations $\delta_\lambda$ in \eqref{automorphic}.
A~distribution $\KK$ is an {\it automorphic flag kernel} on the standard flag if it satisfies the differential inequalities
\beas
\big\vert\partial^{\gammab}K(\x)\big\vert \leq C_{\gammab}\prod_{j=1}^{n}\Big[n_{1}(\x_{1})+\cdots +n_{j}(\x_{j})\Big]^{-Q_{j}-\[\gammab_{j}\]}
\eeas
and the  cancellation conditions in Definition \ref{Def2.2} (\ref{Def2.2B}). 

The differential inequalities for the {\it automorphic flag multiplier} $m=\widehat\KK$ are
\beas
\big\vert\partial^{\gammab}m(\xib)\big\vert &\le
C_{\gammab}\,\prod_{j=1}^{n}\big[n_{j}(\xib_{j})+\cdots + n_{n}(\xib_{n})\big]^{-\[\gammab_{j}\]}.
\eeas
(See Definition 2.3.2 in \cite{MR1818111} for details.)
\medskip

 We show that, if $\EEE$ is doubly monotone, a distribution $\KK\in \PP_0(\EEE)$ is always an automorphic  flag kernel for a standard flag.

\begin{proposition}\label{Prop9.7}
If $\KK\in \PP_{0}(\EEE)$ where the class of distributions is defined by a doubly monotone matrix, then $\KK$ is an automorphic flag kernel on the standard flag.
\end{proposition}

\begin{proof}
Let $\KK\in \PP_{0}(\EEE)$ where the class of distributions is defined by a doubly monotone matrix $\EEE=\{e(j,k)\}$ and let $m=\widehat \KK$. Note that for $\xib\in \B(1)^c$ we have
\beas
\widehat N_{j}(\xib) &= \sum_{k=1}^{n}n_{k}(\xib_{k})^{1/e(k,j)}
\geq \sum_{k=j}^{n}n_{k}(\xib_{k})^{1/e(k,j)}
\geq \sum_{k=j}^{n}n_{k}(\xib_{k})
\eeas
since for $k\geq j$ we have $e(k,j) \leq 1$. Thus if $\gammab\in \N^{C_{1}}\times\cdots\times\N^{C_{n}}$ we have
\beas
\big\vert\partial^{\gammab}m(\xib)\big\vert &\leq C_{\gammab}\prod_{j=1}^{n}\big(1+\widehat N_{j}(\xib)\big)^{-\[\gammab_{j}\]}
\leq
C_{\gammab}\,\prod_{j=1}^{n}\big[n_{j}(\xib_{j})+\cdots + n_{n}(\xib_{n})\big]^{-\[\gammab_{j}\]}.
\eeas
\end{proof}

\begin{theorem}\label{Thm9.8}
 Let $G$ be a homogeneous nilpotent Lie group, endowed with automorphic dilations $\delta_\lambda$, $G\cong \R^N=\R^{C_{1}}\oplus\cdots\oplus \R^{C_{n}}$, the $C_j$ being as in \eqref{orderedC}. If $\EEE$ is a doubly monotone matrix satisfying the basic hypothesis, and if $\KK\in \PP_0(\EEE)$, then the operator $T_{\KK}[\varphi] = \varphi*\KK$, defined initially for $\varphi\in \SS(\R^{N})$, extends uniquely to a bounded operator on $L^{p}(G)$ for $1 < p < \infty$.  If the rank of $\EEE$ is greater than one, the same holds for $\KK\in \PP(\EEE)$.
\end{theorem}

\begin{proof}
 If $\KK\in \PP_0(\EEE)$, it follows from Proposition \ref{Prop9.7} that $\KK$ is an automorphic flag kernel on the standard flag. The result now follows from Theorem 8.14 in \cite{MR2949616}.
 Assume now that the rank of $\EEE$ is greater than 1 and $\KK\in\PP(\EEE)$. By Corollary \ref{Cor5.3ww},
we can write $\KK=\KK_{0}+K_{\infty}$ where $\KK_{0}\in \PP_{0}(\EEE)$ and $K_{\infty}\in L^{1}(\R^{N})$. The operator $f\to f*K_{\infty}$ is bounded on $L^{p}(G)$ for $1 \leq p \leq \infty$. 
\end{proof}

\section{Composition of operators}\label{Composition}

If $\KK\in \PP_{0}(\EEE)$ let $T_{\KK}[f]=f*\KK$ \index{T2K@$T_{\KK}$} be the corresponding left-invariant convolution operator on the homogeneous nilpotent Lie group $G$. $T_{\KK}$ is initially defined on the Schwartz space $\SS(\R^{N})$. We are interested in the composition of two operators $T_{\KK},T_{\LL}$ with $\KK,\LL\in \PP_{0}(\EEE)$.
Formally $T_{\LL}\circ T_{\KK}[f] = T_{\KK}[f]*\LL=(f*\KK)*\LL= f*(\KK*\LL) = T_{\KK*\LL}$. 

To give a concrete meaning to these equalities is  easy  if convolution is relative to the (abelian) additive structure of $\R^N$, because we can refer to the Fourier multipliers $m=\widehat\KK$, $m'=\widehat\LL$.  For $f\in\SS(\R^N)$, $T_{\LL}\circ T_{\KK}[f] =\FF^{-1}(mm'\widehat f)$, where $\FF$ denotes the Fourier transform, and the convolution kernel of $T_{\LL}\circ T_{\KK}$ is $\FF^{-1}(mm')=\KK*\LL$. In fact, since $\MM_\infty(\EEE)$ is closed under pointwise product, we can conclude as follows.

\begin{theorem}\label{Thm10.1abelian}
Let $\EEE$ be a matrix satisfying the basic hypotheses and let $\KK,\LL\in\PP_0(\EEE)$. With respect to ordinary convolution on $\R^N$, then
  $\KK*\LL\in\PP_0(\EEE)$.
  \end{theorem}

Suppose now that $\R^N\cong G$ is a noncommutative homogeneous group. It is still easy to verify that  $\KK*\LL$ is a well-defined distribution. In fact,
$$
\KK=\KK_0+\varphi_0,\qquad \LL=\LL_0+\psi_0,
$$
where $\KK_0,\LL_0$ have compact support and $\varphi_0,\psi_0$ are Schwartz functions. Hence
$$
\KK*\LL=\KK_0*\LL_0+\KK_0*\psi_0+\varphi_0*\LL_0+\varphi_0*\psi_0,
$$
where each summand is well defined.
However, it is not guaranteed that the convolution $\KK*\LL$ is in the same class.

The main result of this section is that, under  the assumptions introduced in Section~\ref{Groups}, $\KK*\LL\in \PP_{0}(\EEE)$, and so $\PP_{0}(\EEE)$ is an algebra under convolution on the group $G$. We recall the assumptions:
\begin{enumerate}
\item[(i)] the underlying dilations on $\R^N$ are automorphisms of $G$;
\item[(ii)] the subspaces $\R^{C_j}$ are  as in \eqref{orderedC};
\item[(iii)] the matrix $\EEE$ is doubly monotone.
\end{enumerate}

\begin{theorem}\label{Thm10.1}
Suppose that conditions {\rm (i)-(iii)} above are satisfied.
If $\KK, \LL \in \PP_{0}(\EEE)$ then $\KK*\LL \in \PP_{0}(\EEE)$.
\end{theorem}

The following outline provides an informal \emph{road map} for the long proof which begins in Section \ref{PrelimDecomp}. Let $\KK, \LL \in \PP_{0}(\EEE)$.

\begin{enumerate}[(A)]

\smallskip

\item According to Theorem \ref{Thm8.5}, 
\beas
\KK&=\varphi_{0}+\sum_{S\in \SS'(\EEE)}\KK_{S}&&\text{and}& \LL&=\psi_{0}+\sum_{T\in \SS'(\EEE)}\LL_{T}
\eeas
where $\varphi_{0}, \psi_{0}\in \SS(\R^{N})$ and 
\beas
\KK_{S}&=\sum_{I\in \Gamma_{\Z}(\EEE_{S})}[\varphi^{I}_{S}]_{I}&&\text{and}& 
\LL_{T}&=\sum_{J\in \Gamma_{\Z}(\EEE_{T})}[\psi^{J}_{T}]_{J}.
\eeas
Here $\big\{\varphi^{I}_{S}\big\}$ and $\big\{\psi^{J}_{T}\big\}$ are uniformly bounded families of functions in $\CC^{\infty}_{0}(\R^{N})$, and
\beas
\Gamma_{\Z}(\EEE_{S})&=\Big\{I=(i_{1}, \ldots, i_{s})\in \Z^{s}:\tau_{S}(k_{p},k_{r})i_{r}\leq i_{p}<0,\,\,1\leq r,p\leq s\Big\},\\
\Gamma_{\Z}(\EEE_{T}) &=\Big\{J=(j_{1}, \ldots, j_{q})\in\Z^{q}:\tau_{T}(\ell_{b},\ell_{a})j_{a}\leq j_{b}<0,\,\,1\leq a,b\leq q\Big\}.
\eeas
Then the problem of studying the composition $T_{\LL}\circ T_{\KK}$ reduces essentially to the study of the finite number of compositions $\big\{T_{\LL_{T}}\circ T_{\KK_{S}}\big\}$ for $S,T\in \SS'(\EEE)$. This reduction is discussed in Section \ref{PrelimDecomp}.
\smallskip

\item We fix marked partitions $S= \big((I_{1},k_{1});\ldots;(I_{s},k_{s})\big)$ and $T=\big((J_{1},\ell_{1});\ldots;(J_{q},\ell_{q})\big)$ in $\SS'(\EEE)$, and we simplify notation by writing $\varphi^{I}_{S}=\varphi^{I}$ and $\psi^{J}_{T}=\psi^{J}$. 
 Thus we need to study $T_{\LL_{T}}\circ T_{\KK_{S}}$ where 
\beas
\KK_{S}&=\sum_{I\in \Gamma_{\Z}(\EEE_{S})}[\varphi^{I}]_{I}=\lim_{F_{1}\nearrow \Gamma_{\Z}(\EEE_{S})}\sum_{I\in F_{1}}[\varphi^{I}]_{I}=\lim_{F\nearrow \Gamma_{\Z}(\EEE_{S})}\KK_{F_{1},S},\\
\LL_{T}&=\sum_{J\in \Gamma_{\Z}(\EEE_{T})}[\psi^{J}]_{J}=\lim_{F_{2}\nearrow \Gamma_{\Z}(\EEE_{T})}\sum_{J\in F_{2}}[\psi^{J}]_{J}=\lim_{F_{2}\nearrow \Gamma_{\Z}(\EEE_{S})}\LL_{F_{2},T}.
\eeas
Here $F_{1}, F_{2}$ are finite sets of indices, and the limits exist in the sense of distributions. 

\smallskip

\item Recall from  Corollary \ref{Remark3.10} that, since $S,T\in\SS'(\EEE)$, $\Gamma_{\Z}(\EEE_{S})$ and $\Gamma_{\Z}(\EEE_{T})$ have nonempty interiors.
By Lemma \ref{Lem5.10} in Section \ref{NewMatrices}, there are new matrices $\EEE_{S}=\{e_{S}(p,r)\}$ and $\EEE_{T}=\{e_{T}(b,a)\}$ which satisfy the basic hypothesis such that $e_{S}(p,r)\leq \tau_{S}(k_{p},k_{r})$, $e_{T}(b,a)\leq \tau_{T}(\ell_{b},\ell_{a})$.

If $F_{1}\subset  \Gamma_{\Z}(\EEE_{S})$ and $F_{2}\subset\Gamma_{\Z}(\EEE_{T})$ are any two finite sets of indices then $T_{\LL_{F_{2},T}}\circ T_{\KK_{F_{1},S}}=T_{\HH_{F_{1},F_{2},S,T}}$ where
\beas
\HH_{F_{1},F_{2},S,T} = \KK_{F_{1},S}*\LL_{F_{2},T}= \sum_{I\in F_{1}}\sum_{J\in F_{2}}[\varphi^{I}]_{I}*[\psi^{J}]_{J}.
\eeas
In  Section \ref{FiniteSets} we see that it suffices to show that $\lim_{F_{1}\nearrow\Gamma_{\Z}(\EEE_{S})}\lim_{F_{2}\nearrow\Gamma_{\Z}(\EEE_{T})}\HH_{F_{1},F_{2},S,T}$ exists in the sense of distributions and belongs to $\PP_{0}(\EEE)$. More generally, for arbitrary finite subsets $F\subset \Gamma_{\Z}(\EEE_{S})\times \Gamma_{\Z}(\EEE_{T})$, we show that
\begin{enumerate}[(a)]

\smallskip

\item $\displaystyle \HH_{F,S,T}=\sum_{(I,J)\in F} [\varphi^{I}]_{I}*[\psi^{J}]_{J}\in \PP_{0}(\EEE)$ with estimates independent of the set $F$;

\smallskip

\item In the sense of distributions, $\displaystyle \lim_{F\nearrow \Gamma_{\Z}(\EEE_{S})\times \Gamma_{\Z}(\EEE_{T})}\HH_{F,S,T}$ exists and belongs to $\PP_{0}(\EEE)$.
\end{enumerate}

\smallskip

\item In Section \ref{Properties} we summarize the support, decay, and cancellation properties of the convolutions $[\varphi^{I}]_{I}*[\psi^{J}]_{J}$ that are needed in our analysis.

\smallskip

\item In order to study $\HH_{F,S,T}$ we write
\bes
\Gamma_{\Z}(\EEE_{S})\times \Gamma_{\Z}(\EEE_{T})= \bigcup_{W\in\WW(S,T)}\Gamma_{W}(S,T)
\ees
where the index set $\WW(S,T)$ is finite, and the sets $\big\{\Gamma_{W}(S,T)\subset  \Gamma_{\Z}(\EEE_{S})\times \Gamma_{\Z}(\EEE_{T})\big\}$ are disjoint. This partition of $\Gamma_{\Z}(\EEE_{S})\times \Gamma_{\Z}(\EEE_{T})$ is somewhat complicated, but it has the property that if $(I,J)=\big((i_{1}, \ldots, i_{s}), (j_{1}, \ldots, j_{q})\big)\in \Gamma_{W}(S,T)$, then for $1\leq r\leq s$ and $1\leq a \leq q$, if $I_{r}\cap J_{a}\neq \emptyset$ the ratio $i_{r}/j_{a}$ is suitably restricted. Having made the partition, it then suffices to show that for every $W\in \WW(S,T)$, $\lim_{F\nearrow \Gamma_{W(S,T)}}\HH_{F,S,T}$ converges in the sense of distributions and belongs to $\PP_{0}(\EEE)$. This partition is described in Section \ref{Decomp1}.
 
\smallskip

\item 
Now fix $W\in \WW(S,T)$, and let $F\subset \Gamma_{W}(S,T)$ be a finite set. We study of $\HH_{F,S,T}= \sum_{(I,J)\in F}[\varphi^{I}]_{I}*[\psi^{J}]_{J}$. 
\begin{enumerate}[(a)]

\smallskip

\item For each convolution $[\varphi^{I}]_{I}*[\psi^{J}]_{J}$ there is a normalized bump function $\theta^{I,J}$ and a set of dilations $M(I,J)$ so that  $[\varphi^{I}]_{I}*[\psi^{J}]_{J}=[\theta^{I,J}]_{M(I,J)}$. 

\smallskip

\item There are $s$ parameters $(i_{1}, \ldots, i_{s})$ in $I$ and $q$ parameters $(j_{1}, \ldots, j_{q})$ in $J$, and the function $[\theta^{I,J}]_{M(I,J)}$ is dilated by some collection of these dilations.  However not all of these need appear in $M(I,J)$. We will say that a parameter $i_{r}$ or $j_{a}$ from $I\cup J$ is `fixed' if it appears among the parameters in $M(I,J)$, and `free' if it does not appear. These notions are defined precisely in Section \ref{FixedFree}.

\smallskip

\item Let $I_{r}^{*}$ be the indices in $I_{r}$ which are fixed, and let $J_{a}^{*}$ be the indices in $J_{a}$ which are fixed. Some of these sets may be empty (if $I_{r}$ or $J_{a}$ consists entirely of free indices), and we suppose there are $\alpha\leq s$ non-empty sets $I_{r_{1}}^{*}, \ldots, I_{r_{\alpha}}^{*}$ and $\beta\leq q$ non-empty set $J_{a_{1}}^{*}, \ldots J_{a_{\beta}}^{*}$. 

\smallskip

\item The sets $I_{r_{1}}^{*}, \ldots, I_{r_{\alpha}}^{*}, J_{a_{1}}^{*}, \ldots J_{a_{\beta}}^{*}$ are disjoint and their union is $\{1, \ldots, n\}$. This leads to a new decomposition 
\beas
\R^{N}=\bigoplus_{j=1}^{\alpha}\R^{I_{r_{j}}^{*}}\oplus\bigoplus_{k=1}^{\beta}\R^{J_{a_{k}}^{*}}.
\eeas
This decomposition is discussed in Section \ref{Finer}.
\end{enumerate}
\smallskip

\item\label{part(e)}
Associated to this last decomposition, we show in Section \ref{TheMatrix} that there is an $(\alpha+\beta)\times(\alpha+\beta)$ matrix $\EEE(S,T,W)$ with associated family of norms $N_{1}^{\#}, \ldots, N_{\alpha+\beta}^{\#}$ on $\R^{N}$. In Section \ref{TheClass} we show that the corresponding family of distributions $\PP_{0}(\EEE(S,T,W))$ is contained in the original family $\PP(\EEE)$. Theorem \ref{Thm10.1} then follows if we can show that 
\bes
\sum_{(I,J)\in \Gamma_{W}(S,T)}[\varphi^{I}]_{I}*[\psi^{J}]_{J}\in\PP_{0}(\EEE(S,T,W)).
\ees 
\textit{This is the heart of the matter,} and we argue as follows.

\begin{enumerate}[a)]

\medskip

\item Each convolution $[\varphi^{I}]_{I}*[\psi^{J}]_{J}=[\theta^{I,J}]_{M(I,J)}$ where $\theta^{I,J}\in \CC^{\infty}_{0}(\R^{N})$ is normalized relative to $\varphi^{I}$ and $\psi^{J}$, and where $M(I,J)$ is an $\alpha+\beta$-tuple of integers coming from $\alpha+\beta$ of the integers in $I=(i_{1}, \ldots, i_{s})$ and $J=(j_{1}, \ldots, j_{q})$. Since each function $\varphi^{I}$ and $\psi^{J}$ has strong cancellation, the function $\theta^{I,J}$ will have appropriate decay in the free variables and weak cancellation in the fixed variables. (See Lemma \ref{Lem10.13bb} below.) 
\medskip

\item
We let $\mathfrak M(S,T,W)$ \index{L4ambda@$\Lambda(S,T,W)$} denote the (infinite) set of $\alpha+\beta$-tuples $M=M(I,J)$ that arise in this way as the dilation for some convolution $[\varphi^{I}]_{I}*[\psi^{J}]_{J}$. For each $M\in \mathfrak M(S,T,W)$, we let $F(M)$ be the set of pairs $(I,J)\in \Gamma_{W}(S,T)$ such that $M(I,J)=M$. It follows that
\beas
\sum_{(I,J)\in \Gamma_{W}(S,T)}[\varphi^{I}]_{I}*[\psi^{J}]_{J}= \sum_{M\in \mathfrak M(S,T,W)}\Big[\sum_{(I,J)\in F(M)}\theta^{I,J}\Big]_{M}.
\eeas

\medskip

\item Because of the decay properties of $\theta^{I,J}$, for each $M\in \mathfrak M(S,T,W)$,  the inner infinite sum $\sum_{(I,J)\in F(M)}\theta^{I,J}$ converges to a function $\Theta^{M}\in\CC^{\infty}_{0}$, and the family of functions $\{\Theta^{M}\}$ is uniformly bounded and has weak cancellation relative to $M$. Thus \label{double sum}
\beas
\sum_{(I,J)\in \Gamma_{W}(S,T)}[\varphi^{I}]_{I}*[\psi^{J}]_{J}= \sum_{M\in \mathfrak M(S,T,W)}[\Theta^{M}]_{M}.
\eeas
Theorem \ref{Thm3.7} then shows that $\sum_{(I,J)\in \Gamma_{W}(S,T)}[\varphi^{I}]_{I}*[\psi^{J}]_{J}\in\PP_{0}(\EEE(S,T,W))$. The details are provided in Sections \ref{Estimates} and \ref{FixedandFree}.

\end{enumerate}
\end{enumerate}

\subsection{A preliminary decomposition}\label{PrelimDecomp}

Let $\KK, \LL \in \PP_{0}(\EEE)$. According to Corollary \ref{Cor8.8} we can write 
\beas
\KK &= \varphi_{0}+\sum_{S\in \SS'(\EEE)}\KK_{S}=\varphi_{0}+\sum_{S\in \SS'(\EEE)}\sum_{I\in  \Gamma_{\Z}(\EEE_{S})}[\varphi^{I}_{S}]_{I},\\
\LL &= \psi_{0}+\sum_{T\in \SS'(\EEE)}\LL_{T}=\psi_{0}+\sum_{T\in \SS'(\EEE)}\sum_{J\in\Gamma_{\Z}(\EEE_{T})}[\psi^{J}_{T}]_{J},
\eeas
where $\varphi_{0},\psi_{0}\in \SS(\R^{N})$ and $\{\varphi^{I}_{S}\}$ and $\{\psi^{J}_{T}\}$ are uniformly bounded families of functions in $\CC^{\infty}_{0}(\R^{N})$ supported in $\B(1)$. Each $\varphi^{I}_{S}$ has cancellation in the variables $\{\x_{k_{1}}, \ldots, \x_{k_{s}}\}$ and each $\psi^{J}_{T}$ has cancellation in the variables $\{\x_{l_{1}}, \ldots, \x_{l_{q}}\}$. We have
\beas
T_{\KK}\circ T_{\LL}= \varphi_{0}*\LL +\KK*\psi_{0} +\sum_{S,T\in\SS'(\EEE)}T_{\KK_{S}}\circ T_{\LL_{T}}.
\eeas
Since the first two terms on the right-hand side are Schwartz functions,  to prove Theorem \ref{Thm10.1} it suffices to show that each composition $T_{\KK_{S}}\circ T_{\LL_{T}}$ is a convolution operator $T_{\HH_{S,T}}$ where $\HH_{S,T}\in \PP_{0}(\EEE)$.

\subsection{Reduction to the case of finite sets}\label{FiniteSets}

It follows from Corollary \ref{Cor8.8} that
\beas
\KK_{S}(\x) &= \lim_{F\nearrow \Gamma_{\Z}(\EEE_{S})}\sum_{I\in F}[\varphi^{I}_{S}]_{I}(\x),
&&&
\LL_{T}(\x) &= \lim_{G\nearrow \Gamma_{\Z}(\EEE_{T})}\sum_{J\in G}[\psi^{J}_{T}]_{J}(\x)
\eeas
where the limits are taken in the sense of distributions, $F$ ranges over finite subsets of $\Gamma_{\Z}(\EEE_{S})$, and $G$ ranges over finite subsets of $\Gamma_{\Z}(\EEE_{T})$. It follows that
\beas
T_{\KK_{S}}\circ T_{\LL_{T}}[f] &= \lim_{F\nearrow \Gamma_{\Z}(\EEE_{S})}\lim_{G\nearrow \Gamma_{\Z}(\EEE_{T})}\Big[f*\sum_{I\in F}\sum_{J\in G}[\varphi^{I}_{S}]_{I}*[\psi^{J}_{T}]_{J}\Big].
\eeas
Thus it will suffice to show that $\lim_{F\nearrow\Gamma_{\Z}(\EEE_{S})}\lim_{G\nearrow \Gamma_{\Z}(\EEE_{T})}\sum_{I\in F}\sum_{J\in G}[\varphi^{I}_{S}]_{I}*[\psi^{J}_{T}]_{J} \in \PP_{0}(\EEE)$ where the limits are take in the sense of distributions. To do this, we will show that for any finite sets $F_{1}\subset \Gamma_{\Z}(\EEE_{S})$ and $F_{2}\subset\Gamma_{\Z}(\EEE_{T})$, 
\bes
\HH_{F_{1},F_{2}}=\sum_{I\in F_{1}}\sum_{J\in F_{2}}[\varphi^{I}_{S}]_{I}*[\psi^{J}_{T}]_{J}\in \PP_{0}(\EEE)
\ees
with estimates that are independent of the choice of the finite subsets $F$ and $G$, and
\bes
\HH_{S,T}= \lim_{F_{1}\nearrow \Gamma_{\Z}(\EEE_{S})}\lim_{F_{2}\nearrow \Gamma_{\Z}(\EEE_{T})}\HH_{F,G}\in \PP_{0}(\EEE)
\ees
\label{second lemma}
Thus to prove Theroem \ref{Thm10.1}, it will suffice establish the following result.

\medskip

\begin{lemma}\label{Lem10.2}
For each $S,T \in \SS'(\EEE)$ there exists $\HH_{S,T}\in \PP_{0}(\EEE)$ so that
\be\label{10.12ss}
\KK_{S}*\LL_{T}=\sum_{(I,J)\in \Gamma_{\Z}(\EEE_{S})\times\Gamma_{\Z}(\EEE_{T})}\,[\varphi^{I}_{S}]_{I}*[\psi^{J}_{T}]_{J}=\HH_{S,T}
\ee
where the series converges in the sense of distributions as the limit of sums over finite subsets of $\Gamma_{\Z}(\EEE_{S})\times\Gamma_{\Z}(\EEE_{T})$.
\end{lemma}

\subsection{Properties of the convolution $[\varphi^{I}_{S}]_{I}*[\psi^{J}_{T}]_{J}$}\label{Properties}

In order to proceed further, we will need information about the support, decay, and cancellation properties of each of the convolutions $[\varphi^{I}_{S}]_{I}*[\psi^{J}_{T}]_{J}$ that appear in equation (\ref{10.12ss}). The required estimates are given in Lemma \ref{Lem10.13bb} below. The statement essentially incorporates the results of Lemmas 6.17 and 6.18 and Corollary 6.19 in \cite{MR2949616}. Since the setting now is somewhat more complicated, we include a sketch of the proof.

Suppose that $G$ is a homogeneous nilpotent Lie group with underlying manifold $\R^{N}$, and that the one-parameter family of automorphic dilations on $G$ is given by \beas
\delta_{r}(\x)=(r^{d_{1}}x_{1}, \ldots, r^{d_{N}}x_{N}).
\eeas 
Let $*$ denote the convolution on $G$.

\begin{lemma}\label{Lem10.13bb}
Let $\varphi, \psi\in \CC^{\infty}_{0}(\R^{N})$ have support in the unit ball and put 
\beas[]
[\varphi]_{\lambda}(\x)&= \Big[\prod_{m=1}^{N}\lambda_{m}^{-d_{m}}\Big]\varphi(\lambda_{1}^{-d_{1}}x_{1}, \ldots, \lambda_{N}^{-d_{N}}x_{N}),\\
[\psi]_{\mu}(\x) &= \Big[\prod_{m=1}^{N}\mu_{m}^{-d_{m}}\Big]\psi(\mu_{1}^{-d_{1}}x_{1}, \ldots, \mu_{N}^{-d_{N}}x_{N}).
\eeas
Assume that  $0<\lambda_{1}\leq \lambda_{2}\leq \cdots \leq \lambda_{N}$ and $0<\mu_{1}\leq \mu_{2}\leq\cdots \leq \mu_{N}$. 
\begin{enumerate}[{\rm (a)}]

\item \label{Lem10.13bb1} \emph{[Support properties]} 
There exists $\theta\in \CC^{\infty}_{0}(\R^{N})$, normalized with respect to $\varphi$ and $\psi$, so that, 
\beas[]
[\varphi]_{\lambda}*[\psi]_{\mu}(\x) &= \Big[\prod_{m=1}^{n}\nu_{m}^{-d_{m}}\Big]\theta(\nu_{1}^{-d_{1}}\x_{1}, \ldots, \nu_{N}^{-d_{N}}x_{N})=[\theta]_{\nu}(\x)
\eeas
where $\nu_{j}=\lambda_{j}\vee\mu_{j}=\max\{\lambda_{j},\mu_{j}\}$.

\item  \label{Lem10.13bb2}\emph{[Cancellation and decay properties]}  The estimates come in three parts, depending on whether there is cancellation in any of the first $N-1$ variables $\{x_{1}, \ldots, x_{N-1}\}$, or in the last variable $x_{N}$.

\begin{enumerate}[{\rm(i)}]

\medskip

\item Let $C,D\subset \{1, \ldots, N-1\}$, and suppose that $\varphi$ has cancellation in the variables $x_{j}$ for $j\in C$, and that $\psi$ has cancellation in the variables $x_{k}$ for $k\in D$. Put
\beas
C_{1}& =\Big\{j\in \{1, \ldots, N-1\}:\text{$\varphi$ has cancellation in $x_{j}$ and $\lambda_{j}< \mu_{j}$}\big\},\\
C_{2}& =\Big\{j\in \{1, \ldots, N-1\}:\text{$\varphi$ has cancellation in $x_{j}$ and $\lambda_{j}\geq \mu_{j}$}\big\},\\
D_{1}&=\Big\{j\in \{1, \ldots, N-1\}:\text{$\psi$ has cancellation in $x_{j}$ and $\mu_{j}\leq\lambda_{j}$}\big\},\\
D_{2}&= \Big\{j\in \{1, \ldots, N-1\}:\text{$\psi$ has cancellation in $x_{j}$ and $\mu_{j}> \lambda_{j}$}\big\},
\eeas
so that
\beas
C_{1}\cup C_{2}&= C,&&&C_{1}\cap C_{2}&= \emptyset,\\
D_{1}\cup D_{2} &= D,&&&D_{1}\cap D_{2}&=\emptyset,\\
C_{1}\cap D_{1}&=\emptyset, &&& C_{2}\cap D_{2}&=\emptyset.
\eeas
There exists $\epsilon>0$ depending only on the group structure so that $[\varphi]_{\lambda}*[\psi]_{\mu}$  can be written as a finite sum of terms of the form

\bea\label{NewEquation}
\prod_{j\in C_{1}}\left[\left(\frac{\lambda_{j}}{\mu_{j}}\right)^{\epsilon}+\left(\frac{\lambda_{j}}{\lambda_{j+1}}\right)^{\epsilon}\right]& \times 
\prod_{k\in D_{1}}\left[\left(\frac{\mu_{k}}{\lambda_{k}}\right)^{\epsilon}+\left(\frac{\mu_{k}}{\mu_{k+1}}\right)^{\epsilon}\right]\times\\
&\times \prod_{j\in C_{2}'}\left(\frac{\lambda_{j}}{\lambda_{j+1}}\right)^{\epsilon}\times
\prod_{k\in D_{2}'}\left(\frac{\mu_{k}}{\mu_{k+1}}\right)^{\epsilon}
\big[\Theta_{C_{2}'',D_{2}''}\big]_{\nu}
\eea
\noindent where
\beas
C_{2}&= C_{2}'\cup C_{2}'',&&&&C_{2}'\cap C_{2}''&=\emptyset,\\
D_{2}&= D_{2}'\cup D_{2}'',&&&&D_{2}'\cap D_{2}''&=\emptyset.
\eeas
The functions $\big\{\Theta_{C_{2}'',D_{2}''}\big\}\subset \CC^{\infty}_{0}(\R^{N})$ are normalized relative to $\varphi$ and $\psi$ and have cancellation in the variables $x_{j}$ for $j\in C_{2}''\cup D_{2}''$.

\medskip

\item Suppose that $\varphi$ also has cancellation in the variable $x_{N}$. If $\lambda_{N}<\mu_{N}$, there is an additional factor $\Big(\frac{\lambda_{N}}{\mu_{N}}\Big)^{\epsilon}$ in equation (\ref{NewEquation}). If $\lambda_{N}\geq \mu_{N}$ then the function $\Theta_{C_{2}'',D_{2}''}$ in equation (\ref{NewEquation}) also has cancellation in the variable $x_{N}$.

\medskip

\item Suppose that $\psi$ also has cancellation in the variable $x_{N}$. If $\lambda_{N}<\mu_{N}$, then the function $\Theta_{C_{2}'',D_{2}''}$ in equation (\ref{NewEquation}) also has cancellation in the variable $x_{N}$.  If $\lambda_{N}\geq \mu_{N}$ there is an additional factor $\Big(\frac{\lambda_{N}}{\mu_{N}}\Big)^{\epsilon}$ in equation (\ref{NewEquation}).
\end{enumerate}
\end{enumerate}
\end{lemma}

\begin{proof}[Sketch of the proof]\quad

\medskip

Because of the monotonicity of the $\lambda$'s and the $\mu$'s, assertion (\ref{Lem10.13bb1}) follows immediately from Lemma \ref{Lem9.2}, so the real content of the lemma is assertion (\ref{Lem10.13bb2}). Let us first consider the case of  Abelian convolution.  Then there is no need to distinguish between the first $N-1$ variable and the last variable. If (say) $\varphi$ has cancellation in the variable $x_{j}$ then we can write $\varphi=\partial_{x_{j}}\phi$ where $\phi$ is normalized in terms of $\varphi$. Suppose that $\lambda_{j}<\mu_{j}$. Then moving the derivative $\partial_{x_{j}}$ across the convolution, we have
\beas[]
[\varphi]_{\lambda}*[\psi]_{\mu}
&=
[\partial_{x_{j}}\phi]_{\lambda}*[\psi]_{\mu}
=
\lambda_{j}^{d_{j}}\partial_{x_{j}}\big([\phi]_{\lambda})*[\psi]_{\mu}\\
&=
\lambda_{j}^{d_{j}}[\phi]_{\lambda}*\partial_{x_{j}}\big([\psi]_{\mu}\big)
=
\Big(\frac{\lambda_{j}}{\mu_{j}}\Big)^{d_{j}} [\phi]_{\lambda}*[\partial_{x_{j}}\psi]_{\mu}
=
\Big(\frac{\lambda_{j}}{\mu_{j}}\Big)^{d_{j}}[\theta]_{\nu}
\eeas
where we have use part (\ref{Lem10.13bb1}) to write $[\phi]_{\lambda}*[\partial_{x_{j}}\psi]_{\mu}= [\theta]_{\nu}$ where $\theta$ is normalized.  On the other hand, if $\lambda_{j}\geq \mu_{j}$, we can put the derivative in front of the whole convolution, and we have
\beas[]
[\varphi]_{\lambda}*[\phi]_{\mu}
&=
[\partial_{x_{j}}\phi]_{\lambda}*[\phi]_{\mu}
=
\lambda_{j}^{d_{j}}\partial_{x_{j}}\big([\phi]_{\lambda})*[\phi]_{\mu}\\
&=
\lambda_{j}^{d_{j}}\partial_{x_{j}}\Big([\phi]_{\lambda}*[\psi]_{\mu}\Big)
=
\lambda_{j}^{d_{j}}\partial_{x_{j}}\big([\theta']_{\nu}\big)
=
[\partial_{x_{j}}\theta']_{\nu}
\eeas
where $[\phi]_{\lambda}*[\psi]_{\mu}=[\theta']_{\nu}$ by part (\ref{Lem10.13bb1}), and now $\partial_{x_{j}}\theta'$ has cancellation in the variable $x_{j}$. Thus the lemma follows in the case of Abelian convolution on $\R^{n}$.

\medskip

For convolution $*$ on a homogeneous nilpotent Lie group, in general it is \emph{not} the case that $\partial_{x_{j}}f*g=f*\partial_{x_{j}}g$ or that $\partial_{x_{j}}f*g=\partial_{x_{j}}[f*g]$.  However, we have the following facts.\footnote{See Proposition 6.16 in \cite{MR2949616} or page 22 in \cite{MR657581}.}
\begin{enumerate}[(a)]

\smallskip

\item If $X$ is a left-invariant vector field and $Y$ is a right-invariant vector field then $X[f*g]=f*X[g]$ and $Y[f*g]= Y[f]*g$;

\smallskip

\item If $X$ is a left-invariant vector field and $Y=\widetilde X$ is the unique right-invariant vector field agreeing with $X$ at the origin, then $X[f]*g=f*Y[g]$;

\smallskip

\item For any $1 \leq j \leq N-1$ we can write
\beas
\frac{\partial f}{\partial x_{j}}(\x) = Z_{j}[f](\x)+\sum_{k=j+1}^{N}P_{j,k}(\x)\frac{\partial f}{\partial {x_{k}}}(\x)= Z_{j}[f](\x)+\sum_{k=j+1}^{N}\frac{\partial}{\partial {x_{k}}}\big[P_{j,k}(\x)f(\x)\big]
\eeas
where $Z_{j}$ is either a left- or a right-invariant vector field, and $P_{j,k}$ is a polynomial of the form
\beas
P_{j,k}(x_{1}, \ldots, x_{k-1}) = \sum_{\alphab\in\mathfrak A(j,k)}c_{j,k}(\alpha)x_{1}^{\alpha_{1}}\cdots x_{k-1}^{\alpha_{k-1}}
\eeas
where $\mathfrak A(j,k)=\big\{(\alpha_{1}, \ldots, \alpha_{k-1})\in \Z^{k-1}:\alpha_{1}d_{1}+\cdots+\alpha_{k-1}d_{k-1}=d_{k}-d_{j}\big\}$ .
\end{enumerate}

\bigskip

\noindent We can establish part (\ref{Lem10.13bb2}) of the lemma by repeating the Abelian argument, but replacing the derivative $\partial_{x_{j}}$ by the appropriate invariant vector field $Z_{j}$. When $j<N$ this introduces an error which is a sum of terms with cancellation in variables $x_{k}$ with $k>j$. Explicitly,
\beas[]
[\partial_{x_{j}}\phi]_{\lambda}&=
\lambda_{j}^{d_{j}}\partial_{x_{j}}\big([\phi]_{\lambda}\big)
=
\lambda_{j}^{d_{j}}Z_{j}\big([\phi]_{\lambda}\big)+\lambda_{j}^{d_{j}}\sum_{k=j+1}^{N}\partial_{x_{k}}\big(P_{j,k}(\x)[\phi]_{\lambda}(\x)\big)\\
\eeas
Now it follows as in Proposition 4.5 of \cite{MR2949616} that
\beas[]
[\partial_{x_{j}}\phi]_{\lambda}&=
\lambda_{j}^{d_{j}}Z_{j}\big([\phi]_{\lambda}\big)+\sum_{k=j+1}^{N}\Big(\frac{\lambda_{j}}{\lambda_{k}}\Big)^{d_{j}}
\lambda_{k}^{d_{k}}\partial_{x_{j}}\big([\Theta_{j,k}]_{\lambda}\big)\\
&=
\lambda_{j}^{d_{j}}Z_{j}\big([\phi]_{\lambda}\big)+\sum_{k=j+1}^{N}\Big(\frac{\lambda_{j}}{\lambda_{k}}\Big)^{d_{j}}
[\partial_{x_{j}}\Theta_{j,k}]_{\lambda}
\eeas
where the functions $\Theta_{j,k}$ are normalized relative to $\phi$. The term $\lambda_{j}^{d_{j}}Z_{j}\big([\phi]_{\lambda}\big)$ is handled as in the Abelian case, since now the vector field $Z_{j}$ can either be moved across a convolution or put in front of the convolution. This gives a gain of $\Big(\frac{\lambda_{j}}{\mu_{j}}\Big)^{d_{j}}$. The remaining terms now have cancellation in variables $x_{k}$ with $k>1$, and there is a gain  $\Big(\frac{\lambda_{j}}{\lambda_{k}}\Big)^{d_{j}}\leq\Big(\frac{\lambda_{j}}{\lambda_{j+1}}\Big)^{d_{j}} $. Of course when $j=N$, there is no correction term. This completes the sketch. Complete details can be found in Section 6 of  \cite{MR2949616}.
\end{proof}

\subsection{A further decomposition of $ \Gamma_{\Z}(\EEE_{S})\times \Gamma_{\Z}(\EEE_{T})$}\label{Decomp1}

Fix two marked partitions $S=\big\{(I_{1},k_{1});\ldots;(I_{s},k_{s})\big\},\,T=\{(J_{1},\ell_{1});\ldots;(J_{q},\ell_{q})\}\in \SS'(\EEE)$. We return to the study of the infinite sum $\sum_{(I,J)\in \Gamma_{\Z}(\EEE_{S})\times \Gamma_{\Z}(\EEE_{T})}[\varphi^{I}]_{I}*[\psi^{J}]_{J}$.\footnote{Since $S$ and $T$ are fixed, we now write $\varphi^{I}$ and $\psi^{J}$ instead of $\varphi^{I}_{S}$ and $\psi^{J}_{T}$.} To analyze this, we need a further decomposition of the index set $ \Gamma_{\Z}(\EEE_{S})\times \Gamma_{\Z}(\EEE_{T})$.  Recall from equation (\ref{1.15iou}) that
\bea\label{10.3aaa}
[\varphi^{I}]_{I}(\x) &= 2^{-\sum_{r=1}^{s}i_{r}\widehat Q_{S,r}}\,\varphi^{I}(2^{-i_{1}}\,\hat\cdot_{S}\,\x_{I_{1}}, \ldots, 2^{-i_{s}}\,\hat\cdot_{S}\,\x_{I_{s}}),\\
[\psi^{J}]_{J}(\x) &= 2^{-\sum_{a=1}^{q}j_{a}\widehat Q_{T,a}}\psi^{J}(2^{-j_{1}}\,\hat\cdot_{T}\,\x_{J_{1}},\ldots, 2^{-j_{q}}\hat \cdot_{T}\,\x_{J_{q}}),
\eea
where, using the identities $e(k_{r},j)=\tau_{S}(j,k_{r})^{-1}$ and $e(\ell_{a},l)=\tau_{T}(l,\ell_{a})^{-1}$ from Proposition \ref{Prop5.8}, together with \eqref{5.12aa} and Definition \ref{Def3.12iou} (b), we have
\bea\label{10.4aaa}
\widehat Q_{S,r}&=\sum_{j\in I_{r}}Q_{j}\,e(k_{r},j)^{-1}&&=\sum_{j\in I_{r}}Q_{j}\tau_{S}(j,k_{r}),\\
2^{-i_{r}}\,\hat\cdot_{S}\,\x_{I_{r}}&= \Big\{2^{-i_{r}/e(k_{r},j)}\cdot \x_{v}:j\in I_{r}\Big\}&&= \Big\{2^{-i_{r}\tau_{S}(j,k_{r})}\cdot \x_{v}:j\in I_{r}\Big\},\\
\widehat Q_{T,a}&= \sum_{l\in J_{a}}Q_{a}e(\ell_{a},l)^{-1}
&&=\sum_{l\in J_{a}}Q_{a}\tau_{T}(l,\ell_{a}),
\\2^{-j_{a}}\,\hat\cdot_{T}\,\x_{J_{a}}&= \Big\{2^{-j_{a}/e(\ell_{a},v)}\cdot\x_{v}:v\in J_{a}\Big\}
&&=
\Big\{2^{-j_{a}\tau_{T}(v,\ell_{a})}\cdot\x_{v}:v\in J_{a}\Big\}.
\eea

According to Lemma \ref{Lem10.13bb}, for each index $j$, $1\leq j \leq n$, the scale of the support of the convolution $[\varphi^{I}]_{I}*[\psi^{J}]_{J}$ in the direction of each coordinate $\x_{j}$ is the maximum of the scales of the supports of each factor. We let $1\leq r\leq s$ and $1 \leq a\leq q$ and focus on indices $j\in I_{r}\cap J_{a}$ provided that the intersection is non-empty. The size of the support of $[\varphi^{I}]_{I}$ in the direction of $\x_{j}$ is on the order of $2^{i_{r}\tau_{S}(j,k_{r})}$, while the size of the support of $[\psi^{J}]_{J}$ is on the order of $2^{j\tau_{T}(v_{j},\ell_{a})}$. Since the indices $i_{r}$ and $j_{a}$ are negative, we have
\beas
2^{i_{r}\tau_{S}(j,k_{r})} \geq 2^{j_{a}\tau_{T}(j,\ell_{a})}
&\Longleftrightarrow 
i_{r}\tau_{S}(j,k_{r}) \geq j_{a}\tau_{T}(j,\ell_{a})
\Longleftrightarrow
\frac{i_{r}}{j_{a}}\leq \frac{\tau_{T}(j,\ell_{a})}{\tau_{S}(j,k_{r})}=\frac{e(k_{r},j)}{e(\ell_{a},l)}.
\eeas
Write\footnote{We should write $I_{r}\cap J_{a} =\{v_{1}^{r,a}, \ldots, v_{m(r,a)}^{r,a}\}$, but as long as the context is clear we omit the dependence on $(r,a)$ for simplicity of notation.} $I_{r}\cap J_{a}=\big\{v_{1}, \ldots, v_{m}\big\}$, and order these indices so that
\be\label{10.15tt}
0<\frac{\tau_{T}(v_{1},\ell_{a})}{\tau_{S}(v_{1},k_{r})}
\leq \frac{\tau_{T}(v_{2},\ell_{a})}{\tau_{S}(v_{2},k_{r})}\leq \cdots \leq \frac{\tau_{T}(v_{m-1},\ell_{a})}{\tau_{S}(v_{m-1},k_{r})}\leq \frac{\tau_{T}(v_{m},\ell_{a})}{\tau_{S}(v_{m},k_{r})}<\infty.
\ee
The ratios in equation (\ref{10.15tt}) divide the positive real axis into at most $m+1$ subintervals. For any $I=(i_{1}, \ldots, i_{s})\in  \Gamma_{\Z}(\EEE_{S})$  and $J=(j_{1}, \ldots, j_{q})\in   \Gamma_{\Z}(\EEE_{T})$, the ratio $i_{r}/j_{a}$ must lie in one of these intervals. We partition $ \Gamma_{\Z}(\EEE_{S})\times \Gamma_{\Z}(\EEE_{T})$ into disjoint sets so that for $(I,J)$ in each set, the ratios $i_{r}/j_{a}$ always lie in the same interval. 

To describe and index this decomposition, we proceed as follows. Let 
\beas \index{S4igma@$\Sigma_{S,T}$}
\Sigma_{S,T}=\{(I_{r},J_{a}):1\leq r\leq s,\,\,1\leq a \leq q,\,\,I_{r}\cap J_{a}\neq \emptyset\},
\eeas 
For each 
$(I_r,J_a)\in\Sigma_{S,T}$ consider the set of $m_{r,a}+1$ intervals
\bea\label{UUU}
\Bigg\{\bigg[0,&\frac{\tau_{T}(v_{1},\ell_{a})}{\tau_{S}(v_{1},k_{r})}\bigg)=E^{r,a}_{0},
\left[\frac{\tau_{T}(v_{1},\ell_{a})}{\tau_{S}(v_{1},k_{r})},\frac{\tau_{T}(v_{2},\ell_{a})}{\tau_{S}(v_{2},k_{r})}\right)=E^{r,a}_{1},\ldots\\
&\ldots, \left[\frac{\tau_{T}(v_{j},\ell_{a})}{\tau_{S}(v_{j},k_{r})},\frac{\tau_{T}(v_{j+1},\ell_{a})}{\tau_{S}(v_{j+1},k_{r})}\right)=E^{r,a}_{j},\ldots\\
&
\ldots,
\left[\frac{\tau_{T}(v_{m_{r,a}-1},\ell_{a})}{\tau_{S}(v_{m_{r,a}-1},k_{r})},\frac{\tau_{T}(v_{m_{r,a}},\ell_{a})}{\tau_{S}(v_{m_{r,a}},k_{r})}\right)=E^{r,a}_{m_{r,a}-1}, 
\left[\frac{\tau_{T}(v_{m_{r,a}},\ell_{a})}{\tau_{S}(v_{m_{r,a}},k_{r})},\infty\right)=E^{r,a}_{m_{r,a}}\Bigg\},
\eea

and let \index{W1@$\widetilde {\WW}(S,T),\WW(S,T)$}
\beas
\widetilde {\WW}(S,T)=\prod_{(I_r,J_a)\in\Sigma_{S,T}}\big\{E^{r,a}_{0},E^{r,a}_{1},\dots,E^{r,a}_{m_{r,a}}\big\},
\eeas
i.e., the set of functions which assign to each  $(I_r,J_a)\in \Sigma_{S,T}$ one of the intervals in \eqref{UUU}.
Hence, if $W\in \widetilde \WW(S,T)$
 and $(I_{r},J_{a})\in \Sigma_{S,T}$, $W(I_{r},J_{a})$ is one of these intervals, and we write
\beas
W(I_{r},J_{a})=E^{r,a}_{w(I_{r},J_{a})}
\eeas
so that $w(I_{r},J_{a})\in\big\{0,1,\dots,m_{r,a}\big\}$. (Thus $W(I_{r},J_{a})$ is one of the $m_{r,a}+1$ intervals, while $w(I_{r},J_{a})$ is the number of that interval.)   Setting $j=w(I_{r},J_{a})$, let
\beas
\Lambda_{W}(J_{a},I_{r})&=
\begin{cases}
\infty &\text{if }j=0,\\\\
\frac{\tau_{S}(v_{j},k_{r})}{\tau_{T}(v_{j},\ell_{a})} &\text{if $j>0$,}
\end{cases}&&&
\Lambda_{W}(I_{r},J_{a}) &=
\begin{cases}
\frac{\tau_{T}(v_{j+1},\ell_{a})}{\tau_{S}(v_{j+1},k_{r})}&\text{if $j<m_{r,a}$,}\\\\
\infty &\text{if $j=m_{r,a}$,}
\end{cases}
\eeas
so that $\Lambda_{W}(J_{a},I_{r})^{-1}$ and $\Lambda_{W}(I_{r},J_{a})$ are the left and right endpoints of $E_{w(I_{r},J_{a})}$.

\medskip

We now construct the desired partition. For $W\in \widetilde{\WW}(S,T)$, put 
\bea\label{CCC} \index{G3ammaW@$\Gamma_{W}(S,T)$}
\Gamma_{W}(S,T) = \Big\{\big((i_{1},\ldots,i_{s}),&(j_{1}, \ldots,j_{q})\big)\in \Gamma_{\Z}(\EEE_{S})\times \Gamma_{\Z}(\EEE_{T}): \\
&
(I_{r},J_{a})\in \Sigma_{S,T} \Longrightarrow \frac{1}{\Lambda_{W}(J_{a},I_{r})} \leq \frac{i_{r}}{j_{a}}<\Lambda_{W}(I_{r},J_{a})\Big\}.
\eea
It may happen  that for some $W\in \widetilde \WW(S,T)$ the set $\Gamma_{W}(S,T)$ is empty,\footnote{For instance,  if there are equalities among the weak inequalities in (\ref{10.15tt}), some of the intervals are empty. } and so we put
\beas
\WW(S,T)=\Big\{W\in \widetilde \WW(S,T):\Gamma_{W}(S,T)\neq \emptyset\Big\}.
\eeas
The following Proposition is an immediate consequence of the definition of $\Gamma_{W}(S,T)$ and the ordering in equation (\ref{10.15tt}).

\goodbreak

\begin{proposition}\label{Prop10.4yy}Let $S=\big\{(I_{1},k_{1});\ldots;(I_{s},k_{s})\big\},\,T=\{(J_{1},\ell_{1});\ldots;(J_{q},\ell_{q})\}\in \SS'(\EEE)$.

\begin{enumerate}[{\rm(a)}]

\smallskip

\item The sets $\{\Gamma_{W}(S,T):W\in \WW(S,T)\}$ partition the set $\Gamma_{\Z}(\EEE_{S})\times \Gamma_{\Z}(\EEE_{T})$:
\beas
\Gamma_{\Z}( {\EEE}_{S})\times \Gamma_{\Z}({\EEE}_{T})&= \bigcup_{W\in \WW(S,T)}\Gamma_{W}(S,T),\\
W_{1}, W_{2}\in \WW(S,T&),\quad W_{1}\neq W_{2}\Longrightarrow\Gamma_{W_{1}}(S,T)\cap \Gamma_{W_{2}}(S,T)=\emptyset.
\eeas

\smallskip

\item

Let $W\in \WW(S,T)$ and suppose $I_{r}\cap J_{a} =\{v_{1}, \ldots, v_{m}\}\neq\emptyset$, ordered as in equation (\ref{10.15tt}). Suppose that $W(I_{r},J_{a}) = E_{w(I_{r},J_{a})}$. Let $\big((i_{1},\ldots,i_{s}),(j_{1}, \ldots,j_{q})\big)\in \Gamma_{W}(S,T)$. Then for $1\leq k \leq m$,
\beas
k\leq w(I_{r},J_{a}) &\Longleftrightarrow j_{a}\tau_{T}(v_{k},\ell_{a})\geq i_{r}\tau_{S}(v_{k},k_{r})\\
k>w(I_{r},J_{a}) &\Longleftrightarrow j_{a}\tau_{T}(v_{k},\ell_{a})< i_{r}\tau_{S}(v_{k},k_{r})\\
\eeas
In particular, 
\begin{enumerate}[{\rm(1)}]

\smallskip

\item if $w(I_{r},J_{a})=0$, then $i_{r}\tau_{S}(v_{k},k_{r})>j_{a}\tau_{T}(v_{k},\ell_{a}) $ for all $1 \leq k \leq m$;

\smallskip

\item if $w(I_{r},J_{a})= m$ then $i_{r}\tau_{S}(v_{k},k_{r})\leq j_{a}\tau_{T}(v_{k},\ell_{a}) $ for all $1 \leq k \leq m$;

\smallskip

\item if $k_{r}\in I_{r}\cap J_{a}$ then
\beas
k_{r}\leq w(I_{r},J_{a}) &\Longrightarrow i_{r}\leq \tau_{T}(k_{r},\ell_{a})\,j_{a}<0,\\
k_{r}> w(I_{r},J_{a}) &\Longrightarrow \tau_{T}(k_{r},\ell_{a})\,j_{a}<i_{r}<0;\\
\eeas
\item 
if $\ell_{a}\in I_{r}\cap J_{a}$ then
\beas
\ell_{a}> w(I_{r},J_{a}) &\Longrightarrow j_{a}<\tau_{S}(\ell_{a},k_{r})\, i_{r}<0,\\
\ell_{a}\leq w(I_{r},J_{a}) &\Longrightarrow \tau_{S}(\ell_{a},k_{r})\, i_{r}\leq j_{a}<0.
\eeas
\end{enumerate}
\end{enumerate}
\end{proposition}

To summarize, \label{summary}we have partitioned $\Gamma_{\Z}( {\EEE}_{S})\times \Gamma_{\Z}({\EEE}_{T})$ into a disjoint union of sets $\Gamma_{W}(S,T)$. Let $I=(i_{1}, \ldots, i_{s})$ and $J=(j_{1}, \ldots, j_{q})$. Then $(I,J)\in \Gamma_{W}(S,T)$ if and only if
\bea\label{10.8zxc}
 e_{S}(p,r) i_{r}&\leq i_{p}<0&&&&\text{for $r,p\in\{1, \ldots, s \}$,}\\
 e_{T}(b,a)j_{a}&\leq j_{j}<0&&&&\text{for $a,b\in \{1, \ldots, q\}$,}\\
\Lambda_{W}(I_{r},J_{a})j_{a}&< i_{r}&&&&\text{if $I_{r}\cap J_{a}\neq\emptyset$,}\\
\Lambda_{W}(J_{a},I_{r})i_{r}&\leq j_{a}&&&&\text{if $I_{r}\cap J_{a}\neq\emptyset$.}
\eea
It can happen that $\Lambda_{W}(I_{r},J_{a})$ or $\Lambda_{W}(J_{a},I_{r})$ is $+\infty$, in which case there is simply no corresponding inequality between $i_{r}$ and $j_{a}$. However, if there is no bound of an index $i_{r}$ by \emph{any} $j_{a}$, then $\Lambda_{W}(I_{r},J_{a})=\infty$ for every $a$ such that $I_{r}\cap J_{a}\neq \emptyset$, and this means that the scale of the dilation by $i_{r}$ is \emph{always} smaller than the scale of the dilation by the corresponding $j_{a}$. Similarly, if there is no bound of $j_{a}$  by \emph{any} $i_{r}$, then the scale of the dilation by $j_{a}$ is \emph{always} less than or equal to the scale of of the dilation by the corresponding $i_{r}$.\footnote{In the notation of the next section, this means that in the first case, $i_{r}$ is a free index, and in the second case $j_{a}$ is a free index. See Definition \ref{Def10.5zxc}.}

\subsection{Fixed and free indices}\label{FixedFree}

Continue to fix $S=\big\{(I_{1},k_{1});\ldots;(I_{s},k_{s})\big\}, \,\,T=\{(J_{1},\ell_{1});\ldots, (J_{q},\ell_{q})\}\in \SS'(\EEE)$. The index set $\WW(S,T)$ is finite, so formally, using the partition from Proposition \ref{Prop10.4yy}, we have
\beas
\KK_{S}*\LL_{T}(\x) &=\sum_{W\in \WW(S,T)}\Bigg[\sum_{(I,J)\in\Gamma_{W}(S,T)}\,[\varphi^{I}]_{I}*[\psi^{J}]_{J}(\x)\Bigg]=\sum_{W\in \WW(S,T)}\RR_{S,T,W}(\x).
\eeas
Thus to prove Lemma \ref{Lem10.2} (and hence prove Theorem \ref{Thm10.1}) it suffices to show that for each $W\in \WW(S,T)$ the inner infinite sum 
\beas
\RR_{S,T,W}=\sum_{(I,J)\in\Gamma_{W}(S,T)}\,[\varphi^{I}]_{I}*[\psi^{J}]_{J}
\eeas
converges in the sense of distributions, and that the sum belongs to the class $\PP_{0}(\EEE)$. We will do this by showing that there is a matrix $\EEE(S,T,W)$ so that $\RR_{S,T,W}\in \PP_{0}(\EEE(S,T,W))\subset\PP_{0}(\EEE)$. 

\begin{definition}\label{Def10.5zxc}
Let $S=\big\{(I_{1},k_{1});\ldots;(I_{s},k_{s})\big\}, T=\{(J_{1},\ell_{1});\ldots, (J_{q},\ell_{q})\}\in \SS'(\EEE)$ and $W\in \WW(S,T)$. Recall that if $(I_{r},J_{a})\in \Sigma_{S,T}$ then $I_{r}\cap J_{a}\neq \emptyset$ and $W(I_{r},J_{a})=E_{w(I_{r},J_{a})}$ where $\{E_{0}, \ldots, E_{m}\}$ are the intervals listed in equation (\ref{UUU}).

\begin{enumerate}[{\rm(a)}]

\item
Set 
\bea\label{IJstar}
I_{r}^{*}
&=
\bigcup_{\substack{a=1\\I_{r}\cap J_{a}\neq \emptyset}}^{q}\Big\{j\in I_{r}\cap J_{a}: i_{r}\,\tau_{S}(j,k_{r})>j_{a}\,\tau_{T}(j,\ell_{a})\
\\
&\qquad\qquad\qquad\qquad\qquad\qquad \forall \big((i_{1},\ldots,i_{s}),(j_{1}, \ldots,j_{q})\big)\in\Gamma_W(S,T)\Big\}\\
&=
\bigcup_{\substack{a=1\\I_{r}\cap J_{a}\neq \emptyset}}^{q}\Big\{v_k\in I_{r}\cap J_{a}: k>w(I_{r},J_{a})\Big\} ;\\
J_{a}^{*}&=
\bigcup_{\substack{r=1\\I_{r}\cap J_{a}\neq \emptyset}}^{w}\Big\{j\in I_{r}\cap J_{a}: i_{r}\,\tau_{S}(j,k_{r})\leq j_{a}\,\tau_{T}(j,\ell_{a})\ 
\\
&\qquad\qquad\qquad\qquad\qquad\qquad\forall \big((i_{1},\ldots,i_{s}),(j_{1}, \ldots,j_{q})\big)\in\Gamma_W(S,T)\Big\}\\
&=
\bigcup_{\substack{a=1\\I_{r}\cap J_{a}\neq \emptyset}}^{q}\Big\{v_k\in I_{r}\cap J_{a}: k\leq w(I_{r},J_{a})\Big\} .\index{I1rJa@$I_{r}^{*},J_{a}^{*}$}
\eea

\item An index $r$, $1\leq r\leq s$ is \textbf{fixed} if $I_{r}^{*}\neq \emptyset$, and is \textbf{free} if $I_{r}^{*}= \emptyset$.

\smallskip

\item An index $a$, $1 \leq a \leq q$ is \textbf{fixed} if $J_{a}^{*}\neq \emptyset$, and is \textbf{free} if $J_{a}^{*}= \emptyset$.

\end{enumerate}
\end{definition}
\begin{remark}
The notions of fixed and free indices and the definitions of the sets $I_{r}^{*}$ and $J_{a}^{*}$ depend on the choice of $S,T\in\SS'(\EEE)$ and on the choice of $W\in \WW(S,T)$. In particular, they are the same for all choices of $(I,J)\in \Gamma_{W}(S,T)$.
\end{remark}
Clearly $I_{r}^{*}\subset I_{r}$, $J_{a}^{*}\subset J_{a}$ and $I_{r}^{*}\cap J_{a}^{*}= \emptyset$. Also, for every $j\in \{1, \ldots, n\}$ there is a unique pair $(I_{r},J_{a})\in \Sigma_{S,T}$ so that $j\in I_{r}\cap J_{a}$, and then either $j\in I_{r}^{*}$ or $j\in J_{a}^{*}$. Thus we have a disjoint decomposition 
\bea\label{VVV}
\{1, \ldots, n\} = \bigcup_{r=1}^{s}I_{r}^{*}\,\cup\,\bigcup_{a=1}^{q}J_{a}^{*}.
\eea 

The following Proposition provides an explanation of the significance of fixed and free indices. The proof just involves unwinding the definitions.

\begin{proposition}\quad
\begin{enumerate}[{\rm(a)}]

\smallskip

\item An index $r\in \{1, \ldots, s\}$ is free if and only if for {every} $j\in I_{r}$ and  every $(I,J)\in \Gamma_{W}(S,T)$,  the scale of the support of $[\varphi^{I}]_{I}$ is smaller in the direction of $\x_{j}$ than the scale of the support of $[\psi^{J}]_{J}$.

\smallskip

\item An index $r\in \{1, \ldots, s\}$ is fixed if and only if there exists {at least one} $j\in I_{r}$ so that for every $(I,J)\in \Gamma_{W}(S,T)$, the scale of the support of $[\varphi^{I}]_{I}$ in the direction of $\x_{j}$ is larger than the scale of the support of $[\psi^{J}]_{J}$.

\smallskip

\item The indices $j\in \{1, \ldots, n\}$ for which the scale of the support of $[\varphi^{I}]_{I}$ is larger than the scale of the support of $[\psi^{J}]_{J}$ in the direction of $\x_{j}$ is precisely $\bigcup_{r=1}^{s}I_{r}^{*}$.

\smallskip

\item An index $a\in \{1, \ldots, q\}$ is free if and only if for {every} $j\in J_{a}$ and for every $(I,J)\in \Gamma_{W}(S,T)$, the scale of the support of $[\psi^{J}]_{J}$ in the direction of $\x_{j}$ is smaller than the scale of the support of $[\varphi^{I}]_{I}$.

\smallskip

\item An index $a\in \{1, \ldots, q\}$ is fixed if and only if there exists {at least one} $j\in J_{a}$ so that for every $(I,J)\in \Gamma_{W}(S,T)$ the scale of the support of $[\psi^{J}]_{J}$ is larger than the scale of the support of $[\varphi^{I}]_{I}$ in the direction of $\x_{j}$.

\smallskip

\item The indices $j\in \{1, \ldots, n\}$ for which the scale of the support of $[\psi^{J}]_{J}$ is larger than the scale of the support of $[\varphi^{I}]_{I}$ in the direction of $\x_{j}$ is precisely $\bigcup_{a=1}^{q}J_{a}^{*}$.
\end{enumerate}
\end{proposition}

Let $S^{*}=\{r_{1}, \ldots, r_{\alpha}\}$ denote the subset of $\{1, \ldots, s\}$ consisting of fixed indices, and let $T^{*}= \{a_{1}, \ldots, a_{\beta}\}$ denote the subset of $\{1, \ldots, q\}$ consisting of fixed indices. It may happen that $\alpha=0$ or that $\beta=0$, but we always have $\alpha+\beta\geq 1$. Then we can write
\be\label{10.28ss}
\{1, \ldots, n\} =\bigcup_{r\in S^{*}}I_{r}^{*}\cup \bigcup_{a\in T^{*}}J_{a}^{*}= \bigcup_{j=1}^{\alpha}I_{r_{j}}^{*} \cup \bigcup_{k=1}^{\beta}J_{a_{k}}^{*}.
\ee

\subsection{{A finer decomposition of $\R^{N}$}}\label{Finer}\quad

\medskip

Using the decomposition in equation (\ref{10.28ss}) we study the new decomposition of $\R^{N}$ indexed by the fixed indices in $\{1, \ldots, s\}$ and $\{1, \ldots, q\}$: 
\bes
\R^{N}= \bigoplus_{r\in S^{*}}\R^{I_{r}^{*}}\oplus\bigoplus_{a\in T^{*}}\R^{J_{a}^{*}}=\bigoplus_{j=1}^{\alpha}\R^{I_{r_{j}}^{*}}\,\oplus \,\bigoplus_{j=1}^{\beta}\R^{J_{a_{j}}^{*}}
\ees
where
\beas
\R^{I_{r}^{*}}&=\big\{\x=(\x_{1}, \ldots, \x_{n})\in \R^{N}: \x_{j}\neq 0 \Longrightarrow j\in I_{r}^{*}\big\},\\
\R^{J_{a}^{*}}&=\big\{\x=(\x_{1}, \ldots, \x_{n})\in \R^{N}: \x_{j}\neq 0 \Longrightarrow j\in J_{a}^{*}\big\}.
\eeas
Write elements of $\R^{I_{r}^{*}}$ and $\R^{J_{a}^{*}}$ as $\x_{I_{r}^{*}}=\{\x_{j}:j\in I_{r}^{*}\}$ and  $\x_{J_{a}^{*}}=\{\x_{j}:j\in J_{a}^{*}\}$. This gives us a set of coordinates on $\R^{N}$. 

Order these coordinates so that the corresponding set of indices $\{k_{r_{1}},\ldots, k_{r_{\alpha}}, \ell_{a_{1}}, \ldots, \ell_{a_{\beta}}\}$ together are increasing. With this ordering, rename the subspaces  as $\{\R^{V_{1}}, \ldots, \R^{V_{\alpha+\beta}}\}$. Then  every element $\x\in\R^{N}$ is an $(\alpha+\beta)$-tuple $\x=(\overline \x_{1}, \ldots,\overline \x_{m}, \ldots, \overline \x_{\alpha+\beta})$  where $\overline \x_{m}\in \R^{V_{k}}$ and $V_{k}$ is either $I_{r}^{*}$ for some $r\in S^{*}$ or $J_{a}^{*}$ for some $a\in T^{*}$. By the choice of the ordering, if $m_{1}<m_{2}$ then
\beas
V_{m_{1}}&=I_{r}^{*},\,&V_{m_{2}}&=I_{p}^{*}&&\quad \Longrightarrow&\quad k_{r}&<k_{p},\\ V_{m_{1}}&=I_{r}^{*},\,&V_{m_{2}}&=J_{a}^{*}&&\quad \Longrightarrow&\quad k_{r}&<\ell_{a},\\
V_{m_{1}}&=J_{a}^{*},\,&V_{m_{2}}&=I_{r}^{*}&&\quad \Longrightarrow&\quad \ell_{a}&<k_{r},\\
V_{m_{1}}&=J_{a}^{*},\,&V_{m_{2}}&=J_{b}^{*}&&\quad \Longrightarrow&\quad \ell_{a}&<\ell_{b}.
\eeas

On each component $\R^{V_{m}}$ we use the dilation structure inherited from either $I_{r}$ or $J_{a}$. Thus  from (\ref{10.4aaa}) we define a dilation which we write $\,\circ\,$ by setting
\bea\label{10.36ss}
\lambda\circ\overline \x_{m}=
\begin{cases}
\Big\{\lambda^{\tau_{S}(j,k_{r})}\cdot \x_{j}:j\in I_{r}^{*}\Big\}&\text{if $\overline \x_{m}=\x_{I_{r}^{*}}$}\\\\
\Big\{\lambda^{\tau_{T}(j,\ell_{a})}\cdot \x_{j}:j\in J_{a}^{*}\Big\}&\text{if $\overline \x_{m}=\x_{J_{a}^{*}}$} \end{cases}.
\eea
(To keep notation simple we write  $\,\circ\,$ instead of $\,\circ_{S,T,W}\/$.) The corresponding homogeneous norms are given by
\beas
n_{S,r}(\x_{I_{r}^{*}}) &= \sum_{j\in I_{r}^{*}}n_{j}(\x_{j})^{1/\tau_{S}(j,k_{r})},&&&
n_{T,a}(\x_{J_{a}^{*}})&= \sum_{j\in J_{a}^{*}}n_{j}(\x_{j})^{1/\tau_{T}(j,\ell_{a})},
\eeas
and the homogeneous dimension of $\R^{V_{m}}$ is then
\beas
Q_{m}=
\begin{cases}
Q_{I_{r}^{*}}= \sum_{j\in I_{r}^{*}}Q_{j}\,\tau_{S}(j,k_{r}) &\text{if $\overline \x_{m}=\x_{I_{r}^{*}}$};\\\\
Q_{J_{a}^{*}}= \sum_{j\in J_{a}^{*}} Q_{j}\,\tau_{T}(j,\ell_{a})&\text{if $\overline \x_{m}=\x_{J_{a}^{*}}$}.
\end{cases}
\eeas

\subsection{{The matrix $\EEE_{S,T,W}$}}\label{TheMatrix}\quad

\medskip

If $\big((i_{1}, \ldots, i_{s}),(j_{1}, \ldots, j_{q})\big)\in \Gamma_{W}(S,T)\subset \Gamma_{\Z}({\EEE}_{S})\times\Gamma_{\Z}({\EEE}_{T})$,
   we obtain an $(\alpha+\beta)$-tuple $\{i_{r_{1}}, \ldots, i_{r_{\alpha}}, j_{a_{1}}, \ldots, j_{a_{\beta}}\}$ by keeping only the $i_{r}$ and $j_{s}$ corresponding to fixed indices $S^{*}$ and $T^{*}$. Let $\pi:\Z^{s+q}\to \Z^{\alpha+\beta}$ be the mapping which takes $(i_{1}, \ldots, i_{s}), (j_{1}, \ldots, j_{q})\in\Z^{s+q}$ to  $(i_{r_{1}}, \ldots, i_{r_{\alpha}}, j_{a_{1}}, \ldots, j_{a_{\beta}})\in \Z^{\alpha+\beta}$. Our next objective is to describe the inequalities satisfied by the entries of these $(\alpha+\beta)$-tuples.  

It follows from equation (\ref{10.8zxc}) that 
\bea\label{10.12zxc}
e_{S}(r_{j},r_{k}) i_{r_{k}}&\leq i_{r_{j}}<0&&&&\text{for $j,k\in\{1, \ldots, \alpha \}$,}\\
 e_{T}(a_{l},a_{m})j_{a_{m}}&\leq j_{a_{l}}<0&&&&\text{for $l,m\in \{1, \ldots, \beta\}$,}\\
\Lambda_{W}(I_{r_{j}},J_{a_{m}})j_{a_{m}}&< i_{r_{j}}&&&&\text{if $I_{r_{j}}\cap J_{a_{m}}\neq\emptyset$,}\\
\Lambda_{W}(J_{a_{m}},I_{r_{j}})i_{r_{j}}&\leq j_{a_{m}}&&&&\text{if $I_{r_{j}}\cap J_{a_{m}}\neq\emptyset$.}
\eea
Let
\bea \index{G4ammaSTW@$\Gamma(S,T,W)$}
\Gamma(S,T,W) =\Big\{(i_{r_{1}}, \ldots, i_{r_{\alpha}},j_{a_{1}}, \ldots, j_{a_{\beta}})\in\Z^{\alpha+\beta}\,\,\text{which satisfy (\ref{10.12zxc})}\Big\}.
\eea
Thus $\pi:\Gamma_{W}(S,T) \to \Gamma(S,T,W)$. 

\begin{proposition}
The mapping $\pi:\Gamma_{W}(S,T) \to \Gamma(S,T,W)$ is onto.
\end{proposition}

\begin{proof} Thus suppose that  the $(\alpha+\beta)$-tuple $(i_{r_{1}},\ldots, i_{r_{\alpha}},j_{a_{1}}, \ldots, j_{a_{\beta}})$ satisfies the inequalties in equation (\ref{10.12zxc}). First observe that the first two lines of inequalities describe the projections of $\Gamma({\EEE}_{S})$ and $\Gamma({\EEE}_{T})$ onto the smaller sets of coordinates. Since $\EEE_{S}$ and $ \EEE_{T}$ satisfy the basic hypotheses (\ref{2.5}), it follows from Lemma \ref{Lem14.4zxc} that there exists $(i_{1}, \ldots, i_{s})\in \Gamma({\EEE}_{S})$ and $(t_{1}, \ldots, t_{q})\in \Gamma({\EEE}_{T})$ such that 
\beas
\pi(i_{1}, \ldots, i_{s},j_{1}, \ldots, j_{q}) = (i_{r_{1}},\ldots, i_{r_{\alpha}},j_{a_{1}}, \ldots, j_{a_{\beta}}).
\eeas
We claim that $(i_{1}, \ldots, i_{s},j_{1}, \ldots, j_{q})\in \Gamma_{W}(S,T)$, so we need to check all the inequalities in equation (\ref{10.8zxc}). The first two lines are clearly satisfied, since they only depend on the fact that $(i_{1}, \ldots, i_{s})\in \Gamma(\EEE_{S})$ and $(j_{1}, \ldots, j_{q})\in \Gamma( \EEE_{T})$. Also the second two lines are satisfied provided that $r=r_{j}$ and $a=a_{m}$ for some $j\in \{1, \ldots, \alpha\}$ and $m\in \{1, \ldots, \beta\}$. It remains to check what happens if $i_{r}$ is a free index or $j_{a}$ is a free index. But as we observed in the summary on page \pageref{summary}, if $i_{r}$ is a free index and $I_{r}\cap J_{a}\neq \emptyset$ then $\Lambda_{W}(I_{r},J_{a})=\infty$ so there is no third line in the inequalities. Also if $i_{r}$ is a free index, then the scale of the dilation $\tau_{S}(j,k_{r})i_{r}$ must be smaller than the scale of the dilation $\tau_{T}(j,\ell_{a})j_{a}$ for every $j\in I_{r}\cap J_{a}$. This says that the ratio $i_{r}/j_{a}$ must lie in the largest (\textit{i.e.} the unbounded) interval, and this is the statement of the fourth inequality. A similar argument works if $j_{a}$ is a free index.
\end{proof}

The cone $\Gamma(S,T,W)\subset \R^{\alpha+\beta}$ is defined by the inequalities in (\ref{10.12zxc}), which, using the notation in Section \ref{PartialMatrix}, 
is defined by a \emph{partial matrix} $\BBB_W$, with indices  varying in the disjoint union $S^*\sqcup T^*$ and given by
\be\label{bW}
b(\mu,\nu)=\begin{cases} e_S(\mu,\nu)&\text{ if }\mu,\nu\in S^*\\ e_T(\mu,\nu)&\text{ if }\mu,\nu\in T^*\\ \Lambda_W(I_\mu,J_\nu)&\text{ if }\mu\in S^*\,,\,\nu\in T^*\\ 
\Lambda_W(J_\mu,I_\nu)&\text{ if }\mu\in T^*\,,\,\nu\in S^*.
\end{cases}
\ee 

\begin{lemma} The matrix $\BBB_W$ is connected.
\end{lemma}

\begin{proof}
We must prove that, given $\mu\ne\nu \in S^*\sqcup T^*$, there are $\lambda_1,\dots,\lambda_a\in S^*\sqcup T^*$ such that 
$$
b_W(\mu,\lambda_1)b_W(\lambda_1,\lambda_2)\cdots b_W(\lambda_a,\nu)<\infty.
$$
There is no doubt that this holds if both $\mu$ and $\nu$ belong either to $S^*$ or to $T^*$ because all the $e_S$ and $e_T$ are finite. For the other two cases it is necessary and sufficient to prove that there exist $r\in S^*$, $a\in T^*$ with $\Lambda_W(I_r,J_a)<\infty$ and $r'\in S^*$, $a'\in T^*$  with $\Lambda_W(J_{a'},I_{r'})<\infty$ (equivalently with  $w(I_r,J_a)\ne0$ and $w(I_{r'},J_{a'})\ne m_{r',a'}$).

Assume that one of these conditions is violated, e.g., that $w(I_r,J_a)=0$ for all $r\in S^*, a\in T^*$. By \eqref{IJstar}, this implies that $I_r\cap J_a\subset I_r^*$ for every $r,a$. Hence $T^*=\emptyset$ and the proof is completed.
\end{proof}

By Lemma \ref{Lem3.2zxc},  there is matrix $\EEE_{S,T,W}$ \index{E5ESTW@$\EEE_{S,T,W}$} which satisfies the basic hypotheses (\ref{2.5}) so that we have
$\Gamma(S,T,W) = \Gamma(\EEE_{S,T,W})$. 

\medskip

\begin{proposition}
If $(m_{1}, \ldots, m_{\alpha+\beta})\in \Gamma_{\Z}(\EEE_{S,T,W})$, the corresponding dilation of $\R^{N}$ given by
\beas
(\overline \x_{1}, \ldots, \overline \x_{\alpha+\beta}) \longrightarrow (2^{-m_{1}}\circ\overline \x_{1}, \ldots, 2^{-m_{\alpha+\beta}}\circ\overline \x_{\alpha+\beta})
\eeas
is still compatible (in the sense of Section \ref{Compatibility}) with the nilpotent Lie group structure.
\end{proposition}

\begin{proof}
Let $1 \leq l_{1}<l_{2}\leq n$. Then $\x_{l_{1}}$ and $\x_{l_{2}}$ are coordinates that belong to one of the sets $\{I_{r_{j}}^{*}, 1\leq j \leq \alpha\}$ or $\{J_{a_{k}}^{*}, 1\leq k \leq \beta\}$. Write the $(\alpha+\beta)$-tuple $(m_{1}, \ldots, m_{\alpha+\beta})=(i_{r_{1}}, \ldots, i_{r_{\alpha}},j_{a_{1}}, \ldots, j_{a_{q}})$. There are thus four cases to consider.

If $\x_{l_{1}}\in I_{r_{1_{1}}}^{*}$ and $\x_{l_{2}}\in I_{r_{l_{2}}}^{*}$, we proceed as in Lemma \ref{Lem9.6}. After taking logarithm to base $2$ the scales of the corresponding dilations are $i_{r_{l_{1}}}\tau_{S}(l_{1},k_{r_{1_{1}}})$ and $i_{r_{l_{2}}}\tau_{S}(l_{2},k_{r_{l_{2}}})$. Now $ e_{S}(r_{l_{2}},r_{l_{1}})i_{r_{l_{1}}}\leq i_{r_{l_{2}}}<0$ so
\beas
\frac{i_{r_{l_{2}}}}{i_{r_{l_{1}}}} 
&\leq 
 e_{S}(r_{l_{2}},r_{l_{1}})
\leq 
\tau_{S}(k_{r_{l_{2}}},k_{r_{l_{1}}})
=
\min_{j\in I_{r_{l_{1}}}}\frac{e(k_{r_{l_{2}}},j)}{e(k_{r_{l_{1}}},j)}\\
&\leq 
\frac{e(k_{r_{l_{2}}},l_{1})}{e(k_{r_{l_{1}}},l_{1})}
\leq 
\frac{e(k_{r_{l_{2}}},l_{2})e(l_{2},l_{1})}{e(k_{r_{l_{1}}},l_{1})}
\leq 
\frac{e(k_{r_{l_{2}}},l_{2})}{e(k_{r_{l_{1}}},l_{1})}
\eeas
since $e(l_{2},l_{1})\leq 1$. Thus
\beas
i_{r_{l_{1}}}\tau_{S}(l_{1},k_{r_{1_{1}}})=\frac{i_{r_{l_{1}}}}{e(k_{r_{l_{1}}},l_{1})}\leq 
\frac{i_{r_{l_{2}}}}{e(k_{r_{l_{2}}},l_{2})}=i_{r_{l_{2}}}\tau_{S}(l_{2},k_{r_{1_{2}}}),
\eeas
which is the desired inequality. A similar argument works if $\x_{l_{1}}\in J_{a_{1_{1}}}^{*}$ and $\x_{l_{2}}\in J_{a_{l_{2}}}^{*}$.

Now suppose $\x_{l_{1}}\in I_{r_{1_{1}}}^{*}\cap J_{a_{l_{1}}}$ and $\x_{l_{2}}\in I_{r_{l_{2}}}\cap J_{a_{l_{2}}}^{*}$. This time, after taking logarithm to base $2$ the scales of the corresponding dilations are $i_{r_{l_{1}}}\tau_{S}(l_{1},k_{r_{1_{1}}})$ and $j_{a_{l_{2}}}\tau_{T}(l_{2},j_{a_{l_{2}}})$. Just as before, it follows that
\beas
i_{r_{l_{1}}}\tau_{S}(l_{1},k_{r_{1_{1}}})&\leq i_{r_{l_{2}}}\tau_{S}(l_{2},k_{r_{1_{2}}}).
\eeas
But now since $l_{2}\in J_{a_{l_{2}}}^{*}$, we know that the $i$-scale at this variable is less than or equal to the $j$-scale at this variable. Thus 
\beas
i_{r_{l_{2}}}\tau_{S}(l_{2},k_{r_{1_{2}}})\leq j_{a_{l_{2}}}\tau_{T}(l_{2},l_{a_{l_{2}}}).
\eeas
Putting the last two inequalities together gives the required estimate. A similar argument works if $\x_{l_{1}}\in I_{r_{1_{1}}}\cap J_{a_{l_{1}}}^{*}$ and $\x_{l_{2}}\in I_{r_{l_{2}}}^{*}\cap J_{a_{l_{2}}}$, and this completes the proof.
\end{proof}

\subsection{{The class $\PP_{0}(\EEE_{S,T,W})$}}\label{TheClass}

Given the decomposition $\R^{N}= \bigoplus_{j=1}^{\alpha}\R^{I_{r_{j}}^{*}}\,\oplus\,\bigoplus_{j=1}^{\beta}\R^{j_{a_{j}}^{*}}$ and the matrix $\EEE_{S,T,W}$, we have the space of distributions $\PP_{0}(\EEE_{S,T,W})$ and the corresponding space of multipliers $\MM_{\infty}(\EEE_{S,T,W})$. For notational convenience we use the symbol $e_W$ to denote the entries of $\EEE_{S,T,W}$. Recall that their indices  vary in the disjoint union $S^*\sqcup T^*$ of the sets $S^*\subset\{1,\dots,s\}$ and $T^*\subset\{1,\dots,q\}$ of fixed indices relative to $W$.

\begin{lemma}\label{Lem10.10gg}
$\MM_{\infty}(\EEE_{S,T,W})\subset\MM_{\infty}(\EEE)$ and consequently $\PP_{0}(\EEE_{S,T,W})\subset\PP_{0}(\EEE)$.
\end{lemma}

\begin{proof}
The dual norms for the space $\MM_{\infty}(\EEE_{S,T,W})$ are given by
\beas \index{N6sharp@$\widehat N_{r,W}$}
\widehat N_{r,W}(\xib) &= \sum_{p\in S^{*}}\widehat n_{S,p}(\xib_{I_{p}^{*}})^{1/e_{W}(p,r)}
+
\sum_{b\in T^{*}}\widehat n_{T,b}(\xib_{J_{b}^{*}})^{1/e_{W}(b,r)},\qquad r\in S^{*};\\
\widehat N_{a,W}(\xib)&= \sum_{p\in S^{*}}\widehat n_{S,p}(\xib_{I_{p}^{*}})^{1/e_{W}(I_{p},J_{a})}
+
\sum_{b\in T^{*}}\widehat n_{T,b}}(\xib_{J_{b}^{*})^{1/e_{W}(b,a)},\qquad a\in T^{*}.
\eeas
By Corollary \ref{sharp-inclusion}, it suffices to show, for each $(j,k)$, the entries $e^\sharp_W(j,k)$ of the $n\times n$ matrix $\EEE^\sharp_{S,T,W}$ satisfy the inequality
$$
e^\sharp_W(j,k)\le e(j,k).
$$
By \eqref{esharp},  $e^\sharp_W(j,k)=\frac{\alpha_k}{\alpha_j}e_W(\mu,\nu)$, where
\begin{enumerate}
\item[(i)]
$\mu,\nu\in S^*\sqcup T^*$ are such that $j\in I^*_\mu$ (if $\mu\in S^*$) or $j\in J^*_\mu$ (if $\mu\in T^*$) and $k\in I^*_\nu$ or $J^*_\nu$ accordingly;
\item[(ii)] $\alpha_k$ is the exponent of the norm of $\xib_k$ in
$$
\widehat n_{S,p}(\xib_{I_{p}^{*}})=\sum_{k\in I_p^*}n_{k}(\xib_{k})^{e(k_{p},k)}\ (\text{if }k\in I_p^*),\qquad \widehat n_{T,b}(\xib_{J_{b}^{*}})=\sum_{k\in J_b^*}n_{k}(\xib_{k})^{e(\ell_{b},k)}\ (\text{if }k\in J_b^*).
$$
\end{enumerate}
We then have the inequalities
$$
e^\sharp_W(j,k)= \begin{cases} \displaystyle\frac{e(k_{\nu},k)}{e(k_{\mu},j)} e_W(\mu,\nu)&\text{ if }\mu,\nu\in S^*\,,\,j\in I_\mu^*\,,\,k\in I_\nu^*\\ \displaystyle\frac{e(\ell_{\nu},k)}{e(\ell_{\mu},j)}e_W(\mu,\nu)&\text{ if }\mu,\nu\in T^*\,,\,j\in J_\mu^*\,,\,k\in J_\nu^*\\ \displaystyle\frac{e(\ell_{\nu},k)}{e(k_{\mu},j)}e_W(\mu,\nu)&\text{ if }\mu\in S^*\,,\,\nu\in T^*\,,\,j\in I_\mu^*\,,\,k\in J_\nu^*\\ 
\displaystyle\frac{e(k_{\nu},k)}{e(\ell_{\mu},j)}e_W(\mu,\nu)&\text{ if }\mu\in T^*\,,\,\nu\in S^*\,,\,j\in J_\mu^*\,,\,k\in I_\nu^*.
\end{cases}
$$

If $\mu,\nu\in S^*$ and  $j\in I_\mu^*\,,\,k\in I_\nu^*$, using the inequality $e_W(\mu,\nu)\le b_W(\mu,\nu)$, we have
$$
e^\sharp_W(j,k)\le \frac{e(k_{\nu},k)}{e(k_{\mu},j)}b_W(\mu,\nu)= \frac{e(k_{\nu},k)}{e(k_{\mu},j)}\min_{l\in I_\nu}\frac{e(k_{\mu},l)}{e(k_{\nu},l)}\le\frac{e(k_{\mu},k)}{e(k_{\mu},j)} \le e(j,k),
$$
and similarly if $\mu,\nu\in T^*$.

Assume now that $\mu\in S^*$, $\nu\in T^*$, $j\in I_\mu^*\,,\,k\in J_\nu^*$. In this case the inequality $e_W(\mu,\nu)\le b_W(\mu,\nu)$ may not be sufficient (e.g., we may have $b_W(\mu,\nu)=\infty$). However, we can use the basic inequality
$$
e_W(\mu,\nu)\le e_W(\mu,\nu')e_W(\nu',\nu)\le b_W(\mu,\nu')b_W(\nu',\nu)\le \Lambda_W(I_{\mu},J_{\nu'})\tau_T(\ell_\nu,\ell_{\nu'}),
$$
which holds for all $\nu'\in T^*$. Choosing $\nu'$ so that $j\in J_{\nu'}$, we have, with $m=w(\mu,\nu')$,
\bea
e_W^\sharp(j,k)&\le\frac{e(\ell_{\nu},k)}{e(k_{\mu},j)}  \Lambda_W(I_{\mu},J_{\nu'})\tau_T(\ell_\nu,\ell_{\nu'})\\
&= \frac{e(\ell_{\nu},k)}{e(k_{\mu},j)} \frac{e(k_{\mu},v_m)}{e(\ell_{\nu'},v_m)}\min_{i\in J_{\nu}}\frac{e(\ell_{\nu'},i)}{e(\ell_{\nu},i)}\\
&\le \frac{e(\ell_{\nu},k)}{e(k_{\mu},j)} \frac{e(k_{\mu},v_m)}{e(\ell_{\nu'},v_m)}\frac{e(\ell_{\nu'},k)}{e(\ell_{\nu},k)}\\
&= \frac{e(k_{\mu},v_m)}{e(\ell_{\nu'},v_m)}\frac{e(\ell_{\nu'},j)}{e(k_\mu,j)}\frac{e(\ell_{\nu'},k)}{e(\ell_{\nu'},j)}.
\eea
 Since $j\in I_{\mu}^*\cap J_{\nu'}$, it follows from  \eqref{IJstar} that  $k=v_{m'}$ with $m'> m$ in the ordering \eqref{10.15tt} on $ I_{\mu}\cap J^{\nu'}$, i.e.,
 $$
\frac{e(k_{\mu},v_m)}{e(\ell_{\nu.},v_m)}= \frac{\tau_{T}(v_m,\ell_{\nu'})}{\tau_{S}(v_m,k_{\mu})}
\leq \frac{\tau_{T}(j,\ell_{\nu'})}{\tau_{S}(j,k_{\mu})}=\frac{e(k_{\mu},j)}{e(\ell_{\nu'},j)}.
$$
Hence
$$
e_W^\sharp(j,k)\le \frac{e(\ell_{\nu'},k)}{e(\ell_{\nu'},j)}\le e(j,k).
$$
  
The last case, with  $\mu\in T^*$, $\nu\in S^*$, $j\in J_\mu^*\,,\,k\in I_\nu^*$, is treated similarly.
\end{proof}

\subsection{Estimates for $[\varphi]_{I}*[\psi]_{J}$}\label{Estimates} 

Fix marked partitions $S=\big((I_{1},k_{1});\ldots;(I_{s},k_{s})\big)$ and $T=\big((J_{1}, \ell_{1});\ldots;(J_{q}, \ell_{q})\big)$, and let $W\in \WW(S,T)$. Recall that
\beas
I_{r}^{*}&=\bigcup_{\substack{a=1\\I_{r}\cap J_{a}\neq\emptyset}}^{q}\Big\{j\in I_{r}\cap J_{a}:\frac{i_{r}}{e(k_{r},j)}>\frac{j_{a}}{e(\ell_{a},j)}\ \forall \big((i_{1},\ldots,i_{s}),(j_{1}, \ldots,j_{q})\big)\in\Gamma_{\EEE_{S,T,W}}\Big\},\\
J_{a}^{*}&=\bigcup_{\substack{r=1\\I_{r}\cap J_{a}\neq\emptyset}}^{s}\Big\{j\in I_{r}\cap J_{a}:\frac{i_{r}}{e(k_{r},j)}\leq\frac{j_{a}}{e(\ell_{a},j)}\ \forall \big((i_{1},\ldots,i_{s}),(j_{1}, \ldots,j_{q})\big)\in\Gamma_{\EEE_{S,T,W}}\Big\}.
\eeas
Then 
\beas
S^{*}&=\big\{r\in \{1, \ldots, s\}:I_{r}^{*}\neq \emptyset\big\}\,=\,\{r_{1}, \ldots, r_{\alpha}\}\\
T^{*}&=\big\{a\in \{1, \ldots, q\}:J_{a}^{*}\neq \emptyset\big\}=\{a_{1}, \ldots, a_{\beta}\}
\eeas 
are the fixed indices. Write $\R^{N}=\bigoplus_{j=1}^{\alpha+\beta}V_{j}$ as in subsection \ref{Finer}. Thus if $\x\in \R^{N}$, $\x=(\overline \x_{1}, \ldots, \overline \x_{\alpha+\beta})$ where each $\overline \x_{m}$  is either equal to $\x_{I_{r}^{*}}= \big\{\x_{j}:j\in I_{r}^{*}\big\}$ for some $r\in S^{*}$ or is equal to $\x_{J_{a}^{*}}=\big\{\x_{j}:j\in J_{a}^{*}\big\}$ for some $a\in T^{*}$. In equation (\ref{10.36ss}) we introduced a family of dilations on the space $V_{m}$ by setting
\beas
\lambda\circ \overline \x_{m}&=
\begin{cases}
\Big\{\lambda^{1/e(k_{r},j)}\cdot \x_{j}:j\in I_{r}^{*}\Big\}&\text{if $\overline \x_{m}=\x_{I_{r}^{*}}$},\\\\
\Big\{\lambda^{1/e(\ell_{a},j)}\cdot \x_{j}:j\in J_{a}^{*}\Big\}&\text{if $\overline \x_{m}=\x_{J_{a}^{*}}$}.
\end{cases} 
\eeas

Let $
I=(i_{1}, \ldots, i_{s})$ and $J=(j_{1}, \ldots, j_{q})$ with $(I,J)\in \Gamma_{W}(S,T)$. If $\varphi, \psi\in \CC^{\infty}_{0}(\R^{N})$ have support in the unit ball, it follows from from part (\ref{Lem10.13bb1}) of Lemma \ref{Lem10.13bb} that  there is a function $\theta\in \CC^{\infty}_{0}(\R^{N})$, normalized with respect to the functions $\varphi$ and $\psi$, so that 
\bes
[\varphi]_{I}*[\psi]_{J}(\x) = 2^{-\sum_{j=1}^{\alpha+\beta}m_{j}Q_{j}^{*}}\,\,\theta(2^{-m_{1}}\,\hat\cdot\,\overline \x_{1}, \ldots, 2^{-m_{\alpha+\beta}}\,\hat\cdot\,\overline \x_{\alpha+\beta}),
\ees
where $Q_{j}^{*}$ is the homogeneous dimension of $V_{j}$, and 
\beas
m_{j}&= i_{r}&&&&\text{if $\overline \x_{m}=\x_{I_{r}^{*}}$}\\
m_{j}&= j_{a}&&&&\text{if $\overline \x_{m}= \x_{J_{a}^{*}}$}.
\eeas
Thus the indices $\{m_{1}, \ldots, m_{\alpha+\beta}\}$ are some permutation of the indices $\{r_{1}, \ldots,r_{\alpha}, a_{1}, \ldots, a_{\beta}\}$. Note that if  $S^{**}=\{r_{\alpha+1}, \ldots, r_{s}\}\subset\{1, \ldots, s\}$ and $T^{**}=\{a_{\beta+1}, \ldots, a_{q}\}\subset\{1, \ldots,q\}$ are the free indices, they do not appear among the indices $\{m_{1}, \ldots, m_{\alpha+\beta}\}$. 

Now suppose that $\varphi$ has cancellation in the variables $\{\x_{k_{1}}, \ldots, \x_{k_{s}}\}$ and that $\psi$ has cancellation in the variables $\{\x_{\ell_{1}}, \ldots, \x_{\ell_{q}}\}$. In order to obtain additional estimates of the function $\theta$, we will use part (\ref{Lem10.13bb2}) of Lemma \ref{Lem10.13bb}. Using the notation of Lemma \ref{Lem10.13bb}, we have 
\beas
C&=\{k_{1}, \ldots, k_{s}\}\subset\{1, \ldots, n\},\\
D&=\{\ell_{1}, \ldots, \ell_{q}\}\subset\{1, \ldots, n\}.
\eeas 
The dilations are given by
\beas
\lambda_{j}&= 2^{i_{r}/e(k_{r},j)}&&&&\text{if $j\in I_{r}$},\\
\mu_{k}&= 2^{j_{a}/e(\ell_{a},j)}&&&&\text{if $j\in J_{a}$}.
\eeas
The set $C_{1}=\big\{k_{r}:\lambda_{k_{r}}<\mu_{k_{r}}\big\}$ is now the union of two disjoint subsets:
\beas
C_{1}'&=\Big\{k_{r}:I_{r}^{*}=\emptyset\Big\} &&= \Big\{k_{r}:\text{$r$ is free}\Big\},\\
C_{1}''&= \Big\{k_{r}:\text{$I_{r}^{*}\neq\emptyset$ and $k_{r}\notin I_{r}^{*}$}\Big\}&&=\Big\{k_{r}:\text{$r$ is fixed and $k_{r}\notin I_{r}^{*}$}\Big\}.
\eeas
The same is true of the set $D_{1}=\big\{\ell_{a}:\mu_{\ell_{a}}\leq\lambda_{\ell_{a}}\big\}$:
\beas
D_{1}'&=\Big\{\ell_{a}:J_{a}^{*}=\emptyset\Big\} &&= \Big\{\ell_{a}:\text{$a$ is free}\Big\},\\
D_{1}''&= \Big\{\ell_{a}:\text{$J_{a}^{*}\neq\emptyset$ and $\ell_{a}\notin J_{a}^{*}$}\Big\}&&=\Big\{\ell_{a}:\text{$a$ is fixed and $\ell_{a}\notin J_{a}^{*}$}\Big\}.
\eeas
The sets $C_{2}$ and $D_{2}$ are given by
\beas
C_{2}&= \Big\{k_{r}:\text{$r$ is fixed and $k_{r}\in I_{r}^{*}$}\Big\},\\
D_{2}&= \Big\{\ell_{a}:\text{$a$ is fixed and $\ell_{a}\in J_{a}^{*}$}\Big\}.
\eeas
Define maps
\beas
\sigma_{1}&:\big\{1, \ldots, s\big\}\to \big\{1,\ldots,q\big\}\,\, \text{such that $k_{r}\in J_{\sigma_{1}(r)}$},\\
\sigma_{2}&:\big\{1, \ldots, q\big\}\to \big\{1,\ldots,s\big\}\,\, \text{such that $\ell_{a}\in I_{\sigma_{2}(a)}$},
\eeas
and note that
\beas
\text{$1\leq r \leq s$ and $r$ is free} \quad \Longrightarrow \quad \text{$\sigma_{1}(r)$ is fixed},\\
\text{$1\leq a \leq q$ and $a$ is free} \quad \Longrightarrow \quad \text{$\sigma_{2}(a)$ is fixed}.
\eeas
Applying part (\ref{Lem10.13bb2}) of Lemma \ref{Lem10.13bb}  we have the following result.

\begin{lemma}\label{Lem10.12gg}
Let $(I,J)\in \Gamma(\EEE_{S,T,W})$,  let $\varphi, \psi$ be normalized bump functions, and suppose that $\varphi$ has cancellation in the variables $\{\x_{k_{1}}, \ldots, \x_{k_{s}}\}$ and that $\psi$ has cancellation in the variables $\{\x_{\ell_{1}}, \ldots  \x_{\ell_{q}}\}$. Write $[\varphi]_{I}*[\psi]_{J}=[\theta]_{M}$ as in Lemma \ref{Lem10.2}. There exists $\epsilon >0$ depending only on the group structure so that $\theta$ can be written as a finite sum of terms of the form

\beas
\prod_{k_{r}\in C_{1}'}2^{-\epsilon G_{1}(i_{r})}\!\!\!
\prod_{k_{r}\in C_{1}''}2^{-\epsilon G_{1}(i_{r})}\!\!\!
\prod_{\ell_{a}\in D_{1}'}2^{-\epsilon H_{1}(j_{a})} \prod_{\ell_{a}\in D_{1}''}2^{-\epsilon H_{1}(a)}\!\!\!
\prod_{k_{r}\in C_{2}'}2^{-\epsilon G_{2}(r)}\!\!\!
\prod_{\ell_{a}\in D_{2}'}2^{-\epsilon H_{2}(j_{a})}\,\,
\theta_{C_{2}'',D_{2}''}
\eeas
where the factors are defined as follows.
\begin{enumerate}[{\rm a)}]

\item $C_{2}$and $D_{2}$ are each partitioned into two disjoint subsets:
\beas
C_{2}&= C_{2}'\cup C_{2}'',& C_{2}'\cap C_{2}''&=\emptyset,\\
D_{2}&= D_{2}'\cup D_{2}'',& D_{2}'\cap D_{2}''&=\emptyset.
\eeas

\smallskip

\item The functions $\big\{\theta_{C_{2}'',D_{2}''}\big\}\subset \CC^{\infty}_{0}(\R^{n})$ are normalized relative to $\varphi$ and $\psi$.

\smallskip

\item The function $\theta_{C_{2}'',D_{2}''}$ has cancellation in each variable $\overline \x_{m} = \x_{I_{r}^{*}}$ if $r\in C_{2}''$.

\smallskip

\item The function $\theta_{C_{2}'',D_{2}''}$ has cancellation in each variable $\overline \x_{m}= \x_{J_{a}^{*}}$ if $a\in D_{2}''$.

\smallskip

\item The expressions $G_{j}(r)$ and $H_{j}(a)$ are given by: 

\beas
G_{1}(i_{r}) &= \min\Big\{\frac{j_{\sigma_{1}(r)}}{e(\ell_{\sigma_{1}(r)},k_{r})}-i_{r},\,\, i_{r+1}-i_{r}\Big\};\\
H_{1}(j_{a}) &= \min\Big\{\frac{k_{\sigma_{2}(a)}}{e(k_{\sigma_{2}(a)},\ell_{a})}-j_{a}, \,\,j_{a+1}-j_{a}\Big\};\\
G_{2}(i_{r}) &= i_{r+1}-i_{r};\\
H_{2}(j_{a}) &= j_{a+1}-j_{a}.\\
\eeas
\end{enumerate}
\end{lemma}

The meaning of this result is as follows. The original functions $\varphi^{I}$ or $\psi^{J}$ have cancellation in variables whose indices belong to $C_{0}\cup C_{1}\cup C_{2}$ and to $D_{0}\cup D_{1}\cup D_{2}$. Then with $\varphi^{I}*\psi^{J}=[\theta^{I,J}]_{M}$, the function $\theta^{I,J}$ can be written as a sum of terms in which a normalized bump function $\Theta^{I,J}_{C_{1}',D_{1}/}$ are multiplied by certain small factors. 

\begin{enumerate}[A)]

\item If $r\in \{1, \ldots, s\}$ or $a\in \{1, \ldots, q\}$ is a free variable then $r\in C_{1}'$ or $a\in D_{1}'$. If $r$ is free, the scale of $[\varphi^{I}]_{I}$ in the coordinate $\x_{k_{r}}$ must be smaller than the scale of $[\psi^{J}]_{J}$ in $\x_{k_{r}}$, which is $j_{\sigma_{1}(r)}/e(\ell_{\sigma_{1}(r)},k_{r})$. Thus there is a gain, either of order $-\epsilon[j_{\sigma_{1}(r)}/e(\ell_{\sigma_{1}(r)},k_{r})-i_{r}]$ or of order $-\epsilon[i_{r+1}-i_{r}]$, due either to integration by parts or to replacing the derivative $\partial_{\x_{k_{r}}}$ with the corresponding invariant differential operator. This leads to the term $2^{-\epsilon G_{1}(r)}$, and there is a similar explanation for the term $2^{-\epsilon H_{1}(a)}$. In particular, there is an gain associated with all free indices.

\smallskip

\item If $r\in \{1, \ldots, s\}$ or $a\in \{1, \ldots, q\}$ is a fixed variable, then $r\in C_{1}''\cup C_{2}$ or $a\in D_{1}''\cup D_{2}$. 
\begin{enumerate}[a)]

\smallskip

\item For indices $r\in C_{1}''$, $k_{r}\notin I_{r}^{*}$ and so the scale of $[\varphi^{I}]_{I}$ in the coordinate $\x_{k_{r}}$ must be smaller than the scale of $[\psi^{J}]_{J}$ in $\x_{k_{r}}$, which is $j_{\sigma_{1}(r)}/e(\ell_{\sigma_{1}(r)},k_{r})$. Thus there is a gain, either of order $-\epsilon[j_{\sigma_{1}(r)}/e(\ell_{\sigma_{1}(r)},k_{r})-i_{r}]$ or of order $-\epsilon[i_{r+1}-i_{r}]$, due either to integration by parts or to replacing the derivative $\partial_{\x_{k_{r}}}$ with the corresponding invariant differential operator. This leads to the term $2^{-\epsilon G_{1}(r)}$ and there is a similar explanation for the term $2^{-\epsilon H_{1}(a)}$.

\smallskip

\item For an index $r\in C_{2}$, it may happen that the cancelation of $[\varphi^{I}]_{I}$ in the variable $\x_{k_{r}}$ leads to cancellation in this variable of the convolution $[\varphi^{I}]_{I}*[\psi^{J}]_{J}$. We let $C_{2}''$ be the elements of $C_{2}$ where this happens and $C_{2}'$ be the complementary set. Then $\theta^{I,J}_{C_{2}'',D_{2}''}$ does have cancellation in the variables $\x_{I_{r}^{*}}$ for $r\in C_{2}''$. In the same way,$\theta^{I,J}_{C_{2}'',D_{2}''}$ does have cancellation in the variables $\x_{J_{a}}^{*}$ for $a\in D_{2}''$.

\smallskip

\item For indices $r\in C_{2}'=C_{2}\setminus C_{2}''$ or $a\in D_{2}'=D_{2}\setminus D_{2}''$, since cancellation does not persist in the convolution, there is an exponential gain of order $-\epsilon[i_{r+1}-i_{r}|$ or $-\epsilon[j_{a+1}-j_{a}|$ due replacing an ordinary derivative by a left- or right-invariant operator. This leads to the terms $2^{-\epsilon G_{2}(r)}$ and $2^{-\epsilon H_{2}(a)}$.

\end{enumerate}
\end{enumerate}

\begin{remarks} Let $(I,J)\in \Gamma(\EEE_{S,T,W})$.

\begin{enumerate}[{\rm 1)}]

\item If $k_{r}\in C_{1}'$ then $i_{r}\leq j_{\sigma_{1}(r)}e(\ell_{\sigma_{1}(r)},k_{r})^{-1}$ and $\sigma_{1}(r)$ is a fixed index. Similarly, if $\ell_{a}\in D_{1}'$ then $j_{a}\leq k_{\sigma_{2}(a)}e(k_{\sigma_{2}(a)},\ell_{a})^{-1}$ and $\sigma_{2}(a)$ is a fixed index.

\smallskip

\item If $k_{r}\in C_{1}''$ then $\sigma_{1}(r)$ is also fixed, and 
\beas
\frac{i_{r}}{j_{\sigma_{1}(r)}}&\geq \tau_{W}(J_{\sigma_{1}(r)},I_{r})^{-1}
\geq \Lambda_{W}(J_{\sigma_{1}(r)},I_{r})^{-1}
=
\frac{e(k_{r},v_{m_{r,\sigma_{1}(r)}})}{e(\ell_{\sigma_{1}(r)},v_{m_{r,\sigma_{1}(r)}})}\geq e(\ell_{\sigma_{1}(r)},k_{r})^{-1}
\eeas
where the second inequality follows from the definition of $\tau_{W}$ since $J_{\sigma(r)}^{*}\cap I_{r}\neq \emptyset$, and the last inequality follows from the basic hypothesis. Thus $i_{r}-e(\ell_{\sigma_{1}(r)},k_{r})^{-1}j_{\sigma_{1}(r)}\leq 0$. Similarly, if $j_{a}\in D_{1}''$, it follows that $j_{a}-e(k_{\sigma_{2}(a)},\ell_{a})^{-1}i_{\sigma_{2}(a)}\leq 0$.
\end{enumerate}
\end{remarks}

\subsection{Proof of Theorem \ref{Thm10.1}}\label{FixedandFree}

We have seen on page \pageref{second lemma} that Theorem \ref{Thm10.1} follows from Lemma \ref{Lem10.2} which deals with sums of convolutions $[\varphi^{I}]_{I}*[\psi^{J}]_{J}$ where $(I,J)\in \Gamma_{\Z}(\EEE_S)\times \Gamma_{\Z}(\EEE_T)$. In section \ref{Decomp1} we partitioned $\Gamma_{\Z}(\EEE_S)\times \Gamma_{\Z}(\EEE_T)$ into disjoint sets $\Gamma_{W}(S,T)$ for $W\in \WW(S,T)$. Thus the proof of Lemma \ref{Lem10.2}, and hence Theorem \ref{Thm10.1}, follows from the following result.

\begin{lemma}\label{Lem10.5}
If $F\subset\Gamma_{W}(S,T)$ is any finite set, then 
\bes
\HH_{F}=\sum_{(S,T)\in F}[\varphi^{I}]_{I}*[\psi^{J}]_{J}\in \PP_{0}(\EEE_{S,T,W})
\ees
with constants independent of the finite set $F$. Moreover, $\lim_{F\nearrow\Gamma_{W}(S,T)}\HH_{F}=\HH_{S,T,W}\in \PP_{0}(\EEE_{S,T,W})$ with convergence in the sense of distributions. In particular, it follows from Lemma \ref{Lem10.10gg} that $\HH_{S,T,W}\in \PP_{0}(\EEE)$.
\end{lemma}

\begin{proof}
Recall from page \pageref{double sum} that
\beas
\sum_{(I,J)\in \Gamma_{W}(S,T)}[\varphi^{I}]_{I}*[\psi^{J}]_{J}= \sum_{M\in \mathfrak M(S,T,W)}\Big[\sum_{(I,J)\in F(M)}\theta^{I,J}\Big]_{M}= \sum_{M\in \mathfrak M(S,T,W)}[\Theta^{M}]_{M}.
\eeas
We first show that for each $M\in \mathfrak M(S,T,W)$, the sum $\sum_{(I,J)\in F(M)}\theta^{I,J}$ converges to a normalized bump function. Given $M=(m_{1}, \ldots, m_{\alpha+\beta})$, this determines the indices $i_{r}$ and $j_{a}$ for which $r$ and $a$ are fixed; \textit{i.e.}  $r\in\{{r_{1}}, \ldots, {r_{\alpha}}\}$ and $a\in \{{a_{1}}, \ldots, {a_{\beta}}\}$. Thus $i_{r}$ and $j_{a}$ belong to the sets $C_{1}'', D_{1}'', C_{2}, D_{2}$. We must sum over the remaining indices corresponding to free variables $r$ and $a$ for which $k_{r}\in C_{1}'$ and $\ell_{a}\in D_{1}'$.  We write the free variables as $\{r_{\alpha+1}, \ldots, r_{s}\}$ and $\{a_{\beta+1}, \ldots, a_{q}\}$. Using Lemma \ref{Lem10.12gg} it follows that $\sum_{(I,J)\in F(M)}\theta^{I,J}$ is a finite sum of terms  (indexed by subsets $C_{2}''\subset C_{2}$ and $D_{2}''\subset D_{2}$) of the form
\beas
\prod_{k_{r}\in C_{1}''}2^{-\epsilon G_{1}(r)}
\prod_{\ell_{a}\in D_{1}''}2^{-\epsilon H_{1}(a)}
&\prod_{k_{r}\in C_{2}'}2^{-\epsilon G_{2}(r)}
\prod_{\ell_{a}\in D_{2}'}2^{-\epsilon H_{2}(a)}\\
&\Big[\sum_{\substack{
(i_{r_{\alpha+1}}, \ldots, i_{r_{s}})\in C_{1}'\\
(j_{a_{\beta+1}}, \ldots, j_{a_{q}})\in D_{1}'}}
\prod_{v=\alpha+1}^{s}2^{-\epsilon G_{1}(r_{v})}\prod_{v=\beta+1}^{q}2^{-\epsilon H_{1}(a_{v})} \,\,\theta_{C_{2}'',D_{2}''}\Big]
\eeas
where each function $\theta_{C_{2}'',D_{2}''}$ has cancellation in the variables $\overline \x_{m}=\x_{I_{r}^{*}}$ if $r\in C_{2}''$ and $\overline \x_{m}=\x_{J_{a}^{*}}$ if $a\in D_{2}''$.

We have
\beas
\sum_{\substack{
(i_{r_{\alpha+1}}, \ldots, i_{r_{s}})\in C_{1}'\\
(j_{a_{\beta+1}}, \ldots, j_{a_{q}})\in D_{1}'}}
\prod_{v=\alpha+1}^{s}&2^{-\epsilon G_{1}(r_{v})}\prod_{v=\beta+1}^{q}2^{-\epsilon H_{1}(a_{v})} \,\,\theta_{C_{2}'',D_{2}''}\\
&\leq
\prod_{v=\alpha+1}^{s}\Big[\sum_{i_{r_{v}}\leq j_{\sigma_{1}(r_{v})}e(\ell_{\sigma_{1}(r_{v})},k_{r_{v}})}2^{-\epsilon[j_{\sigma_{1}(r_{v})}e(\ell_{\sigma_{1}(r_{v})},k_{r_{v}})-i_{r_{v}}]}\Big]\\
&\qquad \quad
\prod_{v=\beta+1}^{q}\Bigg[
\sum_{j_{a_{v}}\leq i_{\sigma_{2}(a_{v})}e(k_{\sigma_{2}(a_{v})},\ell_{a_{v}})}2^{-\epsilon[i_{\sigma_{2}(a_{v})}e(k_{\sigma_{2}(a_{v})},\ell_{a_{v}})-j_{a_{v}}]}\Bigg]<\infty.
\eeas
It follows that $\sum_{(I,J)\in F(M)}\theta^{I,J}$ is a finite sum of terms of the form
\beas
\prod_{k_{r}\in C_{1}''}2^{-\epsilon G_{1}(r)}
\prod_{\ell_{a}\in D_{1}''}2^{-\epsilon H_{1}(a)}
&\prod_{k_{r}\in C_{2}'}2^{-\epsilon G_{2}(r)}
\prod_{\ell_{a}\in D_{2}'}2^{-\epsilon H_{2}(a)}\,\,\theta^{I,J}_{M,C_{2}'',D_{2}''}
\eeas
where $\theta^{I,J}_{M,C_{2}'',D_{2}''}$ is a normalized bump function. It remains to show that 
\beas
\sum_{M\in \mathfrak M(S,T,W) }\prod_{k_{r}\in C_{1}''}2^{-\epsilon G_{1}(r)}
\prod_{\ell_{a}\in D_{1}''}2^{-\epsilon H_{1}(a)}
&\prod_{k_{r}\in C_{2}'}2^{-\epsilon G_{2}(r)}
\prod_{\ell_{a}\in D_{2}'}2^{-\epsilon H_{2}(a)}\,\,\theta^{I,J}_{M,C_{2}'',D_{2}''}
\eeas
converges in the sense of distributions to an element of $\PP_{0}(\EEE_{S,T,W})$. However, this follows from Theorem \ref{Thm3.7} once we observe that each term in the sum is a function with weak cancellation with parameter $\epsilon>0$ relative to the multi-index $M$. This completes the proof. 
\end{proof}

\section{Convolution of  Calder\'on-Zygmund kernels}\label{CZKernels}

In this section we study the convolution of two or more compactly supported \CZ kernels with different homogeneities, and show that the result belongs to an appropriate class $\PP_{0}(\EEE)$. 

\subsection{\CZ kernels}

Recalling the notion of Calder\'on-Zygmund kernel given in Definition \ref{CZa}, we define  the corresponding class of kernels with rapid decay at infinity.

\begin{definition}\label{CZa0}\index{C2Za0@$\CC\ZZ_{\a,0}$}
$\CC\ZZ_{\a,0}$ is the space of tempered distributions $\KK$ on $\R^{N}$ such that
\begin{enumerate}[{\rm(a)}]
\smallskip

\item away from the origin, $\KK$ is given by integration against a smooth function $K$ satisfying the differential inequalities
\beas
\big\vert\partial^{\gammab}K(\x)\big\vert \leq C_{\gammab, M}N_{\a}(\x)^{-Q_{\a}-\[\gammab\]_{\a}}\big(1+N_{\a}(\x)\big)^{-M},
\eeas
for every $\gammab$ and $M$;

\item there is a constant $C>0$ so that for any normalized bump function $\psi\in \CC^{\infty}_{0}(\R^{N})$ with support in the unit ball and for any $R>0$, $\big\vert\big\langle\KK,\psi_{R}\big\rangle\big\vert\leq C$, where $\psi_{R}(\x) = \psi(R\cdot_{\a}\x)$.
\end{enumerate}
\end{definition}

It  follows from the results of Section \ref{Duality} that, for distributions $\KK\in \CC\ZZ_{\a,0}$, the corresponding multipliers  $m=\widehat \KK$ are characterized by the differential inequalities 
$$
\big\vert\partial^{\gammab}m(\xib)\big\vert\leq C_{\gammab}\big[1+N_{\a}(\xib)\big]^{-\[\gammab\]_{\a}}.
$$
 These differential inequalities characterize the multipliers belonging to the class $\MM_{\infty}(\EEE_{\a})$ where $\EEE_{\a}$ is the matrix in \eqref{Ea}. It follows that, if $\KK\in \CC\ZZ_{\a,0}$, then $\KK\in \PP_{0}(\EEE_{\a})$. Moreover, the matrix $\EEE_{\a}$ is doubly monotone if and only if $a_{1}\geq a_{2}\geq \cdots \geq a_{n}$. Hence this is the compatibility condition with an underlying nilpotent Lie group structure in the sense of Section \ref{Compatibility}.

Finally, it follows from our decomposition result that if $\KK\in \PP_{0}(\EEE_{\a})$, there exists $\Phi\in \SS(\R^{N})$ and a uniformly bounded family $\{\varphi^{j}\}\subset\CC^{\infty}_{0}(\R^{N})$ such that $\int_{\R^{N}}\varphi^{j}(\x)\,d\x=0$ for all $j$, and, in the sense of distributions,
\bea\label{11.1iou}
\KK(\x) = \sum_{j=-\infty}^{0}[\varphi^{j}]_{j}(\x) +\Phi(\x) = \sum_{j=-\infty}^{0}2^{-jQ_{a}}\,\varphi^{j}(2^{-ja_{1}}\cdot \x_{1}, \ldots, 2^{-ja_{n}}\cdot \x_{n})+\Phi(\x).
\eea

\subsection{A general convolution theorem}

Let $\KK_{1}, \ldots, \KK_{p}$ be \CZ kernels on a homogeneous nilpotent Lie group $G\cong\R^{N}=\R^{C_{1}}\oplus\cdots \oplus\R^{C_{n}}$. Suppose that the kernel $\KK_{\ell}$ is in the class $\CC\ZZ_{\a_\ell,0}$ associated with the dilations $\lambda\,\cdot_{\a_\ell}\x = \big(\lambda^{a^{\ell}_{1}}\cdot\x_{1}, \ldots, \lambda^{a^{\ell}_{n}}\cdot\x_{n}\big)$. If $G$ is not abelian, we assume that \eqref{orderedC} holds, as well as the compatibility condition $a^{\ell}_{j}\geq a^{\ell}_{j+1}$ for all $j,\ell$.

\begin{theorem}\label{Lem11.5}\quad
\begin{enumerate}[{\rm (a)}]

\item\label{Lem11.5a}
The convolution  $\KK_{1}*\cdots*\KK_{p}$  belongs to the class $\PP_{0}(\EEE)$ where $e(j,k) = \max_{1\leq \ell\leq p}\frac{a^{\ell}_{j}}{a^{\ell}_{k}}$.

\item \label{Lem11.5b}
Conversely, every $\EEE$ arises in this way. Precisely, let $\EEE = \{e(j,k)\}$ be any $n\times n$ matrix satisfying \eqref{2.5}.  There are $n$ sequences $\a_k$
such that $e(j,k) = \max_{1\leq \ell\leq n}\frac{a^{\ell}_{j}}{a^{\ell}_{k}}$. If $\EEE$ is doubly monotone, each $\a_k$ can be taken non-increasing.
\end{enumerate}
\end{theorem}

\begin{proof}
 We have the identity $\CC\ZZ_{\a_\ell,0}=\PP_0(\EEE_\ell)$, where $e_\ell(j,k)=\big(a^{\ell}_{j}\big)\big(a^{\ell}_{k}\big)^{-1}$. By  Proposition~\ref{inclusion}, each $\KK_{\ell}\in \PP_{0}(\EEE)$.  Moreover, if $a^{\ell}_{j}\geq a^{\ell}_{j+1}$ for all $j,\ell$, $\EEE$ is doubly monotone. Then part (a) follows from Theorems \ref{Thm10.1abelian} and  \ref{Thm10.1}. 

To establish (\ref{Lem11.5b}), put $a^{\ell}_{j}= e(j,\ell)$. Since $\EEE$ is doubly monotone, $a^{\ell}_{j}=e(j,\ell)\geq e(j+1,\ell)=a^{\ell}_{j+1}$, so $a^{\ell}_{1}\geq a^{\ell}_{2}\geq \cdots \geq a^{\ell}_{n}$. For any $1 \leq j,\ell\leq n$ we have
\beas
\frac{a^{\ell}_{j}}{a^{\ell}_{k}}=\frac{e(j,\ell)}{e(k,\ell)}\leq \frac{e(j,k)e(k,\ell)}{e(k,\ell)}=e(j,k)
\eeas
so $\max_{1\leq \ell\leq m}\big(a^{\ell}_{j}\big)\big(a^{\ell}_{k}\big)^{-1}\leq e(j,k)$. On the other hand
\bes
e(j,k)=\frac{e(j,k)}{e(k,k)}\leq \max_{1\leq \ell \leq n}\frac{e(j,\ell)}{e(k,\ell)}=\max_{1\leq \ell \leq n}\frac{a^{\ell}_{j}}{a^{\ell}_{k}}.
\ees
This completes the proof.
\end{proof}

\subsection{Convolution of two \CZ kernels}

Part (a) of Theorem \ref{Lem11.5} determines the minimal class $\PP_0(\EEE)$ containing a finite family of \CZ classes $\CC\ZZ_{\a_\ell,0}$ and part (b) shows that every class $\PP_0(\EEE)$ arises in this way. It is possible, however, that a convolution
$$
\KK_1*\KK_2*\cdots*\KK_p,
$$
with $\KK_j\in\CC\ZZ_{\a_j,0}$,
is contained in a proper subalgebra of the class $\PP_0(\EEE)$ which appears in Theorem \ref{Lem11.5} (a). 

We show that such a situation occurs with two \CZ classes having different homogeneities. Let $\a=(a_{1}, \ldots, a_{n})$ and $\b=(b_{1}, \ldots, b_{n})$ with $a_{j}\geq a_{j+1}$ and $b_{j}\geq b_{j+1}$ for all $j$, and set 
\beas
\lambda\,\cdot_{\a}\x &= \left(\lambda^{a_{1}}\x_{1}, \ldots, \lambda^{a_{n}}\x_{n}\right),\\
\lambda\,\cdot_{\b}\x &= \left(\lambda^{b_{1}}\x_{1}, \ldots, \lambda^{b_{n}}\x_{n}\right).
\eeas

Let $\KK\in \PP_{0}(\EEE_{\a})$ and $\LL\in \PP_{0}(\EEE_{\b})$. If we apply part (a) of Theorem \ref{Lem11.5}, we obtain that $\KK*\LL\in \PP_0(\EEE)$, where the entries of $\EEE$ are
\be\label{aveeb}
e(j,k)=\max\Big\{\frac{a_j}{a_k},\frac{b_j}{b_k}\Big\}.
\ee
Notice that
$$
e(j,k)e(k,j)=\max\Big\{\frac{b_j/a_j}{b_k/a_k},\frac{b_k/a_k}{b_j/a_j}\Big\},
$$
so that the reduced rank of $\EEE$ can be as large as $n$. We will show  that $\KK*\LL$ belongs to the sum of $n-1$ classes $\PP_0(\EEE_m)$ where each $\EEE_m$ has reduced rank equal to 2.

As in (\ref{11.1iou}), write
\beas
\KK(\x) &= \sum_{j=-\infty}^{0}[\varphi^{j}]_{j}(\x)+\Phi(\x) =\sum_{j=-\infty}^{0}2^{-jQ_{\a}}\,\varphi^{j}(2^{-ja_{1}}\cdot \x_{1}, \ldots, 2^{-ja_{n}}\cdot \x_{n})+\Phi(\x),\\
\LL(\x) &= \sum_{k=-\infty}^{0}[\psi^{k}]_{k}(\x)+\Psi(\x)=\sum_{k=-\infty}^{0}2^{-kQ_{b}}\,\psi^{k}(2^{-kb_{1}}\cdot \x_{1}, \ldots, 2^{-kb_{n}}\cdot \x_{n})+\Psi(\x).
\eeas
Since $\KK*\Psi, \,\,\Phi*\LL, \,\,\Phi*\Psi\in \SS(\R^{N})$, we only need to analyze the double sum 
\bea\label{11.2iou}
\sum_{j=-\infty}^{0}\sum_{k=-\infty}^{0}[\varphi^{j}]_{j}*[\psi^{k}]_{k}.
\eea
Let $\sigma$ be the permutation of $\{1, \ldots, n\}$ so that $\frac{b_{\sigma(1)}}{a_{\sigma(1)}}\leq \frac{b_{\sigma(2)}}{a_{\sigma(2)}}\leq \cdots \leq \frac{b_{\sigma(n)}}{a_{\sigma(n)}}$. For simplicity, we will suppose that all the inequalities are strict\footnote{ The proof below shows that this restriction does not affect the validity of the final result.}.

  We split the double sum in (\ref{11.2iou}) into  $n+1$ two-parameter sums. Set

\bea\label{MlNm}
\NN_{m}&= \sum_{(j,k)\in \Gamma_{m}}[\varphi^{j}]_{j}*[\psi^{k}]_{k},&0&\leq m \leq n,
\eea
where
\bea\label{Gammam}
\Gamma_{0}&=\Big\{(j,k)\in \Z\times\Z: j<0,\quad k<0, \quad \frac{j}{k} \le \frac{b_{\sigma(1)}}{a_{\sigma(1)}}\Big\},\\
\Gamma_{m}&= \Big\{(j,k)\in \Z\times\Z: j<0,\quad k<0, \quad \frac{b_{\sigma(m)}}{a_{\sigma(m)}} <\frac{j}{k} \le\frac{b_{\sigma(m+1)}}{a_{\sigma(m+1)}}\Big\},\,1\leq m\leq n-1,\\
\Gamma_{n}&= \Big\{(j,k)\in \Z\times\Z: j<0,\quad k<0, \quad \frac{b_{\sigma(n)}}{a_{\sigma(n)}} <\frac{j}{k} \Big\}.
\eea
Then clearly $\sum_{j=-\infty}^{0}\sum_{k=-\infty}^{0}[\varphi^{j}]_{j}*[\psi^{k}]_{k}=
\sum_{m=0}^{n}\NN_{m}$.

\begin{lemma}\label{lemmaNm}
For $0\leq m \leq n$ let $A_{m}=\{\sigma(1), \ldots, \sigma(m)\}$ and $B_{m}=\{\sigma(m+1), \ldots, \sigma(n)\}$. Then the distribution $\NN_{m}\in \PP_{0}(\EEE_{m})$ where $\EEE_{m}=\{e_{m}(j,k)\}$ is the $n\times n$ matrix whose entries are given by
\beas
e_{m}(j,k) &= \frac{b_{j}}{b_{k}} &&\text{if $j\in A_{m}$ and $k\in A_{m}$},\\
e_{m}(j,k) &= \frac{b_{j}a_{\sigma(m+1)}}{a_{k}b_{\sigma(m+1)}} &&\text{if $j\in A_{m}$ and $k\in B_{m}$},\\
e_{m}(j,k) &= \frac{a_{j}b_{\sigma(m)}}{b_{k}a_{\sigma(m)}} &&\text{if $j\in B_{m}$ and $k\in A_{m}$},\\
e_{m}(j,k) &= \frac{a_{j}}{a_{k}} &&\text{if $j\in B_{m}$ and $k\in B_{m}$}.
\eeas
\end{lemma}
\begin{proof}
We use the fact that $[\varphi^{j}]_{j}*[\psi^{k}]_{k}(\x) =2^{-\sum_{p=1}^{n}l_{p}Q_{p}}\theta^{j,k}(2^{-l_{1}}\cdot \x_{1}, \ldots, 2^{-l_{n}}\cdot\x_{n})$
where $l_{p}=ja_{p}\vee k b_{p}$. Now if $(j,k)\in \Gamma_{m}$ then
\bea\label{L(j,k)}
l_{p}=
\begin{cases}
kb_{p} & \text{if $p\in \{\sigma(1), \ldots, \sigma(m)\}=A_{m}$},\\\\
ja_{p} &\text{if $p\in \{\sigma(m+1),\ldots, \sigma(n)\}=B_{m}$.}
\end{cases}
\eea
If $L(j,k)$ is the multiindex with components $l_p$ as above, one  easily verifies that
$$
(j,k)\in \Gamma_m\Longrightarrow L(j,k)\in\Gamma(\EEE_m).
$$
Consider first the case $1\le m\le n-1$, where $A_m$ and $B_m$ are both nonempty. Then the map $(j,k)\in\Gamma_m\longmapsto L(j,k)$ is injective. Since each term $[\varphi^{j}]_{j}*[\psi^{k}]_{k}$ has weak cancellation, it follows\footnote{ The fact that the components $l_p$ are not integers is irrelevant as long as they belong to a fixed lattice.} from Lemma \ref{Thm3.7}  that the series \eqref{MlNm} convergens to a kernel in the class $\PP_{0}(\EEE_{m})$. 

Suppose now that $m=0$ (the case $m=n$ can be treated in the same way). In this case the matrix $\EEE_0$ has rank 1 and the statement is that $\NN_0\in\CC\ZZ_{\a,0}$. 
Notice that
$L(j,k)=j\a$ only depends on $j$. Hence for each $j$ we must verify that 
$$
\sum_{k\le \frac{a_{\sigma(1)}}{b_{\sigma(1)}}j}[\varphi^{j}]_{j}*[\psi^{k}]_{k}=[\widetilde\varphi^j]_j
$$
with  uniformly bounded functions $\widetilde\varphi^j$ with mean value zero (the scaling by $j$ being as in \eqref{11.1iou}). It follows from Lemma \ref{Lem10.13bb} that each term in the sum is a normalized bump function times a factor that decreases exponentially with $k$. The mean value zero property is obvious. 
\end{proof}

\begin{remarks}\quad
\begin{enumerate}[{\rm 1.}]
\item Recall that \,\,\, $\displaystyle \overbrace{\frac{b_{\sigma(1)}}{a_{\sigma(1)}}<\frac{b_{\sigma(2)}}{a_{\sigma(2)}}<\cdots<\frac{b_{\sigma(m)}}{a_{\sigma(m)}}}^{{A_{m}}}
<
\overbrace{\frac{b_{\sigma(m+1)}}{a_{\sigma(m+1)}}<\cdots <\frac{b_{\sigma(n-1)}}{a_{\sigma(n-1)}}<\frac{b_{\sigma(n)}}{a_{\sigma(n)}}}^{{B_{m}}}$,\,\,\, and that\\ $a_{1}\geq a_{2}\geq \cdots\geq a_{n}$, $b_{1}\geq b_{2}\geq \cdots\geq b_{n}$. We claim that $e(j+1,j)\leq 1$ for $1\leq j \leq n-1$.
\begin{itemize}
\item If $j, j+1 \in A_{m}$, then since $b_{j+1}\leq b_{j}$,
\bes
e(j+1,j) =\frac{b_{j+1}}{b_{j}}\leq \frac{b_{j}}{b_{j}} = 1.;
\ees

\item If $j\in A_{m}$ and $j+1\in B_{m}$ then $\frac{a_{j+1}}{b_{j+1}}\leq \frac{a_{\sigma(m+1)}}{b_{\sigma(m+1)}}\leq \frac{a_{\sigma(m)}}{b_{\sigma(m)}}$ and so
\bes
e(j+1,j) = \frac{a_{j+1}b_{\sigma(m)}}{b_{j}a_{\sigma(m)}}\leq \frac{a_{j+1}b_{\sigma(m)}}{b_{j+1}a_{\sigma(m)}} \leq 1.
\ees

\item If $j\in B_{m}$ and $j+1\in A_{m}$ then $\frac{b_{j+1}}{a_{j+1}}\leq \frac{b_{\sigma(m)}}{a_{\sigma(m)}}$ and so
\bes
e(j+1,j) = \frac{b_{j+1}a_{\sigma(m+1)}}{a_{j}b_{\sigma(m+1)}}
\leq
\frac{b_{j+1}a_{\sigma(m+1)}}{a_{j+1}b_{\sigma(m+1)}}
\leq
\frac{b_{\sigma(m)}a_{\sigma(m+1)}}{a_{\sigma(m)}b_{\sigma(m+1)}}\leq 1.
\ees

\item If $j,j+1\in B_{m}$, then since $a_{j+1}\leq a_{j}$,
\bes
e(j+1,j) = \frac{a_{j+1}}{a_{j}} \leq 1.
\ees
\end{itemize}
It now follows from Proposition \ref{Prop9.4} that the matrix $\EEE_{m}$ is doubly monotone.

\smallskip

\item For $m=0,n$, $\EEE_m$ has rank one and $\NN_0\in\CC\ZZ_{\a,0}$, $\NN_n\in\CC\ZZ_{\b,0}$. For $1\le m\le n-1$, we can also think of the distribution $\NN_{m}$ on the coarser decomposition 
$$
\R^{N} = \big(\oplus_{j\in A_{m}}\R^{C_{j}}\big) +\big(\oplus_{j\in B_{m}}\R^{C_{j}}\big).
$$ 
with exponents, in the notation of Section \ref{subs.coarser},  $\alphab_m=(\alpha_{m,1},\dots,\alpha_{m,n})$ given by
$$
\alpha_{m,j}=\begin{cases}\frac1{b_j}&\text{ if }j\in A_m\\ \frac1{a_j}&\text{ if }j\in B_m.\end{cases}
$$
\beas
\EEE_{m}^\flat=
\left[\begin{matrix}
1&\frac{a_{\sigma(m+1)}}{b_{\sigma(m+1)}}\\
\frac{b_{\sigma(m)}}{a_{\sigma(m)}}&1
\end{matrix}\right].
\eeas
(Note that the product of the off-diagonal elements is at least $1$.) 

\smallskip

\item The proof of Lemma \ref{lemmaNm} shows that, if $G$ is the additive group $\R^N$, the inclusion \eqref{CZ*CZ} holds without any monotonicity assumption on the $a_j$ and the $b_j$ and with the same definition of $\sigma(m)$.\footnote{ This can also be proved more directly by decomposing the multiplier $\mu=\widehat\KK\widehat\LL$. Assuming, as we may, that the variables have been ordered so that $\sigma(m)=m$ for every $m$, $\mu$ can be decomposed as a sum $\mu=\sum_{m=1}^{n-1}\mu_m$, where  $\mu_m\in \MM_\infty(\EEE_m)$ and
$\B^c\cap\supp \mu_m$ is contained in the set where
$$
N_\b(\xib_1,\dots,\xib_{m})^{b_{m}}\le A N_\a(\xib_{m},\dots,\xib_n)^{a_{m}}, \qquad N_\a(\xi_{m+1},\dots,\xib_n)^{a_{m+1}}\le A N_\b(\xib_1,\dots,\xib_{m+1})^{b_{m+1}}.
$$}
\end{enumerate}
\end{remarks}

\begin{proposition} We have the inclusion
\be\label{CZ*CZ}
\CC\ZZ_{\a,0}* \CC\ZZ_{\b,0}\subseteq\sum_{m=1}^{n-1} \PP_0(\EEE_m),
\ee
Moreover, the class $\sum_{m=1}^{n-1} \PP_0(\EEE_m)$ is contained in $\PP_0(\EEE)$, where $\EEE$ is the matrix with entries \eqref{aveeb}, and is closed under convolution.
\end{proposition}

\begin{proof}
The inclusion follows since $\NN_0\in\PP_0(\EEE_1)$, $\NN_n\in \PP_0(\EEE_{n-1})$. By Proposition \ref{inclusion}, the first part of the statement follows from the inequalities 
$$
e_m(j,k)\le \max\Big\{\frac{a_j}{a_k},\frac{b_j}{b_k}\Big\},
$$
which can be easily verified.
To prove the second part of the statement, observe that each $\KK\in \sum_{m=1}^{n-1} \PP_0(\EEE_m)$ decomposes as
$$
\KK=\sum_{m=1}^{n-1}\NN_m+\sum_{l=1}^n \MM_l,
$$
where $\NN_m\in\PP_0(\EEE_m)$ has the form
$$
\NN_m=\sum_{(j,k)\in\Gamma_m}[\ph^{j,k}]_{L(j,k)},
$$
where the $\ph^{j,k}$ are uniformly bounded and have cancellation in the variables $\x_{A(m)}$ and $\x_{B(m)}$,
and $\MM_l\in\CC\ZZ_{\gammab_l,0}$, where $\gammab_{l}=(\gamma^{l}_{1}, \ldots, \gamma^{l}_{n})$ and
\beas
\gamma^{l}_{m}=
\begin{cases}
a_{\sigma(l)}b_{m}&\text{if $\frac{b_{m}}{a_{m}}\leq \frac{b_{\sigma(l)}}{a_{\sigma(l)}}$},\\
b_{\sigma(l)}a_{m}&\text{if $\frac{b_{m}}{a_{m}}\geq \frac{b_{\sigma(l)}}{a_{\sigma(l)}}$}.
\end{cases}
\eeas

Given a second kernel $\KK'=\sum_{m=1}^{n-1}\NN'_m+\sum_{l=1}^n \MM'_l$, we  consider the convolutions
$$
\NN_m*\NN'_{m'},\qquad \NN_m*\MM'_{l'},\qquad \MM_l*\NN'_{m'},\qquad \MM_l*\MM'_{l'}.
$$

We prove that, assuming $m\le m'$, $\NN_m*\NN'_{m'}\in \sum_{p=m}^{m'} \PP_0(\EEE_p)$. The proofs of the following statements are left to the reader:
\beas
\NN_m*\MM'_{l'}&\in\begin{cases} \sum_{p=m}^{l'-1} \PP_0(\EEE_p)&\text{ if }m<l'\\
 \sum_{p=l'}^{m} \PP_0(\EEE_p)&\text{ if }m\ge l'.\end{cases}\qquad (\text{similarly for }\MM_l*\NN'_{m'})\\
 \MM_l*\MM'_{l'}&\in \sum_{p=l}^{l'-1} \PP_0(\EEE_p)\qquad (l<l').
\eeas 

Let
$$
\NN_m=\sum_{(j,k)\in\Gamma_m}[\ph^{j,k}]_{L(j,k)},\qquad \NN'_{m'}=\sum_{(j',k')\in\Gamma_{m'}}[\psi^{j',k'}]_{L'(j',k')},
$$
where the components of $L'(j',k')$ are defined by \eqref{L(j,k)} with $m'$ in place of $m$.
We isolate a single term $[\ph^{j,k}]_{L(j,k)}*[\psi^{j',k'}]_{L'(j',k')}$ in the convolution, which is scaled by the multi-index $L(j,k)\vee L'(j',k')=(l_p)_{1\le p\le n}$, where
$$
L(j,k)_p=\begin{cases} kb_p&\text{ if }p\in A_m\\ja_p&\text{ if }p\in B_m\end{cases}\qquad 
L'(j',k')_p=\begin{cases} k'b_p&\text{ if }p\in A_{m'}\\j'a_p&\text{ if }p\in B_{m'}.\end{cases}
$$

Consider first the case $k\le k'$.

If $p\in A_m\subset A_{m'}$, then $l_p=k'b_p$.

Assume now that $r,s\in\{m+1,\dots,m'\}$, $r<s$. Then $p=\sigma(r),q=\sigma(s)\in B_m\cap A_{m'}$. If $ja_p\ge k'b_p$, then
\be\label{AcapB}
j\ge k'\frac{b_p}{a_p}\ge k'\frac{b_q}{a_q}.
\ee
If 
\be\label{m''}
m''=\max\big\{r:k'b_{\sigma(r)}>ja_{\sigma(r)}\big\},
\ee

$$
l_p=\begin{cases} k'b_p&\text{ if }p=\sigma(r)\text{ with }r\le m''\\ ja_p&\text{ if }p=\sigma(r)\text{ with }m''<r\le m'.\end{cases}
$$
Assume that $R'$ is nonempty, i.e., that, for $q=\sigma(m')$, $ k'b_q\le ja_q$. Since $(j',k')\in\Gamma_{m'}$, we have
$$
j'< k'\frac{b_q}{a_q}\le j.
$$
Therefore $l_p=ja_p$ for all $p\in B_{m'}$. To conclude,
$$
l_p=\begin{cases} k'b_p&\text{ if }p\in A_{m''}\\ ja_p&\text{ if }p\in B_{m''}.\end{cases}
$$

Notice that $(j,k')\in \Gamma_{m''}$ by  \eqref{m''}. We then have
\bea\label{k>k'}
\sum_{\substack{(j,k)\in\Gamma_m \\ (j',k')\in\Gamma_{m'}\\k\le k'}}&[\ph^{j,k}]_{L(j,k)}*[\psi^{j',k'}]_{L'(j',k')}=\\
&\sum_{m''=m}^{m'}\sum_{(j,k')\in\Gamma_{m''}}\sum_{\substack{(j,k)\in\Gamma_m \\ (j',k')\in\Gamma_{m'}\\k\le k'\,,\,j'\le j}}[\ph^{j,k}]_{L(j,k)}*[\psi^{j',k'}]_{L'(j',k')}.
\eea
By Lemma \ref{Lem10.13bb}, the terms in the innermost sum are uniformly bounded up to a factor which decays exponentially in $j'$ and $k$ and have weak cancellation in the variables $\x_{A_{m''}}$, $\x_{B_{m''}}$. So the sum in \eqref{k>k'} belongs to $\sum_{m''=m}^{m'} \PP_0(\EEE_{m''})$.

Consider now $k>k'$. Then $l_p=ka_p$ for $p\in A_m$. Since $(j,k)\in \Gamma_m$, for $p=\sigma(m+1)$ we have
$ja_p\ge kb_p>k'b_p$. Hence $l_p=ja_p$. By \eqref{AcapB}, we obtain that $l_p=ja_p$ for all $p\in  B_m\cap A_{m'}$.

In particular, since $(j',k')\in \Gamma_{m'}$, for $p=\sigma(m')$ we have
$ja_p\ge k'b_p>j'a_p$ so that $j>j'$. Then $l_p=ja_p$ for all $p\in B_{m'}$. To conclude,
$L(j,k)\vee L'(j',k')=L(j,k)\in\Gamma_m$ and 
\bea\label{k>k''}
\sum_{\substack{(j,k)\in\Gamma_m \\ (j',k')\in\Gamma_{m'}\\k> k'}}&[\ph^{j,k}]_{L(j,k)}*[\psi^{j',k'}]_{L'(j',k')}=\\
&\sum_{(j,k)\in\Gamma_{m}}\sum_{\substack{ (j',k')\in\Gamma_{m'}\\k'< k\,,\,j'< j}}[\ph^{j,k}]_{L(j,k)}*[\psi^{j',k'}]_{L'(j',k')}.
\eea
Repeating the same arguments used in the other case, we conclude that the sum in \eqref{k>k''} belongs to $\PP_0(\EEE_m)$.
\end{proof}

\section{Two-flag kernels and multipliers}\label{Multi-Flag}

In this Section we show the surprising fact, already mentioned in Section \ref{SME}, that distributions which are simultaneously flag kernels for two opposite flags are in fact  the  elements of an appropriate class $\PP(\EEE)$. We will recall the definition of flag-kernel in Section \ref{Flags}, and define two-flag kernels in Section \ref{multi}. We consider the special case of step-two flags  (which is considerably simpler than the general case)  in Section \ref{Heisenberg}, and the general situation in Section \ref{2Flag}. The key point is that if $m$ is the Fourier transform of a compactly supported distribution satisfying the differential inequalities of flag kernels, then $m$ actually satisfies improved estimates. This is established in  Section \ref{Improved}.

\subsection{Flag kernels and multipliers}\label{Flags}

We recall the definition of \textit{flag kernel} introduced in \cite{MR1818111}. Let $\R^{N}= \R^{C_{1}}\oplus\cdots\oplus\R^{C_{n}}$, and use the notation of Section \ref{Kernels}, so that $n_{j}$ is a homogeneous norm on $\R^{C_{j}}$, which has homogeneous dimension $Q_{j}$.  An  \textit{$n$-step flag} is a strictly increasing sequence of $n$ subspaces of $\R^{N}$, each of which is a direct sum of the subspaces $\{\R^{C_{1}}, \ldots, \R^{C_{n}}\}$. In particular, denote by $\FF$ \index{F2@$\FF$} the flag
\be\label{8.1aa}
\FF:\quad (0) \subsetneq \R^{C_{n}}\subsetneq \R^{C_{n-1}}\oplus \R^{C_{n}}\subsetneq \cdots \subsetneq \R^{C_{2}}\oplus \cdots \oplus \R^{C_{n}}\subsetneq \R^{N}.
\ee
Given  an $n$-tuple $\a=(a_{1}, \ldots, a_{n})$ of positive real numbers, we consider the dilations 
\be\label{8.2aa}
\lambda\cdot_{\a}\x = (\lambda^{a_{1}}\cdot\x_{1}, \ldots, \lambda^{a_{n}}\cdot\x_{n})
\ee 
defined in \eqref{lambdaa} and introduce on $\R^N$ the \textit{partial norms} \index{N7jj@$N'_{j}$}
\be\label{8.3aa} 
N_{j}'(\x) = n_{1}(\x_{1})^{1/a_{1}}+ \cdots + n_{j}(\x_{j})^{1/a_{j}},
\ee
so that $N_{j}'$ does \emph{not} involve the variables $\x_{j+1}, \ldots, \x_{n}$. Recall that if $L=\{l_{1}, \ldots, l_{r}\}\subset \{1, \ldots, n\}$ with $l_{1}<\cdots < l_{r}$, we denote by  $\R^L$ the space $\R^{C_{l_{1}}}\oplus\cdots\oplus\R^{C_{l_{r}}}$. The ``quotient flag'' of $\FF$ on $\R^L$, denoted by $\FF_{L}$, is then \index{F3L@$\FF_{L}$}
\be\label{8.4aa}
(0)\subsetneq \R^{C_{l_{r}}}\subsetneq\R^{C_{l_{r-1}}}\oplus\R^{C_{l_{r}}}\subsetneq \cdots \subsetneq \R^{C_{l_{2}}}\oplus \cdots \R^{C_{l_{r}}}\subsetneq  \R^L
\ee
and the corresponding partial norms are given for $1 \leq j \leq r$ by
\be\label{8.5aa}\index{N8Lj'@$N'_{L,j}$}
N_{L,j}'(\x_{l_{1}}, \ldots, \x_{l_{j}})= n_{l_{1}}(\x_{l_{1}})^{1/a_{l_{1}}}+ \cdots +n_{l_{j}}(\x_{l_{j}})^{1/a_{l_{j}}}.
\ee

 Notice that intrinsic to the flag structure is an ordering of the subspaces $\R^{C_j}$, which must be taken into account in the labeling of variables, norms etc.

\goodbreak

\begin{definition}\label{defFlagKernel}
A flag kernel associated to the flag $\FF$ and the dilations $\lambda\cdot_{\alphab}\x$ is a distribution $\KK$ on $\R^{N}$ with the following size estimates and cancellation conditions:

\begin{enumerate}[{\rm(A)}]

\item {\rm[Differential Inequalities]} Away from the subspace where $\x_{1}=0$ the distribution $\KK$ is given by integration against a smooth function $K$ which satisfies
\bes
\big\vert\partial^{\gammab_{1}}_{\x_{1}}\cdots \partial^{\gammab_{n}}_{\x_{n}}K(\x_{1}, \ldots, \x_{n}) \big|
\leq
C_{\gammab}\prod_{j=1}^{n}N_{j}'(\x_{1}, \ldots, \x_{j})^{-a_{j}(Q_{j}+\[\gammab_{j}\])}
\ees

\item {\rm [Cancellation Conditions]} Let $L=\{l_{1},\ldots,l_{r}\}$ and $M=\{m_{1},\ldots, m_{s}\}$ be complementary subsets of $\{1,\ldots, n\}$, let $R=\{R_{1}, \ldots, R_{s}\}$ be positive real numbers, and let $\psi\in \mathcal C^{\infty}_{0}(\R^{M})$
 be a normalized bump function with support in the unit ball. Define a distribution $\mathcal K_{\psi,R}$ on $\R^{C_{l_{1}}}\oplus \cdots \oplus\R^{C_{l_{r}}}$ by setting $\big\langle \mathcal K_{\psi,R},\varphi\big\rangle = \big\langle \mathcal K, \varphi\otimes\psi_{R}\big\rangle$ for every $\varphi\in \mathcal C^{\infty}_{0}(\R^{L})$. Then $\mathcal K_{\psi,R}$ uniformly satisfies the analogue of the estimates in {\rm(\ref{Def2.2A})} on the space $\R^{L}$;
  \textit{i.e.} away from the subspacee of $\R^{L}$ where $\x_{c_{r}}=0$, the distribution $\KK_{\psi,R}$ is given by integration against a smooth function $K_{\psi,R}$ and for every $\gammab=(\gammab_{l_{1}}, \ldots, \gammab_{l_{r}}) \in \mathbb N^{C_{l_{1}}}\times\cdots\times\mathbb N^{C_{l_{r}}}$  there is a constant $C_{\gammab}$ depending only on the constants in (\ref{Def2.2A}) and in particular independent of $\psi$ and $R$ so that
\bes
\big\vert\partial^{\gammab_{l_{1}}}_{\x_{l_{1}}}\cdots \partial^{\gammab_{l_{r}}}_{\x_{l_{r}}}K_{\psi,R}(\x_{l_{1}}, \ldots, \x_{l_{r}}) \big|\leq C_{\gammab}\prod_{j=1}^{r}N_{L,j}'(\x_{l_{1}}, \ldots, \x_{l_{j}})^{-a_{l_{j}}(Q_{l_{j}}+\[\gammab_{l_{j}}\])}.
\ees
As usual, if $s=n$, we require that $|\big\langle \KK,\psi_{R}\big\rangle|$ is bounded independently of $R$.
\end{enumerate}
\end{definition}

\noindent The following is Theorem 2.3.9 in \cite{MR1818111}.

\begin{theorem}
If $\KK$ is a flag kernel adapted to the flag {\rm(\ref{8.1aa})} and the family of dilations {\rm(\ref{8.2aa})}, the Fourier transform $m=\widehat{\KK}$ is function smooth away from the subspace where $\xib_{n}= 0$ which satisfies the differential inequalities
\be\label{8.6aa}
\big\vert\partial^{\gammab_{1}}_{\xib_{1}}\cdots \partial^{\gammab_{n}}_{\xib_{m}}m(\xib_{1}, \ldots, \xib_{n})\big\vert\leq
C_{\gammab}\,
\prod_{j=1}^{n}\Big(n_{j}(\xib_{j})^{1/a_{j}}+\cdots+n_{n}(\xib_{n})^{1/a_{n}}\Big)^{-\alpha_{j}\[\gammab_{j}\]}.
\ee
Conversely, any function satisfying {\rm(\ref{8.6aa})} is the Fourier transform of a flag kernel adapted to the flag {\rm(\ref{8.1aa})} and dilations {\rm(\ref{8.2aa})}.
\end{theorem}

\subsection{Pairs of opposite flags}\label{multi}\quad

\medskip

In Section \ref{Flags}, we chose a particular ordering of the subspaces $\{\R^{C_{1}}, \ldots, \R^{C_{n}}\}$ to define the flag $\FF$ in (\ref{8.1aa}) and a particular family of dilations on $\R^{N}$ in (\ref{8.2aa}). We can consider the \emph{opposite flag} obtained by reversing the indices $\{1, \ldots, n\}$ and choosing a different one-parameter family of dilations. Our objective is then to study \textit{two-flag distributions} which are simultaneously flag kernels for the two different flags and two different homogeneities. Thus set \index{F4perp@$\FF^{\perp}$} 
\bea
\FF:\quad &(0) \subsetneq \R^{C_{n}}\subsetneq \R^{C_{n-1}}\oplus \R^{C_{n}}\subsetneq \cdots \subsetneq \R^{C_{2}}\oplus \cdots \oplus \R^{C_{n}}\subsetneq \R^{N},\\
\FF^{\perp}\quad &(0) \subsetneq \R^{C_{1}}\subsetneq \R^{C_{1}}\oplus \R^{C_{2}}\subsetneq \cdots \subsetneq \R^{C_{1}}\oplus \cdots \oplus \R^{C_{n-1}}\subsetneq \R^{N}.
\eea
Let $\a=(a_{1}, \ldots, a_{n})$ and $\b = (b_{1}, \ldots, b_{n})$ be two $n$-tuples of positive real numbers. Define two one-parameter families of dilations 
\bea
\lambda\cdot_{\a}\x&=(\lambda^{a_{1}}\cdot\x_{1}, \ldots, \lambda^{a_{n}}\cdot\x_{n}),\\
\lambda\cdot_{\b}\x&=(\lambda^{b_{1}}\cdot\x_{1}, \ldots, \lambda^{b_{n}}\cdot\x_{n}),
\eea
and associate the first dilation with the flag $\FF$ and the second with the flag $\FF^{\perp}$. Along with satisfying the appropriate cancellation conditions, a two-flag distribution $\KK$ then satisfies the following differential inequalities. Away from the set where $\x_{1}=0$ or where $\x_{n}=0$, $\KK$ is given by integration against a smooth function $K$ and
\beas
\big\vert\partial^{\gammab}K(\x)\big\vert
&\leq
C_{\gammab}\,
\begin{cases}
\prod_{j=1}^{n}\Big(n_{1}(\x_{1})^{a_{j}/a_{1}}+ \cdots +n_{j-1}(\x_{j-1})^{a_{j}/a_{j-1}}+ n_{j}(\x_{j})\Big)^{-(Q_{j}+\[\gammab_{j}\])}\\\\
\prod_{j=1}^{n}\Big(n_{j}(\x_{j}) + n_{j+1}(\x_{j+1})^{b_{j}/b_{j+1}}+\cdots +n_{n}(\x_{n})^{b_{j}/b_{n}}\Big)^{-(Q_{j}+\[\gammab_{j}\])}
\end{cases}.
\eeas
The corresponding Fourier transform $m=\widehat\KK$ satisfies
\beas
\big\vert\partial^{\gammab}m(\xib)\big\vert
&\leq C_{\gammab}\,
\begin{cases}
\prod_{j=1}^{n}\Big(n_{j}(\xib_{j})+n_{j+1}(\xib_{j+1})^{a_{j}/a_{j+1}}+\cdots +n_{n}(\xib_{n})^{a_{j}/a_{n}}\Big)^{-\[\gammab_{j}\]}\\\\
\prod_{j=1}^{n}\Big(n_{1}(\xib_{1})^{b_{j}/b_{1}}+\cdots +n_{j-1}(\xib_{j-1})^{b_{j}/b_{j-1}}+n_{j}(\xib_{j})\Big)^{-\[\gammab_{j}\]}
\end{cases}.
\eeas

\subsection{Two-step flags}\label{Heisenberg}\quad 
\medskip

We first consider the case of two-step flags, where $\R^{N}$ is written as a direct sum of two subspaces: $\R^{N}=\R^{n}\oplus\R^{m}$. As we will see, this situation is much simpler than the general case considered later in Section \ref{2Flag}.  Write elements of $\R^{N}$ as $(\x,\y)$ with $\x\in \R^{n}$ and $\y\in \R^{m}$.  For notational convenience we assume isotropic dilations on each component. There are two flags $\FF$ and $\FF^{\perp}$ adapted to this decomposition:
\beas
\FF:\quad && (0) &\subsetneq \R^{n} \subsetneq \R^{n}\oplus\R^{m} = \R^{N}&&\text{and}&&&
\FF^{\perp}:\quad && (0) &\,\subsetneq \,\R^{m} \,\,\subsetneq \,\,\R^{n}\oplus\R^{m} = \R^{N}.\\
\eeas
Let $\a=(a_{1},a_{2})$ and $\b=(b_{1},b_{2})$ be two pairs of positive real numbers and put
\bea\label{Eqn8.9asd}
\lambda\cdot_{\a}(\x,\y) &= (\lambda^{a_{1}}\x,\lambda^{a_{2}}\y),\\
\lambda\cdot_{\b}(\x,\y) & = \,(\lambda^{b_{1}}\x,\lambda^{b_{2}}\y).
\eea 
Associate the family of dilations $\lambda\cdot_{a}(\x,\y)$ with the flag $\FF$ and the family $\lambda\cdot_{b}\,(\x,\y)$ with the flag $\FF^{\perp}$. A distribution $\KK$ on $\R^{n}\times\R^{m}$ is then a \textit{two-flag kernel} relative to this data if it satisfies appropriate cancellation conditions and
is given by integration against a smooth function $K$ which satisfies the differential inequalities
\bea\label{flagEstimatesK}
\big\vert\partial^{\alphab}_{\x}\partial^{\betab}_{\y}K(\x,\y)\big\vert
&\leq
C_{\alphab,\betab}\,
\begin{cases}
|\x|^{-n-|\alphab|}\big(|\x|^{a_{2}/a_{1}}+|\y|\big)^{-m-|\betab|}&\text{for $\x\neq 0$  (flag $\FF$)}\\\\
\big(|\x|+|\y|^{b_{1}/b_{2}}\big)^{-n-|\alphab|}\,|\y|^{-m-|\betab|}&\text{for $\y\neq 0$  (flag $\FF^{\perp}$)}
\end{cases}
\eea  
If $m(\xib,\etab)=\widehat \KK(\xib,\etab)$ is the Fourier transform, then 
\bea\label{Eqn8.10asd}
\big\vert\partial^{\alphab}_{\xi}\partial^{\betab}_{\etab}m(\xib,\etab)\big\vert
&\leq
C_{\alphab,\betab}\,
\begin{cases}
\big(|\xib|+|\etab|^{a_{1}/a_{2}}\big)^{-|\alphab|}|\etab|^{-|\betab|}&\text{for $\etab\neq 0$  (flag $\FF$)}\\\\
|\xib|^{-|\alphab|}\big(|\xib|^{b_{2}/b_{1}}+|\etab|\big)^{-|\betab|}&\text{for $\xib\neq 0$ (flag $\FF^{\perp}$)}
\end{cases}.  
\eea  
It follows that both $K$ and $m$ are smooth away from the origin.

\begin{proposition}\label{Prop7.3}
Suppose that $\frac{a_{1}}{b_{1}}< \frac{a_{2}}{b_{2}}$,   and let
\beas
N_{1}(\x,\y) &= |\x|+|\y|^{b_{1}/b_{2}} && &\widehat N_1(\xib,\etab)=|\xib|+|\etab|^{a_1/a_2}\\
N_{2}(\x,\y) &= |\x|^{a_{2}/a_{1}}+|\y|&& &\widehat N_2(\xib,\etab)=|\xib|^{b_2/b_1}+|\etab|
\eeas  
be the norms associated with the matrix  
\be\label{E2x2}
\EEE =\left[\begin{matrix}1&\frac{b_{1}}{b_{2}}\\\frac{a_{2}}{a_{1}}&1\end{matrix}\right]
\ee 
satisfying the basic hypotheses. 

If $\KK$ is a two-flag kernel on $\R^{n}\times\R^{m}$ adapted to the dilations given in (\ref{Eqn8.9asd}), then
\begin{enumerate}[{\rm(a)}]
\item \label{Prop7.3a}the corresponding function $K$ is integrable at infinity;

\smallskip

\item \label{Prop7.3b} $\KK\in \PP(\EEE)$ and  $\widehat \KK=m\in \MM(\EEE)$;

\smallskip

\item \label{Prop7.3c} we can write $\KK=\KK_{0}+K_{\infty}$ where $K_{\infty}\in L^{1}(\R^{N})\cap \CC^{\infty}(\R^{N})$, and $\KK_{0}$ is a two-flag kernel supported in $\B(1)$.
\end{enumerate}

Conversely, if $\KK$ is in $\PP(\EEE)$, then $\KK$ is a two-flag kernel (and $\widehat\KK=m$ a two-flag multiplier)
 on $\R^{n}\times\R^{m}$ adapted to the dilations given in (\ref{Eqn8.9asd}). 
\end{proposition}

\begin{proof}
We have already observed that $m$ is smooth away from the origin, and it follows from equation (\ref{Eqn8.10asd}) that we have the following estimates for pure derivatives of $m$:

\beas
\big\vert\partial^{\alphab}_{\xib}m(\xib,\etab)\big\vert 
&\lesssim
\big(|\xib|+|\etab|^{a_{1}/a_{2}}\big)^{-|\alphab|},\\
\big\vert\partial^{\betab}_{\etab}m(\xib,\etab)\big\vert 
&\lesssim
\big(|\xib|^{b_{2}/b_{1}}+|\etab|\big)^{-|\betab|}
\eeas  
Since the rank of $\EEE$ is greater than $1$, it follows from Lemma \ref{Lem4.6} that $K$ is smooth away from the origin and is integrable at infinity. This establishes (\ref{Prop7.3a}).

\medskip

Next let $\chi\in\CC^{\infty}_{0}(\R^{n}\times\R^{n})$ have support in the unit ball, with $\chi(\x,\y) \equiv 1$ if $|\x|+|\y|\leq \frac{1}{2}$.  Obviously $\KK_{0}=\chi\KK$  satisfies the differential inequalities \eqref{flagEstimatesK} and it easy to verify that it also satisfies the cancellations in Definition \ref{defFlagKernel}. By (\ref{Prop7.3a}), $\KK_{\infty}=(1-\chi)\KK\in L^{1}(\R^{n})\cap \CC^{\infty}(\R^{n})$. This establishes (\ref{Prop7.3c}).

\medskip
Being a two-flag kernel we know  that $K_{0}$  satisfies  \eqref{flagEstimatesK}.  A case by case analysis shows that these estimates can be combined and improved. 
\begin{enumerate}[(1)]

\smallskip

\item Suppose $|\x|\geq |\y|^{b_{1}/b_{2}}$. Then $|\x|^{-1}\lesssim \big(|\x|+|\y|^{b_{1}/b_{2}}\big)^{-1}$ and hence the first inequality shows that
\bea\label{Eqn8.11asd}
|\partial^{\alphab}_{\x}\partial^{\betab}_{\y}K_{0}(\x,\y)&|\lesssim 
\big(|\x|+|\y|^{b_{1}/b_{2}}\big)^{-n-|\alpha|}\big(|\x|^{a_{2}/a_{1}}+|\y|\big)^{-m-|\betab|}.
\eea

\smallskip

\item Suppose  $|\y|\geq |\x|^{a_{2}/a_{1}}$.  Then $|\y|^{-1}\lesssim \big(|\x|^{a_{2}/a_{1}}+|\y|\big)^{-1}$, and the second inequality again gives the estimate in (\ref{Eqn8.11asd}). 

\smallskip

\item Finally suppose  $|\x|< |\y|^{b_{1}/b_{2}}$  and $|\y|< |\x|^{a_{2}/a_{1}}$.  Here we must have
\beas
|\x|&<|\x|^{\frac{a_{2}b_{1}}{a_{1}b_{2}}}&&\text{and}&|\y|<|\y|^{\frac{a_{2}b_{1}}{a_{1}b_{2}}}
\eeas
Since $a_{2}b_{1}>a_{1}b_{2}$, this means that $|\x|>1$ and $|\y|>1$, and this is outside the support of $K_{0}$. It follows that for all $(\x,\y)\in \R^{n}\oplus\R^{m}$  we have the \emph{improved} estimates
\beas
|\partial^{\alphab}_{\x}\partial^{\betab}_{\y}K(\x,\y)|\lesssim \big(|\x|+|\y|^{\frac{b_{1}}{b_{2}}}\big)^{-n-|\alpha|}\big(|\x|^{\frac{a_{2}}{a_{1}}}+|\y|\big)^{-m-|\betab|}.
\eeas
\end{enumerate}
This shows that $\KK_{0}\in \PP_{0}(\EEE)$ since the cancellations conditions for the flag kernels are the same as the cancellation conditions for $\PP_{0}(\EEE)$. This establishes (\ref{Prop7.3b}) and completes the proof.

The last part of the statement is obvious. 
\end{proof}

\begin{remarks}\label{remark-flag2} \quad

\begin{enumerate}
\item[\rm(1)] For a matrix $\EEE$ of the form \eqref{E2x2}, the condition $\frac{a_1}{b_1}\le\frac{a_2}{b_2}$ is equivalent to the basic hypotheses. If the inequality is an equality, $\EEE$ is reducible and the class $\PP(\EEE)$ consists of \CZ kernels.

The arguments used in the proof of Proposition \ref{Prop7.3} can adapted to prove that, if the opposite inequality $\frac{a_1}{b_1}>\frac{a_2}{b_2}$ holds, then a kernel $\KK\in\PP(\EEE)$ (defined by the same differential inequalities and cancellations as in Definition \ref{Def2.2}) is integrable near $0$ instead than at infinity. This suggests that the essentially local nature of the kernels strongly depends on the basic hypotheses.
\smallskip

\item[\rm(2)] One may observe that a $2\times2$ matrix satisfying the basic hypotheses can always be put in the form \eqref{E2x2} for appropriate $a_j,b_j$ with $\frac{a_1}{b_1}\le\frac{a_2}{b_2}$. 
This means that, for every two-fold decomposition $\R^N=\R^{C_1}\oplus\R^{C_2}$ and every $\EEE$ satisfying the basic hypotheses, the class $\PP(\EEE)$ conicides either with a class of two-flag kernels or with a class of \CZ kernels, depending on its rank. 

As we will see next, this is specific of two-fold decomposition. However, it follows from Lemma~\ref{sharp} that the same also holds for general  decompositions, as long as the reduced rank of $\EEE$ is not greater than~2.
\end{enumerate}

\end{remarks}

\subsection{General two-flag kernels}\label{2Flag}

In the case of the two-step flags, the improved estimates for the kernel followed easily from the two-flag estimates simply by considering possible cases. However, if the number of norms is greater than two, this argument no longer works.  Observe that the differential inequalities for the kernel $K$ give \textit{no} information when $\x_{1}=\x_{n}=0$ while other $\x_{j}\neq 0$. In particular, it does not follow from the inequalities alone that the kernel is non-singular away from the origin. Nevertheless, we have the following result.

\begin{theorem}\label{Thm7.7} Suppose that $\frac{a_1}{b_1}\le \frac{a_2}{b_2}\le\cdots\le\frac{a_n}{b_n}$ with at least one strict inequality, and let $\KK$ be a two-flag kernel for the flags $\FF$ and $\FF^{\perp}$ with homogeneities $\a$ and $\b$  respectively. 

\begin{enumerate}[{\rm(a)}]

\item\label{Thm7.7a} The function $K$ is integrable at infinity, and we can write $\KK = \KK_{0}+K_{\infty}$ where $K_{\infty}\in L^{1}(\R^{N})\cap \CC^{\infty}(\R^{N})$, and $\KK_{0}$ is a two-flag kernel supported in $\B(1)$.

\item \label{Thm7.7b}The kernel $\KK_{0}$ belongs to the class $\PP_{0}(\EEE)$ associated to the matrix
\bea\label{Enxn}
\EEE =
\left[\begin{matrix}
1&{b}_{1}/{b}_{2}&{b}_{1}/{b}_{3}&\cdots&{b}_{1}/{b}_{n-1}&{b}_{1}/{b}_{n}\\
a_{2}/a_{1}&1&{b}_{2}/{b}_{3}&\cdots &{b}_{2}/{b}_{n-1}&{b}_{2}/{b}_{n}\\
a_{3}/a_{1}&a_{3}/a_{2}&1&\cdots&{b}_{3}/{b}_{n-1}&{b}_{3}/{b}_{n}\\
\vdots&\vdots&\vdots&\ddots&\vdots&\vdots\\
a_{n-1}/a_{1}&a_{n-1}/a_{2}&a_{n-1}/a_{3}&\cdots&1&{b}_{n-1}/{b}_{n}\\
a_{n}/a_{1}&a_{n}/a_{2}&a_{n}/a_{3}&\cdots&a_{n}/a_{n-1}&1\\
\end{matrix}\right].
\eea

\item \label{Thm7.7c} Conversely, every kernel in $\PP_0(\EEE)$ is a two-flag kernel for the flags $\FF$ and $\FF^{\perp}$ with homogeneities $\a$ and $\b$ respectively.
\end{enumerate}
\end{theorem}

\begin{proof}
If two ratios $a_j/b_j,a_k/b_k$ are equal, the product $e(j,k)e(k,j)$ is 1. By Proposition \ref{PEflat}, the matrix $\EEE$ can be reduced to $\EEE^\flat$  of the same type. We may then  assume that all inequalities are strict. 

The first part of the proof proceeds as  the proof  of Proposition \ref{Prop7.3}. Combining the two kinds of estimates for pure derivatives of the Fourier transform $m$, we see that $m$ satisfies $\big\vert\partial^{\gammab_{j}}_{\xib_{j}}m(\xib)\big\vert \lesssim \widehat N_{j}(\xib)^{-\[\gammab_{j}\]}$ where
\beas
\widehat N_{j}(\xib)=n_{1}(\xib_{1})^{b_{j}/b_{1}}+\cdots + n_{j-1}(\xib_{j-1})^{b_{j}/b_{j-1}}&+n_{j}(\xib_{j})+\\
 &+n_{j+1}(\xib_{j+1})^{a_{j}/a_{j+1}}+ \cdots + n_{n}(\xib_{n})^{a_{j}/a_{n}}.
\eeas
We want to apply Lemma \ref{Lem4.6}, and so we need to check that the matrix $\EEE$ satisfies the basic hypothesis of (\ref{2.5}) and has rank greater than $1$. We proceed to check this.

We want to show that $e(j,k) \leq e(j,l)e(l,k)$ where $e(j,k) =
\begin{cases}
b_{j}/b_{k} &\text{if $j \leq k$}\\
a_{j}/a_{k} &\text{if $j \geq k$}
\end{cases}
$.
Thus we need to show
\beas
e(j,k)&= \frac{b_{j}}{b_{k}}
\leq
\begin{cases} 
\frac{a_{j}}{a_{l}}\frac{b_{l}}{b_{k}}=e(j,l)e(l,k)&\text{if $l\leq j\leq k$,}\\
\frac{b_{j}}{b_{l}}\frac{b_{l}}{b_{k}}=e(j,l)e(l,k)&\text{if $j\leq l \leq k$,}\\
\frac{b_{j}}{b_{l}}\frac{a_{l}}{a_{k}}=e(j,l)e(l,k)&\text{if $j\leq k\leq l$,}
\end{cases}&&&&\text{if $j \leq k$},\\\\
e(j,k)&= \frac{a_{j}}{a_{k}}
\leq
\begin{cases} 
\frac{b_{j}}{b_{l}}\frac{a_{l}}{a_{k}}=e(j,l)e(l,k)&\text{if $l\geq j\geq k$,}\\
\frac{a_{j}}{a_{l}}\frac{a_{l}}{a_{k}}=e(j,l)e(l,k)&\text{if $j\geq l \geq k$,}\\
\frac{a_{j}}{a_{l}}\frac{b_{l}}{b_{k}}=e(j,l)e(l,k)&\text{if $j\geq k\geq l$,}
\end{cases}&&&&\text{if $j \geq k$.}
\eeas
However it is easy to check that all of these inequalities follow from the assumption in the Theorem. Thus $\EEE$ satisfies the basic hypotheses. Also one can check that
\beas
\det(\EEE) = \prod_{j=2}^{n}\left(1-\frac{a_{j}}{b_{j}}\frac{b_{j-1}}{a_{j-1}}\right)\neq 0 
\eeas
so in fact the matrix $\EEE$ has rank $n$. 
Thus it follows from Lemma \ref{Lem4.6} that $K$ is integrable at infinity.

\bigskip

 
The proof of part (\ref{Thm7.7b}) will use the following result, which shows that estimates for pure derivatives leads to estimates for mixed derivatives. In order to keep our focus on the proof of Theorem \ref{Thm7.7}, we defer the proof  to Section~\ref{Improved}.  
 \begin{lemma}\label{Lem9.3}
Let $m\in \CC^{\infty}(\R^{N}\setminus \{0\})$. Suppose that for every $\gammab_{j}\in \N^{C_{j}}$ there is a constant $C_{\gammab_{j}}>0$ so that
\bes
\big|\partial^{\gammab_{j}}_{\xib_{j}}m(\xib)\big|\leq C_{\gammab_{j}}\widehat N_{j}(\xib)^{-\[\gammab_{j}\]}.
\ees
Then for every $\gammab=(\gammab_{1}, \ldots, \gammab_{n})\in \N^{C_{1}}\times \cdots \times \N^{C_{n}}$ there is a constant $C_{\gammab}>0$ so that if $\xib\in \B(1)^{c}$,
\bes
\big|\partial^{\gammab_{1}}_{\xib_{1}}\cdots\partial^{\gammab_{n}}_{\xib_{n}}m(\xib)\big|
\leq 
C_{\gammab}\prod_{j=1}^{n}\widehat N_{j}(\xib)^{-\[\gammab_{j}\]}.
\ees
\end{lemma}

\bigskip

We now turn to the proof of Theorem \ref{Thm7.7} part (\ref{Thm7.7b}), which is more involved than in the situation considered in Section \ref{Heisenberg}.  We start by decomposing the Fourier transform $m_{0}$ of the kernel $\KK_{0}$ by means of a smooth cutoff function supported in $B(2)$ and equal to 1 on $B(1)$. We obtain $m_{0}=m'+m''$, where $m'$ is in $\CC_{0}(\R^N)$, while $m''$ is supported in $B(1)^c$ and satisfies the inequalities \eqref{Eqn8.10asd}. Denoting by $\KK',\KK''$  the inverse Fourier transforms of $m',m''$ respectively, we then have $\KK_{0}=\KK'+\KK''$.

We consider $m''$ first. Since it satisfies the inequalities \eqref{Eqn8.10asd},  the two types of estimates can be combined together to give
 \be\label{pure}
 \big\vert\partial^{\gammab_{j}}_{\xib_{j}}m''(\xib)\big\vert \lesssim \widehat N_{j}(\xib)^{-\[\gammab_{j}\]},
 \ee
 whenever we differentiate only in the $\xib_j$-variables ({\it pure} derivatives). Applying Lemma \ref{Lem9.3} and recalling that $m''$ is supported in $B(1)^c$, we obtain inequalities for all derivatives of $m''$, namely
 \be\label{mixed}
\big|\partial^{\gammab_{1}}_{\xib_{1}}\cdots\partial^{\gammab_{n}}_{\xib_{n}}m''(\xib)\big|
\lesssim\prod_{j=1}^{n}\big(1+\widehat N_{j}(\xib)\big)^{-\[\gammab_{j}\]},
\ee
which tells us that $m''\in\MM_\infty(\EEE)$.

By Theorem \ref{Thm6.1} it  follows that $\KK''\in \PP_{0}(\EEE)$. Using a cutoff function in the $\x$-variables, we could then split $\KK''$ as $\KK_0+\KK'''$, with $\KK_0\in \PP_{0}(\EEE)$ and $\KK'''\in\cS(\R^N)$.

We next observe that $\KK'\in C^\infty(\bR^n)$ because $m'$ has compact support. It follows that $\KK$ coincides with a smooth function away from the origin. We can then apply Theorem~\ref{Lem4.6} and conclude that $\KK$ is integrable at infinity.  

Hence $\KK'\in C^\infty(\R^N)\cap L^1(\R^N)$ and the same holds for $\KK'+\KK'''=\KK_\infty$.

Finally, part (\ref{Thm7.7c}) is an obvious consequence of (\ref{Thm7.7b}).
\end{proof}

\begin{remark} \quad

As mentioned in Remark \ref{remark-flag2}, for a general matrix $\EEE$ with reduced rank larger than 2, it is not true that the class $\PP(\EEE)$ coincides with a class of two-flag kernels. In fact, such a coincidence occurs if and only if the equality $e(j,\ell)=e(j,k)e(k,\ell)$ holds whenever the triple $j,k,\ell$ is ordered ($j<k<l$ or vice-versa).

\end{remark}

The following statement is a direct consequence of Theorem \ref{Thm10.1}.
\begin{corollary}\label{convolution2flag}
Let $\KK,\LL$ be two compactly supported two-flag kernels as in Theorem \ref{Thm7.7}, for the same pair of flags and the same homogeneities. Then the convolution $\KK*\LL$ is also a two-flag kernel of the same type.
\end{corollary}

\subsection{Proof of Lemma \ref{Lem9.3}}\label{Improved} 

In this section we complete the proof of Theorem \ref{Thm7.7} by proving Lemma \ref{Lem9.3}; \emph{i.e.} we show that  if $m\in \CC^{\infty}(\R^{N})$ satisfies the estimates for the class of multipliers $\MM(\EEE)$ for \emph{pure} derivatives, then the estimates for mixed derivatives follow automatically. We will use the following result, which is easily proved by estimating the decay of the Fourier transform of~$f$.

\begin{proposition}\label{Prop9.1}
Let $f\in \CC^{\infty}_{0}(\R^{N})$ be supported in
 the unit ball. Suppose that 
\bes
\big\vert\partial^{\gammab_{j}}_{\etab_{j}}f(\etab)\big\vert \leq C_{\gammab_{j}}
\ees
for each  $j$ and each $\gammab_{j}\in \N^{C_{j}}$. Then for every $\gammab=(\gammab_{1}, \ldots, \gammab_{n}) \in \N^{C_{1}}\times \cdots \times \N^{C_{n}}$ we have
\bes
\big\vert \partial^{\gammab_{1}}_{\etab_{1}}\cdots \partial^{\gammab_{n}}_{\etab_{n}}f(\etab)\big\vert\leq C'_{\gammab}
\ees
where the constants $\{C'_{\gammab}\}$ depend only on the constants $\{C_{\gammab_{j}}\}$.
\end{proposition}
\medskip

Next let $S=\{(I_{1}, k_{1});\ldots; (I_{s}, k_{s})\}\in \SS(n)$. For $\xib\in \B(1)^{c}\cap \widehat E_{S}^{A}$, we give a quantitative estimate of the size  of a ball about $\xib$ that we want contained in some larger region $\widehat E_{S}^{A'}$. Recall from Definition \ref{Def3.12iou} and equation (\ref{5.11aa}) that we defined a dilation structure $\lambda\,\widehat\cdot_{S}\,\xib_{I_{r}}=\big\{\lambda^{1/e(k_{r},j)}\cdot\xib_{j}:j\in I_{r}\big\}$ and a corresponding smooth homogeneous norm $\widehat n_{S,r}(\xib_{I_{r}})\approx \sum_{j\in I_{r}}n_{j}(\xib_{j})^{e(k_{r},j)}$ on $\R^{I_{r}}$. Let $C\geq 1$ be a constant such that, for all $1\leq r \leq s$, we have
\be\label{9.1aa}
\widehat n_{S,r}(\xib_{I_{r}}+\etab_{I_{r}}) \leq C\,\Big[\widehat n_{S,r}(\xib_{I_{r}})+\widehat n_{S,r}(\etab_{I_{r}})\Big].
\ee

\begin{proposition}\label{Prop9.2} 
Let $S=\{(I_{1}, k_{1});\ldots; (I_{s}, k_{s})\}\in \SS(n)$ and let $C$ be the constant from equation (\ref{9.1aa}). Let $\xib=(\xib_{1}, \ldots, \xib_{n})\in \B(1)^{c}\cap \widehat E_{S}^{A}$. If $\epsilon <(2C)^{-1}$, then 
\beas
\Big\{\etab\in\R^{N}:\widehat n_{S,r}(\xib_{I_{r}}-\etab_{I_{r}})=\sum_{j\in I_{r}}n_{j}(\etab_{j}-\xib_{j})^{e(k_{r},j)}<\epsilon\,n_{k_{r}}(\xib_{k_{r}}),\,\,1\leq r\leq s\Big\}\subset \widehat E_{S}^{A'}
\eeas 
where $A'$ is a constant that depends only on $A$ and $C$.
\end{proposition}

\begin{proof}
Let $\xib\in \widehat E_{S}^{A}$. Recall from part (\ref{Thm5.6B}) of Lemma \ref{Thm5.6} that 
\beas
n_{k_{r}}(\xib_{k_{r}})&>A^{-1}n_{j}(\xib_{j})^{e(k_{r},j)}& &\forall j\in I_{r}, j\neq k_{r},\quad\text{and}\\
n_{k_{r}}(\xib_{k_{r}})&>A^{-1}n_{k_{t}}(\xib_{k_{t}})^{1/\tau_{S}(k_{r},k_{t})} & &\forall t\in \{1, \ldots, s\}, \,\,t\neq r.
\eeas
Conversely, if $\etab\in \B(1)^{c}$ satisfies 
\beas
n_{k_{r}}(\etab_{k_{r}})&>A^{-1}n_{j}(\etab_{j})^{e(k_{r},j)}& &\forall j\in I_{r}, j\neq k_{r},\quad\text{and}\\
n_{k_{r}}(\etab_{k_{r}})&>A^{-1}n_{k_{t}}(\etab_{k_{t}})^{1/\tau_{S}(k_{r},k_{t})} & &\forall t\in \{1, \ldots, s\}, \,\,t\neq r.
\eeas
then $\etab\in E_{S}^{A^{\tau}}$ where $\tau$ is a constant that depends only on the coefficients $\{e(j,k)\}$. 

\smallskip

If $n_{j}(\etab_{j}-\xib_{j})^{e(k_{r},j)}<\epsilon n_{k_{r}}(\xib_{k_{r}})$ for $1 \leq r \leq s$ and all $j \in I_{R}$, and if $C\epsilon <\frac{1}{2}$, then $n_{k_{r}}(\xib_{k_{r}})\leq C\big[n_{k_{r}}(\xib_{k_{r}}-\etab_{k_{r}})+n_{k_{r}}(\etab_{k_{r}})\big]\leq C\big[\epsilon n_{k_{r}}(\xib_{k_{r}}) +n_{k_{r}}(\etab_{k_{r}})\big]<\frac{1}{2}\,n_{k_{r}}(\xib_{k_{r}})+Cn_{k_{r}}(\etab_{k_{r}})$ so that for $1\leq r \leq s$,
\be\label{9.2aa}
n_{k_{r}}(\xib_{k_{r}}) <2Cn_{k_{r}}(\etab_{k_{r}}).
\ee
Next, if $j\in I_{r}$, 
\beas
n_{j}(\etab_{j}) \leq C\big[n_{j}(\etab_{j}-\xib_{j})+n_{j}(\xib_{j})\big]< C\big[\big(\epsilon n_{k}(\xib_{k_{r}})\big)^{1/e(k_{r},j)}+\big(An_{k_{r}}(\xib_{k_{r}})\big)^{1/e(k_{r},j)}\big],
\eeas
and it follows from (\ref{9.2aa}) that
\be\label{9.3aa}
n_{j}(\etab_{j})\leq 
C'\,n_{k_{r}}(\etab_{k_{r}})^{1/e(k_{r},j)}
\ee
where $C'$ depends on $C$, $\epsilon$, and $A$. Finally, if $1 \leq t \leq s$ with $t\neq r$, it follows from (\ref{9.3aa}) that $n_{k_{t}}(\etab_{k_{t}}) \leq C' n_{k_{t}}(\xib_{k_{t}})^{1/e(k_{t},k_{t})}$, and so
\bea\label{9.4aa}
n_{k_{t}}(\etab_{k_{t}}) 
\leq C'\big(An_{k_{r}}(\xib_{k_{r}})\big)^{\tau_{S}(k_{r},k_{t})}\leq C''\,n_{k_{r}}(\etab_{k_{r}})^{\tau_{S}(k_{r},k_{t})}
\eea
where $C''$ depends on $C$, $\epsilon$, and $A$. It now follows from the inequalities in (\ref{9.3aa}) and (\ref{9.4aa}) that $\etab\in \widehat E_{S}^{A'}$, and this completes the proof.
\end{proof}

\begin{proof}[We now prove Lemma \ref{Lem9.3}]\quad

\medskip
%
%
Let $\xib\in\B(1)^{c}$. Then if $A>1$ there exists $S=\big((I_{1},k_{1});\ldots,;(I_{S},k_{s})\big)\in \SS(n)$, so that $\xib\in \widehat E_{S}^{A}$. We have $\widehat N_{j}(\xib)\approx n_{k_{r}}(\xib_{k_{r}})^{1/e(k_{r},j)}$ for every $r\in\{1, \ldots, s\}$ and every $j\in I_{r}$. (Here the implied constants depend on $A$.) Thus to establish the Lemma we need to show that for every $\gammab=(\gammab_{1}, \ldots, \gammab_{n})\in \N^{C_{1}}\times \cdots \times \N^{C_{n}}$ there is a constant $C_{\gammab}>0$ so that
\beas
\big|\partial^{\gammab_{1}}_{\xib_{1}}\cdots\partial^{\gammab_{n}}_{\xib_{n}}m(\xib)\big| \leq C_{\gammab}\, \prod_{r=1}^{s} n_{k_{r}}(\xib_{k_{r}})^{-\sum_{j\in I_{r}}\[\gammab_{j}\]/e(k_{r},j)}.
\eeas
According to Proposition \ref{Prop9.2}, there exists $\epsilon>0$ and a constant $A'$ depending only on $A$ so that
\beas
\Big\{\etab\in\R^{N}:\text{$n_{j}(\etab_{j}-\xib_{j})^{e(k_{r},j)}<\epsilon\, n_{k_{r}}(\xib_{k_{r}})$ for $1\leq r \leq s$ and $j \in I_{r}$}\Big\}\subset \widehat E_{S}^{A'}.
\eeas
For $1 \leq r \leq s$ let $\varphi_{r}\in \CC^{\infty}_{0}(\R^{I_{r}})$ be supported where $n_{j}(\etab_{j})\leq 1$ for all $j\in I_{r}$, with $\widehat n_{S,r}(\etab_{I_{r}})\equiv 1$ if $n_{j}(\etab_{j})\leq \frac{1}{2}$ for all $j\in I_{r}$. Put
\beas
f(\etab) 
&= 
\Big(\prod_{r=1}^{s}\varphi_{r}(\etab_{I_{r}})\Big)\,
m\Big(\xib_{I_{1}}+[\epsilon\, n_{k_{1}}(\xib_{k_{1}})]\,\widehat{\cdot}_{S}\,\etab_{I_{1}}, \ldots, 
\xib_{I_{S}}+[\epsilon\, n_{k_{s}}(\xib_{k_{s}})]\,\widehat{\cdot}_{S}\,\etab_{I_{s}}\Big).
\eeas
Here, the function $m$ is evaluated at the point whose $j^{th}$-coordinate is 
\bes
\v_{j}=\xib_{j}+ [\epsilon n_{k_{r}}(\xib_{k_{r}})]^{1/e(k_{r},j)}\cdot\etab_{j}
\ees 
when $j\in I_{r}$. Clearly $f\in\CC^{\infty}(\R^{N})$ and is supported where $n_{j}(\etab_{j})\leq 1$ for all $1 \leq j \leq n$. Moreover
\beas
f(\etab) 
&= 
m\Big(\xib_{I_{1}}+[\epsilon\, n_{k_{1}}(\xib_{k_{1}})]\,\widehat{\cdot}_{S}\,\etab_{I_{1}}, \ldots, 
\xib_{I_{S}}+[\epsilon\, n_{k_{s}}(\xib_{k_{s}})]\,\widehat{\cdot}_{S}\,\etab_{I_{s}}\Big)
\eeas
if $n_{j}(\etab_{j}) \leq \frac{1}{2}$ for all $j$. If $f(\etab)\neq 0$, so that $n_{j}(\etab_{j})\leq 1$, then
\beas
n_{j}(\v_{j}-\xib_{j}) &= n_{j}\Big(n_{k_{r}}([\epsilon\,\xib_{k_{r}})]^{1/e(k_{r},j)}\cdot\etab_{j}\Big)\\
&= 
[\epsilon\,n_{k_{r}}(\xib_{k_{r}})]^{1/e(k_{r},j)}\,n_{j} (\etab_{j})\leq[\epsilon\,n_{k_{r}}(\xib_{k_{r}})]^{1/e(k_{r},j)},
\eeas
and it follows that $m$ is evaluated at a point of $\widehat E_{S}^{A'}$.

Consider `pure' derivatives $\partial^{\gammab_{j}}_{\etab_{j}}f$. Some of the derivatives will fall on $\big(\prod_{r=1}^{s}\varphi_{r}(\etab_{I_{r}})\big)$ and the rest fall on the term involving $m$. Derivatives of $\big(\prod_{r=1}^{s}\varphi_{r}(\etab_{I_{r}})\big)$ are bounded by constants independent of $m$. On the other hand, if $\gammab_{j}\in \N^{(I_{r})}$, using the chain rule, the hypothesis about pure derivatives of $m$, and the fact that we are evaluating $m$ at points in $\widehat E_{S}^{A'}$, we get
\beas
&\partial^{\gammab_{j}}_{\etab_{j}}
\Bigg[
m\Big(\xib_{I_{1}}+[\epsilon\, n_{k_{1}}(\xib_{k_{1}})]\,\widehat{\cdot}_{S}\,\etab_{I_{1}}, \ldots, 
\xib_{I_{S}}+[\epsilon\, n_{k_{s}}(\xib_{k_{s}})]\,\widehat{\cdot}_{S}\,\etab_{I_{s}}\Big)
\Bigg]\\
&=
\Big[\partial^{\gammab_{j}}_{\etab_{j}}m\Big]\Big(\xib_{I_{1}}+[\epsilon\, n_{k_{1}}(\xib_{k_{1}})]\,\widehat{\cdot}_{S}\,\etab_{I_{1}}, \ldots, 
\xib_{I_{S}}+[\epsilon\, n_{k_{s}}(\xib_{k_{s}})]\,\widehat{\cdot}_{S}\,\etab_{I_{s}}\Big) \prod_{j\in I_{r}}\big[\epsilon\,n_{k_{r}}(\xib_{k_{r}})\big]^{\[\gammab_{j}\]/e(k_{r},j)}\\
&\leq
C_{\gammab_{j}}\prod_{j\in I_{r}}n_{k_{r}}(\xib_{k_{r}})^{-\[\gammab_{j}\]/e(k_{r},j)}\,\prod_{j\in I_{r}}\big[\epsilon\,n_{k_{r}}(\xib_{k_{r}})\big]^{\[\gammab_{j}\]/e(k_{r},j)}
\leq C_{\gammab_{j}}.
\eeas
Thus $f\in\C^{\infty}(\R^{N})$ is supported in the unit ball, and pure derivatives are bounded. It follows from Proposition \ref{Prop9.1} that all derivatives are bounded by constants depending only on the $\{C_{\gammab_{j}}\}$. But for $n_{j}(\etab_{j})\leq \frac{1}{2}$, we have
\beas
\partial^{\gammab_{1}}_{\etab_{1}}\cdots \partial^{\gammab_{n}}_{\etab_{n}}f(\etab)
=
[\partial^{\gammab_{1}}_{\etab_{1}}\cdots \partial^{\gammab_{n}}_{\etab_{n}}m](\v_{1}, \ldots, \v_{n})
 \prod_{r=1}^{s} n_{k_{r}}(\xib_{k_{r}})^{\sum_{j\in I_{r}}\[\gammab_{j}\]/e(k_{r},j)}.
\eeas
where $\v_{j}= \xib_{j}+ [\epsilon n_{k_{r}}(\xib_{k_{r}})]^{1/e(k_{r},j)}\cdot\etab_{j}$. When $\etab=0$, the uniform bounds on $f$ give the desired estimates for the derivatives of $m$, and this completes the proof.
\end{proof}


\section{Extended kernels and operators}\label{6:kernels}
\medskip

Operators of the form $T_{\KK}f=f*\KK$ with $\KK\in \PP_0(\EEE)$ are \textquotedblleft constant coefficient\textquotedblright operators. In this section we consider a more general class of operators which have \textquotedblleft variable coefficients\textquotedblright. Identify the space $\bR^N$ with a nilpotent Lie group $G\cong\R^N$ as in Section \ref{Groups}, and consider kernels of the form $K(\x, \z)$, with $(\mbx, \mbz) \in G\times G$, so that for each $\mbx \in G$, $K (\mbx, \cdot )\in \PP(\EEE)$, with smooth dependence on $\mbx$.    More precisely, assume that for each norm $|\,\cdot\,|_{M}$ defined in (\ref{2.8iou}) and each~$\beta$,
\begin{equation}\label{beta-norms}
 \sup_\mbx\limits \|  X^\beta \KK(\mbx, \cdot ) \|_M< \infty\ ,
\end{equation} 
where $X^\beta$ is a monomial in the left-invariant vector fields of $G$ acting on the $\x$-variable. We also assume (unless $G=\R^N$ as an additive group) the hypotheses of Theorems \ref{Thm9.8} and \ref{Thm10.1} to guarantee that proper kernels define bounded convolution on $L^p$, $1<p<\infty$ and form a class closed under convolution.

Kernels of this kind will be said to belong to the ``extended class'', to distinguish them from  the ``proper'' kernels in $\PP_0(\EEE)$  which are independent of $\mbx$. Operators corresponding to these extended kernels are given formally by
\begin{equation}\label{6.2EQ}
   Tf(\mbx) = \int_{G} K(\mbx,\mby^{-1}\mbx) f(\mby)d\mby;
\end{equation}
that is, if $f\in \SS(\R^{N})$, $Tf(\mbx) = \langle K(\mbx, \cdot ), f_\mbx \rangle$, with  $f_\mbx(\mby) = f(\mbx \mby^{-1})$. 

Note that the extended class of operators is left-invariant in the following sense. If $\a\in G$ set $L_{\a}f(\x)= f(\a^{-1}\x)$. Then  $L^{-1}_\a TL_\a$ is an operator again of the form \eqref{6.2EQ}, with $K(\mbx,\mbz)$ replaced by $K(\a \mbx,\mbz)$. Observe that the semi-norms  \eqref{beta-norms} of $K(\mbx,\mbz)$ and $K(\a\mbx, \mbz)$ are the same.
\medskip

Our main results concerning this extended class of kernels and operators are as follows.

\begin{theorem}\label{the:6.1}
   Each operator $T$ given by ~(\ref{6.2EQ}) is bounded on $L^p(\mathbb{R}^N )$, \ for $1 < p < \infty$.
\end{theorem}

\begin{theorem}\label{the:6.2}
   The class of these operators form an algebra under addition and composition of operators.
\end{theorem}

\subsection{The $L^p$ boundedness of the operators}\label{sec:6.1}

Recall that if $\delta_{r}:G\to G$ is the one-parameter group of automorphic dilations of the group $G$, the coordinates on $\R^{N}$ are chosen so that $\delta_{r}(x_{1}, \ldots, x_{N})=(r^{d_{1}}x_{1}, \ldots, r^{d_{N}}x_{N})$ with $d_{1}\leq d_{2}\leq \cdots\leq d_{N}$. Let $|\,\cdot\,|_{G}$ be a smooth homogeneous norm,  so that  $|\delta_{r}\x|_{G}= r\,|\x|_{G}$. The homogeneous dimension of $G$ is $Q=\sum_{j=1}^{N}d_{j}$, and $d(\x,\y)=|\y^{-1}\x|_{G}$ is a left-invariant quasi-distance. Let $B_{r}(\x) =\big\{\y\in G:d(\x,\y)<r\big\}$. Then $\big|B_{r}(\x)\big|=\big|B_{r}(0)\big| =r^{Q}\big|B_{1}(0)\big|$. We will write $B_{r}(0)=B_{r}$.

In proving Theorem \ref{the:6.1}, we make a preliminary reduction to kernels $\KK(\mbx, \mbz)$ that are supported when $\mbz \in B_1$. In fact let $\eta$ be a $C^{\infty}$ function supported in $B_1$ which equals 1 when $\mbz \in B_{1/2}$. Then writing $\KK(\mbx, \mbz) = \KK_0 (\mbx, \mbz) + \KK_\infty (\mbx, \mbz) = \KK(\mbx,  \mbz)\eta (\mbz) + \KK(\mbx, \mbz)\big(1- \eta (\mbz)\big)$, we see that $K_\infty (\mbx, \mbz)$ is bounded and rapidly decreasing as $| \mbz | \to \infty$. Now the operator $T_\infty$, that is \eqref{6.2EQ} with $\KK_\infty$ in place at $\KK$, can be written as $T_\infty (f)(\mbx) = \int K_\infty (\mbx,\mby)f(\mbx \mby^{-1})d\mby$, \ and hence, by Minkowski's inequality, is bounded on $L^p$, for all $1 \leq p \leq \infty$.

Thus suppose  $\KK(\mbx, \mbz)$ is supported when $\mbz \in B_1$. We  make a further assumption that $\KK(\mbx, \mbz)$ is supported when $\mbx \in B_\lambda$, where $\lambda$ is so small so that $B_{2 \lambda}$ is strictly contained in the torus $\mathbb{T} = \left\{  \mbx : |x_j | < \frac{1}{2} , \ \ j=1 ,\dotsc, N    \right\}$. (Later we will lift this hypothesis.) In this case we can extend $\KK$ outside of $\T$ by periodicity, and expand $\KK(\x,\z)$ as Fourier series in the $\mbx$-variable, 
\begin{equation*}
   \KK(\mbx, \mbz) = \sum_{\mbk \in \mathbb{Z}^N}\limits e^{2 \pi i\mbk\cdot\mbx}  \HH_\mbk (\mbz)\ ,
\end{equation*}
   with
\begin{equation}\label{fourier}
   \HH_\mbk (\mbz) = \frac{1}{(2 \pi )^N} \int_{\mathbb{T}}\limits e^{-2 \pi i\mbk\cdot\mbx} K(\mbx, \mbz)d\mbx\ .
\end{equation}

Each kernel $\HH_\mbk(\mbz)$ can be written as $\al_\mbk \KK_\mbk(\mbz)$ in such a way that the $\al_\mbk$ are rapidly decreasing coefficients and the $\KK_\mbk$ are  proper kernels, supported for $\mbz \in B_1$, whose semi-norms $\| \KK_\mbk \|_M$ are uniformly bounded in $\mbk$ for every $M$.
This can be done via the following diagonal argument. Integrating by parts in \eqref{fourier}, we obtain that, for every integer $q$,
$$
\|\HH_\mbk\|_M\le \frac{C_q}{|\mbk|^q}\|\KK\|_{q,M}\ ,
$$
where $\|\KK\|_{q,M}$ denotes the supremum of the semi-norms implied by \eqref{beta-norms} over $|\beta|\le q$. Then, given $q$, we inductively choose increasing integers $m_q$ such that
$$
\frac{C_q}{m_q^q}\|\KK\|_{q,q}<\frac1{m_q^{q/2}}\ .
$$
For $m_q\le |\mbk|<m_{q+1}$, we set 
$$
\al_\mbk=\frac{C_q}{|\mbk|^q}\|\KK\|_{q,q}\ .
$$
Then $\al_\mbk<|\mbk|^{-q/2}$, showing that the $\al_\mbk$ are rapidly decreasing, and the kernels $\KK_\mbk=\al_\mbk\inv \HH_\mbk$ are uniformly bounded in every $M$-norm (we remark that the $M$-norms are increasing in $M$). Upon relabeling the series and with a slight abuse of notation, this gives
\begin{equation}\label{EQ:6.3}
   \KK(\mbx, \mbz) = \sum^\infty_{n=1}\limits a_n (\mbx) \KK_n (\mbz)\ .
\end{equation}
Here $a_n(\mbx) = \eta (\mbx)\al_{\mbk}e^{2 \pi i\mbk\cdot\mbx}$, $\KK_n(\mbz) = \KK_\mbk(\mbz)$ (for a $\mbk$ that depends on $n$); also $\eta$ is a suitable $C^\infty$ cut-off function supported in $\mbx \in B_\lambda$.  With this choice,  the $a_n (\mbx)$ are a rapidly decreasing sequence of $C^\infty$ functions, each supported in $B_\lambda$.
Clearly now the operator $T$ whose kernel is given by (\ref{EQ:6.3}) is bounded on $L^p$, in view of Theorem \ref{Thm9.8}.

Now, with the above $\lambda$ fixed, we will drop the restriction that $\KK(\mbx, \mbz)$ is supported for $\mbx \in B_\lambda$, by a further decomposition which writes $\mathbb{R}^N$ as a union of balls $B_\lambda (\mbx_m ) = \mbx_m \cdot B_\lambda$, with these balls chosen so they are \textquotedblleft almost\textquotedblright  disjoint. For this we need the following lemma.
\begin{lemma}\label{LEM:6.3}
 Given $\lambda > 0$, there is a collection $\left\{ B_\lambda (\mbx_m )  \right\}$
so that  
\begin{enumerate}[{\rm(1)}]
\item $\left\{ B_{\lambda/2 }(\mbx_m )  \right\}$  covers $\mathbb{R}^N $;
\item for every $\mu > 0$, the collection $\left\{ B_\mu (\mbx_m )  \right\}$ has the bounded intersection property: there exists an integer $M = M_\mu$ so that no point $\mbx \in \mathbb{R}^M$ is contained in more than $M$ of the $\left\{ B_\mu (\mbx_m )  \right\}$.
\end{enumerate}
 \end{lemma} 
 
 \begin{proof}  
Fix a constant $c_1$ (which is small with $\lambda$, and which will be chosen momentarily) and choose the centers $\mbx_m$ so that the balls  $\left\{ B_{c_1} (\mbx_m )  \right\}$ are disjoint, and so that this is a maximal family with respect to that property. Now if $\operatorname*{\cup}_m\limits B_{\lambda/2} (\mbx_m )$ does not cover $\mathbb{R}^N$, there exists an $\mbx_0$, so that $d(\mbx_0 , \mbx_m ) \ge \lambda/2$ \ for all $m$. We then claim that $B_{c_1} (\mbx_0 )$ must be disjoint from all the $B_{c_1} (\mbx_m )$, violating the maximality. In fact if this disjointness fails, then there is an $\mbx^\prime \in B_{c_1} (\mbx_0 ) \cap B_{c_1} (\mbx_m )$ for some $m$. This means $d(\mbx^\prime , \mbx_0 ) < c_1$ and $d(\mbx^\prime , \mbx_m ) < c_1$, which by the quasi-triangle inequality $d(\mbx_0 , \mbx_m ) \leq c(d(\mbx^\prime , \mbx_0 ) + d(\mbx^\prime , \mbx_m ))$ gives a contradiction with \ $d(\mbx_0 , \mbx_m ) \geq \lambda/2$, when $\lambda/2 \geq 2 c c_1$. Thus if we choose $c_1 = \lambda/4c$ we get conclusion (1).

To prove the second conclusion we may choose $\mu \geq c_1$. Suppose $\mbx^\prime$ \ and \ $\mbx^{\prime\prime}$ belong to $B_\mu (\mbx_m )$. Then $d(\mbx^\prime , \mbx^{\prime\prime}) \leq 2 c\mu$, so $\mbx^{\prime\prime} \in B_{2 c \mu} (\mbx^\prime )$, and hence $B_\mu (\mbx_m ) \subset B_{2c\mu} (\mbx^\prime )$. So if $\mbx^\prime$ belongs to $M$ of the $B_\mu (\mbx_m )$, then $B_\mu (\mbx_m ) \subset B_{2c\mu} (\mbx^\prime )$ for each of these $m$. Since then $B_{c_1}(\mbx_m ) \subset B_{2c \mu}(\mbx^\prime )$ and the $B_{c_1}(\mbx_m )$ are disjoint, we have
\begin{equation*}
   \sum m \left( B_{c_1} (\mbx_m )  \right) \leq m \left(  B_{2c\mu}(\mbx^\prime )   \right) ,
\end{equation*}
 where the sum is taken over those $m$. By homogeneity we get $Mc_1^Q \leq (2c\mu )^Q$, and conclusion (2) is established with $M = (2c/c_1)^Q$.
 \end{proof}
 
We now return to the proof of Theorem \ref{the:6.1}. Let us fix a positive $C^\infty$ function $\varphi$ that equals $1 \text{ in } B_{\lambda/2}$ and vanishes outside $B_\lambda$. Let $\varphi_m (\mbx) = \varphi \left(  \mbx^{-1}_m \cdot \mbx  \right)$, with the $\mbx_m$ as given in Lemma ~(\ref{LEM:6.3}). Then because of conclusions (1) and (2) of the lemma, \ $1 \leq \sum \varphi_m (\mbx) \leq M$. Set $\psi_m (\mbx) =\varphi_m (\mbx)/\sum \varphi_m (\mbx)$. Since $\left\{  X^\alpha \varphi_m  \right\}$ are bounded uniformly in $m$, the same is true of $\{\psi_m \}$. Notice that the $\{ \psi_m \}$ give a partition of unity, with $\psi_m$ supported in $B_\lambda (\mbx_m )$.

We let $\KK_m (\mbx, \mbz) = \psi_m (\mbx) \KK(\mbx, \mbz)$ and $T_m (f)(\mbx) = \int \KK_m (\mbx,\mby^{-1}\mbx )f(\mby )d \mby$. Then $T = \sum T_m$. Observe that $T_m$ is of the form $L_{\mbx_m} T^\prime L^{-1}_{\mbx_m}$, where  $L_{\x}$ is left-translation by $\x$ and $T^\prime$ is an operator for which we already know its $L^p$ boundedness by Theorem \ref{Thm9.8}. Hence
\begin{equation}\label{EQ:6.4}
|| T_m (f) ||_{L^p} \leq A_p || f ||_{L^p},
\end{equation}
uniformly in $m$, \ \ for \ \ $1 < p < \infty$. Now because $T = \sum T_m$, the lemma (with $\mu = \lambda )$ gives us
\begin{equation*}
   \int | T |^p d\mbx \leq M_\lambda \sum_m\limits \int | T_m (f) |^p d\mbx
\end{equation*}
 However by ~(\ref{EQ:6.4}), 
\begin{equation*}
\int \left| T_m (f) \right|^p d\mbx \leq A^p_p \int_{B_\mu (\mbx_m )}\limits | f |^p d\mbx.
\end{equation*}
This is because $\mbx \in B_\lambda (\mbx_m )$ \ and \ $\mby^{-1} \mbx \in B_1 \ \text{ imply } \mby \in B_\mu (\mbx_m )$ here with \ $\mu = c(\lambda +1)$, in view of the fact that $d(\mby,\mbx_m ) \leq c(d(\mbx,\mbx_m ) + d(\mby,\mbx))$, and the support properties of $K_m (\mbx,\mbz)$. Again by conclusion (2) of Lemma ~\ref{LEM:6.3}, 
\begin{equation*}
   \sum_m\limits \int_{B_\mu (\mbx_m )}\limits | f |^p d\mbx \leq M_\mu \int | f |^p d\mbx,
\end{equation*}
giving the conclusion $|| T(f) ||_{L^p} \leq A_p^\prime || f ||_{L^p}$ \ and so proving Theorem \ref{the:6.1}.

\subsection{The algebra of operators}\label{sec:6.2}

Proceeding to the proof of Theorem ~\ref{the:6.2}, we first summarize and slightly rephrase the essential idea of the previous subsection. Whenever $\KK(\mbx, \mbz)$ is a kernel of the extended class that is supported for $| \mbz | \leq 1$, then we can write it as 
\begin{equation}\label{6.5EQ}
   \KK(\mbx, \mbz) = \sum_{n,m}\limits a_{n,m}(\mbx) \KK_{n,m} (\mbz)\ , 
\end{equation}
where
\begin{enumerate}[(a)]
   \item the $\KK_{n,m}$ are proper kernels, each supported for $\mbz$ in the unit ball, with  $M$-norms that are uniform in $n$ and $m$;
\item the $a_{n,m}(\mbx)$ are $C^{\infty}$ functions, each supported in $\mbx \in B_\lambda (\mbx_m )$;
\item  for every $\al$, $\sup_\mbx\limits \left|     X^\alpha a_{n,m}(\mbx)  \right|$ \ is rapidly decreasing as $n \to \infty$, uniformly in $m$.
\end{enumerate} 
\medskip

\indent To prove Theorem ~\ref{the:6.2} we study the composition of two operators $T \text{ and }S$ given by
\begin{equation*}
   T(f)(\mbx) = \int \KK(\mbx,\mby^{-1} \mbx)f(\mby)d\mby , \quad S(f)(\mbx) = \int \LL(\mbx,\mby^{-1}\mbx)f(\mby) d\mby
\end{equation*}
where $\KK$ and $\LL$ are both kernels of the extended class. We consider first their essential parts, that is, we restrict ourselves to $K(\mbx, \mbz)$ and $L(\mbx, \mbz)$ that are supported when $\mbz$ is in the unit ball $B_1$. Now by \eqref{6.5EQ}, $T = \sum_{n,m} T_{n,m}$ and $S = \sum_{n', m'} S_{n' , m'}$ where 
\beas
T_{n,m}(f)(\mbx) &= a_{n,m}(\mbx) (f *\KK_{n,m})(\mbx),&&&S_{n' , m'}(f)(\mbx) &= b_{n', m'}(\mbx)( f * \LL_{n' , m'})(\mbx).
\eeas
and where $\{ \KK_{n,m}\}$ and $\{ \LL_{n' ,m'} \}$ are uniform families of proper kernels, each supported in $| \mbz | \leq 1$.

We next observe that, for each $m$, \ $T_{n,m}\circ S_{n',m'} = 0$ unless $m'$ belongs to a subset of boundedly many $m'$ (depending on $m$). In fact, by the support properties of $a_{n,m}, \ b_{n^\prime , m^\prime}$, and the kernels $\KK_{n,m}$, the product $T_{n,m} \cdot S_{n^\prime , m^\prime}$ is non-vanishing only if the quasi-distance between $B_\lambda (\mbx_m )$ and $B_\lambda (\mbx_{m^\prime} )$ does not exceed 1. Thus by conclusion (2) of Lemma ~\ref{LEM:6.3}, this only happens for at most $M_\mu$ of the $m^\prime$, where $\mu = c^2 (2 \lambda + 1)$.

From this, and the rapid decrease of $a_{n,m} $ and $b_{n^\prime , m^\prime}$ as $n \to \infty$ and $n^\prime \to \infty$, we see that the analysis of the composition can be reduced to the case when both $T$ and $S$ are each one summand. So we write $T = a\circ U, \ S = b \circ V$, where  $U(f) = f* \KK$ and $V(f) = f* \LL$, with both $\KK$ and $\LL$ proper kernels supported for $| \mbz | \leq 1$. Now 
\bes
T\circ S = a \circ U\circ  b \circ V = a \circ b \circ U \circ V + a\circ[b, U] \circ V.
\ees 
The term $a \circ b \circ U \circ V$ is of the right form because it corresponds to the kernel $a(\mbx)b(\mbx) (\LL\ast \KK)(\mbz)$, and $\LL \ast \KK$ is a proper kernel by Theorem \ref{Thm10.1}. On the other hand,
\begin{equation}\label{[b,U]}
[b,U] (f)(\mbx) = \int \big( b(\mbx) - b(\mby)\big)  \KK(\mby^{-1}\mbx ) f(\mby)\,d\mby\ .
\end{equation}
Write $b(\mbx) - b(\mby) = c(\mbx, \mbz)$, \ with $\mbz=\mby^{-1}\mbx$; thus $c(\mbx, \mbz) = b(\mbx) - b(\mbx \mbz^{-1})$. Recall that $c(\mbx, \mbz)$ can be restricted to $\mbz \in B_1$, (the support of $K$), and so can be taken to have compact $\mbz$\textendash support. Also, $c(\mbx, \mbz)$ is jointly $C^\infty$, and $\sup_\mbx \left|  X^\alpha_\mbx \partial^\beta_\mbz c(\mbx, \mbz)   \right| < \infty$ , uniformly in $\mbz$.

At this stage we invoke the following lemma.

\begin{lemma}\label{LEM:6.4}
 Let $\KK\in \PP_{0}(\EEE)$  and $c\in \CC^{\infty}_{0}(\R^{N})$. Then $c \,\KK\in \PP_{0}(\EEE)$.
 \end{lemma} 
 
\begin{proof} We use the Fourier transform characterization of Theorems \ref{Thm6.1} and \ref{Thm6.2}: if  $m=\widehat{\KK}$, then $m\in \MM_{\infty}(\EEE)$ and
\bes
\widehat{c\,\KK} (\bxi ) = m \ast \widehat{c} \ (\bxi ) = \int_{\mathbb{R}^N} m(\bxi - \beeta ) \widehat{c} (\beeta )d \beeta.
\ees 
We have $\widehat N_{j}(\xib+\etab)\leq C\big[\widehat N_{j}(\xib)+\widehat N_{j}(\etab)\big]$, and if $|\xib|$ is the Euclidean norm of $\xib$, there are positive real numbers $a$ and $b$ so that $|\xib|^{a}\leq \widehat N_{j}(\xib)\leq |\xib|^{b}$ for $1 \leq j \leq n$ and all $|\xib|\geq 1$. Let $0<\delta<ab^{-1}$ and write
\beas
\partial^{\gammab} (m \ast \ \widehat{c} \ ) &= \int_{\| \beeta \| \leq \| \bxi \|^\delta}\partial^{\gammab}_\bxi m(\bxi - \beeta ) \ \widehat{c} \ ( \beeta ) d \beeta + \int_{\| \beeta \| \geq \| \bxi \|^\delta} \partial^{\gammab}_\bxi m(\bxi - \beeta ) \widehat{c}( \beeta ) d \beeta
=\text{ I + II}.
\eeas
If $|\etab|\leq |\xib|^{\delta}$ then $\widehat N_{j}(\xib-\etab)\approx \widehat N_{j}(\xib)$ for all $1\leq j \leq n$ and all sufficiently large $|\xib|$, and hence
\beas
\big| I\big| \leq \int_{|\etab|<|\xib|^{\delta}}\prod_{j=1}^{n}(1+\widehat N_{j}(\xib-\etab))^{-\[\gammab_{j}\]}\,\big|\widehat c(\etab)\big|\,d\etab\lesssim \prod_{j=1}^{n}\big(1+\widehat N_{j}(\xib)\big)^{-\[\gammab_{j}\]}.
\eeas
On the other hand, because of the rapid decrease of $\widehat{c}$,  
\bes
\big| II\big| \leq c_M | \bxi |^{-M} \leq c_{M} \,\prod^n_{j=1}\big( 1 + \widehat N_{j} (\bxi ) \big)^{-[\![\alpha_j ]\!]}\ 
\ees
if $M \geq \sum^n_{j=1} [\![ \alpha_j ]\!]$. As a result,
$c\,\KK\in \PP_{0}(\EEE)$.
\end{proof}

We now apply Lemma \ref{LEM:6.4} to the kernel $c(\mbx,\mbz)\KK(\mbz)$ for each fixed $\mbx$. This shows that $c(\mbx,\mbz)\KK(\mbz)$ is an extended kernel.
An examination of the argument guarantees that every semi-norm \eqref{beta-norms} of this kernel is controlled by only finitely many of the semi-norms $\sup_{\mbx, \mbz} \left| \de^\alpha_\mbx \partial^\beta_\mbz c(\mbx, \mbz)  \right|$ of $c$ and the $M$-norms of $\KK$.

Coming to the operator $[b,U] \circ V$, we see that 
\bes
([b, U] \circ V) (f)(\mbx)=\int \MM(\mbx,\mby^{-1} \mbx)f(\mby)d\mby,
\ees 
where $\MM(\mbx, \,\cdot \,) = \LL\ast \widetilde{K}(\mbx, \cdot )$, where $ V (f)= f \ast \LL$, and where $\widetilde{K} (\mbx,\mbz) = c(\mbx,\mbz)K(\mbz)$.
 
Then $\MM(\mbx, \mbz)$ is an extended kernel, in view of Theorem \ref{Thm10.1}; moreover any semi-norm \eqref{beta-norms} of $\MM$ is controlled by finitely many semi-norms of $b, \KK$ and $\LL$. Finally, the rapid decrease of the $a_{n,m} $, as $n \to \infty$ (and the corresponding decrease of the $ b_{n^\prime , m^\prime} $ as $n^\prime \to \infty )$ then assures that $T \circ S$ is an operator of the extended class under the assumption to both $\KK(\mbx, \mbz) \text{ and } \LL(\mbx,\mbz)$ are supported when $| \mbz | \leq 1$.

The case when either one or both $\KK(\mbx, \mbz)$ and $\LL(\mbx,\mbz)$ are supported for $\mbz$ away from the origin is a simpler version of the above argument. In effect, when $\KK$ is supported when $| \mbz | \leq 1$,  and $\LL$ is supported when $| \mbz| \geq c > 0$, then we need only use the facts that a convolution of a Schwartz function 
with a distribution of compact support is a Schwartz function, or that the convolution of two Schwartz functions is again a Schwartz function. This allows us to finish the proof of Theorem \ref{the:6.2}.

\section{The role of pseudo-differential operators}\label{7:pseudo}\vspace*{.1in}
In this section we study the interplay of pseudo-differential operators with the operators of the extended class treated in Section \ref{6:kernels}. While more general cases can also be treated, for simplicity we restrict ourselves to the situation where the homogeneous group underlying our Euclidean space $\mathbb{R}^N$ is a stratified group, and where the matrix $\EEE$ arises from the two-flag example in Theorem \ref{Thm7.7} where $a_{1}= \cdots = a_{n}=1$ and $b_j = \frac{1}{j}$. Thus we take 
\bes
\EEE =
\left[\begin{matrix}
1&2&3&\cdots&n-1&n\\
1&1&3/2&\cdots &(n-1)/2&n/2\\
1&1&1&\cdots&(n-1)/3&n/3\\
\vdots&\vdots&\vdots&\ddots&\vdots&\vdots\\
1&1&1&\cdots&1&n/(n-1)\\
1&1&1&\cdots&1&1\\
\end{matrix}\right].
\ees

In this case of a stratified group we have the double identification of $\mathbb{R}^N=\bR^{C_{1}}\oplus\cdots\oplus\bR^{C_{n}}$ with the underlying group and also (via exponential coordinates) with its Lie algebra $\mathfrak{g}$\index{g5frak@$\mathfrak{g}$}.  We also assume that the subspace $\bR^{C_k}=\mathfrak{h}_k$\index{h1frak@$\mathfrak{h}_k$} is the subspace homogeneous of degree $k$ under the automorphic dilations $\delta_{r}$. The fact that $\mathfrak{g}$ is stratified guarantees that $\mathfrak{h}_{k+1} = [\mathfrak{h}_k , \mathfrak{h}_1 ]$. Suppose the cardinality of $C_{k}$ is $c_{k}$. If $\x=(\x_{1}, \ldots, \x_{n})\in \R^{N}$ with $\x_{k}\in \R^{C_{k}}$, write $\x_{k}=(x_{k,1}, \ldots, x_{k,c_{k}})$. 
 The automorphic dilations on $\mathfrak h_{k}$ are given by $\delta_{\lambda}\x_{k}=(\lambda^{k}x_{k,1}, \ldots, \lambda^{k}x_{k,c_{k}})$, the homogeneous dimension of $\mathfrak h_{k}$ is $k\,c_{k}$,  and for $\alphab_{k}\in\bN^{c_k}$, we have $\[\alphab_{k}\]=k\,|\alphab_{k}|$, where $|\alphab_{k}|$ is the standard length of the multi-index.

We take the homogeneous norm on $\mathfrak h_{k}$ to be $n_{k}( \t_{k}) = \| \t_{k} \|^{\frac{1}{k}}=\Big[\sum_{j=1}^{C_{k}}|t_{k,j}|^{2}\Big]^{1/2k}$, where  $|\,\cdot\, |$ denotes the Euclidean norm. Then
\bea\label{13.1iou}
N_{j}(\x) &= |\x_{1}|+|\x_{2}|^{\frac{1}{2}}+ \cdots +|\x_{j-1}|^{\frac{1}{j-1}}+|\x_{j}|^{\frac{1}{j}}+|\x_{j+1}|^{\frac{1}{j}}+ \cdots +|\x_{n}|^{\frac{1}{j}}
\\
\widehat N_{j}(\xib) &= |\xib_{1}|^{\frac{1}{j}}+|\xib_{2}|^{\frac{1}{j}}+ \cdots |\xib_{j-1}|^{\frac{1}{j}}+|\xib_{j}|^{\frac{1}{j}}+|\xib_{j+1}|^{\frac{1}{j+1}}+ \cdots +|\xib_{n}|^{\frac{1}{n}}
\eea
Notice that $N_n=\widehat N_0$ is a homogeneous norm for the automorphic dilations $\delta_\lambda$, and $N_{1}$ is equal to the Euclidean norm.

We will consider standard pseudo-differential operators of order 0. By this we mean operators of the form
\begin{equation}\label{pseudo}
   f \mapsto P(f)(\mbx) = \int_{\mathbb{R}^N}\limits a(\mbx, \bxi ) \ \widehat{f} \ (\bxi ) e^{2 \pi i \mbx \cdot\bxi}d \bxi,
\end{equation}
defined for Schwartz functions $f$, where the symbol $a(\mbx, \bxi ) $ is assumed to be of compact support in $\mbx$ and satisfy the differential inequalities 
\begin{equation}
   \left| \partial^\beta_\mbx \partial^\alpha_\bxi a(\mbx, \bxi ) \right| \leq A_{\alpha , \beta} (1 + \| \bxi \| )^{- | \alpha |}\ .
\end{equation}

In what follows the errors introduced by commutators will be \textit{smoothing operators} in the following sense.  For $1<p<\infty$,  the space $L^p_1$ consists of all $f$ in $L^p$ so that $Xf$ (taken in the weak sense) belongs to $L^p(G)$, for each left-invariant vector field $X$ of the generating sub-space $\mathfrak{h}_1$ of $\mathfrak{g}$. We define a norm on $L^p_1$ by setting 
\begin{equation}
   || f ||_{L_1^p} = || f ||_{L^p} + \sum_{j=1}^{g} || X_j (f) ||_{L^p} ,
\end{equation}
where $\{X_{1}, \ldots, X_{g}\}$ a basis for the subspace $\mathfrak{h}_1$. We shall also see below that an $f \in L^p$ belongs to $L^p_1$ if and only if
\begin{equation}\label{lipschitz}
   || f(\mbx  \h) - f(\mbx) ||_{L^p} = O\big(N_n( \h )\big), \ \text{ as }\  \h  \to 0.
\end{equation}
An operator will be said to be \textit{smoothing} if it maps $L^p$ to $L^p_1$, \ for all $1 < p < \infty$. 

The following theorem will be proved in the next sections.

\begin{theorem}\label{the:7.1}
 Suppose $P$ is a standard pseudo-differential operator of order 0. Then $P$ belongs to the extended class treated in Section \ref{6:kernels}. Moreover if $T$ is any operator of this class, then the commutator $[T, P]$ is a smoothing operator.
\end{theorem}

\subsection{The isotropic extended kernels}\label{sec:7.1}

For the proof of Theorem \ref{the:7.1} it is necessary to consider a special sub-class $\CC\ZZ_0\subset\PP_{0}(\EEE)$\index{C3Zzero@$\CC\ZZ_0$}, (the isotropic Calder\'{o}n-Zygmund kernels on $\R^N$ which agree with a Schwartz function at infinity) and a corresponding sub-class of the ``extended'' kernels. A distribution $\KK\in\CC\ZZ_0$ is assumed to satisfy the stronger inequalities
\begin{equation}\label{isotropic-ineq}
   \left| \partial^{\gammab} K(\mbz)   \right| \leq A_\alpha \| \mbz \|^{-N - | \gammab |}\big(1+|\z|\big)^{-M}\ ,
\end{equation}
for every $\gammab$ and $M$. Moreover these kernels are required to satisfy only the one type of cancellation property which arises when $n=1$.  It follows from Theorems \ref{Thm6.1} and \ref{Thm6.2} (for the case $n=1$) that this is equivalent to the condition that the Fourier transform $m=\widehat\KK$ satisfies $|\partial^{\gammab}m(\xib)|\leq A_{\alpha}'(1+|\xib|)^{-|\gammab|}$. But then it follows from equation (\ref{13.1iou}) that  $m\in \MM_{\infty}(\EEE)$, and hence $\CC\ZZ_0 \subset \PP_{0}(\EEE)$. We call elements of $\CC\ZZ_0$ \textit{isotropic proper kernels}.

Together with $\CC\ZZ_0$ we consider the extended kernels $\KK(\mbx, \mbz)$ such that for $\mbx$  fixed, $\KK(\x,\,\cdot\,)\in \CC\ZZ_0$. More precisely, we shall assume that these have compact support in $\mbx$ and all $\mbx$-derivatives $\de_\mbx^\al \KK(\mbx,\,\cdot\,)$ are isotropic proper kernels uniformly in $\mbx$. Notice that, having assumed compact support in $\mbx$, this condition can be equivalently formulated with the $\mbx$-derivatives replaced by left-invariant, or right-invariant, vector fields. We will refer to these kernels as \textit{isotropic extended kernels.}

The key fact we need is the following.
\begin{proposition}\label{pseudo-alt}
   Suppose $P$ is a pseudo-differential operator \eqref{pseudo} of order 0, with symbol $a(\mbx, \bxi )$ having compact support in $\mbx$. Then $P$ can be represented in the following three alternative forms: for any Schwartz function $f$:
\begin{enumerate}[{\rm (a)}]
\item \label{pseudo-alta} $P(f)(\mbx) = \int_{\mathbb{R}^N}\limits \KK(\mbx, \mbx-\mby) f(\mby) d\mby$;
         
\item \label{pseudo-altb}  $P(f)(\mbx) = \int_{\mathbb{R}^N}\limits \KK^L(\mbx, \mby^{-1} \mbx) f(\mby) d\mby$;
    
\item \label{pseudo-altc}  $P(f)(\mbx) = \int_{\mathbb{R}^N}\limits \KK^R(\mbx, \mbx \mby^{-1}) f(\mby) d\mby$.
\end{enumerate}
Here $\KK, \ \KK^L$, and $\KK^R$ are each isotropic extended kernels.
\end{proposition}
The realization (\ref{pseudo-alta}) is well-known; see for instance \cite{MR0290095}, Sections 4 and 7.4 of Chapter 6. In this case $\widehat \KK(\mbx, \cdot )= a(\mbx , \cdot )$ (where the Fourier transform acts on the second variable), and reduces to the fact that a multiplier $m(\bxi )$ that satisfies the differential inequalities $\left| \partial^\alpha_\bxi m(\bxi ) \right|\le A_\alpha (1 + \| \bxi \| )^{- | \alpha |}$ arises as the Fourier transform of an isotropic proper kernel, and so this is covered by Theorems \ref{Thm6.1} and \ref{Thm6.2}. The other two representations need further analysis.

First we assert that
\begin{equation}\label{yinvx}
   \mby^{-1}  \mbx = L_{\mbx, \mby} (\mbx - \mby),
\end{equation}
where for each $(\mbx, \mby)\in G\times G$, $L_{\mbx, \mby}$ is a linear transformation that depends polynomially on $\mbx$ and $\mby$. In fact, by the homogeneity of our group and using the Campbell-Hausdorff formula we can write for each $1 \leq k \leq n$ and each $1 \leq j \leq c_{k}$,
\begin{equation*}
   \left( \mby^{-1} \mbx   \right)_{k,j} = x_{k,j} - y_{k,j} + \sum_{\ell < k}\sum_{j=1}^{c_\ell}  p^{\ell,j}_k (\mbx,\mby) \cdot (x_{\ell,j} - y_{\ell,j} )
\end{equation*}
with $p^{\ell,j}_k$ polynomials in $\mbx$ and $\mby$ jointly homogeneous at degree $k - \ell$ (see for example Section 6.6 in \cite{MR2949616}, or Chapter 1 of \cite{MR657581}). This gives \eqref{yinvx} with $L_{\mbx, \mby}$ represented by a block-triangular matrix which in particular implies that $\det(L_{\mbx, \mby}) = 1$, \ \ for all $\mbx$ and $\mby$. Now let $M_{\mbx, \mby}$ denote the inverse of $L_{\mbx, \mby}$. Then $\mbx - \mby = M_{\mbx, \mby}(\mby^{-1} \mbx)$ and $M_{\mbx, \mby}$ is again polynomially dependent on $\mbx$ and $\mby$.

Define a (non-linear) mapping $\Phi = \Phi_\mbx:\R^{N}\to \R^{N}$ by setting $\Phi_\mbx (\mbz) = M_{\mbx, \mbx \mbz^{-1}} (\mbz)$. Given $\x$ and $\w$,
\bes
\Phi_{\x}\big((\x-\w)^{-1}\x\big) = M_{\x,\x\x^{-1}(\x-\w)}\big((\x-\w)^{-1}\x\big)= \x-(\x-\w)=\w
\ees
so $\Phi_{\x}$ maps $\R^{N}$ onto $\R^{N}$ and the inverse of $\Phi_{x}$ is given by $\Phi_{\x}^{-1}(\w) =(\x-\w)^{-1}\x=L_{\x,\x-\w}(\w)$. Thus both $\Phi_\mbx$ and $\Phi_\mbx^{-1}$ are polynomial mappings of $\mathbb{R}^N$ to itself (and hence diffeomorphisms) that each depend polynomially on $\mbx$. In our applications of these mappings below, the range of $\mbx$ is restricted to a compact set, in view of the support hypothesis on the symbol $a(\mbx,\bxi)$ and the resulting fact that $\KK(\mbx, \mbz)$ has compact-support in $\mbx$. 

We set down the following two simple properties of the mappings $\Phi_\mbx$:
\begin{enumerate}[1)]
\item $ \left\| \Phi_\mbx (\mbz)  \right\| \approx \| \mbz \|$, for  $\| \mbz \| \leq 1$;
\smallskip
\item $\| \mbz \|^{\frac{1}{n}} \lesssim \left\| \Phi_\mbx (\mbz) \right\| \lesssim \| \mbz \|^n , \quad \text{ for } \quad \| \mbz \| \geq 1$.
\end{enumerate}
The first assertion is a consequence of the fact that $\Phi$ is a diffeomorphism that fixes the origin. For the second, we see first that $\| \Phi_\mbx (\mbz) \| \leq c \| \mbz \|^n$, $ \| \mbz \| \geq 1$ because of the polynomial dependence (of degree $n$) of the linear mappings $M_{\mbx,\mby}$. Since the same inequality holds for the inverse $\Phi^{-1}_\mbx$ we get the reverse inequality $\| \mbz \|^{\frac{1}{n}} \leq c \| \Phi_\mbx (\mbz) \|$, again for large $|\mbz|$.

\medskip

Our proof of the second representation (\ref{pseudo-altb}) in Proposition \ref{pseudo-alt} starts with knowing the first representation and defining $\KK^L(\mbx,\mbz) = \KK(\mbx, \Phi_\mbx (\mbz))$. Then, at least formally, we have
\begin{equation}
   \int \KK^L (\mbx, \mby^{-1} \mbx)f(\mby) d\mby = \int \KK(\mbx, \mbx-\mby) f(\mby)d\mby = P(f)(\mbx).
\end{equation}
In order to justify this it will suffice to see that $\KK^L (\mbx, \mbz)$ is, for each $\mbx$, an isotropic proper kernel, and (in both variables) an isotropic extended kernel. For this we make the following observation.

\begin{lemma}\label{LEM:7.3}
 Let $\KK(\mbz)\in \CC\ZZ_0$. Suppose $\Phi$ is a $C^\infty$ diffeomorphism of $\mathbb{R}^N$ fixing the origin, and assume that $\| \Phi (\mbz) \| \gtrsim | \mbz \|^\delta$, for some $\delta > 0$ and all $\| \mbz \| \geq 1$. Then $\KK\circ\Phi\in \CC\ZZ_0$ where the distribution $\KK\circ\Phi$ is defined by $\big\langle\KK\circ\Phi, \ \varphi \rangle = \langle \KK, \varphi ( \Phi^{-1} )J_{\Phi\inv} \big\rangle$ for every test function $\varphi$, where $J_\Phi$ is the Jacobian  determinant of $\Phi$.
\end{lemma} 
\begin{proof}
Away from the origin $\KK\circ\Phi$ is  the $C^\infty$ function $K(\Phi (\mbz))$. Since $\Phi(0) = 0$ it follows that $\| \Phi (\mbz) \| \gtrsim c \| \mbz \|$ for small $|\mbz|$, and hence $\KK\circ\Phi$  satisfies \eqref{isotropic-ineq} on the unit ball because $K$ does.  Moreover the assumption $\| \Phi (\mbz) \| \geq c \| \mbz \|^\delta$ for $\| \mbz \| \geq 1$, implies the rapid decrease at infinity of $K(\Phi (\mbz))$, since the same holds for $K(\mbz)$. It remains to verify the cancellation conditions for $\KK\circ\Phi $, \textit{i.e.} $\big\vert\big\langle \KK\circ\Phi, \varphi(R\,\cdot\,)\big\rangle\big\vert \leq C$ for all $R\geq 1$ wherever $\varphi$ is a normalized $C^\infty$ function supported in the unit ball. However $\big\vert\big\langle \KK\circ\Phi, \varphi(R\,\cdot\,)\big\rangle\big\vert=\big\langle\KK,\Psi_{R}\big\rangle$, where it is easy to verify that for each $R$, $\Psi_R (\mbx)$ is supported in the ball $\| \mbx \| \leq cR\inv$, and satisfies$ \ \ | \partial^\alpha_\mbx \Psi_R (\mbx) | \leq c_\alpha R^{ | \alpha |}$, with bounds $c \text{ and }c_\alpha$ independent of $R$. Thus we may think of $\Psi_R (\mbx)$ as of the form $c_1 \Psi (c_2 R\mbx )$, for some $c_1$ and $c_2$ independent of $R$, verifying that $\KK\circ\Phi$ satisfies the requisite cancellation condition. Thus the proof of Lemma \ref{LEM:7.3} is complete. 
\end{proof}

Hence we have proved the representation (\ref{pseudo-altb}) of Proposition \ref{pseudo-alt}. The representation (\ref{pseudo-altc}) is proved in the same way.

\subsection{Proof of Theorem \ref{the:7.1}}\label{sec:7.2}

Turning to the proof of Theorem \ref{the:7.1}, we see by the identity (\ref{pseudo-altb}) in Proposition \ref{pseudo-alt} that the pseudo-differential operator $P$ belongs to the extended class of operators, establishing the first assertion of Theorem \ref{the:7.1}. Now let $T$ be any operator in the extended class, 
\begin{equation*}
   T(f)(\mbx) = \int \LL(\mbx, \mby^{-1}\mbx) f(\mby) d\mby,
\end{equation*}
with $\LL(\mbx, \mbz)$ a kernel of the extended class. We examine the commutator $[T, P]$. Let $B$ be a ball containing the $\mbx$\textendash support of the symbol $a(\mbx, \bxi )$ of $P$, and hence containing the $\mbx$-supports of $\KK(\mbx, \mbz)$, $\KK^L (\mbx, \mbz)$ and $\KK^R (\mbx, \mbz)$. We now first decompose $T$ as $T_0 + T_\infty$ via $\LL(\mbx, \mbz) = \LL_0 (\mbx, \mbz) + \LL_\infty (\mbx, \mbz)$ where $\LL_0 (\mbx, \mbz) = \eta (\mbx) \LL(\mbx, \mbz), \ \LL_\infty (\mbx, \mbz) = \big(1 - \eta (\mbx)\big)\LL_\infty (\mbx, \mbz)$ with $\eta \in C^\infty$ of compact support which = 1 on the double $B^\ast$ at the ball $B$. Thus $\KK_0 (\mbx, \mbz)$ has compact $\mbx$\textendash support, while $\KK_\infty (\mbx, \mbz)$ is supported on the complement of $B^\ast$. (Note that this is different from the decomposition $T = T_0 + T_\infty$ in Section \ref{sec:6.1}).

Now by reasoning as in Section \ref{sec:6.1}, we can write
\begin{equation}\label{T_0-sum}
   T_0 (f) (\x)= \sum_n\limits T_n (f)(\mbx) = \sum_n\limits a_n (\mbx)(f \ast \HH_n)(\x)
\end{equation}
where $a_n (x)$ are $C^\infty$ functions, supported in a common compact set, for which $\sup_\mbx | \partial^\alpha_\mbx a_n (\mbx) |$ is rapidly decreasing in $n$, for each fixed $\alpha$. The kernels $\HH_n$ are proper kernels and their semi-norms are each uniformly bounded in $n$.

For the pseudo-differential operator $P$ we shall use the \textquotedblleft right\textendash invariant\textquotedblright \ representation (\ref{pseudo-altc}) in Proposition~\ref{pseudo-alt}. This gives by the reasoning establishing \eqref{T_0-sum}, that 
\begin{equation}\label{P-sum}
   P(f)(\x) = \sum_{n^\prime}\limits P_{n^\prime} (f)(\mbx) = \sum_{n^\prime}\limits b_{n^\prime} (\mbx) (\KK_{n^\prime} \ast f)(\mbx)\ .
\end{equation}
Here the $b_{n^\prime} (\mbx)$ are $C^\infty$ function, supported in a compact set (in view of the support property of $\KK^R (\mbx, \mbz))$, for which $\sup_\mbx | \partial^\alpha_\mbx b_{n^\prime} (\mbx) |$ is rapidly decreasing in $n^\prime$ for each fixed $\alpha$, and the $\KK_{n^\prime}$ are isotropic proper kernels and their semi-norms are each uniformly bounded in $n^\prime$. The rapid decrease in $n$ and $n^\prime$ pointed out above reduces the study of the commutator $[ T_0 , P]$ to that where only a single term of each of the sums, \eqref{T_0-sum} and \eqref{P-sum} appears. 

Thus we look at the commutator $[a \circ U,  b\circ V ]$,
where $U(f) = f * \HH$ and $ V(f) = \KK * f$.
This commutator is handled by the following

\begin{lemma}\label{LEM:7.4}\quad
 \begin{enumerate}[{\rm(1)}] 
\item $U$ and $V$ commute, 
\item $[b, U]$ is a smoothing operator,
\item $[a, V]$ is a smoothing operator,
 \end{enumerate}
 and the bounds of the mappings $[b, U]$ and $[a, V]$ depend only on finitely many of the semi-norms controlling $a, b, \HH, \text{ and }\KK$.
 \end{lemma} 

Suppose first that the lemma is proved. There $aU \circ bV = ab U \circ V + a[b, U ] V$. However $U \circ V = V \circ U$, by conclusion (1). Thus
$$
\begin{aligned}{}
[a \circ U,  b\circ V ]&=a\circ[U,  b\circ V ]+[a,  b\circ V ]\circ U\\
&=a\circ[U,  b]\circ V +ab\circ[U,  V ]+b\circ [a,  V ]\circ U+[a,  b ]\circ V\circ U\\
&=a\circ[U,  b]\circ V +b\circ [a,  V ]\circ U\ .
\end{aligned}
$$

As well as multiplication by $a$ and $b$, both $U$ and $V$ map $L^p$ to itself, since they are convolution operators with proper kernels. Thus by conclusions (2) and (3) of the lemma, $[a \circ U,  b\circ V ]$ maps $L^p$ to $L^p_1$, and we have achieved our desired result for $[T_0 , P]$.

\begin{proof} [Proof of Lemma \ref {LEM:7.4} ]
  The first conclusion (which is key) follows immediately because $U$ is left-invariant and $V$ is right-invariant. Let us now consider $[b, U]$, and begin with the representation given for $[b, U]$ in \eqref{[b,U]} (here $\HH$ plays the role of $\KK$). We have that $[b, U]$ is an operator of the extended class, represented by the kernel $\widetilde{\HH}(\mbx, \mbz) = c(\mbx, \mbz) \HH(\mbx, \mbz)$, where $c(\mbx, \mbz)$ is the $C^\infty$ function given by $b(\mbx) - b(\mbx \mbz^{-1})$. Notice however that $c(\mbx, 0) = 0$, and thus 
\begin{equation*}
   c(\mbx, \mbz) = \sum^n_{k=1}\sum_{j=1}^{c_k} z_{k,j} \cdot c_{k,j} (\mbx, \mbz),
\end{equation*}
where $c_{k,j}$  are again $C^\infty$ functions. Hence
\begin{equation}\label{tildeH}
   \widetilde{\HH}(\mbx, \mbz) = \sum_{k,j} z_{k,j} \cdot \widetilde{\HH}_{k,j} (\mbx, \mbz) 
\end{equation}
where by Lemma \ref{LEM:6.4} we know that each $\widetilde{\HH}_{k,j} (\mbx, \mbz)$ is an extended kernel. Now if $X$ is a left-invariant vector field, then 
$$
\begin{aligned}
  X[b, U] (f)(\mbx) = X \left( \int \widetilde{\HH} (\mbx,\mby^{-1} \mbx) f(\mby)d\mby \right) = \int \HH' (\mbx,\mby^{-1}\mbx) f(\mby)d\mby \ ,
\end{aligned}
$$
where by \eqref{tildeH} we have that 
$$
\HH' (\mbx,\mbz) = \sum^n_{j=1} \Big(X_\mbx \big( z_j \cdot \widetilde{\HH}_j (\mbx,\mbz) \big) + X_\mbz  \big( z_j \cdot \widetilde{\HH}_j (\mbx,\mbz)\big)\Big)\ .
$$

\noindent At this stage we need the following observation.
\begin{lemma}\label{LEM:7.5}
 Suppose $\HH$ is a proper kernel and $X$ is a left-invariant vector field of degree 1 (that is, in the subspace $\mathfrak{h}_1$ of the Lie algebra). Then for each  $(k,j)$, $1 \leq  k \leq n$, $1\le j\le c_k$, $ X \big(z_{k,j} \HH(\mbz)\big)$ is a proper kernel.
\end{lemma}

This can be proved by using the Fourier transform characterization of proper kernels. First write $X$~as 
\begin{equation*}
 X=  \frac{\partial}{\partial z_{1,r}} + \sum^n_{\ell = 2}\sum_{s=1}^{c_\ell} q_{\ell,s} (\mbz) \frac{\partial}{\partial z_{\ell,s}}=\frac{\partial}{\partial z_{1,r}} + \sum^n_{\ell = 2}Q_\ell(\mbz)\cdot \nabla_{\mbz_\ell}
\end{equation*}
where $Q_\ell (\mbz) =  Q_\ell (\mbz_1 , \ \cdots , \ \mbz_{\ell - 1})$  are homogeneous polynomials of degree $\ell - 1$ (in the homogeneity of the automorphic dilations).  Consider first the case $k=1$. The Fourier transform of  $X\big(z_{1,j} \HH(\mbz)  \big)$ is then a linear combination of $m=\widehat \HH$  and $\bxi_\ell \cdot Q_\ell ( 2 \pi i \de_{\bxi})\, \de_{ \xi_{1,j}} m(\bxi )$. However,  it is clear from \eqref{13.1iou} that $\widehat N_j (\bxi )^k \geq \| \bxi_k \|$ if $j \leq k$. The the effect of the  factor $\bxi_\ell$, which has size $\|\bxi_\ell\|$, is counterbalanced by the effect of  $Q_\ell ( 2 \pi i \de_\bxi   ) \de_{\xi_{k,j}}$ on $m(\bxi )$, in view of the estimates for $m$. The situation when  $k>1$ (and hence $\de_{\xi_{1,j}}$ is replaced by $\de_{\xi_{k,j}}$) leads to an even better estimate, which then establishes Lemma \ref{LEM:7.5}.

\medskip

Using the lemma we see therefore that $X_\mbx [b, U]$ is an operator with an extended kernel, and hence by Theorem \ref{the:6.1} we have that $[b, U]$ maps $L^p$ to $L^p_1$, \ establishing conclusion (2) of Lemma \ref{LEM:7.4}.

Finally, for conclusion (3) we may appeal to the standard theory of pseudo-differential operators that guarantees that the commutator $[a, V]$ is an operator of \textquotedblleft order 1\textquotedblright, and in particular maps $L^p \text{ to }L^p_1$; or we may argue in the same spirit as in the case of conclusion (2), but here the details are simpler. In any case, Lemma 7.4 is now established and this proves that $[T_0 , P]$ is bounded from $L^p$ to $L^p_1$. 
\end{proof}

To complete the proof of Theorem \ref{the:7.1} it remains to deal with $[T_\infty , P] = T_\infty P - PT_\infty$. We will see that the terms $T_\infty P$ and $PT_\infty$ are separately smoothing operators (in fact, infinitely smoothing) because of disjointness of relevant supports. Consider first $T_\infty P$. Then $T_\infty (Pf) = T_\infty (F)$, where $F$ is supported in the ball $B$ (where the symbol $a(\mbx, \bxi )$ is supported). However
\begin{equation*} 
   T_\infty (F)(\mbx) = \int \KK_\infty (\mbx, \mby^{-1}\mbx) F(\mby) d\mby
\end{equation*}
and $\KK_\infty (\mbx, \mbz)$ is supported for $\mbx$ outside $B^\ast$. Hence in the integral above we have that $| \mby^{-1}\mbx | \geq c > 0$, for an appropriate $c$. Moreover if $X$ is any left\textendash invariant operator then $X(T_\infty F)(\mbx))$ is given by $\int \KK_\infty^\prime (\mbx, \mby^{-1}\mbx) F(\mby) d\mby$ with $K_\infty^\prime (\mbx, \mbz) = X_\mbx \KK_\infty (\mbx, \mbz) + X_\mbz \KK_\infty (\mbx, \mbz)$. Since now $\KK^\prime_\infty (\mbx, \mbz)$ is bounded and has rapid decay in $\mbz$, (we can restrict ourselves here to $| \mbz | \geq c > 0$), this ensures that $T_\infty P$ maps $L^p$ to $L^p_1$. In fact the same argument shows that $X^\alpha (T_\infty P)$ maps $L^p$ to itself for any monomial $X^\alpha$ of left\textendash invariant differential operators.

Essentially the same argument works for $PT_\infty$. Here $(PT_\infty )(f) = \int K^L (\mbx, \mby^{-1} \mbx) f(\mby) d\mby$ by Proposition \ref{pseudo-alt}, with $F$ supported in the complement of $B^\ast$ and $\KK(\mbx, \mbz)$ supported where $\mbx \in B$. Then again $| \mby^{-1} \mbx | \geq c > 0$, and we may proceed as before. Altogether then, the proof of Theorem \ref{the:7.1} is completed.

\subsection{The space $L^p_1$}

We conclude by proving the equivalence of the condition that $f \in L^p_1$ with the Lipschitz\textendash type condition~\eqref{lipschitz}.

\begin{proposition} Let $1 < p < \infty$. Then
\beas
f\in L^{p}_{1} \qquad\Longleftrightarrow \qquad \text{$|| f(\mbx \h) - f(\mbx) ||_{L_p} = O\big( N_n( \h ) \big)$, as $h \to 0$}.
\eeas
\end{proposition}

\begin{proof}
Let $X$ be a left-invariant vector field in the subspace $\mathfrak{h}_1$ of the Lie algebra, and set $\mbh_t = \exp (t X)$. Then $| \mbh_t |=\|\mbh_t\| \approx | t | , \ t \in \mathbb{R}$. So 
\beas
\text{$|| f(\mbx \h) - f(\mbx) ||_{L_p} = O\big( N_n( \h ) \big)$, as $h \to 0$}
\Longrightarrow
\text{$||\frac{1}{t} \big(f(\mbx \mbh_t ) - f(\mbx)\big)||_{L^{p}}= O(1)$ \ \ as $t \to 0.$}
\eeas
Thus by the weak compactness of $L^p$, there is a sequence $\{ t_n \} , \ t_n \to 0$ so that $\frac{1}{t_n} \big(  f(\mbx  \mbh_{t_n} ) - f(\mbx) \big) $ converges  weakly in $L^p$ and hence in the sense of distributions. Thus in that sense, $X(f) \in L^p$. Since this holds for all $X \in \mathfrak{h}_1$, we have that $f \in L^p_1$.

Conversely, suppose $f \in L^p_1$, \ and $\mbh = \exp (s X)$ where $X \in \mathfrak{h}_1$. Then $N_n( \h )=| \mbh | \lesssim \| s \|$  and
\bes
f(\mbx \mbh ) - f(\mbx) = \int^s_0 \frac{d}{dt} f\big(\mbx  \exp (t X)\big)\, dt = \int^s_0 Xf\big(\mbx \exp (t X)\big)\, dt\ .
\ees
Hence, $|| f(\mbx \mbh) - f(\mbx) ||_{L^p} \leq \| s \| \ || X(f) ||_{L^p} = O\big( N_n( \h ) \big)$.  To extend this to any element $\mbh$ we use the following assertion.

\begin{lemma}\label{LEM:7.7}
Suppose $G$ is a stratified group. Then there is an integer $M$, so that every $\mbh \in G$, can be written as $\mbh = \mbh_1  \mbh_2 \ \cdots \ \mbh_M$, where each $\mbh_j \in \exp \mathfrak{h}_1$ and $ | \h_j | \leq c N_n( \h )$ for all $1 \leq j \leq M$.
\end{lemma}

\noindent Assuming for a moment this to be true, the desired estimate for $f(\mbx \mbh) - f(\mbx)$ then follows for the special case when $\mbh \in\exp \fh_1$ by writing 
\bes
f(\mbx \mbh) - f(\mbx) = \sum^M_{j=1}\limits f(\mbx  \mbh_1 \cdots  \mbh_j ) - f(\mbx  \mbh_1 \cdots \mbh_{j-1} ).\qquad\qedhere
\ees
\end{proof}

\begin{proof}[Proof of Lemma \ref{LEM:7.7}]
Let $V=\exp(B_1)$, where $B_1$ is the unit ball in $\fh_1$, Then $V$ is an analytic submanifold of $G$ which generates $G$. By Proposition 1.1 of \cite{MR937632}, there exists an integer $m$ such that $V^m$ contains an open set $A$. Since $V=V\inv$, $V^{2m}$ contains a neighborhood of the identity. This proves the statement for $|\mbh|<\del$ for some $\del>0$ with $M=2m$. By homogeneity, the same holds for every $\mbh\in G$.
\end{proof}

\section{Appendix I: Properties of cones $\Gamma(\AAA)$}\label{Cones}

In this appendix we study properties of the cone $\Gamma(\EEE)$ defined in (\ref{5.1qwe}). Our results provide motivation for the basic hypothesis (\ref{2.5}) imposed on the matrix $\EEE$.

\subsection{Optimal inequalities and the basic hypothesis}

\begin{definition}
Let $\AAA = \big\{a(j,k):1\leq j,k\leq n\big\}$ be an $n\times n$ matrix of strictly positive real numbers. Then 
\beas
\Gamma^{o}(\AAA) &= \Big\{\t=(t_{1}, \ldots, t_{n})\in \R^{n}: a(j,k)t_{k}< t_{j}<0\Big\}\\
\Gamma(\AAA) &= \Big\{\t=(t_{1}, \ldots, t_{n})\in \R^{n}: a(j,k)t_{k}\leq t_{j}<0\Big\}
\eeas 
are (possibly empty) convex polyhedral cones contained in the negative orthant of $\R^{n}$. The dimension of either cone is the maximal number of linearly independent vectors it contains. Equivalently, this is the dimension of the smallest subspace of $\R^{n}$ containing the cone. 
\end{definition}

$\Gamma(\AAA)$ is defined by a collection of linear inequalities involving pairs of coordinates. It may happen that different matrices (\textit{i.e.} different sets of such inequalities) give rise to the same cone. The following Lemma shows that there is a unique optimal set of inequalities, and that the corresponding matrix satisfies the basic hypothesis (\ref{2.5}).

\begin{lemma}\label{Lem3.2}
Let $\AAA=\{a(j,k)\}$ be an $n\times n$ matrix of strictly positive real numbers, and suppose that the corresponding cone $\Gamma(\AAA)$ is not empty. Then there exists a unique $n\times n$ matrix $\widetilde{\AAA}=\big\{\tilde a(j,k)\big\}$ with the following properties:
\begin{enumerate}[{\rm(a)}]
\item \label{Lem3.2a}
The coefficients of $\widetilde{\AAA}$ satisfy the inequalities (\ref{2.5}) of the basic hypothesis:
\beas
1&= \tilde a(j,j) &&\text{for $1 \leq j \leq n$},\\
0&< \tilde a(j,k) \leq \tilde a(j,l)\tilde a(l,k) &&\text{for $1 \leq j,k,l \leq n$}.
\eeas
\item\label{Lem3.2b} 
$\Gamma(\AAA) = \Gamma(\widetilde{\AAA})$.
\end{enumerate}
Moreover, if $\widetilde{\AAA}$ satisfies {\rm(\ref{Lem3.2a})} and {\rm(\ref{Lem3.2b})}, then $\tilde a(j,k)\leq a(j,k)$ for $1 \leq j,k \leq n$.
\end{lemma}

\begin{proof}
For any $\t=(t_{1}, \ldots, t_{n})\in\Gamma(\AAA)$, $t_{j}<0$ and $0<t_{j}t_{k}^{-1}\leq a(j,k)$ for $1 \leq j,k \leq n$. If $\Gamma(\AAA) \neq \emptyset$ define
\bea\label{Eqn14.2zxc}
0< \tilde a(j,k) = \sup_{\t\in \Gamma(\AAA)}\frac{t_{j}}{t_{k}}\leq a(j,k).
\eea
Then $\tilde a(j,j) = 1$ and for any $j,k,l\in \{1, \ldots, n\}$
\beas
\tilde a(j,k) &= \sup_{\t\in \Gamma(\AAA)}\frac{t_{j}}{t_{k}} =\sup_{\t\in \Gamma(\AAA)}\Big[\frac{t_{j}}{t_{l}}\frac{t_{l}}{t_{k}}\Big]\leq \sup_{\t\in \Gamma(\AAA)}\frac{t_{j}}{t_{l}}\,\sup_{\t\in \Gamma(\AAA)}\frac{t_{l}}{t_{k}}= \tilde a(j,l)\,\tilde a(l,k).
\eeas
This establishes (\ref{Lem3.2a}). 

Let $\v=(v_{1}, \ldots, v_{n})\in \Gamma(\widetilde{\AAA})$. Since $v_{k}<0$ and $\tilde a(j,k)\leq a(j,k)$, it follows that $a(j,k)v_{k}\leq \tilde a(j,k)v_{k}\leq v_{j}$, and so $v_{j}v_{k}^{-1}\leq a(j,k)$. Thus $\v\in \Gamma(\AAA)$ and it follows that $\Gamma(\widetilde{\AAA}) \subset \Gamma(\AAA)$. Conversely, if $\v=(v_{1}, \ldots, v_{n})\in \Gamma(\AAA)$, then $v_{j}v_{k}^{-1}\leq \sup_{\t\in \Gamma(\AAA)}t_{j}t_{k}^{-1}=\tilde a(j,k)$ and so $\tilde a(j,k)v_{k}\leq v_{j}$. It follows that $\Gamma(\AAA) \subset \Gamma(\widetilde{\AAA})$, and this establishes (\ref{Lem3.2b}). 

Next we show that the matrix $\widetilde{\AAA}$ is unique. Let $\CCC=\big\{c(j,k):1\leq j,k\leq n\big\}$ be any matrix satisfying (\ref{Lem3.2a}) and (\ref{Lem3.2b}). Since $\Gamma({\CCC})=\Gamma(\AAA)$, if $\t\in \Gamma({\CCC})$ then $c(j,k)t_{k}\leq t_{j}$, and so
\bes
\tilde a(j,k) = \sup_{\t\in \Gamma(\AAA)}t_{j}t_{k}^{-1}= \sup_{\t\in \Gamma({\CCC})}t_{j}t_{k}^{-1}\leq c(j,k).
\ees
To establish the reverse inequality, consider the vectors $\w_{l}=\big(-c(1,l), \ldots, -c(n,l)\big)$. Since  $\CCC$ satisfies (\ref{Lem3.2a}), we have $c(j,l)\leq c(j,k)c(k,l)$, or equivalently  $c(j,k)\big(-c(k,l)\big)\leq \big(-c(j,l)\big)$. This shows that  $\w_{l}\in\Gamma({\CCC})=\Gamma(\widetilde{\AAA})$ for $1 \leq l \leq n$. It follows that $\tilde a(j,k)(-c(k,l))\leq -c(j,l)$ for all $1 \leq j,k,l \leq n$. Letting $l=k$, we have $\tilde a(j,k)(-c(k,k))\leq -c(j,k)$, and since $c(k,k)=1$ we have $c(j,k)\leq \tilde a(j,k)$. Thus $c(j,k)=\tilde a(j,k)$, which shows that $\widetilde{\AAA}$ is unique.  We have already observed that $\tilde a(j,k)\leq a(j,k)$, and since the matrix $\widetilde \AAA$ is unique, this completes the proof.
\end{proof}
 It can be shown that, in the hypotheses of Lemma \ref{Lem3.2}, 
 \beas
 \tilde a(j,k) = 
 \begin{cases}
 \min\big\{a(j,i_1)a(i_1,i_2)\cdots a(i_r,k) : r\geq 0, \{i_1,\ldots, i_r\}\subset \{1, \ldots, n\}\big\}&\text{if $j\neq k$}\\
 1 &\text{if $j=k$}
 \end{cases}.
 \eeas

\subsection{Partial matrices}\label{PartialMatrix}

In Lemma \ref{Lem3.2} it is assumed that the cone $\Gamma(\AAA)$ is defined by the full number $n(n-1)$ inequalities $a(j,k)t_{k}\leq t_{j}$ where $1\leq j\neq k\leq n$. We also need to consider cases in which we begin with fewer inequalities. We formulate this as follows. Let $\BBB=\{b(j,k)\}$ be an $n\times n$ matrix, where this time each $b(j,k)$ is either a positive real number or the symbol $\infty$. As before we put
\beas
\Gamma(\BBB) = \Big\{\t=(t_{1}, \ldots, t_{n})\in \R^{n}: b(j,k)t_{k}\leq t_{j}<0\Big\},
\eeas
but now if $b(j,k)=\infty$, the inequality $b(j,k) t_{k}\leq t_{j}<0$ puts \emph{no} constraint on the relation between the two negative real numbers $t_{k}$ and $t_{j}$. It is still the case that $\Gamma(\BBB)$ is a polyhedral cone in the negative orthant of $\R^{n}$. We call $\BBB$ a \emph{partial matrix}.  We say that a partial matrix $\BBB=\{b(j,k)\}$ is \emph{connected} if for any two indices $j,k\in \{1, \ldots, n\}$ there are indices $i_{1}, \ldots, i_{a}\in \{1, \ldots, n\}$ (with $a$ possibly dependent on $j$ and $k$) so that 
\beas
b(j,i_{1})b(i_{1},i_{2})\cdots b(i_{l},i_{l+1})\cdots b(i_{a-1},i_{a})b(i_{a},k)<+\infty.
\eeas
It then follows that
\beas
b(j,i_{1})b(i_{1},i_{2})\cdots b(i_{j},i_{j+1})\cdots b(i_{a-1},i_{a})b(i_{a},k)t_{k}\leq t_{j
}<0
\eeas
so if $\BBB$ is connected and $\t\in \Gamma(\BBB)$ it follows that
\beas
\sup_{\t\in\Gamma(\BBB)}\frac{t_{j}}{t_{k}}\leq b(j,i_{1})b(i_{1},i_{2})\cdots b(i_{j},i_{j+1})\cdots b(i_{a-1},i_{a})b(i_{a},k)<+\infty.
\eeas

We then have the following analogue of Lemma \ref{Lem3.2}.
\begin{lemma}\label{Lem3.2zxc}
Let $\BBB=\{b(j,k)\}$ be an $n\times n$  connected, partial matrix,  and suppose that the corresponding cone $\Gamma(\BBB)$ is not empty. Then there exists a unique $n\times n$ matrix $\widetilde{\BBB}=\big\{\tilde b(j,k)\big\}$ with the following properties:
\begin{enumerate}[{\rm(a)}]
\item \label{Lem3.2azxc}
The coefficients of $\widetilde{\BBB}$ are all finite positive real numbers, and
\beas
1&= \tilde b(j,j) &&\text{for $1 \leq j \leq n$},\\
0&< \tilde b(j,k) \leq \tilde b(j,l)\tilde b(l,k) &&\text{for $1 \leq j,k,l \leq n$}.
\eeas
\item\label{Lem3.2bzxc} 
$\Gamma(\BBB) = \Gamma(\widetilde{\BBB})$.
\end{enumerate}
Moreover, if $\widetilde{\BBB}$ satisfies {\rm(\ref{Lem3.2a})} and {\rm(\ref{Lem3.2b})}, then $\tilde b(j,k)\leq b(j,k)$ for $1 \leq j,k \leq n$.
\end{lemma}

\begin{proof} There is only a minor change from the proof of Lemma \ref{Lem3.2}.  Since $b(j,k)$ might equal $\infty$ we cannot directly conclude  that $\sup_{\t\in \Gamma(\BBB)} t_{j}\,t_{k}^{-1}$ is finite by observing that it is bounded by $b(j,k)$ as in equation (\ref{Eqn14.2zxc}). However, since $\BBB$ is connected, it is still the case that  $\sup_{\t\in \Gamma(\BBB)} t_{j}\,t_{k}^{-1}<\infty$, and we can define $\widetilde b(j,k)$ to the this supremum. The proof of (\ref{Lem3.2azxc}) then proceeds as before, and we set $\Gamma(\widetilde\BBB) = \Big\{\t=(t_{1}, \ldots, t_{n})\in\R^{n}: \widetilde b(j,k)t_{k}\leq t_{j}<0\Big\}$. Note that if $b(j,k)<\infty$ then $\widetilde b(j,k) =\sup_{\t\in\Gamma(\BBB)}\frac{t_{j}}{t_{k}}\leq b(j,k)$ as before. The proofs that $\Gamma(\widetilde \BBB)=\Gamma(\BBB)$ and the uniqueness of $\BBB$ also follow as before.
\end{proof}

\subsection{Projections}

Let $1<p<n$ and let $\pi_{n,p}:\R^{n}\to \R^{p}$ be the projection onto the first $p$ coordinates: $\pi_{n,p}(x_{1}, \ldots, x_{n})=(x_{1}, \ldots, x_{p})$.  Let $\AAA_{n}=\{a(j,k):1\leq j,k\leq n\}$ be an $n\times n$ matrix with strictly positive entries which satisfies the basic hypotheses
\beas
a(j,j) &= 1&&&&1\leq j \leq n,\\
a(j,l)&\leq a(j,k)a(k,l) &&& &1\leq j,k,l\leq n.
\eeas
Let $\AAA_{p}=\{a(j,k):1\leq j,k\leq p\}$ be the corresponding $p\times p$ matrix (which of course also satisfies the basic hypotheses). We than have two cones:
\beas
\Gamma(\AAA_{n}) &= \Big\{\t=(t_{1}, \ldots, t_{n})\in \R^{n}: a(j,k)t_{k}\leq t_{j}<0,\,\, 1\leq j,k\leq n\Big\},\\
\Gamma(\AAA_{p}) &= \Big\{\u=(u_{1}, \ldots, u_{p})\in \R^{p}: a(j,k)u_{k}\leq u_{j}<0,\,\,\leq j,k\leq p\Big\}.
\eeas
\begin{lemma}\label{Lem14.4zxc}
The projection $\pi_{n,p}$ maps the cone $\Gamma(\AAA_{n})$ onto the cone $\Gamma(\AAA_{p})$.
\end{lemma}

\begin{proof}
If $ a(j,k)t_{k}\leq t_{j}<0$ for $1\leq j,k\leq n$, then $ a(j,k)t_{k}\leq t_{j}<0$ for $1\leq j,k\leq p$, and so it is clear that $\pi$ maps $\Gamma(\AAA_{n})$ into $\Gamma(\AAA_{p})$. The main point of the Lemma is that the mapping is \emph{onto}. But since
\beas
\pi_{n,p}= \pi_{p+1,p}\circ\pi_{p+2,p+1}\circ\cdots\circ \pi_{n,n-1}
\eeas
it clearly suffices to show that the mapping $\pi_{n,n-1}:\Gamma(\AAA_{n})\to \Gamma(\AAA_{n-1})$ is onto.

Let $(u_{1}, \ldots, u_{n-1})\in \Gamma(\AAA_{n-1})$. We must show that there exists $u_{n}<0$ so that
\beas
1\leq j \leq n-1\quad \Longrightarrow \quad a(j,n)u_{n}\leq u_{j} \quad \text{and}\quad  a(n,j)u_{j}\leq u_{n}
\eeas
 If we let $E_{j}$ be the closed interval $\left[a(n,j)u_{j}, \frac{u_{j}}{a(j,n)}\right]$, then we must show that $\bigcap_{j=1}^{n-1}E_{j}\neq \emptyset$.  For any $j,k\in \{1, \ldots, n-1\}$ we have
\beas
a(n,j)u_{j}&\leq \frac{a(k,j)}{a(k,n)}u_{j} &&&&\text{since $a(k,j)\leq a(k,n)a(n,j)$ and $u_{j}<0$,}\\
&\leq
\frac{u_{k}}{a(k,n)} &&&&\text{since $(u_{1}, \ldots, u_{n-1})\in \Gamma(\AAA_{n-1})$ and hence $a(k,j)u_{j}\leq u_{k}$.}
\eeas
It follows that for any index $1\leq\ell\leq n-1$ we have
\beas
a(n,\ell)u_{\ell}\leq \sup\Big\{a(n,j)u_{j}:1\leq j \leq n-1\Big\} \leq \inf\Big\{\frac{u_{k}}{a(k,n)}:1\leq k \leq n-1\Big\}\leq \frac{u_{\ell}}{a(\ell,n)}
\eeas
and hence we have produce a closed, non-empty interval in $\bigcap_{j=1}^{n-1}E_{j}$:
\beas
\left[\sup\Big\{a(n,j)u_{j}:1\leq j \leq n-1\Big\},\inf\Big\{\frac{u_{k}}{a(k,n)}:1\leq k \leq n-1\Big\}\right] \subset \bigcap_{j=1}^{n-1}E_{j}.
\eeas
\end{proof}

\subsection{The dimension of $\Gamma(\EEE)$}

\begin{proposition}\label{Prop3.3} Let $\EEE$ be an $n\times n$ matrix satisfying the basic hypothesis (\ref{2.5}). Then the cone $\Gamma(\EEE)$ is not empty. In particular, the $2n$ vectors
\beas\label{v_l,w_l}
\v_{l}&=\big(-e(1,l), \ldots, -e(n,l)\big), &1\leq l \leq n\\
\w_{l}&=\big(-e(l,1)^{-1}, \ldots, -e(l,n)^{-1}\big),&1\leq l \leq n
\eeas
all belong to $\Gamma(\EEE)$.
\end{proposition}
\begin{proof}
Fix $l\in \{1, \ldots, n\}$ and let $v_{j}= -e(j,l)$. Then it follows from (\ref{2.5}) that $e(j,l)\leq e(j,k)e(k,l)$ and so $e(j,k)v_{k}=-e(j,k)e(k,l) \leq -e(j,l) = v_{j}$. This shows that $\v_{l}\in \Gamma(\EEE)$. Similarly, let $w_{j}=-e(l,j)^{-1}$. Then since $e(l,k)\leq e(l,j)e(j,k)$, $e(j,k)w_{k}= -\frac{e(j,k)}{e(l,k)} \leq -\frac{1}{e(l,j)}$, and so $\w_{l}\in \Gamma(\EEE)$. This completes the proof.
\end{proof}

Since the vectors $\v_{l}$ in Proposition \ref{Prop3.3} are the negatives of the columns of $\EEE$, we have the following corollary.

\begin{corollary}\label{Cor3.4}
Let $\EEE$ be an $n\times n$ matrix satisfying the basic hypothesis (\ref{2.5}). Then the dimension of $\Gamma(\EEE)$ is greater than or equal to the rank of $\EEE$.
\end{corollary}

\begin{lemma}\label{Lem3.5}
 Let $\EEE$ be an $n\times n$ matrix satisfying the basic hypotheses \eqref{2.5}. The following are equivalent:
\begin{enumerate}
\item[\rm(i)] the interior of $\Gamma(\EEE)$ is non-empty;
\item[\rm(ii)] the dimension of $\Gamma(\EEE)$ is $n$;
\item[\rm(iii)]  $1<e(j,k)e(k,j)$ for every pair $1\leq j\neq k\leq n$.
\end{enumerate}
\end{lemma}

\begin{proof}
Obviously (i) implies (ii). Conversely, since $\Gamma(\EEE)$ is convex, if it  contains a basis of $\bR^n$ it also containes the convex hull of this basis together with 0, which is open.

In order to prove that (iii) implies (i), we argue by induction on $n$. For $n=1$, $\Gamma(\EEE)=(-\infty,0)$ which has dimension $1$. For $n=2$, $\det(\EEE) = 1-e(1,2)e(2,1)<0$, so the rank of $\EEE$ is $2$, and hence by Corollary \ref{Cor3.4} the dimension of $\Gamma(\EEE)$ is $2$ and its interior is a nontrivial angle. 

Thus assume that the Lemma is true for all $(n-1)\times (n-1)$ matrices $\EEE'$ satisfying equation (\ref{2.5}). Let $\EEE$ be an $n\times n$ matrix satisfying equation (\ref{2.5}) such that $1<e(j,k)e(k,j)$ for every pair $1\leq j\neq k\leq n$. Let $\EEE'$ be the sub-matrix consisting of the first $n-1$ rows and columns. Then clearly $\EEE'$ also satisfies equation (\ref{2.5}) and the hypothesis of the Lemma, so by the induction hypothesis, the dimension of $\Gamma(\EEE')$ is $n-1$, and $\Gamma(\EEE')$ contains a non-empty open set $U\subset\R^{n-1}$. 

Let $\t'=(t_{1}, \ldots, t_{n-1})\in U$. Then the components of $\t'$ satisfy the strict inequalities $e(k,j)t_j<t_k$ for all $j\ne k$. Since $t_{j}<0$ and $e(k,j)\leq e(k,n)e(n,j)$, it follows that $e(k,n)e(n,j)t_{j}\leq e(k,j)t_{j}< t_{k}$ and so $e(n,j)t_{j}< e(k,n)^{-1}t_{k}$ for all $j,k$. Therefore
\be\label{3.3rr}
\max\Big\{e(n,j)t_{j}:1\leq j \leq n-1\Big\}< \min\Big\{e(k,n)^{-1}t_{k}:1\leq k \leq n-1\Big\}.
\ee
But if $\t'\in \Gamma(\EEE')$ then $\t=(\t',t_{n})=(t_{1}, \ldots, t_{n-1}, t_{n})$ in $\Gamma(\EEE)$ if and only if for all $1 \leq j,k \leq n-1$,
\beas
\max\Big\{e(n,j)t_{j}:1\leq j \leq n-1\Big\}\leq t_{n}\leq \min\Big\{e(k,n)^{-1}t_{k}:1\leq k \leq n-1\Big\}.
\eeas 

It follows from (\ref{3.3rr}) that the set of $t_{n}$ for which this is true is non empty and contains the open interval $\big(m(\t'),M(t')\big)$ where 
\beas
m(\t') &=\max\Big\{e(n,j)t_{j}:1\leq j \leq n-1\Big\},\\
M(\t') &= \min\Big\{e(k,n)^{-1}t_{k}:1\leq k \leq n-1\Big\}.
\eeas

 Then
\bes
\Big\{(\t',t_{n}):\text{$\t'\in U$ and $m(\t')<t_{n} <M(\t')$}\Big\}
\ees
is an open subset of $\Gamma(\EEE)$. 

Conversely, if $1<e(j,k)e(k,j)$ for every pair $1\leq j\neq k \leq n$, it follows from Lemma \ref{Lem3.5} in Appendix I that $\Gamma(\EEE)$ has dimension $n$ and in particular has non-empty interior. On the other hand, if $\Gamma^{o}(\EEE)$ is non-empty, then the projection of $\Gamma(\EEE)$ onto the two dimensional subspace spanned by $x_{j}$ and $x_{k}$ is open and non-empty. Since this projection is contained in $\{(x_{j},x_{k})\in\R^{2}: e(j,k)x_{k}<x_{j}<e(k,j)^{-1}x_{k}\}$, it follows that $1<e(j,k)e(k,j)$.
\end{proof}

\begin{proposition}\label{Prop3.61}
Let $\EEE$ be an $n\times n$ matrix satisfying the basic hypothesis (\ref{2.5}), and suppose that $1 <e(j,k)e(k,j)$ for all $j\neq k$. Then for any indices $k_{1}, \ldots, k_{s}$ it follows that
\bes
1<e(k_{1},k_{2})e(k_{2},k_{3}) \cdots e(k_{s-1},k_{s})e(k_{s},k_{1})
\ees
provided that the indices $\{k_{1}, \ldots, k_{s}\}$ are not all the same.
\end{proposition}
\begin{proof}
The basic hypothesis implies that we always have
\bes
1\leq e(k_{1},k_{2})e(k_{2},k_{3}) \cdots e(k_{s-1},k_{s})e(k_{s},k_{1}).
\ees
We argue by induction on $s$. If $s=2$, the hypothesis of the Proposition implies that we cannot have $1 = e(k_{1},k_{2})e(k_{2},k_{1})$ unless $k_{1}=k_{2}$. Now suppose the statement is true for any set of indices of length $s-1$, and let $k_{1}, \ldots, k_{s}$ be indices which are not all equal. Consider the product $e(k_{1},k_{2})e(k_{2},k_{3}) \cdots e(k_{s-1},k_{s})e(k_{s},k_{1})$. Without loss of generality we can assume $k_{1}\neq k_{2}$ for otherwise $e(k_{1},k_{2})=1$ and the result follows by induction. But then
\beas
1&<e(k_{1},k_{2})e(k_{2},k_{1})\leq e(k_{1},k_{2})e(k_{2},k_{3})e(k_{3},k_{1})
\leq \cdots \\
&\cdots \leq e(k_{1},k_{2})e(k_{2},k_{3}) \cdots e(k_{s-1},k_{s})e(k_{s},k_{1})
\eeas
which completes the proof.
\end{proof}

\subsection{The reduced matrix $\EEE^{\flat}$}\label{Reduced}
 
In this section we show that if the dimension of $\Gamma(\EEE)$ is $s<n$, there is an $s\times s$ sub-matrix $\EEE^{\flat}$ of $\EEE$ satisfying the hypothesis of Lemma \ref{Lem3.5}, so that $\Gamma(\EEE^{\flat})\subset\R^{s}$ has dimension $s$, and that $\Gamma(\EEE)$ is the graph of a linear mapping from $\R^{s}$ to $\R^{n}$.

Let $\EEE=\{e(j,k)\}$ be an $n\times n$ matrix satisfying the basic hypothesis (\ref{2.5}). If $j,k\in \{1, \ldots, n\}$ write $j \sim k$ if and only if $e(j,k)e(k,j)=1$. It is immediate that $j\sim j$ and $j\sim k \Longleftrightarrow k\sim j$. Suppose that $j\sim k$ and $k \sim l$. Then since $e(j,k)e(k,j)=1$ and $e(k,l)e(l,k)=1$,
\bes
1 \leq e(j,l)e(l,j)\leq \big(e(j,k)e(k,l)\big)\,\big(e(l,k)e(k,j)\big)=\big(e(j,k)e(k,j)\big)\,\big(e(k,l)e(l,k)\big) = 1
\ees
so that $j\sim l$. Thus `$\sim$' is an equivalence relation. 

There are other equivalent formulations of this  relation, based on the following remarks.

\begin{enumerate}
\item[(i)] If $e(j,k)e(k,j)=1$, the basic hypotheses imply that, for every $\ell$, 
$$
\frac1{e(k,j)}\le\frac{e(j,\ell)}{e(k,\ell)}\le e(j,k),\qquad \frac1{e(j,k)}\le\frac{e(\ell,j)}{e(\ell,k)}\le e(k,j).
$$
Hence the above inequalities are all equalities, i.e., the $j$-th and $k$-th row of $\EEE$ are proportional. Conversely, if the $j$-th and the $k$-th rows $\EEE_j,\EEE_k$ are proportional, then $\EEE_j=e(j,k)\EEE_k$ and $\EEE_k=e(k,j)\EEE_j$, hence $e(j,k)e(k,j)=1$.

In the same way, one proves that $j\sim k$ if and only if the $j$-th and $k$-th column are proportional.
\smallskip

\item[(ii)] If $\t=(t_{1}, \ldots, t_{n})\in \Gamma(\EEE)$ and if $j \sim k$, then $e(j,k)t_{k}\leq t_{j}=e(j,k)e(k,j)t_{j}\leq e(j,k)t_{k}$ so for all $\t\in \Gamma(\EEE)$,
\be\label{3.4rrr}
\t=(t_{1}, \ldots, t_{n}) \in \Gamma(\EEE),\,\,j\sim k \qquad \Longrightarrow \qquad t_{j}=e(j,k)t_{k}.
\ee

Conversely, if $\Gamma(\EEE)$ is contained in the hyperplane $t_j=\lambda t_k$, this condition is satisfied in particular by the vectors $\v_k,\w_k$ in \eqref{v_l,w_l}. Hence $e(j,k)=\lambda=e(k,j)^{-1}$ and therefore $j\sim k$.
\end{enumerate}

\medskip

Let $\{k_{1}, \ldots, k_{s}\}$ be representatives of the distinct equivalence classes in $\{1, \ldots, n\}$ and let $\EEE^{\flat}$ be the $s\times s$ matrix whose entries are $\{e(k_{a},k_{b}):1 \leq a,b\leq s\}$. Then $\EEE^{\flat}$ satisfies the same basic hypothesis, and if $a\neq b$ we have $e(k_{a},k_{b})e(k_{b},k_{a})>1$. The matrix $\EEE^{\flat}$ is called a \textit{reduced matrix} for $\EEE$. It is not uniquely determined by $\EEE$ since it depends on the choice of representatives of the $s$ distinct equivalence classes. The number $s$ is called the \textit{reduced rank} of $\EEE$.
\medskip

Let $\Pi:\R^{n}\to\R^{s}$ be the projection given by $\Pi(\t) = \Pi(t_{1}, \ldots, t_{n}) = (t_{k_{1}}, \ldots, t_{k_{s}})$. If $\t\in \Gamma(\EEE)$ then clearly $\Pi(\t)\in \Gamma(\EEE^{\flat})$, so $\Pi:\Gamma(\EEE)\to \Gamma(\EEE^{\flat})$. On the other hand, if $\t^{\flat}= (t_{k_{1}}, \ldots, t_{k_{s}})\in \Gamma(\EEE^{\flat})$, set $\sigma(\t^{\flat}) = (t_{1}, \ldots, t_{n})$ where $t_{j}=e(j,k_{a})t_{k_{a}}$ if $j\sim k_{a}$. We claim $\sigma(\t^{\flat})\in \Gamma(\EEE)$. To see this, let $1 \leq j,k\leq n$ and suppose $j\sim k_{a}$ and $k\sim k_{b}$. Then since $e(k_{a},k_{b})t_{k_{b}}\leq t_{k_{a}}$,
\beas
e(j,k)t_{k}
&=
e(j,k)e(k,k_{b})t_{k_{b}}
=
e(j,k)e(k,k_{b})e(k_{a},k_{b})^{-1}e(k_{a},k_{b})t_{k_{b}}\\
&\leq
e(j,k)e(k,k_{b})e(k_{a},k_{b})^{-1}t_{k_{a}}
=
e(j,k)e(k,k_{b})e(k_{a},k_{b})^{-1}e(j,k_{a})^{-1}e(j,k_{a})t_{k_{a}}\\
&=
e(j,k)e(k,k_{b})e(k_{a},k_{b})^{-1}e(j,k_{a})^{-1}t_{j}\leq t_{j}
\eeas
since $e(j,k_{a})e(k_{a},k_{b})= \frac{e(k_{a},k_{b})}{e(k_{a},j)}\leq \frac{e(k_{a},j)e(j,k_{b})}{e(k_{a},j)}\leq e(j,k)e(k,k_{b})$. It follows that $e(j,k)t_{k}\leq t_{j}$ as required, so $\sigma(\t^{\flat})\in \Gamma(\EEE)$. Thus if $\t=(t_{1}, \ldots, t_{n})\in \Gamma(\EEE)$, the coordinates $t_{j}$ for $j\notin \{k_{1}, \ldots, k_{s}\}$ are uniquely determined by $\Pi(\t)$. It follows that the mapping $\Pi:\Gamma(\EEE)\to \Gamma(\EEE^{\flat})$ is one-to-one and onto, with the inverse mapping given by $\sigma$. We have established the following.

\begin{lemma}\label{Lem3.7}
Let $\EEE$ be an $n\times n$ matrix satisfying the basic hypothesis (\ref{2.5}). Let $j\sim k$ if and only if $1=e(j,k)e(k,j)$. This is an equivalence relation, and suppose there are $s$ distinct equivalence classes. Let $\{k_{1}, \ldots, k_{s}\}$ be representatives of the distinct equivalence classes, and let $\EEE^{\flat}$ be the $s\times s$ reduced matrix whose entries are $\{e(k_{a},k_{b}):1 \leq a,b\leq s\}$.
\begin{enumerate}[{\rm a)}]

\smallskip

\item The reduced matrix $\EEE^{\flat}$ satisfies the basic hypothesis (\ref{2.5}).

\smallskip

\item $1<e(k_{a},k_{b})e(k_{b},k_{a})$ for all $1 \leq a\neq b \leq s$.

\smallskip

\item The dimension of $\Gamma(\EEE)$ is $s$.

\smallskip

\item $\Gamma(\EEE^{\flat})\subset \R^{s}$ has dimension $s$, and thus has non-empty interior.

\smallskip

\item Let $\Pi:\R^{n}\to \R^{s}$ be given by $\Pi(\t) = \Pi(t_{1}, \ldots, t_{n}) = (t_{k_{1}}, \ldots, t_{k_{s}})$. Then $\Pi:\Gamma(\EEE)\to\Gamma(\EEE^{\flat})$ is one-to-one and onto.

\smallskip

\item Let $V=\Big\{\t=(t_{1}, \ldots, t_{n})\in \R^{n}:\text{$t_{j}=e(j,k_{a})t_{k_{a}}$ if $j\sim k_{a}$}\Big\}$. Then $V$ is an $s$-dimensional subspace of $\R^{n}$, and $\Gamma(\EEE)\subset V$.
\end{enumerate}
\end{lemma}

\section{Appendix II: Estimates for homogeneous norms}\label{HomogNorms}

 Consider a family of dilations on $\R^{m}$ given by $\lambda\cdot\x=(\lambda^{d_{1}}x_{1}, \ldots, \lambda^{d_{m}}x_{m})$ with homogeneous dimension $Q=d_{1}+ \cdots + d_{m}$. Let $N:\R^{m}\to [0,\infty)$ be a smooth homogeneous norm relative to this family so that $N(\x) \approx |x_{1}|^{1/d_{1}}+ \cdots + |x_{m}|^{1/d_{m}}$. For any $\gammab=(\gamma_{1}, \ldots, \gamma_{m})\in \N^{m}$ let $\[\gammab\] = d_{1}\gamma_{1}+\cdots +d_{m}\gamma_{m}$.

\medskip

\begin{proposition}\label{Prop12.1}
Let $\rho(\x)$ be homogeneous of degree $p$. For $\gammab \in \N^{m}$ there is a constant $C_{\gammab}>0$ so that\quad $|\partial^{\gammab}_{\x}\rho(\x)| \leq C_{\gammab}\, N(\x)^{p-\[\gammab\]}$.
\end{proposition}

\begin{proof}
Let $\gammab=(\gamma_{1}, \ldots, \gamma_{m}) \in \N^{m}$. Differentiating the equation $\lambda^{p}\,\rho(\x)=\rho(\lambda\cdot\x)$ we obtain
\bes
\lambda^{p} \,\partial^{\gammab}\rho(\x) = \lambda^{\sum_{j=1}^{m}d_{j}\gamma_{j}}\partial^{\gammab}\rho(\lambda\cdot\x)= \lambda^{\[\gammab\]}\partial^{\gammab}\rho(\lambda\cdot\x).
\ees
It follows that $|\partial^{\gammab}_{\x}\rho(\x)|^{1/(p-\[\gamma\])}$ is homogeneous of degree $1$. Since $N$ is continuous and strictly positive on the compact set $\{\x:N(\x)=1\}$, it follows that $|\partial^{\gammab}_{\x}\rho(\x)| \leq C_{\gammab}N(\x)^{p-\[\gammab\]}$.
\end{proof}

\begin{proposition}\label{Prop12.2}
There is a constant $c$ so that $\big|\{\x\in\R^{m}:N(\x)\leq A\}\big| = c\,A^{Q}$.
\end{proposition}

\begin{proof} Let $\B_{m}(1)= \{\x\in\R^{m}: N(\x)<1\}$. Then $N(\x)\leq A$ if and only if $N(A^{-1}\cdot\x) = A^{-1}N(\x) \leq 1$, and this is true if and only if $A^{-1}\cdot \x\in \B_{N}(1)$. If $T_{A}(\x) = A\cdot \x$, this means that $\x\in T_{A}(\B(1))$. Since $\det(T_{A}) = A^{Q}$, the Proposition follows.
\end{proof}

\begin{proposition} \label{Prop12.3}
If $\a=(a_{1}, \ldots, a_{m}) \in \N^{m}$, write $\x^{\a}=\prod_{j=1}^{m}x_{j}^{a_{j}}$. 

\begin{enumerate}[{\rm(a)}]
\smallskip

\item For $\a\in \N^{m}$ and $M<\[a\]+Q$ there exists $C>0$ so that 
\bes
\int\limits_{N(\x)\leq A}|\x^{\a}|N(\x)^{-M}\,d\x \leq C\,A^{\[\a\]+Q-M}.
\ees

\smallskip

\item For $\a\in \N^{m}$, $M >\[a\]+Q$ there exists $C>0$ so that 
\bes
\int\limits_{N(\x)\geq A}|\x^{\a}|N(\x)^{-M}\,d\x\leq C\,A^{\[\a\]+Q-M}.
\ees

\smallskip

\item For $\a\in \N^{m}$, $M =\[a\]+Q$, and $B>4A$, there exists $C>0$ so that 
\bes
\int\limits_{A<N(\x)<B}|\x^{\a}|N(\x)^{-M}\,d\x\leq C\,\log\left[\frac{B}{A}\right].
\ees
\end{enumerate}
\end{proposition}

\begin{proof} 
We have $|\x^{\a}| =\prod_{j=1}^{m}|x_{j}|^{a_{j}}\lesssim \prod_{j=1}^{m}N(\x)^{a_{j}d_{j}}= N(\x)^{\[\a\]}$. Thus if $\[\a\]+Q-M>0$,
\beas
\int\limits_{N(\x)\leq A}|\x^{\a}|N(\x)^{-M}\,d\x 
&\lesssim 
\sum_{j=1}^{\infty}\,\int\limits_{2^{-j}A<N(\x)\leq 2^{-j+1}A}N(\x)^{\[\a\]-M}\,d\x\\
&\approx
\sum_{j=1}^{\infty} (2^{-j}A)^{\[\a\]-M}\big\vert\{\x:N(\x)\leq 2^{-j+1}A\}\big\vert\\
&\approx
\sum_{j=1}^{\infty} (2^{-j}A)^{\[\a\]-M}(2^{-j+1}A)^{Q}\lesssim A^{\[\a\]+Q-M}.
\eeas
Similarly, if $\[\a\]+Q-M<0$,
\beas
\int\limits_{N(\x)\geq A}|\x^{\a}|N(\x)^{-M}\,d\x 
&\lesssim 
\sum_{j=0}^{\infty}\,\int\limits_{2^{j}A\leq N(\x)< 2^{j+1}A}N(\x)^{\[\a\]-M}\,d\x\\&\approx
\sum_{j=0}^{\infty} (2^{j}A)^{\[\a\]-M}\big\vert\{\x:N(\x)\leq 2^{-j+1}A\}\big\vert\\
&\lesssim
\sum_{j=0}^{\infty} (2^{j}A)^{\[\a\]-M}(2^{-j}A)^{Q}\\
&\lesssim A^{\[\a\]+Q-M}.
\eeas
Finally, if $\[\a\]+Q-M=0$,
\beas
\int\limits_{A\leq N(\x)\leq B}|\x^{\a}|N(\x)^{-M}\,d\x 
&\lesssim
\int\limits_{A\leq N(\x)\leq B}N(\x)^{\[\a\]-M}\,d\x.
\eeas
If $B>4A$, choose $j< k$ so that $2^{j-1}<A\leq 2^{j}\leq 2^{k}\leq B <2^{k+1}$. Then
\beas
\int\limits_{A\leq N(\x)\leq B}N(\x)^{\[\a\]-M}\,d\x
&\leq
\sum_{l=j-1}^{k+1}\int\limits_{2^{l}\leq N(\x)<2^{l+1}}N(\x)^{\[\a\]-M}\,d\x\\
&\lesssim 
\sum_{l=j-1}^{k+1}2^{l(\[\a\]-M)}\big\vert\{\x:N(\x)\leq 2^{l+1}\}\big\vert\\
&\lesssim
\sum_{l=j-1}^{k+1}2^{l(\[\a\]-M+Q)}=k-j+2 \approx \log\left[\frac{B}{A}\right].
\eeas
This completes the proof.
\end{proof}

\begin{proposition}\label{Prop12.4}
Let $m\in \CC^{\infty}(\R^{m}\setminus\{0\})$ have compact support, and let $L=\{l_{1}, \ldots, l_{r}\}\subset\{1, \ldots, m\}$. Suppose for every $\gamma\in \N$ there is a constant $C>0$ so that if $j\in L$ then 
\beas
\big\vert\partial^{\gamma}_{\xi_{j}}m(\xi_{1}, \ldots, \xi_{m})\big\vert \leq C N(\xib)^{-d_{j}\gamma}.
\eeas
Put $K(\x)=\int_{\R^{m}}e^{2\pi i \langle \x,\xib\rangle}m(\xib)\,d\xib$. There is a constant $C>0$ independent of the support of $m$ so that
\beas
\big\vert K(\x)\big\vert \leq C\, \big(\sum_{j\in L}|x_{j}|^{1/d_{j}}\big)^{-Q}.
\eeas
\end{proposition}

\begin{proof}
Let $\varphi\in\CC^{\infty}_{0}(\R)$ with $\varphi(t) \equiv 1$ if $|t|\leq 1$, $\varphi(t) \equiv 0$ if $|t|\geq 2$ and $0\leq \varphi(t)\leq 1$ for all $t$. We have
\beas
K(\x)=\int_{\R^{m}}e^{2\pi i \langle \x,\xib\rangle}\varphi\big(\lambda N(\xib)\big)m(\xib)\,d\xib +
\int_{\R^{m}}e^{2\pi i \langle \x,\xib\rangle}\big[1-\varphi\big(\lambda N(\xib)\big)\big]m(\xib)\,d\xib= I+II.
\eeas
Since $m$ is bounded, we have
\beas
\big\vert I\big\vert \leq ||m||_{\infty}\big\vert\big\{\xib:N(\xib)\leq 2\lambda^{-1}\big\}\big\vert\lesssim \lambda^{-Q}.
\eeas
On the other hand, the chain, the product rule, and Proposition \ref{Prop12.1}  show that 
\beas
\big\vert\partial^{\gamma}_{\xi_{j}}\big[1-\varphi\big(\lambda N(\xib)\big)\big]m(\xib)\big\vert \leq C N(\xib)^{-d_{j}\gamma}
\eeas
since any derivative of $\varphi(\lambda N(\xib))$ is supported where $\lambda N(\xib)\leq 2$.  For $j\in L$ we can integrate by parts $M$ times in the variable $\xi_{j}$ and obtain
\beas
\big| II \big| 
&\lesssim 
|x_{j}|^{-M}\,\int_{N(\xib)\geq \lambda^{-1}}\Big\vert \partial^{M}_{\xi_{j}}\Big[\big[1-\varphi\big(\lambda N(\xib)\big)\big]m(\xib)\Big]\Big\vert\,d\xib\\
&\lesssim
|x_{j}|^{-M}\,\int_{N(\xib)\geq \lambda^{-1}}N(\xib)^{-Md_{j}}\,d\xib\\
&\lesssim
|x_{j}|^{-M}\lambda^{-Q+d_{j}M}.
\eeas
Setting $\lambda = |x_{j}|^{1/d_{j}}$, it follows that $|K(\x)|\lesssim |x_{j}|^{-Q/d_{j}}$ for every $j\in L$, and hence that
\beas
|K(\x)| \lesssim \min_{j\in L}|x_{j}|^{-Q/d_{j}} \approx \Big(\sum_{j\in L}|x_{j}|^{1/d_{j}}\Big)^{-Q}.
\eeas 
This completes the proof.
\end{proof}

\begin{remark}
Note that in this Proposition, we only assume estimates on ``pure'' $\xi_{j}$ derivatives of $m$ for $j\in L$. The conclusion is that the inverse Fourier transform $K= \check m$ has the expected decay relative to the partial norm $N_{L}(\x) =\sum_{j\in L}|x_{j}|^{1/d_{j}}$. If $L=\{1, \ldots, n\}$, the conclusion is that $K(\x) \lesssim N(\x)^{-Q}$.
\end{remark}

\section{Appendix III: Estimates for geometric sums}

\begin{proposition}\label{Prop16.3wer}
Let $\epsilon>0$,  $a_{1}, \ldots, a_{m}>0$ with $m\geq 2$ and $\prod_{j=1}^{m}a_{j}\neq 1$. Set
\bes
 \Gamma=\Big\{I=(i_{1}, \ldots, i_{m})\in \Z^{m}:\text{$a_{j}i_{j}\leq i_{j+1}<0$ for $1 \leq j \leq m-1$,\, $a_{m}i_{m}\leq i_{1}<0$}\Big\}.
\ees
For $I=(i_{1}, \ldots, i_{m})\in \Gamma$ set $\Lambda(I) = [i_{2}-a_{1}i_{1}] +\cdots +[i_{m}-a_{m-1}i_{m-1}]+[i_{1}-a_{m}i_{m}]$. Then there is a constant $C>0$ depending on $\{\epsilon,  a_{1}, \ldots, a_{m}\}$ so that $\sum_{I\in \Gamma}2^{-\epsilon\Lambda(I)}\leq C$.
\end{proposition}

\begin{proof}
 $\Gamma$ is a convex cone and $\{\t\in \Gamma:\sum_{j=1}^{m}|t_{j}|=1\}$ is compact. Then $\Lambda(\t)\geq 0$ on $\Gamma$, and if $\Lambda(\t)=0$ with $\t\in \Gamma$, then $t_{1}=a_{m}t_{m}$ and $t_{j+1}=a_{j}t_{j}$ for $1\leq j \leq m-1$, so that $t_{1}=\big(\prod_{j=1}^{m}a_{j}\big)t_{1}$. Since $\prod_{j=1}^{m}a_{j}\neq 1$, it follows that $\Lambda(\t)>0$ for $t\in \Gamma$. In particular, by compactness there exists $\eta >0$ so that if $\t\in\Gamma$ and $|\t|=1$ then $\Lambda(\t)>\eta$. Since $\Lambda$ is continuous, there is an open neighborhood of $\{\t\in \Gamma:\sum_{j=1}^{m}|t_{j}|=1\}$ on which $\Lambda(\t)>\eta$. By homogeneity it follows that $\Lambda(\t)\geq \eta\sum_{j=1}^{m}|t_{j}|$ for all$\t\in \Gamma$.

 If $I\in\Gamma$, then $\Lambda(I) \geq \eta\sum_{j=1}^{n}|i_{j}|$. Thus we have 
\bes
\sum_{I\in \Gamma}2^{-\epsilon\Lambda(I)} 
\leq
\sum_{I\in\Gamma}2^{-\epsilon\eta\sum_{j=1}^{m}|i_{j}|}\\
\leq
\Big[\sum_{i=0}^{\infty}2^{-\epsilon\eta i}\Big]^{m}
\ees 
which is a constant depending only on $\epsilon$ and $a_{1}, \ldots, a_{m}$. This completes the proof.
\end{proof}

Now let $\EEE=\{e(j,k)\}$ be an $n\times n$ matrix satisfying (\ref{2.5}) such that $1<e(j,k)e(k,j)$ for all $1 \leq j,k\leq n$. Let  $\epsilon >0$. Let $\tau:\{1, \ldots, n\}\to\{1, \ldots, n\}$ with $\tau(j)\neq j$ for all $j$. 
Let $0\leq b< a\leq n$. Put
\beas
\Gamma_{\Z}(\EEE) &= \Big\{I=(i_{1}, \ldots, i_{n})\in \Z^{n}: e(j,k)i_{k} \leq i_{j}, &&1\leq j,k \leq n\Big\},\\
\Gamma_{\Z}(\EEE') &=\Big\{I'=(i_{1}, \ldots, i_{a})\in \Z^{a}: e(j,k)i_{k} \leq i_{j}, &&1\leq j,k \leq a\Big\},\\
\Gamma_{\Z}(\EEE'') &=\Big\{I''=(i_{a+1}, \ldots, i_{n})\in \Z^{n-a}: e(j,k)i_{k} \leq i_{j}, &&a+1\leq j,k \leq n\Big\}.
\eeas
Note that $\Gamma_{\Z}(\EEE)\subset\Gamma_{\Z}(\EEE')\times\Gamma_{\Z}(\EEE'')$. Also, if  $I''\in \Gamma_{\Z}(\EEE'')$ then
\beas
\Lambda(I'')=\Big\{I'=(i_{1}, \ldots, i_{a})\in \Z^{a}:(I',I'')\in \Gamma_{\Z}(\EEE)\Big\}\subset\Gamma_{\Z}(\EEE').
\eeas

\begin{proposition}\label{Prop16.1wer}
Let  $R_{b+1}, \ldots, R_{a}$ be positive real numbers. Then
\bes
F= \sum_{I'\in \Gamma_{\Z}(\EEE')} 
\Bigg[\prod_{\substack{j=1}}^{b}2^{-\epsilon\left[i_{j}-e(j,\tau(j))i_{\tau(j)}\right]}\Bigg]
\Bigg[\prod_{\substack{j=b+1}}^{a}\min\Big\{(2^{i_{j}}R_{j})^{+\epsilon}, (2^{i_{j}}R_{j})^{-\epsilon}\Big\}\Bigg]
\ees
converges and is bounded by a constant $C$, depending on $\epsilon, \EEE$ but independent of  $R_{b+1}, \ldots, R_{a}$. Here it is understood that if $b=0$ the first product is empty, and if $0<b=a$ the second product is empty. \end{proposition}

\begin{proof}
We argue by induction on $n\geq 1$. When $n=1$ we have
\bes
F\leq \sum_{j\in\Z}\min\Big\{(2^{j}R)^{+\epsilon},(2^{j}R)^{-\epsilon}\Big\}.
\ees 
which converges and is independent of  $R$. Thus assume that the result is true for some $n-1\geq 1$. There are two possibilities. First suppose that $b=0$. Then 
\beas
F&= \sum_{I'\in \Gamma_{\Z}(\EEE')}\prod_{\substack{j=1}}^{a}\min\Big\{(2^{i_{j}}R_{j})^{+\epsilon}, (2^{i_{j}}R_{j})^{-\epsilon}\Big\}\\
&\leq
\sum_{(i_{1}, \ldots, i_{a})\in\Z^{a}}\prod_{k=1}^{a}\min\Big\{(2^{i_{k}}R_{k})^{+\epsilon},(2^{i_{k}}R_{k})^{-\epsilon}\Big\}\\
&=
\prod_{j=1}^{a}\Bigg[\sum_{i_{j}\in \Z}\min\Big\{(2^{i_{j}}R_{j})^{+\epsilon},(2^{i_{j}}R_{j})^{-\epsilon}\Big\}\Bigg]
\eeas
which is independent of $R_{1}, \ldots, R_{a}$. 

Next suppose that $b\geq 1$ so that the product $\prod_{\substack{j=1}}^{b}2^{-\epsilon\left[i_{j}-e(j,\tau(j))i_{\tau(j)}\right]}$ is non-empty. There are now two possibilities.

\smallskip

\noindent\textbf{Case 1:}

\medskip

Suppose there exists $j_{0}\in \{1, \ldots, b\}$ such that $j _{0}\neq \tau(k)$ for any $k\in \{1, \ldots, b\}$. Without loss of generality, assume $j_{0}=1$. Set
\beas
\widetilde \Gamma_{\Z}(\EEE')&=\Big\{\widetilde I'=(i_{2}, \ldots, i_{a})\in \Z^{a-1}: e(j,k)i_{k} \leq i_{j}, &&2\leq j,k \leq a\Big\}.
\eeas
Then $\Gamma_{\Z}(\EEE') \subset  \Big\{(i_{1},\widetilde I'):\text{$e(1,\tau(1))i_{\tau(1)} \leq i_{1}$ and $\widetilde I'\in \widetilde\Gamma_{\Z}(\EEE')$}\Big\}$.
Since the index $i_{1}$ does not appear in the product $\prod_{\substack{j=2}}^{b}2^{-\epsilon\left[i_{j}-e(j,\tau(j))i_{\tau(j)}\right]}$, the sum $F$ is dominated by
\beas
\sum_{\widetilde I'\in \widetilde\Gamma_{\Z}(\EEE')} 
\Bigg[\prod_{\substack{j=2}}^{b}2^{-\epsilon\left[i_{j}-e(j,\tau(j))i_{\tau(j)}\right]}\Bigg]&
\Bigg[\prod_{\substack{j=b+1}}^{a}\min\Big\{(2^{i_{j}}R_{j})^{+\epsilon}, (2^{i_{j}}R_{j})^{-\epsilon}\Big\}\Bigg]\\
&\times\Bigg[\sum_{\substack{i_{1}\in\Z\\e(1,\tau(1))i_{\tau(1)} \leq i_{1}}}2^{-\epsilon\left[i_{1}-e(1,\tau(1))i_{\tau(1)}\right]}\Bigg]
\eeas
But 
\bes
\sum_{\substack{i_{1}\in\Z\\e(1,\tau(1))i_{\tau(1)} \leq i_{1}}}2^{-\epsilon\left[i_{1}-e(1,\tau(1))i_{\tau(1)}\right]}\leq \Big[1-2^{-\epsilon e(1,\tau(1))}\Big]^{-1},
\ees
and so
\beas
F\leq 
\frac1{1-2^{-\epsilon e(1,\tau(1))}}\!\!\!
\sum_{\widetilde I'\in \widetilde\Gamma_{\Z}(\EEE')} 
\Bigg[\prod_{\substack{j=2}}^{b}2^{-\epsilon\left[i_{j}-e(j,\tau(j))i_{\tau(j)}\right]}\Bigg]&
\Bigg[\prod_{\substack{j=b+1}}^{a}\min\Big\{(2^{i_{j}}R_{j})^{+\epsilon}, (2^{i_{j}}R_{j})^{-\epsilon}\Big\}\Bigg].
\eeas
The result now follows by induction since in the remaining sum we have replaced $\{1, \ldots, n\}$ with $\{2, \ldots, n\}$.

\smallskip

\noindent\textbf{Case 2:}

\medskip

Now assume that each $j\in\{1, \ldots, b\}$ is equal to $\tau(k)$ for at least one $k\in \{1, \ldots, b\}$. Let $\sigma(j)\in \{1, \ldots, b\}$ be a choice of some pre-image of $j$. Then $\sigma:\{1, \ldots, b\} \to \{1, \ldots, b\}$, and since $\tau(\sigma(j))=j$, it follows that $\sigma$ is one-to-one. Thus $\sigma$ (and hence also $\tau$) is a permutation of $\{1, \ldots, b\}$. It follows that we can decompose $\{1, \ldots, b\}$ into a disjoint union of non-empty subsets $C_{1}, \ldots, C_{p}$ which are the minimal closed orbits of the mapping $\tau$. We can write $C_{l}=\{j_{l,1}, \ldots, j_{l,m_{l}}\}$ with $\tau(j_{l,k})=j_{l,k+1}$ for $1\leq k \leq m_{l}-1$ and $\tau(j_{l,k_{l}})=j_{l,1}$. Then for $1 \leq l \leq p$ let
\beas
\Gamma_{l}'=\Big\{(i_{l,1}, \ldots,i_{l,m_{l}})\in \Z^{m_{l}}:e(j_{l,k},\tau(j_{l,k}))i_{\tau(j_{l,k})}\leq i_{j_{l,k}}\,\,1\leq k \leq m_{l}\Big\}.
\eeas
Then $\Gamma_{\Z}(\EEE')\subset \Gamma_{1}'\times\cdots \times \Gamma_{p}$, and so
\beas
F\leq \prod_{l=1}^{p}\Big[\sum_{I_{l}\in\Gamma_{l}'}\prod_{k=1}^{m_{l}}2^{-\epsilon[i_{j_{l,k}}-e(j_{l,k},\tau(j_{l,k}))i_{\tau(j_{l,k})}]}\Big]
\prod_{k=a+1}^{b}\Big[\sum_{i_{k}\in \Z}\min\Big\{(2^{i_{k}}R_{k})^{+\epsilon},(2^{i_{k}}R_{k})^{-\epsilon}\Big\}\Big]
\eeas
However $\sum_{I_{l}\in\Gamma_{l}'}\prod_{k=1}^{m_{l}}2^{-\epsilon[i_{j_{l,k}}-e(j_{l,k},\tau(j_{l,k}))i_{\tau(j_{l,k})}]}$ is uniformly bounded by Proposition \ref{Prop3.61} and Proposition \ref{Prop16.3wer}. Also,  $\sum_{i_{k}\in \Z}\min\Big\{(2^{i_{k}}R_{k})^{+\epsilon},(2^{i_{k}}R_{k})^{-\epsilon}\Big\}$ is bounded independently of $R_{k}$. Thus $F$ is bounded independently of $i_{s+1}$ and $R_{a+1}, \ldots, R_{b}$.
This completes the proof.
\end{proof}

\begin{proposition}\label{Prop13.1}
Let $\EEE=\big\{e(j,k)\big\}$ be an $n\times n$ matrix satisfying (\ref{2.5}), and suppose that $1<e(j,k)e(k,j)$ for all $j\neq k$. Let $\e=\max\big\{e(j,k):1\leq j,k\leq n\big\}$. Let  $\alpha_{j}>0$, $M_{j}\geq 0$ for $1\leq j \leq n$. If $M\geq \sum_{j=1}^{n}(\e(\alpha_{j}+1)+M_{j})$  there is a constant $C=C(M_{1}, \ldots, M_{n},n, \e)$ so that for all $(A_{1}, \ldots, A_{n})$ with $A_{j}>0$,
\beas
\sum_{I\in \Gamma_{\Z}(\EEE)}\Big[\prod_{j=1}^{n}2^{-i_{j}\alpha_{j}}\Big]\Big[1+\sum_{k=1}^{n}2^{-i_{k}}A_{k}\Big]^{-M}\!\!\!\leq C\,
\prod_{j=1}^{n}\Big[A_{1}^{e(j,1)}+\cdots +A_{n}^{e(j,n)}\Big]^{-\alpha_{j}}\prod_{j=1}^{n}(1+A_{j})^{-M_{j}}.
\eeas
\end{proposition}

\begin{proof}
Let $N_{j}(A)=\sum_{k=1}^{n}A_{k}^{e(j,k)}$. If $I=(i_{1}, \ldots, i_{n})\in \Gamma_{\Z}(\EEE)$ then for any $1\leq j \leq n$ we have $-i_{j}<-e(j,k)i_{k}$ and so
\beas
1+2^{-i_{j}}N_{j}(A)&
\leq
1+\sum_{k=1}^{n}2^{-i_{k}e(j,k)}A_{k}^{e(j,k)}
\le
1+\sum_{k=1}^{n}\big[2^{-i_{k}}A_{k}\big]^{e(j,k)}
\lesssim
\Big[1+\sum_{k=1}^{n}2^{-i_{k}}A_{k}\Big]^{\e}.
\eeas
Therefore
\beas
\Big[1+\sum_{k=1}^{n}2^{-i_{k}}A_{k}\Big]^{-M}&\leq
\prod_{j=1}^{n}\Big[1+\sum_{k=1}^{n}2^{-i_{k}}A_{k}\Big]^{-\e (\alpha_{j}+1)}\prod_{j=1}^{n}\Big[1+A_{j}\Big]^{-M_{j}}\\
&\lesssim
\prod_{j=1}^{n}\Big[1+2^{-i_{j}}N_{j}(A)\Big]^{-(\alpha_{j}+1)}\prod_{j=1}^{n}\Big[1+A_{j}\Big]^{-M_{j}},
\eeas
and so
\beas
\sum_{I\in \Gamma_{\Z}(\EEE)}&\Big[\prod_{j=1}^{n}2^{-i_{j}\alpha_{j}}\Big]\Big[1+\sum_{k=1}^{n}2^{-i_{k}}A_{k}\Big]^{-M}\\
 &\lesssim\sum_{I\in \Gamma_{\Z}(\EEE)}\prod_{j=1}^{n}2^{-i_{j}\alpha_{j}}\Big[1+2^{-i_{j}}N_{j}(A)\Big]^{-(\alpha_{j}+1)}\prod_{j=1}^{n}\Big[1+A_{j}\Big]^{-M_{j}}\\
 &\lesssim
C\prod_{j=1}^{n}\Big[\sum_{i_{j}\leq 0}2^{-i_{j}\alpha_{j}}\Big[1+2^{-i_{j}}N_{j}(A)\Big]^{-\alpha_{j}-1}\Big]\prod_{j=1}^{n}\Big[1+A_{j}\Big]^{-M_{j}}\\
&\lesssim
\prod_{j=1}^{n}N_{j}(A)^{-\alpha_{j}}\,\prod_{j=1}^{n}\Big[1+A_{j}\Big]^{-M_{j}}.
\eeas
This completes the proof.
\end{proof}


\printindex

\end{document}